\documentclass{memo-l}

\usepackage{comment,graphicx,amsthm,amsmath,amssymb,bbm,epsfig,enumerate,nicefrac,color}
\allowdisplaybreaks

\usepackage{xcolor}
\usepackage[colorlinks,linkcolor=red,citecolor=green]{hyperref}

\newtheorem{theorem}{Theorem}[chapter]
\newtheorem{lemma}[theorem]{Lemma}
\newtheorem{prop}[theorem]{Proposition}
\newtheorem{corollary}[theorem]{Corollary}

\theoremstyle{definition}

\newtheorem{example}[theorem]{Example}

\theoremstyle{remark}
\newtheorem{remark}[theorem]{Remark}

\numberwithin{section}{chapter}
\numberwithin{equation}{chapter}

\newcommand{\smallsum}{\textstyle\sum}
\newcommand{\calB}{\mathcal B}
\newcommand{\calC}{\mathcal C}
\newcommand{\calF}{\mathcal F}
\newcommand{\calG}{\mathcal G}
\newcommand{\calH}{\mathcal H}
\newcommand{\calL}{\mathcal L}
\newcommand{\calM}{\mathcal M}

\newcommand{\invL}{{\mathpalette\mirrorL\relax}}
\newcommand{\mirrorL}[2]{\reflectbox{$#1L$}}

\renewcommand\theenumi{\roman{enumi}}

\newcommand{\inv}[1]{\frac{1}{#1}}
\newcommand{\tinv}[1]{\tfrac{1}{#1}}
\newcommand{\n}{\|}
\providecommand{\eps}{{\ensuremath{\varepsilon}}}

\newcommand{\C}{C}
\providecommand{\eps}{{\ensuremath{\varepsilon}}}

\providecommand{\N}{{\ensuremath{\mathbbm{N}}}}
\providecommand{\Z}{{\ensuremath{\mathbbm{Z}}}}
\providecommand{\R}{{\ensuremath{\mathbbm{R}}}}

\providecommand{\E}{{\ensuremath{\mathbb{E}}}}
\renewcommand{\P}{{\ensuremath{\mathbb{P}}}}
\providecommand{\1}{{\ensuremath{\mathbbm{1}}}}
\providecommand{\HS}{{\ensuremath{\textup{HS}}}}

\providecommand{\qed}{\hfill\mbox{$\Box $}}

\newcommand{\sgc}{}
\newcommand{\cgs}{}

\begin{document}

\frontmatter

\title{Local Lipschitz continuity in the\\ initial value
and strong completeness
for\\ nonlinear stochastic differential equations}

\author{Sonja Cox}
\address{
Korteweg-de Vries Instituut,
University of Amsterdam}
\curraddr{}
\email{s.g.cox@uva.nl}

\author{Martin Hutzenthaler}
\address{Faculty of Mathematics,
University of Duisburg-Essen}
\curraddr{}
\email{martin.hutzenthaler@uni-due.de }

\author{Arnulf Jentzen}
\address{School of Data Science and Shenzhen Research Institute of Big Data,
The Chinese University of Hong Kong, Shenzhen (CUHK-Shenzhen), China;
Applied Mathematics: Institute of Analysis and Numerics, 
University of M\"unster, Germany}
%
\curraddr{}
\email{ajentzen@cuhk.edu.cn, ajentzen@uni-muenster.de}

\date{}




\begin{abstract}
Recently, Hairer et~al.~\cite{hhj12} showed that
there exist stochastic differential equations (SDEs)
with infinitely often differentiable and globally bounded coefficient 
functions
whose solutions
fail to
be locally Lipschitz continuous
in the strong $L^p$-sense
with respect to the initial value
for every $p\in (0,\infty]$.
In this article we provide conditions on the coefficient functions
of the SDE and on $ p \in (0,\infty] $
that are sufficient for local Lipschitz continuity in the strong $L^p$-sense with respect 
to the initial value 
and we establish explicit estimates for the local Lipschitz continuity 
constants.
In particular, we prove local Lipschitz continuity in the initial value
for several nonlinear stochastic ordinary and stochastic partial 
differential equations in the 
literature such as
the stochastic van der Pol oscillator,
Brownian dynamics,
the Cox-Ingersoll-Ross processes and 
the Cahn-Hilliard-Cook equation.
As an application of our estimates, we obtain strong completeness for several 
nonlinear SDEs.
\end{abstract}

\maketitle

\tableofcontents

\mainmatter

\chapter{Introduction}

Let $ d, m \in \N $,
let $ O \subseteq \R^d $ be an open set,
let $ ( \Omega, \mathcal{F}, \P,
( \mathcal{F}_t )_{ t \in [0,\infty) }
) $
be a filtered probability space satisfying
the usual conditions (i.e., for all $t\in [0,\infty)$ it holds that $\{A \subseteq \Omega \colon \sgc{}(\exists\,\cgs{} B\in \mathcal{F} \sgc{}\colon\cgs{} A \subseteq B \text{ and }\P(B)=0\sgc{})\cgs{} \} \sgc{}\subseteq\cgs{} \mathcal{F}_t$ and $\mathcal{F}_t = \cap_{s\in (t,\infty)}\mathcal{F}_s$), 
let 
$
  W \colon [0,\infty) \times \Omega \to \R^m
$
be a standard $ ( \mathcal{F}_t )_{ t \in [0,\infty) } $-Brownian motion,
let $ \mu \colon O \to \R^d $
and $ \sigma \colon O \to \R^{ d \times m } $
be continuous functions
and let $ X^x \colon [0,\infty) \times \Omega \to O $,
$ x \in O $,
 be adapted stochastic processes with continuous sample paths 
 which solve the stochastic differential equation (SDE)
\begin{equation}
\label{eq:SDE}
   X_t^x=x+\int_0^t \mu(X_s^x)\,ds+\int_0^t\sigma(X_s^x)\,dW_s
\end{equation}
$\P$-a.s.\ for all $ t \in [0,\infty) $ and
all $ x \in O $.

An essential question in stochastic analysis is regularity of
solution processes of the SDE~\eqref{eq:SDE} in the initial value.
In this article, our main objective\sgc{} is to determine\cgs{} sufficient conditions
on $\mu$, $\sigma$ and
$t,p\in(0,\infty)$
to ensure that the function
$O\ni x\mapsto X_t^x\in L^p(\Omega;\R^d)$
is
\emph{locally Lipschitz continuous}.
In particular,
for \sgc{}every $t,p\in(0,\infty)$\cgs{}
our goal is to obtain a continuous function $\varphi_{t,p}\colon 
O^2\to[0,\infty)$
such that for all $ x, y \in O $ it holds that
\begin{equation}
\label{eq:goal}
  \| X^x_t - X^y_t \|_{ L^p( \Omega; \R^d ) }
  \leq
  \varphi_{t,p}(x,y)
  \left\| x - y \right\|
  .
\end{equation}
In addition, we want the values of the functions $\varphi_{t,p}$, 
$t,p\in(0,\infty)$, to be
as small and as explicit as possible.
A well-known sufficient condition for~\eqref{eq:goal} with $p=2$
is the {\it global monotonicity} assumption that there exists a $c\in\R$
such that for all $ x, y \in O $ it holds that
\begin{equation}\label{eq:global.monotonicity}
    \left\langle 
      x - y , \mu(x) - \mu(y)
    \right\rangle
  + 
    \tfrac{ 1 }{ 2 }
    \left\|
      \sigma(x) - \sigma(y)
    \right\|^2_{ \HS( \R^m, \R^d ) }
  \leq 
    c
    \left\| 
      x - y 
    \right\|^2 
    .
\end{equation}
If the monotonicity assumption is satisfied,  
inequality~\eqref{eq:goal} holds with $p=2$
and $\varphi_{t,2}(x,y)=e^{ct}$
for all $t\in(0,\infty)$\sgc{},\cgs{} $x,y\in O$
(see, e.g., Assumption~(H2) and Proposition~4.2.10 in 
Pr{\'e}v{\^o}t \&\R\"{o}ckner~\cite{PrevotRoeckner2007}).
Thus the global monotonicity property~\eqref{eq:global.monotonicity}
implies for all $ t \in [0,\infty) $ 
global Lipschitz continuity of the function
$
  O\ni x\mapsto X_t^x\in L^2(\Omega;\R^d)
$.
Unfortunately, the coefficient functions of a large number of nonlinear
SDEs arising from applications 
do not satisfy the global monotonicity assumption~\eqref{eq:global.monotonicity}
(see
Sections~\ref{sec:examples_SODE} and~\ref{sec:examples_SPDE} for a selection of 
examples, note that by `SDE' we mean both stochastic ordinary and stochastic 
partial differential equations).
It remained an open problem to find conditions on $\mu$, $\sigma$ and $p\in[2,\infty)$
which are satisfied by most of the nonlinear SDEs from applications and which ensure
for all $ t \in [0,\infty) $
local Lipschitz continuity of the function
$
  O \ni x\mapsto X_t^x\in L^p(\Omega;\R^d)
$.
In this article, we \sgc{}contribute to this problem of research\cgs{}.

In the deterministic case $\sigma\equiv0$,
the solution of~\eqref{eq:SDE} is always locally Lipschitz continuous in the initial value
if $\mu$ is locally Lipschitz continuous.
The stochastic case is more subtle than the deterministic case.
To emphasize the challenge of the stochastic case,
we consider the following example SDE.
In the special case $d=2$, $m=1$, $O=\R^d$
and $\mu(x_1,x_2)=(x_1x_2,-(x_1)^2)$ for all $(x_1,x_2)\in\R^2$
the SDE~\eqref{eq:SDE} reads as
\begin{equation}
\label{eq:circle.counterexample}
  d X_t^x
  =
  \left(
    \begin{array}{c}
      X^{ x, 1 }_t
      X^{ x, 2 }_t
    \\[1ex]
      - ( X^{ x, 1 }_t )^2
    \end{array}
  \right)
  dt
  + 
  \sigma( X_t^x ) \, dW_t
\end{equation}
for \sgc{} all $t \in [0,\infty)$ and all $x\in \R^2$\cgs{}.
Hairer et al.~\cite{hhj12}
prove that if $\sigma(x)=x$ for all $x\in\R^2$
in~\eqref{eq:circle.counterexample}, then
for any $t,p\in(0,\infty)$
the function
$\R^2\ni x\mapsto X_t^x\in L^p(\Omega;\R^2)$
is well-defined but not locally Lipschitz continuous.
In addition, Theorem 1.2 in
Hairer et al.~\cite{hhj12}
implies that this 
{\it loss of regularity phenomenon} can happen even in the
case of globally bounded and smooth coefficients.
In contrast,
Corollary~\ref{cor:UV3}
below implies 
for the SDE~\eqref{eq:circle.counterexample}
that if
$ \sigma $ in~\eqref{eq:circle.counterexample}
is globally bounded and globally Lipschitz continuous,
then 
for all $t,p\in(0,\infty)$
the function
$\R^2\ni x\mapsto X_t^x\in L^p(\Omega;\R^2)$
is locally Lipschitz continuous.
More generally, 
Corollary~\ref{cor:UV3}
ensures for the SDE~\eqref{eq:SDE}
that if $\mu$ is differentiable with
$\lim_{x\to\infty}\|\mu'(x)\|/\|x\|^2=0$
and 
$\limsup_{x\to\infty}\langle x,\mu(x)\rangle/\|x\|^2<\infty$
and if $\sigma$ is globally bounded and globally Lipschitz continuous,
then
for all $t,p\in(0,\infty)$
the function
$
  \R^d\ni x\mapsto X_t^x\in L^p(\Omega;\R^d)
$
is locally Lipschitz continuous.

It turns out that for some SDEs
such as Cox-Ingersoll-Ross processes (Subsection~\ref{ssec:cir})
or the Cahn-Hilliard-Cook equation
(Subsection~\ref{ssec:Cahn_Hilliard}),
it is appropriate to measure distance with a general
function
$V\in C^2(O^2,[0,\infty))$ rather than with 
the squared Euclidean distance $O^2\ni(x,y)\mapsto\|x-y\|^2\in[0,\infty)$.
Then It\^o's formula implies that
$dV(X_t^x,X_t^y)=
(\overline{\mathcal{G}}_{\mu,\sigma}V)(X_t^x,X_t^y)\,dt
+
(\overline{{G}}_{\sigma}V)(X_t^x,X_t^y)\,dW_t
$
for all $t\in[0,\infty)$, $x,y\in O$
where 
the linear operators
$  
  \overline{ \mathcal{G} }_{ \mu, \sigma }
  \colon
  C^2( O^2, \R )
  \to
  \C( O^2, \R )
$
(see (1.1) in Maslowski~\cite{Maslowski1986} and Ichikawa~\cite{Ichikawa1984})
and
$  
  \overline{ G }_{ \sigma }
  \colon
  C^2( O^2, \R )
  \to
  \C( O^2, \R^{ 1 \times m } )
$
are defined
by
\begin{align*}
&  ( \overline{ \mathcal{G} }_{ \mu, \sigma } \phi)( x, y )
  :=
  \big(
    \tfrac{ \partial }{ \partial x }
    \phi
  \big)(x,y) \, 
  \mu( x )
  +
  \big(
    \tfrac{ \partial }{ \partial y }
    \phi
  \big)(x,y) \, 
  \mu( y )
\\
\nonumber
  & \quad 
  +
  \tfrac{ 1 }{ 2 }
  \sum_{ i = 1 }^m
  \big(
    \tfrac{ \partial^2 }{ \partial x^2 }
    \phi
  \big)(x,y) 
  \big( 
    \sigma_i( x ) , 
    \sigma_i( x )
  \big)
  +
  \sum_{ i = 1 }^m
  \big(
    \tfrac{ \partial }{ \partial y }
    \tfrac{ \partial }{ \partial x }
    \phi
  \big)(x,y)\big( 
    \sigma_i( x ) , 
    \sigma_i( y )
  \big)
\\ 
  & \quad 
  +
  \tfrac{ 1 }{ 2 }
  \sum_{ i = 1 }^m
  \big(
    \tfrac{ \partial^2 }{ \partial y^2 }
    \phi
  \big)(x,y) 
  \big(
    \sigma_i( y ) , 
    \sigma_i( y )
  \big)
\\
& 
  ( \overline{ G }_{ \sigma } \phi)( x, y )
  :=
  \big(
    \tfrac{ \partial }{ \partial x }
    \phi
  \big)(x,y) \, 
  \sigma( x )
  +
  \big(
    \tfrac{ \partial }{ \partial y }
    \phi
  \big)(x,y) \, 
  \sigma( y )
\end{align*}
for all $ x, y \in O $\sgc{}, \cgs{}$ \phi \in C^2( O^2, \R ) $.
In terms of these operators, we formulate
the first main result of this article.
\begin{theorem}
\label{thm:regularity.introduction}
   Assume the above setting\sgc{},\cgs{} let 
   $t\in(0,\infty)$,
   $ \alpha_0, \alpha_1, \beta_0, \beta_1, c \in [0,\infty) $, 
   $V\in C^2(O^2,[0,\infty))$,
   $ U_0, U_1 \in C^2( O, [0,\infty) ) $,
   $ \overline{U} \in C( O, [0,\infty) ) $,
   $ r,p,q_0,q_1 \in ( 0, \infty ] $
   \sgc{}satisfy\cgs{}
   $
     \tfrac{ 1 }{ r } + \tfrac{ 1 }{ q_0 }  +  \tfrac{ 1 }{ q_1 } = \tfrac{ 1 }{ p }	
   $
   and
   $
     ( V^{ - 1 } )(0)
     \subseteq 
     (
       \overline{\mathcal{G}}_{ \mu,\sigma }V
     )^{ - 1 }(0)
     \cap
     (
       \overline{G}_{ \sigma }V
     )^{ - 1 }(0)
   $\sgc{},\cgs{} assume
   \begin{align}
   \label{eq:assumption.regularity.introduction}
   \tfrac{
      ( \overline{ \mathcal{G} }_{ \mu, \sigma } V)( x, y )
    }{
      V( x, y )
    }
    +
      \tfrac{
	( r - 1 ) \,
	\|
	    ( \overline{ G }_{ \sigma } V )( x, y )
	\|^2
      }{
	2 \,
	( 
	  V( x, y ) 
	)^2
      }
    & \leq   
      c 
      +
      \tfrac{
	  U_0( x ) + U_0( y )
      }{
	2 q_0 t e^{ \alpha_0 t }
      }
      +
      \tfrac{
	\overline{U}( x ) + \overline{U}( y )
      }{
	2 q_1 e^{ \alpha_1 t }
      },
\end{align}
\begin{align}
    \label{eq:condition.exponential.moments0}
    U_0'(x)\mu(x)
    +\tfrac{
    \operatorname{tr}\left(\sigma(x) \sigma(x)^{*}
     (\operatorname{Hess}U_0)(x)\right)
    }{2}
     +
     \tfrac{ 1 }{ 2 }
     \|
       \sigma( x )^* ( \nabla U_0 )( x )
     \|^2
   & \leq
     \alpha_0 U_0( x )
     + \beta_0,
\end{align}
\begin{equation}
\begin{split}
&    U_1'(x)\mu(x)
    +\tfrac{
    \operatorname{tr}\left(\sigma(x) \sigma(x)^{*}
     (\operatorname{Hess}U_1)(x)\right)
    }{2}
     +
     \tfrac{ 1 }{ 2 }
     \|
       \sigma( x )^* ( \nabla U_1 )( x )
     \|^2
     +
     \overline{U}( x )
\\ &  \leq
     \alpha_1 U_1( x )
     + \beta_1
   \label{eq:condition.exponential.moments1}
   \end{split}     \end{equation}
   for all 
   $ x,y\in O$
   with $ V(x,y)\neq 0 $\sgc{}, and assume\cgs{} 
   $\sup_{x,y\in K, V(x,y)\neq 0}\tfrac{\|(\overline{G}_{\sigma}V)(x,y)\|}{V(x,y)}<\infty$
   for all compact 
   sets
   $K \subseteq O$.
   Then
   \begin{equation}
   \begin{split}
   \label{eq:thm:local.Lipschitz}
   &
       \left\|
         V(X^x_{ t } , X^y_{ t })
       \right\|_{
         L^p( \Omega; \R )
       }
   \leq
       V(x,y)
       \exp\!\left(
         c t +
         \smallsum_{ i = 0 }^1
         \tfrac{
          2\beta_i t+ U_i( x ) + U_i(y) 
         }{
           2 q_i
         }
       \right)
   \end{split} \end{equation}
   for all $ x, y \in O $.
\end{theorem}
Theorem~\ref{thm:regularity.introduction}
follows immediately from the more general version
in Theorem~\ref{thm:UV} in Subsection~\ref{sec:marginal} below.
In addition,
the second main result of this article,
Theorem~\ref{thm:UV2} in Subsection~\ref{sec:uniform}
below,
establishes a sufficient condition on
$\mu$, $\sigma$ and $t,p\in(0,\infty)$ which ensures
that the function $O\ni x\mapsto X^x|_{[0,t]}\in L^p(\Omega;\C([0,t],\R^d))$
is locally Lipschitz continuous, that is, loosely speaking,
an estimate such as~\eqref{eq:goal} with supremum over time inside the $L^p$-norm.
Theorem~\ref{thm:regularity.introduction}
implies
local Lipschitz continuity
in the initial value for all examples in Sections~\ref{sec:examples_SODE}
and~\ref{sec:examples_SPDE}.
In many of these examples
the function $V$ in
Theorem~\ref{thm:regularity.introduction}
is the squared Euclidean distance, that is,
$V(x,y)=\|x-y\|^2$ for all $x,y\in O$.
In that case, in the setting of
Theorem~\ref{thm:regularity.introduction},
condition~\eqref{eq:assumption.regularity.introduction}
reads as
   \begin{equation}
   \label{eq:assumption.regularity.introduction.discussion}
   \begin{split}
&  \tfrac{ 
    2\langle
      x - y ,
      \mu( x ) - 
      \mu( y )
    \rangle
    +
    \| 
      \sigma( x ) - \sigma( y ) 
    \|^2_{ \HS( \R^m, \R^d ) }
  }{
    \|
      x - y
    \|^2
  }
  {
    +
    }
    \tfrac{
      2( r - 1 ) 
      \| 
        (
          \sigma( x ) - \sigma( y )
        )^*
        ( x-y )
      \|^2
    }{
      \| x - y \|^4
    }
\\ & 
   \leq   
     c 
     +
     \tfrac{
         U_0( x ) + U_0( y )
     }{
       2 q_0 t e^{ \alpha_0 t }
     }
     +
     \tfrac{
       \overline{U}( x ) + \overline{U}( y )
     }{
       2 q_1 e^{ \alpha_1 t }
     }
   \end{split}
   \end{equation}
   for all 
   $ \sgc{}x,y \in O\cgs{} $
   with $ x \neq y $
   (see Example~\ref{ex:phi_identity} and Corollary~\ref{cor:UV_squared_norm}).
   For Cox-Ingersoll-Ross processes in Subsection~\ref{ssec:cir}
   and for
   Wright-Fisher diffusions in Subsection~\ref{ssec:Wright.Fisher.diffusion},
   we choose $V$ such that $\overline{G}_\sigma V\equiv 0$.
   This
   considerably simplifies
   condition~\eqref{eq:assumption.regularity.introduction}
   with the cost that in a second step we need to derive a local Lipschitz estimate
   from inequality~\eqref{eq:thm:local.Lipschitz}.
   Also the function 
   $
     O^2 \ni (x,y) \mapsto
     \left\|x-y \right\|^2
     \left( 1 + \|x\|^q+\|y\|^q
     \right)
     \in[0,\infty)
   $
   for some $q\in[2,\infty)$ can be a good choice.
We note that in a number of our examples, we could not verify
the 
conditions~\eqref{eq:condition.exponential.moments0}
and
\eqref{eq:assumption.regularity.introduction.discussion}
with $\overline{U}\equiv 0$.
The key to many of our examples is either
inequality~\eqref{eq:assumption.regularity.introduction.discussion}
together with
condition~\eqref{eq:condition.exponential.moments1}
with $\overline{U}\not\equiv 0$
(see Subsections~\ref{ssec:stochastic.van.der.Pol.oscillator}, \ref{ssec:stochastic.Duffing.van.der.Pol.oscillator},
\ref{sec:overdamped_Langevin},
\ref{ssec:stochastic.Burgers.equation},
\ref{ssec:Cahn_Hilliard})
or to find a suitable function $ V $
which is not the squared Euclidean distance
(see Subsections~\ref{ssec:cir}, \ref{ssec:Wright.Fisher.diffusion} and 
\ref{ssec:stochastic.wave.equation}).

The method of proof of
Theorem~\ref{thm:regularity.introduction}
is to show under suitable assumptions 
(see Proposition~\ref{prop:two_solution} for details) 
that
\begin{equation}  \begin{split}
\label{eq:equation.norm.increment}
  V(X_t^x,X_t^y)
  &=V(x,y)\exp\!\left(
       \smallint_0^t\tfrac{
         (
           \overline{\mathcal{G}}_{\mu,\sigma}V
         )(X_s^x,X_s^y)
       }
       {V(X_s^x,X_s^y)}\,ds
       \right)
\\ & \quad \cdot
     \exp\!\left(
       \smallint_0^t\tfrac{
         ( 
           \overline{{G}}_{\sigma}V
         )(X_s^x,X_s^y)
       }{
         V( X_s^x,X_s^y)
       }
       \, dW_s
       -
       \smallint_0^t
       \tfrac{
         \|
           ( \overline{{G}}_{\sigma} V )(X_s^x,X_s^y)
         \|^2
       }{
         2 \,
         ( V(X_s^x,X_s^y) )^2
       }
       \,
       ds
     \right)
\end{split}     \end{equation}
$\P$-a.s.\ for all $x,y\in O$ where $\tfrac{0}{0}:=0$
and then to estimate the $L^p$-norm of the right-hand side for each $p\in(0,\infty)$.
Now due to condition~\eqref{eq:assumption.regularity.introduction},
it suffices to estimate exponential moments.
Exponential moments, in turn, are guaranteed by
conditions~\eqref{eq:condition.exponential.moments0}
and~\eqref{eq:condition.exponential.moments1}.
More precisely, 
Corollary~\ref{cor:exp_mom}
together with
conditions~\eqref{eq:condition.exponential.moments0}
and~\eqref{eq:condition.exponential.moments1}
implies that
\begin{equation}\begin{split}
\label{eq:exponential.moments.type0}
  \left\|
    \exp\!\left(
      \smallint_0^t 
      \tfrac{ 
        U_0(X_s^x) 
      }{
        2 q_0 t e^{ \alpha_0 s }
      }
      \, ds
    \right)
  \right\|_{
    L^{ 2 q_0 }( \Omega; \R ) 
  }
& \leq
  \tfrac{1}{t}\int_0^t
  \left\|
    \exp\!\left(
      \tfrac{U_0(X_s^x)}{2q_0e^{\alpha_0 s}}
    \right)
  \right\|_{L^{2q_0}(\Omega;\R)}\,ds
\\ &
  \leq 
    \exp\!\left(
      \tfrac{
        \beta_0 t+ U_0( x )
      }{
        2 q_0
      }
    \right)
\end{split}\end{equation}
and
\begin{align}
  \left\|
    \exp\!\left(
      \tfrac{U_1(X_t^x)}{2q_1e^{\alpha_1 t}}+
      \smallint_0^t \tfrac{\overline{U}(X_s^x)}{2q_1e^{\alpha_1 s}}\,ds
    \right)
  \right\|_{L^{2q_1}(\Omega;\R)}
  &
  \leq \exp\!\left(
         \tfrac{
          \beta_1 t+ U_1( x )
         }{
           2 q_1
         }
       \right)
\label{eq:exponential.moments.type1}
\end{align}
for all $x\in O$.
Note that the final factor on the right-hand
side of~\eqref{eq:equation.norm.increment}
is a local exponential martingale
so that in the case that $V\sgc{}\colon\cgs{}O^2\rightarrow \R^2$
is the squared Euclidean distance, 
the global monotonicity 
assumption~\eqref{eq:global.monotonicity}
immediately implies for all $x,y\in O$
that 
$
  \|X_t^x-X_t^y\|_{L^2(\Omega;\R)}
  \leq
  \left\|x-y \right\|
  e^{ct}
$.

There are a number of related results in the literature.
The idea of a general function for measuring distance was 
used
in Theorem 1.2 in Maslowski~\cite{Maslowski1986} 
(cf.\ also Ichikawa~\cite{Ichikawa1984} and, e.g., also Leha \&\ Ritter~\cite{LehaRitter1994, LehaRitter2003})
for studying
long-time stability properties of SDEs under the assumption
$
  (\overline{\mathcal{G}}_{\mu,\sigma}V)(x,y) 
  \leq 
  0
$ 
for all 
$
  x, y \in O 
$.
The relation~\eqref{eq:equation.norm.increment}
with $V$ being the squared Euclidean distance
appeared
in (14) in Zhang~\cite{Zhang2010}
and
on page 311 in Taniguchi~\cite{Taniguchi1989}
and in (14) in Li~\cite{Li1994}
in terms of the derivative in probability
(see Definition 4.9 in
Krylov~\cite{Krylov1999}) of  the function
$O\ni x\mapsto X^x \in L^0(\Omega;\R^d)$.
Building on Taniguchi's equation for the squared norm of the derivative process,
Theorem 5.1 and Theorem 3.1
in Li~\cite{Li1994}
proves a conditional result
which implies 
that if $ O $ is a complete connected Riemannian manifold,
if there exists an
$ x \in O $ such that 
$
  \P\big[
    \,
    \forall \, t \in [0, \infty) 
    \colon
    X_t^x \in O
    \,
  \big] = 1
$,
if $ \mu $ and $ \sigma $ 
are twice continuously differentiable
and if there exists a measurable function 
$
  f \colon O \to [0,\infty)
$
and a 
$
  p \in (0, \infty) 
$
such that
for all $ x \in O $, $ v \in \R^d \backslash \{ 0 \} $
it holds that
\begin{equation*}
 \begin{split}
  2
  \langle (\nabla\mu)( x ) v , v \rangle
  +
  \sum_{ i = 1 }^m
  \left\| 
    \sigma_i'( x )( v )
  \right\|^2
  +
  ( p - 2 )
  \sum_{ i = 1 }^m 
  \| v \|^{ - 2 }
  \left| \langle 
    \sigma_i'( x )( v ), v 
  \rangle \right|^2
  & \leq 6 p f(x) \| v \|^2
\\
\text{and }\quad 
  \sum_{ i = 1 }^m
  \| \sigma_i'(x) 
  \|_{
    L(\R^d,\R^d)
  }^2
  & \leq f(x)
\end{split}
\end{equation*}
and such that
for all $ t \in ( 0, \infty ) $
and all compact sets $ K \subseteq O $ 
it holds that
\begin{equation*}
  \sup_{ x \in K }
  \E\big[
    \exp\!\big(
      6 p^2 
      \int_0^t 
        f( X_s^x )
        \,
        \1_{
          \cap_{ r \in [0, s] }
          \{
            X_r^x \in O
          \}
        }
        \,
      ds
    \big)
  \big]
  < \infty,
\end{equation*}
then
for all $ t \in ( 0, \infty ) $
the function
$
  O \ni x \mapsto X^x 
  \in 
  L^p( \Omega; C( [0,t],O) )
$
is
locally Lipschitz continuous
(cf.\ also 
Corollary~\ref{cor:two_solution_supinside} below).
In addition,
Lemma 6.1
in Li~\cite{Li1994}
derives inequality~\eqref{eq:exponential.moments.type0}
from inequality~\eqref{eq:condition.exponential.moments0}
in the case $\alpha_0=0$ and $O=\R^d$.
Moreover, 
Theorem 6.2 in
Li~\cite{Li1994}
proves
inequality~\eqref{eq:condition.exponential.moments0}
with $\alpha_0=0$ and with $U_0(x)=\ln(1+\|x\|^2)$ for all $x\in\R^d$
under a global log-Lipschitz type condition
(see Li~\cite{Li1994} for details).
In addition, Corollary~6.3 in Li~\cite{Li1994}
and Theorem 1.7 in Fang, Imkeller \sgc{}\&\cgs{} Zhang~\cite{FangImkellerZhang2007}
prove local Lipschitz continuity results under
appropriate at most quadratic growth assumptions on 
$\mu$, $\mu'$, $\sigma$ and $\sigma'$
(see Corollary~\ref{cor:UV3} below for details).
Furthermore, Lemma 2.3 in Zhang~\cite{Zhang2010}
implies that if $ O = \R^d $, 
if $ c \in (0,\infty)$  is a real number 
and if
$
  U_0 \in \C^2(\R^d,[1,\infty))
$ 
is a function such that
for all $ x, y \in \R^d $
it holds that
\begin{equation*}
 \begin{split}
  U_0'(x)\mu(x)
    +\tfrac{
    \operatorname{tr}\left(\sigma(x) (\sigma(x))^{*}
     (\operatorname{Hess}U_0)(x)\right)
    }{2}
  & 
    \leq c \, U_0(x)
    ,
\\
  \|
    \sigma( x )^* ( \nabla U_0 )( x )
  \|^2
 &
  \leq 
   c \, U_0(x)
   ,
\\
  \langle 
    x - y,
    \mu(x)-\mu(y)
  \rangle
 & 
  \leq 
  c
  \left(
    U_0(x) + U_0(y)
  \right)
  \|x-y\|^2
\\ 
\text{and} \quad
  \|\sigma(x)-\sigma(y)\|_{\HS(\R^m,\R^d)}^2
 & 
  \leq 
  c\left(U_0(x)+U_0(y)\right)\|x-y\|^2
,
 \end{split}
\end{equation*}
then
for all 
$ p \in [2,\infty) $ 
there exists a 
$ t \in (0,\infty) $
such that the function
$
  \R^d
  \ni x
  \mapsto X^x|_{ [0,t] }
  \in 
  L^p( 
    \Omega; C( [0,t]; \R^d ) 
  )
$
is locally Lipschitz continuous.
In particular, Lemma 2.3 in Zhang~\cite{Zhang2010}
yields local Lipschitz continuity in the initial value for 
sufficiently small positive time points
for the stochastic van der Pol oscillator in the case of globally bounded noise
(Subsection~\ref{ssec:stochastic.van.der.Pol.oscillator}),
for the stochastic Duffing oscillator with a globally Lipschitz continuous 
diffusion \sgc{}coefficient\cgs{}
(see Subsection~\ref{ssec:stochastic.Duffing.van.der.Pol.oscillator}),
for the stochastic Lorenz equation with additive noise 
(see Subsection~\ref{ssec:stochastic.Lorenz.equation}),
for the Langevin dynamics under certain assumptions
(see Subsection~\ref{ssec:Langevin.dynamics}),
for a model from experimental psychology
(see Subsection~\ref{ssec:experimental.psychology})
and for the stochastic Brusselator under certain assumptions
(see Subsection~\ref{ssec:Brusselator}).
Moreover, local Lipschitz continuity in the initial value in the $ L^p $-norm 
for any $ p \in (0,\infty) $
and any time point
for the vorticity formulation of the two-dimensional stochastic Navier-Stokes equations
follows from Lemma 4.10
in Hairer \&\ Mattingly~\cite{HairerMattingly2006}.
Lemma 4.10 in Hairer \&\ Mattingly~\cite{HairerMattingly2006}
also includes an
inequality similar to~\eqref{eq:exponential.moments.type1}
in the case where $ U_1 $ is the squared Hilbert space norm
for the vorticity formulation of the two-dimensional stochastic Navier-Stokes equations.
Further instructive results 
on exponential moments
can
be found, for example, in
\cite{
BouRabeeHairer2013,
EsSarhirStannat2010,
FangImkellerZhang2007,
HieberStannat2013}.
Theorem~\ref{thm:regularity.introduction}
in this article
implies
local Lipschitz continuity
in the initial value for all $t,p\in(0,\infty)$
for the stochastic van der Pol oscillator with unbounded noise
(Subsection~\ref{ssec:stochastic.van.der.Pol.oscillator}),
for the stochastic Duffing-van der Pol oscillator with unbounded noise
(Subsection~\ref{ssec:stochastic.Duffing.van.der.Pol.oscillator}),
for the over-damped Langevin dynamics under certain assumptions
(Subsection~\ref{ssec:overdamped.Langevin.dynamics}),
for the stochastic SIR model
(Subsection~\ref{ssec:stochastic.SIR.model}),
for 
Cox-Ingersoll-Ross processes, for the Ait-Sahalia interest rate model, for Heston's $ 3 / 2 $-volatility,
for constant elasticity of variance processes
(Subsection~\ref{ssec:cir}),
for Wright-Fisher diffusions
(Subsection~\ref{ssec:Wright.Fisher.diffusion}),
for the stochastic Burgers equation with a globally bounded diffusion coefficient
(Subsection~\ref{ssec:stochastic.Burgers.equation}),
for the Cahn-Hilliard-Cook equation
(Subsection~\ref{ssec:Cahn_Hilliard})
and for Galerkin approximations of 
a non-linear wave equation
(Subsection~\ref{ssec:stochastic.wave.equation}).
Note that in the case of stochastic partial differential equations (SPDEs),
we first apply Theorem~\ref{thm:regularity.introduction} 
to spatial discretizations of the considered SPDE 
and then let the dimension of the discretization approach infinity.

We sketch three applications of Theorem~\ref{thm:regularity.introduction}.
First,
Theorem~\ref{thm:regularity.introduction}
and its uniform counterpart in 
Theorem~\ref{thm:UV2} 
can be applied
to establish
\emph{strong completeness}
of the SDE~\eqref{eq:SDE}.
More precisely, 
we show
in 
Theorem~\ref{thm:strong_completeness_uniform}
in Section~\ref{sec:strong_completeness}
that 
if there exist a $ p \in (d , \infty ) $ and an $ \eps\in(0,1) $ such that
$
  \bar{ O } 
  \ni
  x
  \mapsto 
  ( X^x_t )_{ t \in [0, \varepsilon] }
  \in 
  L^p(   
    \Omega; C( [0, \varepsilon] , \bar{ O } ) 
  )
$
is \emph{locally Lipschitz continuous},
then
the SDE~\eqref{eq:SDE} is \emph{strongly complete}
(we assume here that $ X $, $ \mu $ and $ \sigma $ are extended 
continuously to $ \bar{ O } $ 
in
an appropriate way;
see Theorem~\ref{thm:strong_completeness_uniform} for the precise assumptions
and, e.g., Subsection~\ref{ssec:stochastic.SIR.model} 
for the application of Theorem~\ref{thm:strong_completeness_uniform} to an SDE on a domain
which is not equal to $ \R^d $)
and, thus,
there exists a function 
$
  Z \colon \Omega \to C([0,\infty)\times \bar{ O } , \bar{ O } )
$
such that $ Z^x $, $ x \in \bar{ O } $, solve~\eqref{eq:SDE}.
Then combining Theorem~\ref{thm:strong_completeness_uniform}
with
Theorem~\ref{thm:UV2} 
yields strong completeness for all examples in Section~\ref{sec:examples_SODE};
see Section~\ref{sec:examples_SODE} for the precise assumptions.
We emphasize that strong completeness may fail to hold even
in the case of smooth and globally bounded coefficient functions;
see Li \&\ Scheutzow~\cite{LiScheutzow2011}.

Secondly,
a local Lipschitz estimate such as~\eqref{eq:goal}
is an important tool
for proving strong and weak convergence rates 
of \emph{numerical approximations} to the solution processes of the SDE~\eqref{eq:SDE}.
In the literature strong and weak convergence
rates for temporally discrete or spatially discrete
numerical approximation processes for
multi-dimensional SDEs are 
generally (except for in the case of D\"{o}rsek's insightful 
work~\cite{Doersek2012};
see Corollary~3.2 in~\cite{Doersek2012}) 
only known under the global 
monotonicity condition~\eqref{eq:global.monotonicity};
see, e.g.,~\cite{h96, hms02, Liu2003, GyoengyMillet2009,
HutzenthalerJentzenKloeden2012,MaoSzpruch2013Rate,KloedenNeuenkirch2013, 
Sabanis2016, SauerStannat2015}
and the references mentioned therein.
Inequality~\eqref{eq:goal} and Theorem~\ref{thm:regularity.introduction} 
now allow us to establish strong and weak convergence rates for numerical 
approximation
processes of SDEs which fail to satisfy the global monotonicity 
condition~\eqref{eq:global.monotonicity}.
Heuristically, the argument is as follows. 
Let $ T \in ( 0, \infty ) $\sgc{}, \cgs{}let 
$ \hat{X}^{ s, x } \colon [s, T] \times \Omega \to O $,
$ s \in [0,T] $, $ x \in O $,
be solution processes of
$ 
  d \hat{X}^{ s, x }_t 
= 
  \mu( \hat{ X }^{ s, x }_t ) \, dt + 
  \sigma( \hat{ X }^{ s, x }_t ) \, dW_t 
$
and
$
  \hat{X}_{s}^{ s, x } = x
$
for 
$
  t \in [s,T]
$,
$ 
  s \in [0,T] 
$, \sgc{}$ x \in O $\cgs{}
and for every $ h \in ( 0, T] $
let 
$ 
  Y^{x,h} \colon [0,T] \times \Omega \to \R^d 
$,
$ x \in O $,
be a family of one-step numerical approximation stochastic processes
for the SDE~\eqref{eq:SDE}
with step size $ h \in ( 0, T] $
(cf., e.g., Section~2.1.4 in~\cite{HutzenthalerJentzen2014Memoires}).
Heuristically speaking, the exact solution is the best approximation process
so that for estimating 
the quantity
$
  \| X_T^{ x } - Y_T^{ x, h } \|_{ L^2( \Omega; \R^d ) }
$
for $ x \in O $ and small $ h\in(0,T]$
we need
to
estimate
at least
the quantity
\begin{equation*}
  \big\| 
    X_T^{ x } 
    - 
    \hat{ X }^{ h , Y^{x, h}_h }_T 
  \big\|_{
    L^2( \Omega; \R^d )
  }
  =
  \big\| 
    \hat{X}_T^{ h, X_{h}^x } 
    - 
    \hat{ X }^{ h , Y^{x, h}_h }_T
  \big\|_{
    L^2( \Omega; \R^d )
  }
\end{equation*}
for $ x\in O$
and small $h\in(0,T]$.
This can be done  with 
a local Lipschitz estimate such as~\eqref{eq:goal}
together with estimates on
the one-step approximation errors
$
  \|
    X^{ x }_h - 
    Y^{ h, x }_h
  \|_{
    L^p( \Omega; \R^d )
  }
$,  
$
  x \in O
$, 
$
  h \in (0,T]
$,
$ p \in (2,\infty) $,
and together with 
suitable a priori estimates on the approximation processes
(see, e.g.,~Section~2 in~\cite{HutzenthalerJentzen2014Memoires}).
The detailed analysis of strong and weak convergence rates for numerical 
approximation processes
based on~\eqref{eq:goal} and Theorem~\ref{thm:regularity.introduction} 
will be the subject of future work\sgc{}; see, 
e.g.,\cgs{}~\cite{HutzenthalerJentzen2014PerturbationArxiv}.

A third application of Theorem~\ref{thm:regularity.introduction}
is to obtain moment bounds on the derivative process.
If the coefficient functions $\mu$ and $\sigma$ are continuously
differentiable, then Theorem 4.10 in
Krylov~\cite{Krylov1999}
shows that there exist stochastic processes 
$
  D^x \colon 
  [0,\infty) \times \Omega \to \R^{ d \times d } 
$, 
$
  x \in O
$,
such that for all 
$ y \in O $, $ t \in [0, \infty ) $
it holds that
\begin{equation*}
  \sup_{
    s \in [0,t] 
  }
  \|
    X_s^{x} - X_s^{y} - D_s^y 
    \, ( x - y )
  \| /
  \|x-y\|
  +
  \sup_{ s \in [0,t] }
  \| D^x_s - D^y_s \|_{
    L( \R^d )
  }
  \to 0
\end{equation*} 
in probability as $ y \neq x \to y $.
Thus, if
$ \mu $ and $ \sigma $ are continuously differentiable 
and if inequality~\eqref{eq:goal} holds,
then dividing by 
$ \| x - y \| \in ( 0, \infty ) $
in~\eqref{eq:goal}
and applying Fatou's lemma
(cf., e.g., Lemma~3.10 in~\cite{HutzenthalerJentzen2014Memoires})
immediately implies that
\begin{equation}
  \|
    D_t^y
  \|_{
    L^p( \Omega; \R^{ d \times d } )
  }
  \leq
  d \,
  \sup_{ 
    \substack{
      v \in R^d ,
      \| v \| \leq 1
    }
  }
  \|
    D_t^y v
  \|_{
    L^p( \Omega; \R^d )
  }
  \leq 
  d \,
  \varphi_{ t, p }( y, y )
\end{equation}
for all $y\in O$ and with $t,p\in(0,\infty)$ 
as in
inequality~\eqref{eq:goal}.

\section{Notation}

Throughout this article we often calculate and formulate
expressions in the extended real numbers 
$ [ - \infty, \infty ] = \R \cup \{ -\infty, \infty \} $.
In particular, 
we frequently use the conventions
\sgc{}$\sup(\emptyset)=-\infty$,\cgs{}
$
  \frac{ 0 }{ 0 } 
  =
  0 \cdot \infty
  = 0
$,
$0^0=1$,
$
  \frac{ a }{ 0 } = \infty
$,
$
  \frac{ - a }{ 0 } = - \infty
$,
$
  0^{ - a } = 
  \frac{ 1 }{ 0^{ a } } 
  = 
  \infty
$,
$
  \frac{ b }{ \infty } = 0^a= 0
$,
$
    \infty^{a} = \infty
$
for all $ a \in ( 0,\infty) $\sgc{}, \cgs{}$ b \in \R $
.

Throughout this article we also use the following notation. 
For every set $A$ we denote the power set of $A$ by $\mathcal{P}(A):=\{ B\subseteq A\}$.
We define 
$
  x \wedge y := \min(x,y)
$ 
and 
$
  x \vee y
  := \max(x,y)
$ 
for all
$
  x, y \in [-\infty,\infty]
$.
For all $ d, m \in \N := \{ 1, 2, \dots \} $, 
$ v = ( v_1, \dots, v_d ) \in [-\infty,\infty]^d $
and all
$
  A \in [-\infty,\infty]^{ d \times m } 
$
we denote 
the Euclidean norm of $ v $ by
$
  \| v \| 
  := 
  \left[
    | v_1 |^2
    +
    \ldots
    +
    | v_d |^2
  \right]^{ \frac{1}{2} }\in [0,\infty]
$
and the Euclidean operator norm of $ A $
by
$
  \| A \|
  :=
  \| A \|_{ L( \R^m, \R^d ) }
  = 
  \sup_{
    v \in \R^m \backslash \{ 0 \}
  }
  \big[
    \| A v \| / \| v \| 
  \big]\in [0,\infty]
$
and we denote by $A^*$ the transposed matrix of $A$.
For all $d\in \N$, $r\in (0,\infty)$, $x\in \R^d$ we denote the open ball with radius $r$ around $x$
by $B_{r}(x):=\{ y\in \R^d\colon \| x-y \|<r \}$.
For all \sgc{}$d\in \N$\cgs{} and all non-empty $A,B \subseteq \R^d$
we denote the distance of $A$ to $B$ by $\textup{dist}(A,B) := \inf_{y\in A, z\in B} \| y - z \| \in [0,\infty]$\sgc{}. 
For all $d \in \N$, $x\in \R$ and all non-empty $B \subseteq \R^d$
we denote the distance of $x$ to $B$ by $\textup{dist}(x,B) :=\textup{dist}(\{x\},B)$.\cgs{}
For every $\R$-Banach space $(E,\left\| \cdot \right\|_E)$ and all $F\subseteq E$ 
we denote the closure of $F$ in $(E,\left\| \cdot \right\|_E)$ by 
$\overline{F}^{\left\| \cdot \right\|_{E}}:= \cap_{G\subseteq E \textnormal{ closed},\, F \subseteq G}\, G$
and we denote the convex hull of $F$ by $\textup{conv}(F):=\cap_{G\subseteq E \text{ convex},\, F \subseteq G}\, G$.
For all $d\in \N$, $F\subseteq \R^d$ we denote the closure of $F$ in $\R^d$ by 
$\overline{F}:=\overline{F}^{\left\| \cdot \right\|}$.
For all $d\in \N$, $D\subseteq \R^d$ we let $\calB(D)\subseteq \mathcal{P}(\R^d)$ denote the $\sigma$-algebra 
which satisfies $\calB(D):= \sigma(\{ A \cap D \colon A \subseteq \R^d \text{ is open}\})$ and we let $\calB(C([0,\infty),\R^d))\subseteq \mathcal{P}(C([0,\infty),\R^d))$ denote the $\sigma$-algebra which satisfies 
$
    \calB(C([0,\infty),\R^d)) 
    := 
    \sigma(\{ 
        \{  
            f \in C([0,\infty),\R^d) 
            \colon 
            f(t) \in B
        \} 
        \colon  
        t \in [0,\infty), 
        B \in \calB(\R^d)
        \})
$.
For every measurable space $(S,\Sigma)$ and every $x\in S$ we denote the Dirac measure in $x$ by $\delta_x$.
For every two sets
$ A $
and
$ B $
we denote by
$ \mathcal{M}( A, B ) $
the set of all functions from
$ A $ to $ B $.
For all two sets $ A, B $ 
and every function $ f \colon A \to B $ we denote \sgc{}the image of $f$\cgs{} by
$ 
  \textup{im}(f) := \{ f(x) \in B \colon x\in A \} 
$ 
. 
For every two measurable spaces
$ ( A, \mathcal{A} ) $
and
$ ( B, \mathcal{B} ) $
we denote by
$ \mathcal{L}^0( A ; B ) $
the set of all $ \mathcal{A} $/$ \mathcal{B} $-measurable functions 
from $ A $ to $ B $.
For every measure space $(S,\Sigma,\mu)$, every separable
$\R$-Banach space $(E,\left\|\cdot \right\|_E)$
and every $f\in \calL^0(S;E)$ we denote the $\mu$-equivalence class of $f$ by
$
[f]_{\mu}:= \{ g \in \calL^0(S;E) \colon \mu\big[ g \neq f \big] = 0 \} 
$
and we denote the $[0]_{\mu}$-quotient space of $\calL^0(S;E)$ by $L^0(S;E) := \calL^0(S;E)/[0]_{\mu}$.
For all $p\in [1,\infty)$, every measure space $(S,\Sigma,\mu)$ and every separable $\R$-Banach space $(E,\left\|\cdot \right\|_E)$ 
we define $\left\| \cdot \right\|_{L^p(S;E)}\colon \calL^0(S;E) \rightarrow [0,\infty]$ by
$
 \| f \|_{L^p(S;E)}
=
 \left(
    \int_{S} 
        \| f(s) \|_E^p
    \,\mu(ds)
 \right)^{1/p}
$
for all $f\in \calL^0(S;E)$ and we define $\calL^p(S;E)\subseteq \calL^0(S;E)$
and $L^p(S;E)\subseteq L^0(S;E)$ by 
$
 \calL^p(S;E)
 =
 \{
    g \in \calL^0(S;E)
    \colon
    \| g \|_{L^p(S;E)} < \infty
 \} 
 $
 and 
 $
 L^p(S;E)
  =
 \calL^p(S;E)/[0]_{\mu}
 $.
For all $d\in \N$, $\alpha \in (0,1]$, $D\subseteq \R^d$ and every $\R$-Banach space $(E,\left\|\cdot\right\|_{E})$
we define 
$ \left\| \cdot \right\|_{
    \calC^{\alpha}( D, E) 
  } 
  \colon 
  \calM(D,E) 
  \rightarrow 
  [0,\infty] $
  and 
  $ \left\| \cdot \right\|_{
    C^{\alpha}( D, E )
  } 
  \colon 
  \calM(D,E) 
  \rightarrow 
  [0,\infty] $
  by
\begin{align}
 \left\| 
    f
 \right\|_{
    \calC^{\alpha}(
        D, E 
    )
 }
 & =
 0
 \vee
 \sup_{x,y\in D, x\neq y}
    \frac{ 
        \| f(x) - f(y) \|_{E}
    }{
        \| x - y \|^{\alpha}
    }
 \quad\text{and}
 \\    
 \left\| 
    f
 \right\|_{
    C^{\alpha}(
        D, E 
    )
 }
 & =
 \left(
    0\vee 
    \sup_{x\in D} 
        \| f(x) \|_{E}
\right)
 + 
 \| f \|_{\calC^{\alpha}(D, E)}
\end{align}
for all $f\in \calM(D,E)$ and we define 
$C^{\alpha}(D,E)\subseteq \calM(D,E)$ by 
$
 C^{\alpha}(D, E) 
  =
 \{ 
    g \in \calM(D,E)
    \colon 
    \| g \|_{C^{\alpha}(D,E)} < \infty
 \}
 $.
For every two normed vector spaces
$ ( V_1, \left\| \cdot \right\|_{ V_1 } ) $
and
$ ( V_2, \left\| \cdot \right\|_{ V_2 } ) $
with $V_1\neq \{0\}$
and every function 
$ f \colon V_1 \to V_2 $
from $ V_1 $ to $ V_2 $
we define
$
  \| f \|_{
    \operatorname{Lip}( V_1, V_2 )
  }
  :=
  \sup_{ 
    v, w \in V_1, v \neq w 
  }
  \frac{
    \| f( v ) - f( w ) \|_{ V_2 }
  }{
    \| v - w \|_{ V_1 }
  }
$.
For all
$ d,\,m \in \N $,
$ T\in(0,\infty) $,
for every filtered probability space
$
  ( \Omega, \mathcal{F}, \P, ( \mathcal{F}_t )_{ t \in [0 , T] } ) 
$
satisfying the usual conditions, 
every standard 
$ ( \mathcal{F}_t )_{ t \in [0,T] } $-Brownian motion
$ W \colon [0,T] \times \Omega \to \R^m $ 
and every adapted and product measurable stochastic
process
$ 
  X \colon 
  [0,T] \times \Omega \to 
  [-\infty,\infty]^{ d \times m }	
$
with
$
  \int_0^T \| X_s \|^2 \, ds
  < \infty
$
$ \P $-a.s.\ we define
$
  \int_0^T X_s \, dW_s
  \in L^{0}( \Omega; \R^d )
$
by
$
  \int_0^T X_s \, dW_s
  :=
  \int_0^T \mathbbm{1}_{
    \{
      X_s \in \R^{ d \times m }
    \}
  } 
  X_s \, dW_s
$
$ \P $-a.s.\ (see~\cite{WeizsackerWinkler:1990}).
For all $ d,\, m \in \N $,
every open set $ O \subseteq\R^d $
and every two functions
$
  \mu \colon O \to \R^d
$
and
$
  \sigma \colon O \to \R^{ d \times m }
$
we define linear operators
$  
  \mathcal{G}_{ \mu, \sigma }
  \colon
  C^2( O, \R )
  \to
  \mathcal{M}( O, \R )
$,
$  
  G_{ \sigma }
  \colon
  C^1( O, \R )
  \to
  \mathcal{M}( O, \R^{ 1 \times m } )
$,
$  
  \overline{ \mathcal{G} }_{ \mu, \sigma }
  \colon
  C^2( O^2, \R )
  \to
  \mathcal{M}( O^2, \R )
$
and
$  
  \overline{ G }_{ \sigma }
  \colon
  C^1( O^2, \R )
  \to
  \mathcal{M}( O^2, \R^{ 1 \times m } )
$
by
\begin{align}
\label{eq:drift_operator}
&  
  ( \mathcal{G}_{ \mu, \sigma } \phi)( x )
  :=
  \phi'(x) \,
  \mu( x )
  +
  \tfrac{ 1 }{ 2 }
  \operatorname{tr}\!\big(
    \sigma(x) \sigma(x)^* 
    ( \operatorname{Hess} \phi )( x )
  \big)
  ,
\\
\label{eq:diffusion_operator}
&
  ( G_{ \sigma } \psi)( x )
  :=
  \psi'(x) \, \sigma(x) ,
\\
&
  ( \overline{ \mathcal{G} }_{ \mu, \sigma } \Phi)( x, y )
  :=
  \big(
    \tfrac{ \partial }{ \partial x }
    \Phi
  \big)(x,y) \, 
  \mu( x )
  +
  \big(
    \tfrac{ \partial }{ \partial y }
    \Phi
  \big)(x,y) \, 
  \mu( y ) 
\label{eq:extended_drift}
\\ 
\nonumber
  & \quad
  +
  \tfrac{ 1 }{ 2 }
  \sum_{ i = 1 }^m
  \big(
    \tfrac{ \partial^2 }{ \partial x^2 }
    \Phi
  \big)(x,y) 
  \big( 
    \sigma_i( x ) , 
    \sigma_i( x )
  \big)
  +
  \sum_{ i = 1 }^m
  \big(
    \tfrac{ \partial }{ \partial y }
    \tfrac{ \partial }{ \partial x }
    \Phi
  \big)(x,y)\big( 
    \sigma_i( x ) , 
    \sigma_i( y )
  \big)
\\ \nonumber & \quad +
  \tfrac{ 1 }{ 2 }
  \sum_{ i = 1 }^m
  \big(
    \tfrac{ \partial^2 }{ \partial y^2 }
    \Phi
  \big)(x,y) 
  \big(
    \sigma_i( y ) , 
    \sigma_i( y )
  \big)
  ,
\\
\label{eq:extended_diffusion}
&
  ( \overline{ G }_{ \sigma } \Psi)( x, y )
  :=
  \big(
    \tfrac{ \partial }{ \partial x }
    \Psi
  \big)(x,y) \, 
  \sigma( x )
  +
  \big(
    \tfrac{ \partial }{ \partial y }
    \Psi
  \big)(x,y) \, 
  \sigma( y )
\end{align}
for all 
$ x,y \in O $,
$ \phi \in C^2( O, \R ) $,
$ \psi \in C^1( O, \R ) $,
$ \Phi \in C^2( O^2, \R ) $,
$ \Psi \in C^1( O^2, \R ) $.
We call the 
linear operator $ \mathcal{G}_{ \mu, \sigma } $
defined in~\eqref{eq:drift_operator}
the \emph{generator},
we call the 
linear operator $ G_{ \sigma } $
defined in~\eqref{eq:diffusion_operator}
the \emph{noise operator},
we call the 
linear operator $ \overline{ \mathcal{G} }_{ \mu, \sigma } $
defined in~\eqref{eq:extended_drift} 
the \emph{extended generator}
and we call the 
linear operator $ \overline{ G }_{ \sigma } $
defined in~\eqref{eq:extended_diffusion}
the \emph{extended noise operator}.
The extended generator has been
exploited in 
Ichikawa~\cite{Ichikawa1984}
and Maslowski~\cite{Maslowski1986}
(see, e.g., also Leha \&\ Ritter~\cite{LehaRitter1994, LehaRitter2003}).
Whereas these references rely on the extended generator,
in our analysis below both the extended generator and the extended diffusion
operator play an essential role.

\chapter{Strong stability analysis
for solutions of SDEs}
\label{sec:strong_stability}

The main results of  this section are
Theorem~\ref{thm:UV}
and
Theorem~\ref{thm:UV2}
below which establish
marginal and uniform
strong stability estimates
respectively.

\section{Setting}
\label{sec:setting}

Throughout this section we will frequently use the following setting.
Let $ d, m \in \N $, $ T \in (0,\infty) $,
let $ O \subseteq\R^d $ be an open set,
let $ \mu \in \mathcal{L}^0( O; \R^d ) $, 
$ \sigma \in \mathcal{L}^0( O; \R^{ d \times m } ) $,
let
$
  ( 
    \Omega, \mathcal{F}, \P, ( \mathcal{F}_t )_{ t \in [0,T] } 
  )
$
be a filtered probability space satisfying
the usual conditions and let 
$
  W \colon [0,T] \times \Omega \to \R^m
$
be a standard $ ( \mathcal{F}_t )_{ t \in [0,T] } $-Brownian motion.

\section{Exponential integrability bounds for solutions of SDEs}

The main result of this subsection, 
Proposition~\ref{prop:exp_mom} below,
establishes certain exponential integrability properties for solutions
of SDEs. 
Further instructive results 
on exponential moments
can, for example,
be found in
\cite{
HairerMattingly2006,
BouRabeeHairer2013,
EsSarhirStannat2010,
FangImkellerZhang2007,
HieberStannat2013}.
For the proof of Proposition~\ref{prop:exp_mom}, we first present
two well-known auxiliary lemmas.

\begin{lemma}[Positivity]
\label{lem:comparison}
Let $ T \in [0,\infty) $,
let $ \left( \Omega, \mathcal{F}, \P \right) $
be a probability space and let 
$ Z \colon [0,T] \times \Omega \to \R $
be a product measurable stochastic process
satisfying
$
  \int_0^T \max( Z_t,0) \, dt
  < \infty
$
$ \P $-a.s.\ and
$
  Z_t \geq 0
$
$ \P $-a.s.\ for Lebesgue-almost all
$ t \in [0,T] $.
Then $ \int_0^T Z_t \, dt \geq 0 $ $ \P $-a.s.
\end{lemma}

\begin{proof}[Proof
of Lemma~\ref{lem:comparison}]
Note that
\begin{equation*}
  0 
  \leq
  \E\!\left[ 
    \smallint\nolimits_0^T \max( Z_t, 0 ) \, dt
    -
    \smallint\nolimits_0^T Z_t \, dt 
  \right]
=
  \smallint\nolimits_0^T
  \E\!\left[ 
    \max( Z_t, 0 ) 
    -
    Z_t  
  \right]
  dt
  = 0
\end{equation*}
and hence that
$
  0 \leq \int_0^T \max( Z_t, 0 ) \, dt 
  = \int_0^T Z_t \, dt 
$
$ \P $-a.s. This finishes the proof
of Lemma~\ref{lem:comparison}.
\end{proof}

For convenience of the reader,
we recall the following well-known Lyapunov estimate
(cf., e.g., the proof
of Lemma~2.2 in Gy\"{o}ngy \&\ Krylov~\cite{gk96b}).

\begin{lemma}[A Lyapunov estimate]
\label{lem:Lyapunov}
  Assume the setting in Section~\ref{sec:setting},
  let 
  $ V \in C^{ 1, 2 }( [0,T] \times O, [0,\infty) ) $,
  $ \alpha \in \mathcal{L}^0( [0,T]; [0,\infty) ) $
  with
  $
    \int_0^T \alpha(t) \, dt 
    < \infty
  $,
  let $ \tau \colon \Omega \to [0,T] $
  be a stopping time
  and let
  $
    X \colon [0,T] \times \Omega \to O
  $
  be an adapted stochastic
  process with continuous sample paths
  satisfying
  $
    \int_0^{ \tau } \| \mu( X_s ) \| 
    + \| \sigma( X_s ) \|^2 
    \, ds < \infty
  $
  $ \P $-a.s.,
  \begin{equation}
  \label{eq:V_est_Lyapunov}
  \begin{split}  
    \big( 
      \tfrac{ \partial }{ \partial t } V
    \big)( t \wedge \tau, X_{ t \wedge \tau } )
    +
    \big( 
      \tfrac{ \partial }{ \partial x } V
    \big)( t \wedge \tau, X_{ t \wedge \tau } )
    \, \mu( X_{ t \wedge \tau } )
  \\
    +
    \tfrac{ 1 }{ 2 }
    \textup{tr}\big(
      \sigma( X_{ t \wedge \tau } ) \sigma( X_{ t \wedge \tau } )^*
      ( \textup{Hess}_x V)( t \wedge \tau, X_{ t \wedge \tau } )
    \big)
  & \leq 
    \alpha(t \wedge \tau) 
    \,
    V( t \wedge \tau ,X_{ t \wedge \tau } )
  \end{split}
  \end{equation}
  $ \P $-a.s.\ and
  $
    X_{ t \wedge \tau }
    =
    X_0
    + 
    \int_0^{ t \wedge \tau }
    \mu( X_s )
    \, ds
    +
    \int_0^{ t \wedge \tau }
    \sigma( X_s )
    \, dW_s
  $ $ \P $-a.s.\ for all 
  $ 
    t \in [0,T] 
  $.
  Then
  $
    \E\big[
      V( \tau, X_{ \tau } )
    \big]
  \leq
    \exp\!\big(
      \int_0^T 
      \alpha(s) \,
      ds
    \big)
    \,
    \E\big[
      V( 0, X_0 )
    \big]
  \in 
    [0,\infty]
  $.
\end{lemma}

\begin{proof}[Proof
of Lemma~\ref{lem:Lyapunov}]
First, we assume w.l.o.g.\ that 
$
  \E[
    V( 0, X_0 )
  ]
  < \infty
$.
Next define 
stopping times 
$ \rho_n \colon \Omega \to [0,T] $,
$ n \in \N $,
by
\begin{equation*} \begin{split}
&
  \rho_n
  :=
\\
&
  \inf\!\left(
    \{ \tau \} 
    \cup
    \left\{
    t \in [0,T]
    \colon
    \sup\nolimits_{ s \in [0,t] }
    V( s, X_s )
    +
    \smallint\nolimits_0^t
      \|
        \sigma( X_u )^*
        ( \nabla_x V)( u, X_u )
      \|^2
    \, du
  \geq
    n
  \right\}
  \right)
\end{split}
\end{equation*}
for all $ n \in \N $.
Then note that It\^{o}'s formula proves that
\begin{align}
  V( t \wedge \rho_n , X_{ t \wedge \rho_n } )
&  =
  V( 0, X_0 ) 
  +
  \int_0^{ t \wedge \rho_n }
    \big( 
      \tfrac{ \partial }{ \partial x } V
    \big)( s, X_s )
    \, \sigma( X_s )
  \, dW_s
\\ & \quad +
  \int_0^{ t \wedge \rho_n }
    \big( 
      \tfrac{ \partial }{ \partial s } V
    \big)( s, X_s )
    +
    \big( 
      \tfrac{ \partial }{ \partial x } V
    \big)( s, X_s )
    \, \mu( X_s )
    \,
  ds
  \nonumber
\\ & \quad +
  \int_0^{ t \wedge \rho_n }
    \tfrac{ 1 }{ 2 } 
    \text{tr}\big(
      \sigma( X_s ) \sigma( X_s )^*
      ( \text{Hess}_x V)( s, X_s )
    \big)
    \,
  ds
  \nonumber
\end{align}
$ \P $-a.s.\ for all $ (t, n) \in [0,T] \times \N $
and assumption~\eqref{eq:V_est_Lyapunov}
and Lemma~\ref{lem:comparison}
hence imply that
\begin{equation}
\begin{split}
&
  V( t \wedge \rho_n , X_{ t \wedge \rho_n } )
\\ &
\leq
  V( 0, X_0 )
  +
  \int_0^{ t \wedge \rho_n }
    \big( 
      \tfrac{ \partial }{ \partial x } V
    \big)( s, X_s )
    \, \sigma( X_s )
  \, dW_s
  +
  \int_0^{ t \wedge \rho_n }
    \alpha(s) \, V(s, X_s )
    \,
  ds
\end{split}
\end{equation}
$ \P $-a.s.\ for all $ (t, n) \in [0,T] \times \N $.
Taking expectations then shows that
\begin{equation}
\begin{split}
  \E\big[
    V( t \wedge \rho_n , X_{ t \wedge \rho_n } )
  \big]
& \leq
  \E\big[
    V( 0, X_0 )
  \big]
  +
  \int_0^t
    \alpha(s) \, 
    \E\big[ 
      \mathbbm{1}_{
        \{ 
          s \leq \rho_n
        \}
      }
      V(s, X_s )
    \big]
    \,
  ds
\\ & \leq
  \E\big[
    V( 0, X_0 )
  \big]
  +
  \int_0^t
    \alpha(s) \, 
    \E\big[ 
      V( s \wedge \rho_n, X_{ s \wedge \rho_n } )
    \big]
    \,
  ds
\end{split}
\end{equation}
for all $ (t, n) \in [0,T] \times \N $.
The estimate
$
    \E\big[ 
      V( t \wedge \rho_n, X_{ t \wedge \rho_n } )
    \big]
  \leq
    n +
    \E\big[ 
      V( 0, X_0 )
    \big]
  < \infty
$
for all $ (t, n) \in [0,T] \times \N $
and Gronwall's lemma therefore yield
$
  \E\big[
    V( t \wedge \rho_n , X_{ t \wedge \rho_n } )
  \big]
\leq
  e^{ \int_0^t \alpha(s) \, ds }
  \,
  \E\big[
    V( 0 , X_0 )
  \big]
$
for all $ (t, n) \in [0,T] \times \N $.
This and Fatou's lemma complete the proof
of Lemma~\ref{lem:Lyapunov}.
\end{proof}

The next proposition proves exponential
integrability properties for solution processes of
SDEs.

\begin{prop}[Exponential integrability properties]
\label{prop:exp_mom}
  Assume the setting in Section~\ref{sec:setting},
  let $ U \in C^{ 1, 2 }( [0,T] \times O, \R ) $,
  $ \overline{U} \in \mathcal{L}^0( [0,T] \times O ; \R ) $,
  let $ \tau \colon \Omega \to [0,T] $
  be a stopping time
  and
  let
  $
    X \colon [0,T] \times \Omega \to O
  $
  be an adapted stochastic
  process with continuous sample paths
  satisfying
  $
    \int_0^{ \tau }
    \| \mu( X_s ) \| +
    \| \sigma( X_s ) \|^2 
    +
    | \overline{U}( s, X_s ) |
    \, ds
    < \infty
  $
  $ \P $-a.s.,
  $  
    X_{ t \wedge \tau }
    =
    X_0
    + 
    \int_0^{ t \wedge \tau }
    \mu( X_s )
    \, ds
    +
    \int_0^{ t \wedge \tau }
    \sigma( X_s )
    \, dW_s
  $
  $ \P $-a.s.\ for all $ t \in [0,T] $
  and
  \begin{equation}
  \label{eq:assumption_exp_mom}
  \begin{split}
  &
      ( \tfrac{ \partial }{ \partial t } U )(t,x)
      +
      ( \tfrac{ \partial }{ \partial x } U )(t,x)
      \,
      \mu(x)
      +
      \tfrac{
      \operatorname{tr}( 
	\sigma( x ) \, \sigma( x )^* 
	( \operatorname{Hess}_x U )( t, x )   	
      )
      +
        \| \sigma(x)^* (\nabla_x U)(t,x) \|^2
      }{
        2 
      }
  \\ & \leq 
    -
    \overline{U}( t, x )
  \end{split}
  \end{equation}
  for all 
  $ 
    (t,x) 
    \in 
    \cup_{ \omega \in \Omega }
    \{ 
      ( s, X_s( \omega ) )
      \in [0,T] \times O
      \colon
      s \in [0, \tau( \omega ) ]
    \}
  $.
  Then
  \begin{equation}
  \label{eq:statement_exp_mom}
  \begin{split}
  &  \E\!\left[
      \exp\!\left(
        r
        \left[ 
          U( \tau, X_{ \tau } )
          +
          \int_0^{ \tau }
          \overline{U}( s, X_s )
          +
          \tfrac{ (1 - r) }{ 2 }
          \| \sigma(X_s)^* (\nabla_x U)(s,X_s) \|^2
          \, ds
        \right]
      \right)
    \right]
\\ & \leq
    \E\!\left[
      e^{
        r \, U( 0, X_0) 
      }
    \right]
    \in [0,\infty]
  \end{split}
  \end{equation}
  for all $ r \in [0,\infty) $.
\end{prop}

\begin{proof}[Proof of Proposition~\ref{prop:exp_mom}]
  First of all, let $ r \in [0,\infty) $
  and define
  $
    \bar{V} \colon [0,T] \times O \times \R \to \R
  $
  by
  $
    \bar{V}(t,x,y)
    =
    \exp\!\left(
      r 
      \left[
        U(t,x)
        +
        y
      \right]
    \right)
  $
  for all $ (t,x,y) \in [0,T] \times O \times \R $.
  Then note that
  assumption~\eqref{eq:assumption_exp_mom}
  implies that
  \begin{equation}
  \begin{split}
  &
    ( 
      \tfrac{ \partial }{ \partial t } \bar{V}
    )( t, x, y )
    +
    ( 
      \tfrac{ \partial }{ \partial x } \bar{V}
    )( t, x, y )
    \, \mu(x)
    +
    \tfrac{ 
      1
    }{ 2 }
    \,
      \text{tr}(
        \sigma(x) \sigma(x)^*
        ( \operatorname{Hess}_x \bar{V})( t, x, y )
      )
  \\ & \quad
    +
    ( 
      \tfrac{ \partial }{ \partial y } \bar{V}
    )( t, x, y )
    \left[ 
      \overline{U}( t, x )
      +
      \tfrac{ ( 1 - r ) }{ 2 }
      \|
        \sigma( x )^* ( \nabla_x U )( t, x )
      \|^2
    \right]
  \\ & =
    r
    \bar{V}(t,x, y)
    \bigg[
      ( \tfrac{ \partial }{ \partial t } U )(t,x)
      +
      ( \tfrac{ \partial }{ \partial x } U )(t,x)
      \,
      \mu(x)
      +
      \overline{U}( t, x )
      +
      \tfrac{ ( 1 - r ) }{ 2 }
      \|
        \sigma( x )^* ( \nabla_x U )( t, x )
      \|^2
    \\
    &\qquad\qquad
    \qquad\qquad
      +
      \tfrac{
        1
      }{ 2 }
      \operatorname{tr}( 
	\sigma( x ) \, \sigma( x )^* 
	( \operatorname{Hess}_x U )( t, x )   	
      )
      +
      \tfrac{
        r
      }{
        2 
      }
        \| \sigma(x)^* (\nabla_x U)(t,x) \|^2
    \bigg]
  \\ & =
    r
    \bar{V}(t,x, y)
    \bigg[
      ( \tfrac{ \partial }{ \partial t } U )(t,x)
      +
      ( \tfrac{ \partial }{ \partial x } U )(t,x)
      \,
      \mu(x)
      +
      \overline{U}( t, x )
    \\
    &\qquad\qquad
    \qquad\qquad
      +
      \tfrac{
        1
      }{ 2 }
      \operatorname{tr}( 
	\sigma( x ) \, \sigma( x )^* 
	( \operatorname{Hess}_x U )( t, x )   	
      )
      +
      \tfrac{
        1
      }{
        2 
      }
        \| \sigma(x)^* (\nabla_x U)(t,x) \|^2
   \bigg]
  \leq
    0
  \end{split}
  \label{eq:V_estimate_time}
  \end{equation}
  for all 
  $ 
    (t,x,y) 
    \in 
    \cup_{ \omega \in \Omega }
    \{ 
      ( s, X_s( \omega ) )
      \in [0,T] \times O
      \colon
      s \in [0, \tau( \omega ) ]
    \} \times \R
  $.
Next let 
$ Y \colon [0,T] \times \Omega \to \R $
be an adapted stochastic process with continuous
sample paths
satisfying
$
  Y_t =
      \smallint\nolimits_0^{t\wedge\tau} 
      \overline{U}( s, X_s ) 
      +
      \tfrac{ ( 1 - r ) }{ 2 }
      \|
        \sigma( X_s )^* ( \nabla_x U )( s, X_s )
      \|^2
      \, ds
$
$ \P $-a.s.\ for all $ t \in [0,T] $.
Then we get from~\eqref{eq:V_estimate_time} that
  \begin{equation}
  \begin{split}
  &
    ( 
      \tfrac{ \partial }{ \partial t } \bar{V}
    )\!\left( 
      t \wedge \tau , 
      X_{ t \wedge \tau } ,
      Y_{ t \wedge \tau } 
    \right)
    +
    ( 
      \tfrac{ \partial }{ \partial x } \bar{V}
    )\!\left( 
      t \wedge \tau , 
      X_{ t \wedge \tau } ,
      Y_{ t \wedge \tau } 
    \right)
    \mu( X_{ t \wedge \tau } )
  \\ &
    +
    ( 
      \tfrac{ \partial }{ \partial y } \bar{V}
    )\!\left( 
      t \wedge \tau , 
      X_{ t \wedge \tau } ,
      Y_{ t \wedge \tau } 
    \right)\!
    \left[
      \overline{U}( t \wedge \tau, X_{ t \wedge \tau } )
      +
      \tfrac{ ( 1 - r ) }{ 2 }
      \|
        \sigma( X_{ t \wedge \tau } )^* ( \nabla_x U )( t \wedge \tau, X_{ t 
    \wedge \tau } )
      \|^2
    \right]
  \\ &
    +
    \tfrac{ 1 }{ 2 }
    \,
    \text{tr}\big(
      \sigma( X_{ t \wedge \tau } ) \, \sigma( X_{ t \wedge \tau } )^* \,
      ( \text{Hess}_x \bar{V})( t \wedge \tau , X_{ t \wedge \tau } ,
        Y_{ t \wedge \tau } 
      )
    \big)
  \leq
    0
  \end{split}
  \end{equation}
for all 
$ 
  t \in [0,T]
$.
An application of Lemma~\ref{lem:Lyapunov} hence completes the proof
of Proposition~\ref{prop:exp_mom}.
\end{proof}

The next corollary, Corollary~\ref{cor:exp_mom}, 
specializes
Proposition~\ref{prop:exp_mom}
to the case 
where 
$ U( t, x ) = e^{ - \alpha t } \, U( 0, x ) $
and
$ \overline{U}( t, x ) = e^{ - \alpha t } \, \hat{U}( t, x ) $
for all $ (t,x) \in [0,T] \times O $
and some $ \alpha \in \R $.

\begin{corollary}[Exponential integrability properties
(time-independent version)]
\label{cor:exp_mom}
  Assume the setting in Section~\ref{sec:setting},
  let $ \alpha \in \R $, $ U \in C^2( O, \R ) $,
  $ \overline{U} \in \mathcal{L}^0([0,T]\times O ; \R ) $,
  let $ \tau \colon \Omega \to [0,T] $
  be a stopping time
  and
  let
  $
    X \colon [0,T] \times \Omega \to O
  $
  be an adapted stochastic
  process with continuous sample paths
  satisfying
  $
    \int_0^{ \tau }
    \| \mu( X_s ) \| +
    \| \sigma( X_s ) \|^2 +
    | \overline{U}( s, X_s ) |
    \, ds
    < \infty
  $
  $ \P $-a.s.,
  $ 
    X_{ t \wedge \tau }
    =
    X_0
    + 
    \int_0^{ t \wedge \tau }
    \mu( X_s )
    \, ds
    +
    \int_0^{ t \wedge \tau }
    \sigma( X_s )
    \, dW_s
  $
  $ \P $-a.s.\ for all $ t \in [0,T] $
  and
  \begin{equation}
  \label{eq:exp_mom_assumption}
    ( \mathcal{G}_{ \mu, \sigma } U)( x )
    + 
    \tfrac{ 
      1 
    }{ 
      2 e^{ \alpha t } 
    } 
    \left\| 
      \sigma(x)^*
      \left(\nabla U\right)\!(x) 
    \right\|^2 
    +
    \overline{U}(t, x )
  \leq 
    \alpha U(x) 
  \end{equation}
  for all 
  $ 
    (t, x) 
    \in 
    [0,T] \times
    \cup_{ \omega \in \Omega }
    \{ 
      X_s( \omega )
      \in O
      \colon
      s \in [0, \tau( \omega ) ]
    \}
  $.
  Then
  \begin{equation}
  \label{eq:exp_mom_estimate}
  \begin{split}
  &
    \E\!\left[
      \exp\!\left(
        \tfrac{   
          U( X_{ \tau } )
        }{ 
          e^{ \alpha \tau }
        }    
        +
        \smallint_0^{ \tau }
          \tfrac{ 
            \overline{U}(s, X_s ) 
          }{ 
            e^{ \alpha s } 
          }
        \, ds
      \right)
    \right]
  \leq
    \E\Big[\!
      \exp\!\big( 
        U(X_0) 
      \big)
    \Big]
    \in [0,\infty]
    .
  \end{split}
  \end{equation}
\end{corollary}

A slightly different formulation of 
Corollary~\ref{cor:exp_mom}
is presented in the following corollary.

\begin{corollary}
\label{cor:exp_mom_short_version2}
  Assume the setting in Section~\ref{sec:setting},
  let $ \tau \colon \Omega \to [0,T] $
  be a stopping time
  and
  let
  $
    X \colon [0,T] \times \Omega \to O
  $
  be an adapted stochastic
  process with continuous sample paths
  satisfying
  $
    \int_0^{ \tau }
    \| \mu( X_s ) \| +
    \| \sigma( X_s ) \|^2 
    \, ds
    < \infty
  $
  $ \P $-a.s.\ and
  $  
    X_{ t \wedge \tau }
    =
    X_0
    + 
    \int_0^{ t \wedge \tau }
    \mu( X_s )
    \, ds
    +
    \int_0^{ t \wedge \tau }
    \sigma( X_s )
    \, dW_s
  $
  $ \P $-a.s.\ for all $ t \in [0,T] $.
  Then
  \begin{equation}
  \label{eq:cor_exp_mom2}
  \begin{split}
  &
    \E\!\left[
      \exp\!\left(
        e^{ - \alpha \tau }
        U( X_{ \tau } )
          +
          \smallint_0^{ \tau }
	  {\scriptstyle
          e^{ - \alpha s }
          \left[
          \alpha 
          U( X_s)
          -
          (
            \mathcal{G}_{ \mu , \sigma }
            U
          )(X_s)
            -
            \tfrac{ e^{ - \alpha s } }{ 2 }
            \| \sigma(X_s)^* (\nabla U)(X_s) \|^2
          \right] 
          ds
        }
      \right)
    \right]
  \\ &
   \leq
    \E\!\left[
      e^{
        U( X_0 ) 
      }
    \right]
    \in [0,\infty]
  \end{split}
  \end{equation}
  for all $ \alpha \in \R $ and all
  $ U \in C^2( O, \R ) $.
\end{corollary}

We illustrate Corollary~\ref{cor:exp_mom_short_version2} by
three simple examples. First, observe that 
if $ r \in \R $ 
and if $ U $ in Corollary~\ref{cor:exp_mom_short_version2} 
satisfies $ U( x ) = r \, \| x \|^2 $ for all $ x \in O $, then
\eqref{eq:cor_exp_mom2} shows for every $ \alpha \in \R $ that 
\begin{equation*}
  \begin{split}
  &
    \E\!\left[
      \exp\!\left(
      \tfrac{ r }{ e^{ \alpha \tau } }
	\| X_{ \tau } \|^2
	+
	\smallint_0^{ \tau }
	{\scriptstyle
	\tfrac{ r }{ e^{ \alpha s } }
	\left[
	  \alpha 
	\| X_s \|^2
	-
	2
	\left< 
	  X_s ,
	  \mu(X_s)
	\right>
	-
	  \|
	    \sigma( X_s ) 
	  \|^2_{ \HS( \R^m , \R^d ) } 
	  -
	  \tfrac{ 2 r }{ e^{ \alpha s } }
	  \| \sigma(X_s)^* X_s \|^2
	\right] 
	}
	ds\!
      \right)\!
    \right]
\\ &
    \leq
    \E\!\left[
      e^{
        r \| X_0 \|^2
      }
    \right]\!
    .
  \end{split}
  \end{equation*}
  Second, note that 
  if $ \varepsilon \in (0,\infty) $, $d=m$
  and if $ \mu $ and $ \sigma $ in Corollary~\ref{cor:exp_mom_short_version2} 
satisfy $ \mu(x) = - (\nabla U)( x ) $ and
$ \sigma(x) = \sqrt{ \varepsilon } I $ for all $ x \in O $
and some $ U \in C^2( O, \R ) $, then
\eqref{eq:cor_exp_mom2} implies for every $ \alpha, r \in \R $ that
\begin{equation}
  \begin{split}
&    \E\!\left[
      \exp\!\left(
        \tfrac{ r }{ e^{ \alpha \tau } }
          U( X_{ \tau } )
          +
          \smallint\limits_0^{ \tau }
          \tfrac{  r  }{ e^{ \alpha s } }
          \left[
            \alpha 
            U( X_s )
            +
            \left[ 
              1 - 
              \tfrac{ \varepsilon r }{ 2 e^{ \alpha s } }
            \right]
            \| ( \nabla U )( X_s ) \|^2
            -
            \varepsilon ( \Delta U)( X_s )
          \right] 
          ds\!
      \right)\!
    \right]
\\ &
    \leq
    \E\!\left[
      e^{
        r U( X_0 )
      }
    \right]\!
    .
\label{eq:Overd_Langevin00}
\end{split} \end{equation}
A result related to~\eqref{eq:Overd_Langevin00} 
can, e.g., be found in Lemma~2.5 in 
Bou-Rabee \&\ Hairer~\cite{BouRabeeHairer2013}.
Finally, observe that
if $ r \in \R $ and
if $ U $ in Corollary~\ref{cor:exp_mom_short_version2}
satisfies
$
  U(x) = 
  r \ln\!\big( 1 + \| x \|^2 \big)
$
for all $ x \in O $,
then
\eqref{eq:cor_exp_mom2} implies for every
$ \alpha \in \R $ that 
\begin{equation}
  \begin{split}
&    
  \E\!\left[
    \left(1+\|X_{\tau}\|^2\right)^{\frac{r}{e^{\alpha\tau}}}  
    \exp\!\left(
    \int_0^{ \tau }
    \tfrac{ \alpha r }{ e^{ \alpha s } }
      \ln( 1 + \| X_s \|^2 )
    \, ds
  \right)
    \right.
\\ & \quad \cdot 
  \left.
  \exp\!\left(
    \int_0^{ \tau }
    \tfrac{ r }{ e^{ \alpha s } }
    \left[
      \tfrac{
        -
	2 \left< X_s, \mu( X_s ) \right>
	-
	\| \sigma( X_s ) \|^2_{ \HS( \R^m, \R^d ) }
      }{
	( 1 + \| X_s \|^2 )
      }
      +
      \left[ 
      1
      -
      \tfrac{ r }{ e^{ \alpha s } }
      \right]
      \tfrac{
	2 \,
	\| \sigma( X_s )^* X_s \|^2
      }{
	( 1 + \| X_s \|^2 )^2
      }
    \right] 
    \, ds
  \right)
    \right]
\\ & \leq
    \E\Big[
      \left(
        1 + \| X_0 \|^2
      \right)^r
    \Big]
    .
\label{eq:Log_Lyapunov}
\end{split}
\end{equation}
  The following corollary of~\eqref{eq:Log_Lyapunov}
  states a 
  moment estimate for solutions of SDEs which is interesting on its own.
  
\begin{corollary}
  Assume the setting in Section~\ref{sec:setting},
  let $ p, c \in \R $, $ \alpha \in [0,\infty) $,
  let $ \tau \colon \Omega \to [0,T] $
  be a stopping time
  and
  let
  $
    X \colon [0,T] \times \Omega \to O
  $
  be an adapted stochastic
  process with continuous sample paths
  satisfying
  $
    \int_0^{ \tau }
    \| \mu( X_s ) \| +
    \| \sigma( X_s ) \|^2 
    \, ds
    < \infty
  $
  $ \P $-a.s., 
  $  
    X_{ t \wedge \tau }
    =
    X_0
    + 
    \int_0^{ t \wedge \tau }
    \mu( X_s )
    \, ds
    +
    \int_0^{ t \wedge \tau }
    \sigma( X_s )
    \, dW_s
  $
  $ \P $-a.s.\ for all $ t \in [0,T] $
  and
  \begin{equation}
    2
    \langle x, \mu(x) \rangle
    +
    \| \sigma(x) \|_{ \HS(\R^m,\R^d) }^2
    +
    \tfrac{
      2 \, ( p - 1 ) \,
      \| \sigma(x)^* x \|^2
    }{
      ( 1 + \|x\|^2 )
    }
    \leq 
    \left(
      c + \alpha
      \ln( 1 + \|x\|^2 )
    \right)
    \left(
      1 + \|x\|^2
    \right)
  \end{equation}
  for all $ x \in \operatorname{im}( X ) $.
  Then
  $
    \E\big[
      e^{ - c \tau }
      (
        1 + 
        \| X_{ \tau } \|^2
      )^{
        p 
        e^{ - \alpha \tau }
      }
    \big]
    \leq 
    \E\!\left[
      \left(
        1 + 
        \|X_0\|^2
      \right)^p
    \right]
  $.
\end{corollary}

\begin{lemma}
\label{l:exponential.martingale}
  Let
  $ m  \in \mathbb{N} $,
  $ T \in [0,\infty) $,
  let 
  $ 
    ( \Omega, \mathcal{F}, \P, ( \mathcal{F}_t )_{ t \in [0,T] } ) 
  $
  be a filtered probability space satisfying the usual conditions, 
  let
  $
    W \colon [0,T] \times \Omega
    \rightarrow \mathbb{R}^m
  $
  be a standard 
  $
    ( \calF_t )_{ t \in [0,T] }
  $-Brownian motion
  and let 
  $ A \colon [0,T] \times \Omega \to \R^m $
  be an adapted and product measurable 
  stochastic process satisfying
  $
    \int_0^T \|A_s\|^2 \, ds 
  $
  $ 
    < \infty
  $
  $\P$-a.s.
  Then it holds for all $ p \in (1,\infty] $ that
  \begin{align}
  \notag 
  &
  \left\|\sup_{t\in[0,T]}
    \exp\!\left(
      \int_0^t
	\langle A_s , dW_s \rangle
      - \tfrac{1}{2}\int_0^t
	\left\|A_s\right\|^2\,ds
    \right)
    \right\|_{L^p(\Omega;\R)}
  \\& \label{eq:exponential.martingale.1}
  \leq
    \frac{ 1 }{ ( 1 - \frac{ 1 }{ p } ) }
    \inf_{ q \in [p, \infty] }
    \left\n\,
      \exp\!\left(
        \tfrac{ 1 }{ 2 }
        \Big[
          \tfrac{ 1 }{ ( \frac{1}{p} - \frac{1}{q} ) } - 1
        \Big]
	\int_{0}^{T} 
	  \n A_s \n^2
	\,ds
      \right)
    \right\n_{L^{q}(\Omega;\R)} 
  \\& \label{eq:exponential.martingale.2p}
  \leq 
    \frac{ 1 }{ ( 1 - \frac{ 1 }{ p } ) }
    \left\n\,
      \exp\!\left(
        \left(
          p - \tinv{2}
        \right)
	\int_{0}^{T} 
	  \n A_s \n^2
	\,ds
      \right)
    \right\n_{L^{2p}(\Omega;\R)}     \in[0,\infty]
    .
  \end{align}
\end{lemma}

\begin{proof}[Proof
of Lemma~\ref{l:exponential.martingale}]
  Inequality~\eqref{eq:exponential.martingale.2p}
  follows from inequality~\eqref{eq:exponential.martingale.1}
  by taking $ q = 2 p $.
  It thus remains to prove inequality~\eqref{eq:exponential.martingale.1}.
  If the right-hand side of~\eqref{eq:exponential.martingale.1}
  is infinite, then the proof is complete.
  So for the rest of the proof, we assume that
  the right-hand side of~\eqref{eq:exponential.martingale.1}
  is finite.
  If the infimum
  on the right-hand side of~\eqref{eq:exponential.martingale.1}
  is attained at
  $ q = p $, 
  then necessarily 
  $
    \int_0^T \| A_s \|^2 \, ds = 0
  $
  $ \P $-a.s.
  In that case, both sides of~\eqref{eq:exponential.martingale.1}
  are equal to $1$ and this completes the proof in that case.
  So for the rest of the proof, we assume that 
  $p\in(1,\infty)$ and that the infimum
  on the right-hand side of~\eqref{eq:exponential.martingale.1}
  is not attained at $q=p$.
  Let 
  $
    Z^{ (r) } \colon [0,T] \times \Omega \rightarrow \R
  $,
  $ r \in \R $,
  be adapted stochastic processes with continuous sample paths satisfying
  \begin{align}
      Z^{ (r) }_t 
	= \exp \!\left(
	  r\int_0^t \langle A_s , dW_s \rangle 
	  - 
	  \tfrac1{2}
	  r^2
	  \int_0^t \|A_s \|^2\,ds
	\right)
  \end{align}
  $ \P $-a.s.\ for all $ t \in [0,T] $, $ r \in \R $.  
  It follows from, e.g.,~\cite[Lemma 18.21]{Kallenberg2002} 
  that for every $ r \in \R $
  the process $ Z^{ (r) } $
  is a local martingale. 

  For every $r\in \R$, let $\tau_{r,n}\colon\Omega\to[0,T]$, $n\in\N$,
  be a localizing sequence of stopping times for $Z^{(r)}$.

  Doob's martingale inequality and H\"older's inequality
  imply that for every $ q,r\in (p,\infty)$, $ n \in \N $
  with $ \inv{ q } + \inv{ r } = \inv{ p } $
  it holds that
  \begin{equation}
  \begin{split}
&
  \bigg\|
    \sup_{t\in[0,T\wedge\tau_{r,n}]}
      Z^{(1)}_t
  \bigg\|_{L^p(\Omega;\R)}
  \leq
  \frac{p}{(p-1)}
  \left\|
      Z^{(1)}_{T\wedge\tau_{r,n}}
  \right\|_{L^p(\Omega;\R)}
  \\ & =
  \frac{p}{(p-1)}
  \left\|
    \left(Z^{(r)}_{T\wedge\tau_{r,n}}\right)^{\inv{r}}
    \exp \!\left(
      \tfrac1{2}(r-1)
	\int_{0}^{{T\wedge\tau_{r,n}}} 
	  \n A_s \n^2
	\,ds
    \right)
  \right\|_{L^{p}(\Omega;\R)}
  \\ & \leq 
  \frac{p}{(p-1)}
  \left(
    \E\! \left[Z^{(r)}_{T\wedge\tau_{r,n}}\right]
  \right)^{\inv{r}}
  \left\|\,
    \exp \!\left(
	  \tfrac1{2}(r-1)
	    \int_{0}^{T} 
	      \n A_s \n^2
	    \,ds
	\right)
  \right\|_{L^{q}(\Omega;\R)}
  \\ & =
  \frac{p}{(p-1)}
  \left\|\,
    \exp \!\left(
	  \tfrac1{2}
	    \left(
	      \tfrac{1}{(\frac{1}{p}-\frac{1}{q})}
	      - 1
	    \right)
	    \int_{0}^{T} 
	      \n A_s \n^2
	    \,ds
	\right)
  \right\|_{L^{q}(\Omega;\R)}.
   \end{split}\end{equation}
   Letting $n\to\infty$ and applying the monotone convergence theorem
   implies inequality~\eqref{eq:exponential.martingale.1}.
   The proof of 
   of Lemma~\ref{l:exponential.martingale}
   is thus completed.
\end{proof}

\section{An identity for Lyapunov-type functions}

In Lemma~\ref{lem:Lyapunov2} below, a simple identity for suitable
Lyapunov-type functions is proved.
In the proof of Lemma~\ref{lem:Lyapunov2} the following 
stochastic version of the Gronwall lemma is used.
For completeness its proof is given below.

\begin{lemma}
\label{lem:stochastic_Gronwall}
Let $ m \in \N $,
$ T \in (0,\infty) $,
let
$
  ( 
    \Omega, \mathcal{F}, \P, ( \mathcal{F}_t )_{ t \in [0,T] } 
  )
$
be a filtered probability space satisfying the usual conditions, 
let 
$
  W \colon [0,T] \times \Omega \to \R^m
$
be a standard $ ( \mathcal{F}_t )_{ t \in [0,T] } $-Brownian motion,
let $ \tau \colon \Omega \to [0,T] $
be a stopping time,
let $ X \colon [0,T] \times \Omega \to \R $
be an adapted
stochastic process
with continuous sample paths
and let 
$ \hat{A}, A \colon [0,T] \times \Omega \to [-\infty,\infty] $
and
$
  \hat{B} \colon [0,T] \times \Omega \to  [-\infty,\infty] ^{ 1 \times m } 
$
be adapted and product measurable stochastic processes
satisfying
$
  \int_0^{ \tau }
  | A_s | 
  + 
  | \hat{A}_s |
  +
  \| \hat{B}_s \|^2
  \, 
  ds
  <\infty
$
$ \P 
$-a.s.\ and 
$
  X_{ t \wedge \tau }
= 
  X_0 
  + 
  \int_0^{ t \wedge \tau }
  A_s \, ds
  +
  \int_0^{ t \wedge \tau }
  \hat{B}_s X_s \, dW_s
$
$ \P $-a.s.\ for 
all $ t \in [0,T] $
and
$
  \mathbbm{1}_{
    \{ t < \tau \}
  }
  A_{ t  } 
  \leq 
  \mathbbm{1}_{
    \{ t < \tau \}
  }  
  \hat{A}_t X_t
$
$ \P $-a.s.~for Lebesgue-almost all $t\in[0,T]$.
Then 
\begin{equation}
  X_{ \tau }
\leq
  \exp\!\left(
    \smallint_0^{ \tau }
    \left[
      \hat{A}_s
      -
      \tfrac{ 1 }{ 2 }
      \| \hat{B}_s \|^2
    \right] 
    ds
    +
    \smallint_0^{ \tau }
    \hat{B}_s \,
    dW_s
  \right)
  X_0
  \quad
  \P\text{-a.s.}
\label{eq:stochastic_Gronwall}
\end{equation}
If, in addition,
$
  \mathbbm{1}_{
    \{ t < \tau \}
  }
  A_{ t  } 
  =
  \mathbbm{1}_{
    \{ t < \tau \}
  }  
  \hat{A}_t X_t
$
$ \P $-a.s.\ for Lebesgue-almost all $t\in[0,T]$, 
then~\eqref{eq:stochastic_Gronwall}
holds with
equality.
\end{lemma}

\begin{proof}[Proof
of Lemma~\ref{lem:stochastic_Gronwall}]
Let $ Y \colon [0,T] \times \Omega \to \R $
be an adapted 
stochastic process
with continuous sample paths
satisfying
\begin{equation}
  Y_t
=
    X_{ t \wedge \tau }
    \,
    \exp\!\left(
      -
      \smallint_0^{ t \wedge \tau }
      \left[
        \hat{A}_s 
        -
        \tfrac{ 1 }{ 2 }
        \| \hat{B}_s \|^2
      \right] 
      ds
      -
      \smallint_0^{ t \wedge \tau }
        \hat{B}_s
      \, dW_s
    \right)
\end{equation}
$ \P $-a.s.\ for all $ t \in [0,T] $.
Then It\^{o}'s formula proves that
\begin{equation}
\begin{split}
  Y_t
& =
  X_0 +
  \int_0^{ t \wedge \tau }
    A_s
    \exp\!\left(
      -
      \smallint_0^{ s \wedge \tau }
      \left[
        \hat{A}_u 
        -
        \tfrac{ 1 }{ 2 }
        \| \hat{B}_u \|^2
      \right] 
      du
      -
      \smallint_0^{ s \wedge \tau }
        \hat{B}_u
      \, dW_u
    \right)
  ds
\\ & \quad
  -
  \int_0^{ t \wedge \tau }
  \left[ 
    \hat{A}_s 
    -
    \tfrac{ 1 }{ 2 } \| \hat{B}_s \|^2
  \right] Y_s 
  \, ds
  +
  \int_0^{ t \wedge \tau }
  Y_s
  \left[ 
    \hat{B}_s 
    -
    \hat{B}_s 
  \right] 
  dW_s
\\ & \quad
+ \tfrac{ 1 }{ 2 }
  \int_0^{ t \wedge \tau }
  Y_s
  \,
  \|
    \hat{B}_s 
  \|^2
  \, ds
  -
  \int_0^{ t \wedge \tau }
  Y_s
  \,
  \|
    \hat{B}_s 
  \|^2
  \, ds
\\ & =
  X_0 +
  \int_0^{ t \wedge \tau }
    A_s
    \exp\!\left(
      -
      \smallint_0^{ s \wedge \tau }
      \left[
        \hat{A}_u 
        -
        \tfrac{ 1 }{ 2 }
        \| \hat{B}_u \|^2
      \right] 
      du
      -
      \smallint_0^{ s \wedge \tau }
        \hat{B}_u
      \, dW_u
    \right)
  ds
\\ & \quad
  -
  \int_0^{ t \wedge \tau }
  \hat{A}_s 
  Y_s 
  \, ds
\end{split}
\end{equation}
$ \P $-a.s.\ for all $ t \in [0,T] $.
The assumption 
$
  \mathbbm{1}_{
    \{ s < \tau \}
  }
  (
  \hat{A}_s
  X_s
  -
  A_s
  )
\geq 0
$
$ \P $-a.s.\ for Lebesgue-almost all $ s \in [0,T] $
together with Lemma~\ref{lem:comparison}
and 
$\int_0^{\tau}|A_s|+|\hat{A}_sX_s|\,ds<\infty$ $\P$-a.s.~hence implies that
$
  Y_t \leq X_0
$
$ \P $-a.s.\ for all $ t \in [0,T] $
and, in particular, that
$
  Y_T \leq X_0
$
$ \P $-a.s.
In addition, observe that 
if
$
  \mathbbm{1}_{
    \{ s < \tau \}
  }
  A_s
=
  \mathbbm{1}_{
    \{ s < \tau \}
  }
  \hat{A}_s
  X_s
$
$ \P $-a.s.\ for Lebesgue-almost all $ s \in [0,T] $,
then Lemma~\ref{lem:comparison}
implies that
$
  Y_T = X_0
$
$ \P $-a.s.
This 
finishes the proof of Lemma~\ref{lem:stochastic_Gronwall}.
\end{proof}

The following corollary shows that if $X$ in  
Lemma~\ref{lem:stochastic_Gronwall}
is non-negative, then fewer integrability assumptions are needed.

\begin{corollary}
\label{c:stochastic_Gronwall2}
Let $ m \in \N $,
$ T \in (0,\infty) $,
let
$
  ( 
    \Omega, \mathcal{F}, \P, ( \mathcal{F}_t )_{ t \in [0,T] } 
  )
$
be a filtered probability space satisfying the usual conditions, 
let 
$
  W \colon [0,T] \times \Omega \to \R^m
$
be a standard $ ( \mathcal{F}_t )_{ t \in [0,T] } $-Brownian motion,
let $ \tau \colon \Omega \to [0,T] $
be a stopping time,
let $ X \colon [0,T] \times \Omega \to \R $
be an adapted
stochastic process
with continuous sample paths
and let 
$ \hat{A}, A \colon [0,T] \times \Omega \to [-\infty,\infty] $
and
$
  \hat{B} \colon [0,T] \times \Omega \to  [-\infty,\infty] ^{ 1 \times m } 
$
be adapted and product measurable stochastic processes
satisfying
$
  X_{ t \wedge \tau }
  \geq 0
$
$\P$-a.s.,
$
  \int_0^{ \tau }
  | A_s | 
  + 
  \max( \hat{A}_s ,0)
  +
  \| \hat{B}_s \|^2
  \, 
  ds
  <\infty
$
$ \P 
$-a.s.\ and 
$
  X_{ t \wedge \tau }
= 
  X_0 
  + 
  \int_0^{ t \wedge \tau }
  A_s \, ds
  +
  \int_0^{ t \wedge \tau }
  \hat{B}_s X_s \, dW_s
$
$ \P $-a.s.\ for 
all $ t \in [0,T] $
and
$
  \mathbbm{1}_{
    \{ t < \tau \}
  }
  A_{ t  } 
  \leq 
  \mathbbm{1}_{
    \{ t < \tau \}
  }  
  \hat{A}_t X_t
$
$ \P $-a.s.~for Lebesgue-almost all $t\in[0,T]$.
Then 
\begin{equation}
  X_{ \tau }
\leq
  \exp\!\left(
    \smallint_0^{ \tau }
    \left[
      \hat{A}_s
      -
      \tfrac{ 1 }{ 2 }
      \| \hat{B}_s \|^2
    \right] 
    ds
    +
    \smallint_0^{ \tau }
    \hat{B}_s \,
    dW_s
  \right)
  X_0
  \quad
  \P\text{-a.s.}
\label{eq:c:stochastic_Gronwall}
\end{equation}
If, in addition,
$
  \mathbbm{1}_{
    \{ t < \tau \}
  }
  A_{ t  } 
  =
  \mathbbm{1}_{
    \{ t < \tau \}
  }  
  \hat{A}_t X_t
$
$ \P $-a.s.\ for Lebesgue-almost all $t\in[0,T]$, 
then~\eqref{eq:c:stochastic_Gronwall}
holds with
equality.
\end{corollary}

\begin{proof}[Proof
of Corollary~\ref{c:stochastic_Gronwall2}]
As
$
  \mathbbm{1}_{
    \{ t < \tau \}
  }
  A_{ t  } 
\leq
  \mathbbm{1}_{
    \{ t < \tau \}
  }  
  \hat{A}_t X_t
=
  \mathbbm{1}_{
    \{ t < \tau \}
  }  
  \hat{A}_t X_{ t \wedge \tau }
\leq 
  \mathbbm{1}_{
    \{ t < \tau \}
  }  
  \max(\hat{A}_t,-n) X_t
$
$\P$-a.s.\ and $\int_0^{\tau}|\max(\hat{A}_s,-n)|\,ds<\infty$
$\P$-a.s.\ for Lebesgue-almost all $t\in[0,T]$ and all $n\in\N$,
Lemma~\ref{lem:stochastic_Gronwall} implies
\begin{equation}
  X_{ \tau }
\leq
  \exp\!\left(
    \smallint_0^{ \tau }
    \left[
      \max(\hat{A}_s,-n)
      -
      \tfrac{ 1 }{ 2 }
      \| \hat{B}_s \|^2
    \right] 
    ds
    +
    \smallint_0^{ \tau }
    \hat{B}_s \,
    dW_s
  \right)
  X_0
  \quad
  \P\text{-a.s.}
\label{eq:c:stochastic_Gronwall3}
\end{equation}
for all $n\in\N$.
Now the monotone convergence theorem shows
\begin{equation}  \begin{split}
  \lim_{n\to\infty}\int_0^{\tau}\max(\hat{A}_s,-n)\,ds
  &=
  \int_0^{\tau}\max(\hat{A}_s,0)\,ds
  -\lim_{n\to\infty}\int_0^{\tau}\min(\max(-\hat{A}_s,0),n)\,ds
  \\&
  =
  \int_0^{\tau}\max(\hat{A}_s,0)\,ds
  -\int_0^{\tau}\max(-\hat{A}_s,0)\,ds
  =
  \int_0^{\tau}\hat{A}_s\,ds.
\end{split}     \end{equation}
So letting $n\to\infty$ on the right-hand side
of~\eqref{eq:c:stochastic_Gronwall3} yields
\begin{equation}
  X_{ \tau }
\leq
  \exp\!\left(
    \smallint_0^{ \tau }
    \left[
      \hat{A}_s
      -
      \tfrac{ 1 }{ 2 }
      \| \hat{B}_s \|^2
    \right] 
    ds
    +
    \smallint_0^{ \tau }
    \hat{B}_s \,
    dW_s
  \right)
  X_0
  \quad
  \P\text{-a.s.}
\label{eq:c:stochastic_Gronwall4}
\end{equation}
This proves~\eqref{eq:c:stochastic_Gronwall}.
For the remainder of the proof, we assume that
$
  \mathbbm{1}_{
    \{ t < \tau \}
  }
  A_{ t  } 
  =
  \mathbbm{1}_{
    \{ t < \tau \}
  }  
  \hat{A}_t X_t
$
$\P$-a.s.~for Lebesgue-almost all $t\in[0,T]$.
Define a sequence of stopping times
$\rho_{n}\colon\Omega\to[0,T]$, $n\in\N$,
by
$\rho_n
:=
  \inf\!\big(\{\tau\}\{t\in[0,T]\colon \int_0^t|\hat{A}_s|\,ds\geq n\}\big)$
for all $n\in\N$.
Then left-continuity of the function
$[0,T]\ni t\mapsto\int_0^t|\hat{A}_s|\,ds\in[0,\infty]$
implies that $\int_0^{\rho_n}|\hat{A}_s|\,ds\leq n$ for all $n\in\N$.
Consequently,
Lemma~\ref{lem:stochastic_Gronwall} shows
\begin{equation}
  X_{ \rho_n }
=
  \exp\!\left(
    \smallint_0^{ \rho_n }
    \left[
      \hat{A}_s
      -
      \tfrac{ 1 }{ 2 }
      \| \hat{B}_s \|^2
    \right] 
    ds
    +
    \smallint_0^{ \rho_n }
    \hat{B}_s \,
    dW_s
  \right)
  X_0
  \qquad
  \P\text{-a.s.}
\label{eq:c:stochastic_Gronwall5}
\end{equation}
for all $n\in\N$.
Now the assumption that $ X_{ \tau } \geq 0 $ $ \P $-a.s.,
inequality~\eqref{eq:c:stochastic_Gronwall4}
and
$\int_0^{\tau}\|\hat{B}_s\|^2\,ds+|\int_0^\tau\hat{B}_s \, dW_s|<\infty$
$\P$-a.s.\ imply
\begin{equation}
\begin{split}
 0
 & \leq 
 X_\tau\1_{\{\int_0^{\tau}\hat{A}_s\,ds=-\infty\}}
\\ & 
  \leq
  \1_{\{\int_0^{\tau}\hat{A}_s\,ds=-\infty\}}
  \exp\!\left(
    \smallint_0^{ \tau }
    \left[
      \hat{A}_s
      -
      \tfrac{ 1 }{ 2 }
      \| \hat{B}_s \|^2
    \right] 
    ds
    +
    \smallint_0^{ \tau }
    \hat{B}_s \,
    dW_s
  \right)
  X_0
  =0
\label{eq:c:stochastic_Gronwall6}
\end{split}
\end{equation}
$\P$-a.s.
Moreover, it follows from 
\begin{equation*}
  \mathbbm{1}_{
    \{
      \int_0^{ \tau }
      \hat{A}_s \, ds
      > - \infty
    \}
  }
=
  \mathbbm{1}_{
    \{
      \int_0^{ \tau }
      | \hat{A}_s | 
      \, ds
      < \infty
    \}
  }
=
  \mathbbm{1}_{
    (
      \cup_{ n \in \N }
      \{
        \rho_n = \tau
      \}
    )
  }
=
  \lim_{ n \to \infty }
  \left(
    \mathbbm{1}_{
      \{
        \rho_n = \tau
      \}
    }
  \right)
\quad \P\text{-a.s.}
\end{equation*}
and
from~\eqref{eq:c:stochastic_Gronwall5} that
\begin{equation}
\begin{split}
&
   X_\tau
   \1_{\{\int_0^{\tau}\hat{A}_s\,ds>-\infty\}}
 =
   \lim_{n\to\infty}
   \left(
     X_{ \tau }
     \1_{\{\rho_n=\tau\}}
   \right)
 =
   \lim_{n\to\infty}
   \left(
     X_{\rho_n}
     \1_{\{\rho_n=\tau\}}
   \right)
\\ & =
   \lim_{ n \to \infty }
   \left[
     \1_{\{\rho_n=\tau\}}
  \exp\!\left(
    \smallint_0^{ \rho_n }
    \left[
      \hat{A}_s
      -
      \tfrac{ 1 }{ 2 }
      \| \hat{B}_s \|^2
    \right] 
    ds
    +
    \smallint_0^{ \rho_n }
    \hat{B}_s \,
    dW_s
  \right)
  X_0
   \right]
\\ & =
  \1_{\{\int_0^{\tau}\hat{A}_s\,ds>-\infty\}}
  \exp\!\left(
    \smallint_0^{ \tau }
    \left[
      \hat{A}_s
      -
      \tfrac{ 1 }{ 2 }
      \| \hat{B}_s \|^2
    \right] 
    ds
    +
    \smallint_0^{ \tau }
    \hat{B}_s \,
    dW_s
  \right)
  X_0
\end{split}
\label{eq:c:stochastic_Gronwall7}
\end{equation}
$\P$-a.s.
Combining~\eqref{eq:c:stochastic_Gronwall6}
and~\eqref{eq:c:stochastic_Gronwall7}
finishes the proof of Corollary~\ref{c:stochastic_Gronwall2}.
\end{proof}

\begin{lemma}[An identity for Lyapunov-type functions]
\label{lem:Lyapunov2}
Assume the setting in Section~\ref{sec:setting},
let $ V \in C^2( O, \R ) $,
let $ \tau \colon \Omega \to [0,T] $ be a stopping time,
let	
$ X \colon [0,T] \times \Omega \to O $
be an adapted stochastic process
with continuous sample paths
satisfying
$
  V( X_{ t \wedge \tau } )
  \geq 0
$
$ \P $-a.s.,
\begin{equation}
\label{eq:one_solution_integral_finite}
  \int_0^{ \tau }
    \| \mu( X_s ) \| 
    +
    \| 
      \sigma( X_s )
    \|^2
    +
    \max\!\left(
      \tfrac{ ( \mathcal{G}_{\mu,\sigma} V )(X_s)
      }{ V(X_s) 
      },
      0
    \right)
    +
    \tfrac{
      \|
          ( G_{ \sigma } V )( X_s )
      \|^2
    }{
      ( 
        V( X_s )
      )^2
    }
  \,
  ds
  < \infty
\end{equation}
$\P$-a.s.\ and
$
  X_{ t \wedge \tau } = 
  X_0
  + \int_0^{ t \wedge \tau } \mu( X_s ) \, ds
  +
  \int_0^{ t \wedge \tau } \sigma( X_s ) \, dW_s
$
$ \P $-a.s.\ for all 
$ t \in [0,T] $.
Then
\begin{equation}
\label{eq:Lyapunov2}
  V( X_{ \tau } ) 
\leq
  V( X_0 )
  \,
  \exp\!\left(
    \smallint_0^{ \tau }
    \left[
    \tfrac{(\mathcal{G}_{\mu,\sigma}V)(X_s)}{V(X_s)}
    -
    \tfrac{
      \|
          ( G_{ \sigma } V )( X_s )
      \|^2
    }{
      2 \,
      ( 
        V( X_s )
      )^2
    }
    \right]
    ds
    +
    \smallint_0^{ \tau }
    \tfrac{
      ( G_{ \sigma } V )( X_s )
    }{
      V( X_s )
    }
    \,
    dW_s\!
  \right)
\end{equation}
$\P$-a.s. If, in addition,
   $
     ( V^{ - 1 } )( 0 )
     \subseteq 
     (
       \mathcal{G}_{\mu, \sigma }V
     )^{ - 1 }([0,\infty))
   $, then equality holds in~\eqref{eq:Lyapunov2}.
\end{lemma}

\begin{proof}[Proof
of Lemma~\ref{lem:Lyapunov2}]
Path continuity
together with
$
  V \in C^2(O,\R)
$
implies that
$\smallint\nolimits_0^{\tau}
|(\mathcal{G}_{\mu,\sigma}V)(X_s)|
+\|(G_{\sigma}V)(X_s)\|^2\,ds<\infty$
$\P$-a.s.
Next the inequality
$a<0=\tfrac{a}{0}\cdot 0$ for all $a\in(-\infty,0)$
implies
\begin{equation*}
\left((\mathcal{G}_{\mu,\sigma}V)(X_t)
  -\tfrac{(\mathcal{G}_{\mu,\sigma}V)(X_t)}{V(X_t)}V(X_t)\right)
  \1_{ \{ t < \tau \} }\1_{\{(\mathcal{G}_{\mu,\sigma}V)(X_t)< 0\}}
 \leq 0 
 \quad
 \P\text{-a.s.}
\end{equation*}
for all $t\in[0,T]$.
Moreover,
Assumption~\eqref{eq:one_solution_integral_finite}
yields that
\begin{equation}
  \int_0^{\tau}\big((\mathcal{G}_{\mu,\sigma}V)(X_s)
                   \1_{\{(\mathcal{G}_{\mu,\sigma}V)(X_s)\geq 0\}}
                   +\|(G_{\sigma}V)(X_s)\|^2
              \big)
              \1_{\{V(X_s)=0\}}
             \,ds=0
\end{equation}
$\P$-a.s.\ and, consequently, that
\begin{equation} \label{eq:division.doesnt.matter}
\begin{split}
  \int_0^{\tau}
    \left|
      (\mathcal{G}_{\mu,\sigma}V)(X_s)
    -\tfrac{(\mathcal{G}_{\mu,\sigma}V)(X_s)}{V(X_s)}V(X_s)
   \right|
    \1_{\{(\mathcal{G}_{\mu,\sigma}V)(X_s)\geq 0\}}
  \,ds
\\  
  +
  \int_{0}^{\tau}
    \left\|
      G_{\sigma}V)(X_s)
      -\tfrac{(G_{\sigma}V)(X_s)}{V(X_s)}V(X_s)
    \right\|^2
  \,ds
& =0
\end{split}
\end{equation}
$\P$-a.s.
This together with  It\^{o}'s formula 
results in
\begin{equation}
\label{eq:ito_formel}
\begin{split}
  V( X_{ t \wedge \tau } )
& =
  V( X_0 )
  +
  \int_0^{ t\wedge\tau }
  ( \mathcal{G}_{ \mu, \sigma } V)( X_s )
  \, ds
  +
  \int_0^{ t \wedge\tau}
  ( G_{ \sigma } V)( X_s )
  \, dW_s
\\ & =
  V( X_{ 0  } )
  +
  \int_0^{ t\wedge\tau }
  ( \mathcal{G}_{ \mu, \sigma } V)( X_s )
  \, ds
  +
  \int_0^{ t \wedge\tau}
  \tfrac{
    ( G_{ \sigma } V)( X_s )
  }{ 
    V( X_s ) 
  }
  \cdot
  V( X_{ s  } )
  \, dW_s
\end{split}
\end{equation}
$ \P $-a.s.\ for all $ t \in [0,T] $.
Moreover,~\eqref{eq:division.doesnt.matter}
and Fubini's theorem also imply that
\begin{equation}  \begin{split}
  0 & =
  \E\!\left[\int_0^{\tau}
         \left|
	    (\mathcal{G}_{\mu,\sigma}V)(X_s)
           -\tfrac{(\mathcal{G}_{\mu,\sigma}V)(X_s)}{V(X_s)}V(X_s)
         \right|
           \1_{\{(\mathcal{G}_{\mu,\sigma}V)(X_s)\geq 0\}}
          \,ds
      \right]
  \\&
  =\int_0^{T}
      \E\!\left[
	\left|
	  (\mathcal{G}_{\mu,\sigma}V)(X_t)
          -\tfrac{(\mathcal{G}_{\mu,\sigma}V)(X_t)}{V(X_t)}V(X_t)
        \right|
        \1_{ \{ t < \tau \} }
        \1_{\{(\mathcal{G}_{\mu,\sigma}V)(X_t)\geq 0\}}
      \right]
    \,dt
\end{split}     \end{equation}
so that 
$
\1_{ \{ t < \tau \} }\big((\mathcal{G}_{\mu,\sigma}V)(X_t)
  -\tfrac{(\mathcal{G}_{\mu,\sigma}V)(X_t)}{V(X_t)}V(X_t)\big)
  \1_{ \{(\mathcal{G}_{\mu,\sigma}V)(X_t)\geq 0\}}
 =0$
$\P$-a.s.\ for Lebesgue-almost all $t\in[0,T]$.
Combining these observations then shows that
$
\1_{ \{ t < \tau \} }\big((\mathcal{G}_{\mu,\sigma}V)(X_t)
  -\tfrac{(\mathcal{G}_{\mu,\sigma}V)(X_t)}{V(X_t)}V(X_t)\big)
 \leq 0$
$\P$-a.s.\ for Lebesgue-almost all $t\in[0,T]$.
Applying Corollary~\ref{c:stochastic_Gronwall2}
to the stochastic process
$
  V( X_{ t } )
$,
$ t \in [0,T] $,
and to the stopping time $\tau$
together with
Assumption~\eqref{eq:one_solution_integral_finite}
yields that
\begin{equation}
\label{eq:Lyapunov2Y}
  V( X_{ \tau } ) 
\leq
  V( X_0 )
  \,
  \exp\!\left(
    \smallint_0^{ \tau }
    \left[
    \tfrac{(\mathcal{G}_{\mu,\sigma}V)( X_s )}{ V( X_s ) }
    -
    \tfrac{
      \|
          ( G_{ \sigma } V )( X_s )
      \|^2
    }{
      2 \,
      ( 
        V( X_s )
      )^2
    }
    \right]
    ds
    +
    \smallint_0^{ \tau }
    \tfrac{
      ( G_{ \sigma } V )( X_s )
    }{
      V( X_s )
    }
    \,
    dW_s
  \right)
\end{equation}
$\P$-a.s. This
proves
inequality~\eqref{eq:Lyapunov2}.
The additional assumption
   $
     ( V^{ - 1 } )(0)
     \subseteq 
     (
       \mathcal{G}_{\mu, \sigma }V
     )^{ - 1 }([0,\infty))
   $
implies
$
  \1_{ \{ t < \tau \} }\big((\mathcal{G}_{\mu,\sigma}V)( X_t )
  -
  \tfrac{
    (\mathcal{G}_{\mu,\sigma} V)( X_t )
  }{
    V( X_t ) 
  }
    V(X_t)
  \big)
 = 0$
$\P$-a.s.\ for  all $t\in[0,T]$.
In that case,
Corollary~\ref{c:stochastic_Gronwall2}
yields
equality in~\eqref{eq:Lyapunov2}.
This completes the proof
of Lemma~\ref{lem:Lyapunov2}.
\end{proof}

\section{Two solution approach}
\label{sec:two_solution}
In this subsection, we apply 
Lemma~\ref{lem:Lyapunov2} to the bi-variate process of two solutions
of the same SDE
(see Proposition~\ref{prop:two_solution} below).
From this, we then derive marginal and uniform estimates for this bi-variate 
process
(see respectively Propositions~\ref{prop:two_solution_supoutside}
and~\ref{prop:two_solution_supinside} below).

The next remark illustrates the relation between the extended generator
defined in~\eqref{eq:extended_drift}
and the ``standard'' generator defined in~\eqref{eq:drift_operator}
and between the extended noise operator defined
in~\eqref{eq:extended_diffusion}
and the ``standard'' noise operator in~\eqref{eq:diffusion_operator}.
The proof of Remark~\ref{rem:relation} is clear and therefore
omitted.

\begin{remark}[Relation between generator and extended generator
and between noise operator and extended noise operator]
\label{rem:relation}
Let $ d, m \in \N $, let $ O \subseteq\R^d $ be an open set, let $ \mu \colon O 
\to \R^d $
and $ \sigma \colon O \to \R^{ d \times m } $ be functions and define functions
$ \overline{ \mu } \colon O^2 \to \R^{ 2 d } $
and
$ \overline{ \sigma } \colon O^2 \to \R^{ ( 2 d ) \times m } $
by
$
  \overline{ \mu }( x, y ) = ( \mu(x), \mu( y ) )
$
and by
$
  \overline{ \sigma }( x, y ) u = ( \sigma( x ) u, \sigma( y ) u )
$
for all $ x,y \in O $ and all $ u \in \R^m $.
Then
$
  \overline{\mathcal{G}}_{ \mu, \sigma }
  =
  \mathcal{G}_{ \overline{\mu}, \overline{\sigma} }
$
and
$
  \overline{G}_{ \sigma }
  =
  G_{ \overline{\sigma} }
$.
\end{remark}

Using both the extended generator and the extended noise operator,
we now establish in Proposition~\ref{prop:two_solution} 
an elementary identity which is crucial 
for the results developed in this article.
Proposition~\ref{prop:two_solution}
follows immediately 
from Remark~\ref{rem:relation}
and Lemma~\ref{lem:Lyapunov2}.
Proposition~\ref{prop:two_solution}
generalizes a relation on page 1935 in Zhang~\cite{Zhang2010}.

\begin{prop}[Two solution approach]
\label{prop:two_solution}
Assume the setting in Section~\ref{sec:setting},
let $ \tau \colon \Omega \to [0,T] $
be a stopping time,
let $ V \in C^{ 2 }( O^2, \R ) $,
let	
$ X^i \colon [0,T] \times \Omega \to O $,
$ i \in \{ 1, 2 \} $,
be adapted
stochastic processes
with continuous sample paths
satisfying
\begin{equation}
  \int_0^{ \tau }
 \max\left(
  \tfrac{
    (\overline{\mathcal{G}}_{\mu,\sigma}V)(X_s^1,X_s^2)
  }{
    V(X_s^1,X_s^2)},0
 \right)
    +
    \| \mu( X^i_s ) \| 
    +
    \| 
      \sigma( X^i_s )
    \|^2
    +
    \tfrac{
      \|
          ( \overline{ G }_{ \sigma } V )( X^1_s, X^2_s )
      \|^2
    }{
      ( 
        V( X^1_s, X^2_s ) 
      )^2
    }
  \,
  ds
  < \infty
\end{equation}
$ \P $-a.s.,
$
  V( X_{ t \wedge \tau }^1, X_{ t \wedge \tau }^2 )
  \geq 0
$
$ \P $-a.s.\ and
$
  X^i_{ t \wedge \tau } = 
  X^i_0
  + \int_0^{ t \wedge \tau } \mu( X^i_s ) \, ds
  +
  \int_0^{ t \wedge \tau } \sigma( X^i_s ) \, dW_s
$
$ \P $-a.s.\ for all 
$ (t,i) \in [0,T] \in \{ 1, 2 \} $.
Then
\begin{align}
\label{eq:V_identity}
  V( X^1_{ \tau }, X^2_{ \tau } )
& \leq 
  V( X^1_0, X^2_0 )
  \exp\!\left(
    \int_{ 0 }^{ \tau }
    \left[
    \tfrac{(\overline{\mathcal{G}}_{\mu,\sigma}V)(X_s^1,X_s^2)}{V(X_s^1,X_s^2)}
    -
    \tfrac{
      \|
          ( \overline{ G }_{ \sigma } V )( X^1_s, X^2_s )
      \|^2
    }{
      2
      \left( 
        V( X^1_s, X^2_s )
      \right)^2
    }
    \right]
    ds
  \right)
\nonumber
\\ &  \quad
  \cdot
  \exp\!\left(
    \int_0^{ \tau }
    \tfrac{
      ( \overline{ G }_{ \sigma } V )( X^1_s, X^2_s )
    }{
      V( X^1_s, X^2_s )
    }
    \,
    dW_s
  \right)
\end{align}     
$\P$-a.s. If, in addition,
   $
     ( V^{ - 1 } )(0)
     \subseteq 
     (
       \overline{\mathcal{G}}_{\mu, \sigma }V
     )^{ - 1 }([0,\infty))
   $, then equality holds in~\eqref{eq:V_identity}.
\end{prop}

\subsection{Calculations for the extended noise 
operator and the extended generator}

Lemma~\ref{lem:extended_drift} below provides the extended generator and the 
extended noise operator
applied to a suitable class of twice continuously differentiable 
functions. Examples~\ref{ex:phi_identity} and~\ref{ex:vola} below demonstrate 
the use of this result.

\begin{lemma}
\label{lem:extended_drift}
Let 
$ d, m, k \in \N $, 
$ p \in [2,\infty) $,
let $ O \subseteq\R^d $
be an open set, let 
$
  \mu \in \mathcal{L}^0( O ; \R^d )
$,
$
  \sigma \in \mathcal{L}^0( O ; \R^{ d \times m } ) 
$,
$
  \Phi \in C^2( O, \R^k )
$
and 
let $ V \colon O^2 \to \R $
be given by
$
  V(x,y)
=
  \left\| 
    \Phi( x ) - \Phi( y ) 
  \right\|^p
$
for all $ x, y \in O $.
Then
$ V \in C^2( O^2, \R ) $
and
\begin{equation}
\label{eq:extended_diffusion_example}
  \frac{
  \left\|
      ( \overline{G}_{ \sigma } V)( x, y )
  \right\|^2
  }{
    | V(x,y) |^2
  }
=
  \frac{
  p^2
  \sum_{ i = 1 }^m
  \left|
  \left<
    \Phi( x ) - \Phi( y ) ,
    \Phi'(x) \, \sigma_i( x ) - 
    \Phi'(y) \, \sigma_i( y )
  \right>
  \right|^2
  }{
    \left\| \Phi( x ) - \Phi( y ) \right\|^4
  }
  \qquad
  \text{and}
\end{equation}
\begin{equation}
\label{eq:extended_drift_example}
\begin{split}
  \frac{
    ( \overline{ \mathcal{G} }_{ \mu, \sigma } V)( x, y )
  }{
    V( x, y )
  }
& =
  \frac{ 
    p
    \left<
      \Phi( x ) - \Phi( y ) ,
      \Phi'(x) \, \mu( x ) - 
      \Phi'(y) \, \mu( y )
    \right>
  }{
    \left\|
      \Phi( x ) - \Phi( y )
    \right\|^2
  }
\\ & \quad
  +
  \frac{
  (p - 2)
  \left\|
      ( \overline{G}_{ \sigma } V)( x, y )
  \right\|^2
  }{
    2 \, p \,
    | V(x,y) |^2
  }
\\ & \quad  
  +
  \frac{ 
    p 
    \sum\limits_{ i = 1 }^m
    \left< \Phi(x) - \Phi(y) , 
      \Phi''(x)\big( \sigma_i( x ), \sigma_i(x) )
      -
      \Phi''(y)\big( \sigma_i( y ), \sigma_i(y) )
    \right>
  }{
    2
    \left\| \Phi(x) - \Phi(y) \right\|^2
  }
\\ & \quad
  +
  \frac{
    p 
    \1_{ \{ \Phi(x) \neq \Phi(y) \} }
    \left\| \Phi'(x) \, \sigma( x ) - \Phi'(y) \, \sigma(y) 
    \right\|^2_{ \HS( \R^m, \R^k ) }
  }{
    2
    \left\|
      \Phi( x ) - \Phi( y )
    \right\|^2
  }
\end{split}
\end{equation}
for all $ x, y \in O $.
\end{lemma}

\begin{remark}[Derivatives of powers of the norm function]
\label{remark:normd}
Let $ p \in [2,\infty) $ and let
$ 
  ( H, \left< \cdot, \cdot \right>_H, \left\| \cdot \right\|_H ) 
$ 
be an $ \R $-Hilbert space.
Then the function
$ 
  F \colon H  
  \rightarrow \mathbb{R} 
$
given by
$
  F(x) = \| x \|^p
$
for all $ x \in H $
is twice continuously differentiable 
and fulfills
\begin{equation}
\begin{split}
  F'(x)(v) 
& 
  = p \left\| x \right\|^{ (p-2) }
  \left< x, v \right> ,
\\[1ex]
  F''(x)(v,w) 
& 
  =
    p \left\| x \right\|^{ (p - 2) }
    \left< v, w \right>
    +
    p \left( p - 2 \right)
    \left\| x \right\|^{ (p - 4) }
    \left< x, v \right>
    \left< x, w \right>
\end{split}
\end{equation}
for all $ x, v, w \in H $.
In particular, the function
$ 
  G \colon H^2
  \rightarrow
  \mathbb{R}
$
given by
$
  G( x, y )
  =
  \| x - y \|^p
$
for all $ x, y \in H $
is twice continuously differentiable and fulfills
\begin{equation}  \begin{split}
&  
  \Big(
    \big(
      \tfrac{ \partial 
      }{ 
        \partial x  
      }
      G
    \big)(x,y)
  \Big)(v)
  =
  -\Big(
    \big(
      \tfrac{ \partial 
      }{ 
        \partial y  
      }
      G
    \big)(x,y)
  \Big)(v)
=
  p \left\| x - y \right\|^{ (p-2) }
  \left< x - y, 
    v
  \right> ,
\\
&
  \Big(
    \big(
      \tfrac{ \partial^2 
      }{ 
        \partial x^2 
      }
      G
    \big)(x,y)
  \Big)(v,w)
  =
  \Big(
    \big(
      \tfrac{ \partial^2 
      }{ 
        \partial y^2 
      }
      G
    \big)(x,y)
  \Big)(v,w)
  =
  -
  \Big(
    \big(
      \tfrac{ \partial 
      }{ 
        \partial y 
      }
      \tfrac{ \partial
      }{ 
        \partial x 
      }
      G
    \big)(x,y)
  \Big)(v,w)
\\ & \quad =
  p 
  \left\| x - y \right\|^{ (p-2) }
  \left< v, 
    w
  \right>
  + 
  p \left( p - 2 \right)
  \left\| x - y \right\|^{ (p - 4) }
  \left< x - y, 
    v
  \right>
  \left< x - y, 
    w
  \right> 
\end{split}     \end{equation}
for all 
$ x, y, v, w \in H $.
\end{remark}

\begin{proof}[Proof
of Lemma~\ref{lem:extended_drift}]
First, note that the chain rule together with Remark~\ref{remark:normd} shows 
that
$ V \in C^2( O^2, \R ) $.
Next observe that Remark~\ref{remark:normd} implies that
\begin{align}
  ( \overline{ \mathcal{G} }_{ \mu, \sigma } V)( x, y )
& =
  p
  \left\|
    \Phi( x ) - \Phi( y )
  \right\|^{ ( p - 2 ) }
  \left<
    \Phi( x ) - \Phi( y ) ,
    \Phi'(x) \, \mu( x ) - 
    \Phi'(y) \, \mu( y )
  \right>
\nonumber
\\ & \quad
  +
  \frac{ p }{ 2 }
  \left\| \Phi(x) - \Phi(y) \right\|^{ (p - 2) }
\nonumber
\\ & \quad \quad 
  \cdot
  \sum_{ i = 1 }^m
  \left< \Phi(x) - \Phi(y) , 
    \Phi''(x)\big( \sigma_i( x ), \sigma_i(x) )
    -
    \Phi''(y)\big( \sigma_i( y ), \sigma_i(y) )
  \right>
\nonumber
\\ & \quad 
  +
  p \left\| \Phi( x ) - \Phi( y ) \right\|^{ ( p - 2 ) }
\nonumber
\\ & \quad \quad 
  \cdot
  \sum_{ i = 1 }^m
  \Big[ 
    \tfrac{
      \| \Phi'(x) \, \sigma_i( x ) \|^2
      +
      \| \Phi'(y) \, \sigma_i( y ) \|^2
    }{ 2 }
    - 
    \left< \Phi'(x) \, \sigma_i( x ), \Phi'(y) \, \sigma_i( y ) \right>
  \Big]
\nonumber
\\ & \quad  
  +
  \frac{
    p \left( p - 2 \right)
  }{ 2 }
  \left\| \Phi( x ) - \Phi( y ) \right\|^{ ( p - 4 ) }
  \sum_{ i = 1 }^m
  \left|
    \left< \Phi( x ) - \Phi( y ), \Phi'(x) \, \sigma_i( x ) \right>
  \right|^2
\\ & \quad 
  +
  \frac{
    p \left( p - 2 \right)
  }{ 2 }
  \left\| \Phi( x ) - \Phi( y ) \right\|^{ ( p - 4 ) }
  \sum_{ i = 1 }^m
  \left|
    \left< \Phi( x ) - \Phi( y ), \Phi'(y) \, \sigma_i( y ) \right>
  \right|^2
\nonumber
\\ & \quad
  -
    p \left( p - 2 \right)
  \left\| \Phi( x ) - \Phi( y ) \right\|^{ ( p - 4 ) }
\nonumber
\\ & \quad \quad 
  \cdot
  \sum_{ i = 1 }^m
  \left< \Phi( x ) - \Phi( y ), \Phi'(x) \, \sigma_i( x ) \right>
  \left< \Phi( x ) - \Phi( y ), \Phi'(y) \, \sigma_i( y ) \right>
\nonumber
\end{align}
and
\begin{equation}
\begin{split}
&
  ( \overline{G}_{ \sigma } V)( x, y )
=
  p
  \left\|
    \Phi( x ) - \Phi( y )
  \right\|^{ ( p - 2 ) }
  \left(
    \Phi( x ) - \Phi( y )
  \right)^*
  \left(
    \Phi'(x) \, \sigma( x ) - 
    \Phi'(y) \, \sigma( y )
  \right)
\end{split}
\end{equation}
for all $ x, y \in O $.
This shows that
\begin{equation}
\begin{split}
  ( \overline{ \mathcal{G} }_{ \mu, \sigma } V)( x, y )
&  =
  p
  \left\|
    \Phi( x ) - \Phi( y )
  \right\|^{ ( p - 2 ) }
  \left<
    \Phi( x ) - \Phi( y ) ,
    \Phi'(x) \, \mu( x ) - 
    \Phi'(y) \, \mu( y )
  \right>
\\ & \quad +
  \frac{ p }{ 2 }
  \left\| \Phi(x) - \Phi(y) \right\|^{ (p - 2) }
\\ & \quad \quad
  \cdot
  \sum_{ i = 1 }^m
  \left< \Phi(x) - \Phi(y) , 
    \Phi''(x)\big( \sigma_i( x ), \sigma_i(x) )
    -
    \Phi''(y)\big( \sigma_i( y ), \sigma_i(y) )
  \right>	
\\ & \quad +
  \frac{ p }{ 2 }
  \left\| \Phi( x ) - \Phi( y ) \right\|^{ ( p - 2 ) }
  \left\| \Phi'(x) \, \sigma( x ) - \Phi'(y) \, \sigma(y) 
  \right\|^2_{ \HS( \R^m, \R^k ) }
\\ & \quad +
  \frac{
    p \left( p - 2 \right)
  }{ 2 }
  \left\| \Phi( x ) - \Phi( y ) \right\|^{ ( p - 4 ) }
\\ & \quad \quad 
  \cdot
  \sum_{ i = 1 }^m
  \left|
    \left< 
      \Phi( x ) - \Phi( y ), 
      \Phi'(x) \, \sigma_i( x ) -
      \Phi'(y) \, \sigma_i( y ) 
    \right>
  \right|^2
\end{split}
\end{equation}
and
\begin{equation}
\begin{split}
&
  \left\|
    ( \overline{G}_{ \sigma } V)( x, y )
  \right\|^2_{ \HS( \R^m, \R ) }
\\
& =
  p^2
  \left\|
    \Phi( x ) - \Phi( y )
  \right\|^{ ( 2 p - 4 ) }
  \left[ 
  \sum_{ i = 1 }^m
  \left|
  \left<
    \Phi( x ) - \Phi( y ) ,
    \Phi'(x) \, \sigma_i( x ) - 
    \Phi'(y) \, \sigma_i( y )
  \right>
  \right|^2
  \right]
\end{split}
\end{equation}
for all $ x, y \in O $.
Hence, we obtain that
\begin{align*}
&
  \frac{
    ( \overline{ \mathcal{G} }_{ \mu, \sigma } V)( x, y )
  }{
    V( x, y )
  }
=
  \frac{ 
    p
    \left<
      \Phi( x ) - \Phi( y ) ,
      \Phi'(x) \, \mu( x ) - 
      \Phi'(y) \, \mu( y )
    \right>
  }{
    \left\|
      \Phi( x ) - \Phi( y )
    \right\|^2
  }
\nonumber
\\ & \quad +
  \frac{ 
    p 
    \sum_{ i = 1 }^m
    \left< \Phi(x) - \Phi(y) , 
      \Phi''(x)\big( \sigma_i( x ), \sigma_i(x) )
      -
      \Phi''(y)\big( \sigma_i( y ), \sigma_i(y) )
    \right>
  }{
    2
    \left\| \Phi(x) - \Phi(y) \right\|^2
  }
\nonumber
\\ & \quad +
  \frac{
    p 
    \left\|
      \Phi( x ) - \Phi( y )
    \right\|^{ ( p - 2 ) }
    \left\| \Phi'(x) \, \sigma( x ) - \Phi'(y) \, \sigma(y) 
    \right\|^2_{ \HS( \R^m, \R^k ) }
  }{
    2
    \left\|
      \Phi( x ) - \Phi( y )
    \right\|^p
  }
\\ & \quad +
  \frac{
    p \left( p - 2 \right)
  \sum_{ i = 1 }^m
  \left\| \Phi( x ) - \Phi( y ) \right\|^{ ( p - 4 ) }
  \left|
    \left< 
      \Phi( x ) - \Phi( y ), 
      \Phi'(x) \, \sigma_i( x ) -
      \Phi'(y) \, \sigma_i( y ) 
    \right>
  \right|^2
  }{
    2
    \left\| \Phi( x ) - \Phi( y ) \right\|^p
  }
\nonumber
\end{align*}
and
\begin{align}
&
  \frac{
  \left\|
      ( \overline{G}_{ \sigma } V)( x, y )
  \right\|^2
  }{
    | V(x,y) |^2
  }
=
  \frac{
  \left\|
    ( \overline{G}_{ \sigma } V)( x, y )
  \right\|^2_{ \HS( \R^m, \R ) }
  }{
    | V(x,y) |^2
  }  
\nonumber
\\ & =
  \frac{
  p^2
  \left\|
    \Phi( x ) - \Phi( y )
  \right\|^{ ( 2 p - 4 ) }
  \left[ 
  \sum\limits_{ i = 1 }^m
  \left|
  \left<
    \Phi( x ) - \Phi( y ) ,
    \Phi'(x) \, \sigma_i( x ) - 
    \Phi'(y) \, \sigma_i( y )
  \right>
  \right|^2
  \right]
  }{
    \left\| \Phi( x ) - \Phi( y ) \right\|^{ 2 p }
  }
\\ & =
  \frac{
  p^2
  \sum\limits_{ i = 1 }^m
  \left|
  \left<
    \Phi( x ) - \Phi( y ) ,
    \Phi'(x) \, \sigma_i( x ) - 
    \Phi'(y) \, \sigma_i( y )
  \right>
  \right|^2
  }{
    \left\| \Phi( x ) - \Phi( y ) \right\|^4
  }
\nonumber
\end{align}
and therefore
\begin{align}
  \frac{
    ( \overline{ \mathcal{G} }_{ \mu, \sigma } V)( x, y )
  }{
    V( x, y )
  }
& =
  \frac{ 
    p
    \left<
      \Phi( x ) - \Phi( y ) ,
      \Phi'(x) \, \mu( x ) - 
      \Phi'(y) \, \mu( y )
    \right>
  }{
    \left\|
      \Phi( x ) - \Phi( y )
    \right\|^2
  } \nonumber
\\ & \quad +
  \frac{ 
    p 
    \sum\limits_{ i = 1 }^m
    \left< \Phi(x) - \Phi(y) , 
      \Phi''(x)\big( \sigma_i( x ), \sigma_i(x) )
      -
      \Phi''(y)\big( \sigma_i( y ), \sigma_i(y) )
    \right>
  }{
    2
    \left\| \Phi(x) - \Phi(y) \right\|^2
  } \nonumber
\\ & \quad +
  \frac{
    p 
    \left\|
      \Phi( x ) - \Phi( y )
    \right\|^{ ( p - 2 ) }
    \left\| \Phi'(x) \, \sigma( x ) - \Phi'(y) \, \sigma(y) 
    \right\|^2_{ \HS( \R^m, \R^k ) }
  }{
    2
    \left\|
      \Phi( x ) - \Phi( y )
    \right\|^p
  }
\\ & \quad
  +
  \frac{
  (p - 2)
  \left\|
      ( \overline{G}_{ \sigma } V)( x, y )
  \right\|^2
  }{
    2 \, p \,
    | V(x,y) |^2
  } \nonumber
\end{align}
for all $ x, y \in O $.
This completes the proof
of Lemma~\ref{lem:extended_drift}.
\end{proof}

Next we illustrate 
\eqref{eq:extended_diffusion_example}
and
\eqref{eq:extended_drift_example} 
in Lemma~\ref{lem:extended_drift}
by two simple corollaries, 
Example~\ref{ex:phi_identity}
and Example~\ref{ex:vola}.
Example~\ref{ex:phi_identity}
is the special case
of Lemma~\ref{lem:extended_drift}
where
$ \Phi( x ) = x $ for all $ x \in O $.

\begin{example}
\label{ex:phi_identity}
Let 
$ d, m \in \N $, 
$ p \in [2,\infty) $,
let $ O \subseteq\R^d $
be an open set, let 
$
  \mu \in \mathcal{L}^0( O ; \R^d )
$,
$
  \sigma \in \mathcal{L}^0( O ; \R^{ d \times m } )
$
and
let $ V \colon O^2 \to \R $
be given by
$
  V(x,y)
=
  \left\| 
    x - y 
  \right\|^p
$
for all $ x, y \in O $.
Then
$ V \in C^2( O^2, \R ) $
and 
\begin{equation}  \begin{split}
  \frac{
  \left\|
      ( \overline{G}_{ \sigma } V)( x, y )
  \right\|^2
  }{
    | V(x,y) |^2
  }
&=
  \frac{
    p^2
    \left\|
      (
        \sigma( x ) - \sigma( y )
      )^*
      (
        x - y 
      )
    \right\|^2
  }{
    \left\| x - y \right\|^4
  }
\end{split}
\end{equation}
and
\begin{align}
\nonumber
  \frac{
    ( \overline{ \mathcal{G} }_{ \mu, \sigma } V)( x, y )
  }{
    V( x, y )
  }
& =
  p 
  \left[
  \tfrac{ 
    \left<
      x - y ,
      \mu( x ) - 
      \mu( y )
    \right>
    +
    \frac{ 1 }{ 2 }
    \left\| 
      \sigma( x ) - \sigma(y) 
    \right\|^2_{ \HS( \R^m, \R^d ) }
  }{
    \left\|
      x - y
    \right\|^2
  }
  \right]
  +
  \tfrac{
  (p - 2)
  \left\|
      ( \overline{G}_{ \sigma } V)( x, y )
  \right\|^2
  }{
    2 \, p \,
    | V(x,y) |^2
  }
\\ & \leq
  p 
  \left[
  \tfrac{ 
    \left<
      x - y ,
      \mu( x ) - 
      \mu( y )
    \right>
    +
    \frac{ ( p - 1 ) }{ 2 }
    \left\| 
      \sigma( x ) - \sigma(y) 
    \right\|^2_{ \HS( \R^m, \R^d ) }
  }{
    \left\|
      x - y
    \right\|^2
  }
  \right]
\end{align}
for all $ x, y \in O $.
\end{example}

Example~\ref{ex:vola}
is the special case
of Lemma~\ref{lem:extended_drift}
where
$ O = ( 0, \infty ) \subseteq \R $
and where
$ \Phi( x ) = x^q $
for all $ x \in O $
and some $ q \in \R $.

\begin{example}
\label{ex:vola}
Let $ p \in [2,\infty) $,
$ q \in \R $
and let
$ V \colon ( 0, \infty )^2 \to \R $
be given by 
$ V( x, y ) = \left| x^q - y^q \right|^p $
for all $ x, y \in ( 0, \infty ) $.
Then $ V \in C^2( (0,\infty)^2, \R ) $ and
\begin{equation}
\begin{split}
  \frac{
  \left\|
      ( \overline{G}_{ \sigma } V)( x, y )
  \right\|^2
  }{
    | V(x,y) |^2
  }
&
=
  \frac{
  p^2 q^2
  \left(
    x^{ (q-1) } \, \sigma( x ) - 
    y^{ (q-1) } \, \sigma( y )
  \right)^2
  }{
    \left( x^q - y^q \right)^2
  }
\end{split}
\end{equation}
and
\begin{equation}
\begin{split}
&
  \frac{
    ( \overline{ \mathcal{G} }_{ \mu, \sigma } V)( x, y )
  }{
    V( x, y )
  }
\\
& =
  \frac{ 
    p q
    \left(
      x^{ (q - 1) } \, \mu( x ) - 
      y^{ (q - 1) } \, \mu( y )
      +
      \frac{ ( q - 1 ) }{ 2 }
      \left[ 
        x^{ ( q - 2 ) } ( \sigma( x ) )^2
        -
        y^{ ( q - 2 ) } ( \sigma( y ) )^2
      \right]
    \right)
  }{
    \left(
      x^q - y^q
    \right)
  }
\\ & \quad
  +
  \frac{
  p \left( p - 1 \right) q^2
  \left(
    x^{ (q-1) } \, \sigma( x ) - 
    y^{ (q-1) } \, \sigma( y )
  \right)^2
  }{
    2
    \left( x^q - y^q \right)^2
  }
\end{split}
\end{equation}
for all $ x, y \in (0,\infty) $.
\end{example}

\subsection{Marginal strong stability analysis for solutions of SDEs}
\label{sec:marginal}

\begin{prop}[Marginal strong stability analysis]
\label{prop:two_solution_supoutside}
Assume the setting in Section~\ref{sec:setting},
let $ x, y \in O $,
$ V \in C^{ 2 }( O^2, [0,\infty) ) $,
let $ \tau \colon \Omega \to [0,T] $ be a stopping time
and let	
$ X^z \colon [0,T] \times \Omega \to O $,
$ z \in \{ x, y \} $,
be adapted
stochastic processes
with continuous sample paths
satisfying
\begin{equation*}
  \int_0^{ \tau }
    \| \mu( X^z_s ) \| 
    +
    \| 
      \sigma( X^z_s )
    \|^2
    +
    \max\!\left(    
      \tfrac{
	(\overline{\mathcal{G}}_{\mu,\sigma}V)(X_s^x,X_s^y)
      }{
	V(X_s^x,X_s^y)
      }
      ,0\right)
    +
    \tfrac{
      \|
        ( \overline{ G }_{ \sigma } V )( X^x_s, X^y_s )
      \|^2
    }{
      ( 
        V( X^x_s, X^y_s ) 
      )^2
    }
  \,
  ds
  < \infty
\end{equation*}
$ \P $-a.s.\ 
and
$
  X^z_{ t \wedge \tau } 
= 
  z
  + \int_0^{ t \wedge \tau } \mu( X^z_s ) \, ds
  +
  \int_0^{ t \wedge \tau } \sigma( X^z_s ) \, dW_s
$
$ \P $-a.s.\ for all $ (t,z) \in [0,T] \times \{ x, y \} $.
Then it holds for all $ p,q,r \in (0,\infty] $
with $ \frac{ 1 }{ p } + \frac{ 1 }{ q } = \frac{ 1 }{ r } $
that
\begin{equation}
\label{eq:marginal_estimate}
\begin{split}
&
    \left\|
      V( X^x_{ \tau }, X^y_{ \tau } )
    \right\|_{
      L^r( \Omega; \R )
    }
\\ & 
\leq
  V( x, y ) 
  \left\|
 \exp\!\left(
    \int_0^{ \tau }
      \tfrac{
	(\overline{\mathcal{G}}_{\mu,\sigma}V)(X_s^x,X_s^y)
      }{
	V(X_s^x,X_s^y)
      }
    +
    \tfrac{
      ( p - 1 ) \,
      \| 
        ( \overline{ G }_{ \sigma } V )( X^x_s, X^y_s )
      \|^2
    }{
      2 \,
      ( 
        V( X^x_s, X^y_s )
      )^2
    }
    \,
    ds
 \right)
  \right\|_{ L^q( \Omega; \R ) }
  .
\end{split}
\end{equation}
\sgc{}If, in addition,
  it holds that
$
  \P\big( \int_0^{ \tau }
  \|
      ( \overline{G}_{ \sigma } V )( X^x_s, X^y_s )
  \|^2
  \, ds
$ $  = 0  \big)=1
$ and
   $
     ( V^{ - 1 } )(0)
     \subseteq 
     (
       \overline{\mathcal{G}}_{\mu, \sigma }V
     )^{ - 1 }([0,\infty))
   $\cgs{},
then it holds for all $ r \in ( 0, \infty ] $
that
\begin{equation}
\label{eq:marginal_estimate2}
\begin{split}
&
    \left\|
      V( X^x_{ \tau }, X^y_{ \tau } )
    \right\|_{
      L^r( \Omega; \R )
    }
=
  V( x, y ) 
  \left\|
 \exp\!\left(
    \int_0^{ \tau }     
      \tfrac{
	(\overline{\mathcal{G}}_{\mu,\sigma}V)(X_s^x,X_s^y)
      }{
	V(X_s^x,X_s^y)
      }
    \,
    ds
 \right)
  \right\|_{ L^r( \Omega; \R ) } 
  .
\end{split}
\end{equation}
\end{prop}

\begin{proof}[Proof
of Proposition~\ref{prop:two_solution_supoutside}]
Proposition~\ref{prop:two_solution} combined with
H\"{o}lder's inequality implies that
\begin{align*}
&
    \big\|
      V( X^x_{ \tau }, X^y_{ \tau } )
    \big\|_{
      L^r( \Omega; \R )
    }
\\ & \leq
  V( x, y )
\\ & \quad \cdot
  \left\|
  \exp\!\left(
    \smallint_0^{ \tau }
    \left[
   \tfrac{
     ( 
       \overline{ \mathcal{G} }_{ \mu, \sigma } V 
     )( X_s^x, X^y_s )
   }{
     V( X_s^x, X^y_s )
   }
    -
    \tfrac{
      \|
          ( \overline{ G }_{ \sigma } V )( X^x_s, X^y_s )
      \|^2
    }{
      2
      ( 
        V( X^x_s, X^y_s )
      )^2
    }
    \right]
    ds
    +
    \smallint_0^{ \tau }
    \tfrac{
      ( \overline{ G }_{ \sigma } V )( X^x_s, X^y_s )
    }{
      V( X^x_s, X^y_s )
    }
    \,
    dW_s
  \right)
  \right\|_{
    L^r( \Omega; \R )
  }
\\ & \leq
  V( x, y ) 
  \left\|
  \exp\!\left(
    \smallint_0^{ \tau }
    \left[
   \tfrac{
     ( 
       \overline{ \mathcal{G} }_{ \mu, \sigma } V 
     )( X_s^x, X^y_s )
   }{
     V( X_s^x, X^y_s )
   }
    +
    \tfrac{
      \left( p - 1 \right) \,
      \|
          ( \overline{ G }_{ \sigma } V )( X^x_s, X^y_s )
      \|^2
    }{
      2 \,
      ( 
        V( X^x_s, X^y_s )
      )^2
    }
    \right]
    ds
  \right)
  \right\|_{ L^q( \Omega; \R ) }
\\ & \quad 
  \cdot
  \left\|
  \exp\!\left(
    \smallint_0^{ \tau }
    \tfrac{
      ( \overline{ G }_{ \sigma } V )( X^x_s, X^y_s )
    }{
      V( X^x_s, X^y_s )
    }
    \,
    dW_s
    -
    \smallint_0^{ \tau }
    \tfrac{
      p \, 
      \|
          ( \overline{ G }_{ \sigma } V )( X^x_s, X^y_s )
      \|^2
    }{
      2 \,
      ( 
        V( X^x_s, X^y_s )
      )^2
    }
    \,
    ds
  \right)
  \right\|_{ L^p( \Omega; \R ) }
\\ & \leq
  V( x, y )
  \left\|
  \exp\!\left(
    \smallint_0^{ \tau }
   \tfrac{
     ( 
       \overline{ \mathcal{G} }_{ \mu, \sigma } V 
     )( X_s^x, X^y_s )
   }{
     V( X_s^x, X^y_s )
   }
    +
    \tfrac{
      \left( p - 1 \right)
      \,
      \|
          ( \overline{ G }_{ \sigma } V )( X^x_s, X^y_s )
      \|^2
    }{
      2 
      \,
      ( 
        V( X^x_s, X^y_s )
      )^2
    }
    \,
    ds
  \right)
  \right\|_{ L^q( \Omega; \R ) }
\end{align*}
for all $p,q,r\in(0,\infty]$
with $ \frac{ 1 }{ p } + \frac{ 1 }{ q } = \frac{ 1 }{ r } $.
This proves~\eqref{eq:marginal_estimate}.
In the next step we observe that if 
$ 
  \int_0^{ \tau }
  \|
    ( \overline{G}_{ \sigma } V)( X^x_s, X^y_s )
  \|^2
  \, ds
  = 0
$
$ \P $-a.s.~and if
   $
     ( V^{ - 1 } )(0)
     \subseteq 
     (
       \overline{\mathcal{G}}_{\mu, \sigma }V
     )^{ - 1 }([0,\infty))
   $,
then
Proposition~\ref{prop:two_solution} proves that
\begin{equation}
  V( X_{ \tau }^x, X^y_{ \tau } )
=
  \exp\!\left(
  \int_0^{ \tau }
   \tfrac{
     ( 
       \overline{ \mathcal{G} }_{ \mu, \sigma } V 
     )( X_s^x, X^y_s )
   }{
     V( X_s^x, X^y_s )
   }
  \, ds
  \right)
  V(x,y)
\end{equation}
$ \P $-a.s.\ and hence
\begin{equation}
  \left\| 
    V( X_{ \tau }^x, X^y_{ \tau } )
  \right\|_{
    L^r( \Omega; \R )
  }
=
  V(x,y)
  \left\|
  \exp\!\left(
  \int_0^{ \tau }
   \tfrac{
     ( 
       \overline{ \mathcal{G} }_{ \mu, \sigma } V 
     )( X_s^x, X^y_s )
   }{
     V( X_s^x, X^y_s )
   }
  \, ds
  \right)
  \right\|_{
    L^r( \Omega; \R )
  }
\end{equation}
for all $r\in(0,\infty]$.
The proof of
Proposition~\ref{prop:two_solution_supoutside}
is thus completed.
\end{proof}

\begin{remark}
\label{remark:not_extended_system1}
Note in the setting of
Proposition~\ref{prop:two_solution_supoutside}
that if $ \hat{V} \in C^2( O, [0,\infty) ) $,
$ x = y $
and if 
$
  V(v,w) = \hat{V}( v )
$
for all $ v , w \in O $,
then~\eqref{eq:marginal_estimate} in 
Proposition~\ref{prop:two_solution_supoutside}
reduces to an estimate for
$
  \| 
    \hat{V}( X^x_{ \tau } ) 
  \|_{
    L^r( \Omega; \R )
  }
$.
\end{remark}

The corollary below follows 
immediately from 
Proposition~\ref{prop:two_solution_supoutside}
and Example~\ref{ex:phi_identity}.
It establishes an estimate
for the $ L^r $-norm of the difference 
of two solutions of the same SDE
starting in different initial values
for $ r \in (0,\infty] $.

\begin{corollary}
\label{cor:powerful_est}
Assume the setting in Section~\ref{sec:setting},
let $ x, y \in O $,
let
$ \tau \colon \Omega \to [0,T] $
be a stopping time,
let	
$ X^z \colon [0,T] \times \Omega \to O $,
$ z \in \{ x, y \} $,
be adapted
stochastic processes
with continuous sample paths 
satisfying
\begin{equation*}
  \int_0^{ \tau }
    \| \mu( X^z_s ) \| 
    +
    \| 
      \sigma( X^z_s )
    \|^2
    +
    \tfrac{
      \max(
        \langle
          X^x_s - X^y_s ,
          \mu( X^x_s ) - \mu( X^y_s )
        \rangle
      , 0
      )
    +
      \|
        \sigma( X^x_s ) - \sigma( X^y_s )
      \|^2
    }{
      \| X^x_s - X^y_s \|^2
    }
  \, ds
  < \infty
\end{equation*}
$ \P $-a.s.\ 
and
$
  X^z_{ t \wedge \tau } = 
  z
  + \int_0^{ t \wedge \tau } \mu( X^z_s ) \, ds
  +
  \int_0^{ t \wedge \tau } \sigma( X^z_s ) \, dW_s
$
$ \P $-a.s.\ for all $ (t,z) \in [0,T] \times \{ x, y \} $.
Then
\begin{multline}
\label{eq:powerful_est}
    \left\|
      X^x_{ \tau } - X^y_{ \tau }
    \right\|_{
      L^{ r }( \Omega; \R^{d} )
    }
\\ 
  \leq
  \left\| x - y \right\|
  \left\|
  \exp\!\left(
    \int_0^{ \tau }\!\!
    \left[\!
    \begin{aligned}
  \tfrac{ 
    \langle
      X^x_s - X^y_s ,
      \mu( X^x_s ) - 
      \mu( X^y_s )
    \rangle
    +
    \frac{ 1 }{ 2 }
    \| 
      \sigma( X^x_s ) - \sigma( X^y_s ) 
    \|^2_{ \HS( \R^m, \R^d ) }
  }{
    \|
      X^x_s - X^y_s
    \|^2
  }
\\[1ex]
    +
    \tfrac{
      ( \frac{ p }{ 2 } - 1 ) \,
      \| 
        (
          \sigma( X^x_s ) - \sigma( X^y_s )
        )^*
        ( X^x_s - X^y_s )
      \|^2
    }{
      \| X^x_s - X^y_s \|^4
    }
    \end{aligned}
    \right]\!
    ds\!
  \right)\!
  \right\|_{ L^{ q }( \Omega; \R ) }
  \!\!\!\!\!
\end{multline}
for all
$ p, q,r \in (0,\infty] $
with 
$
  \frac{ 1 }{ p } + \frac{ 1 }{ q } = \frac{ 1 }{ r }
$.
\end{corollary}

Corollary~\ref{cor:powerful_est} simplifies in the
case of one-dimensional SDEs.
More precisely,
\eqref{eq:powerful_est}
reads as
\begin{equation}
\begin{split}
\label{eq:powerful_est_one_dimensional}
&    \left\|
      X^x_t - X^y_t
    \right\|_{
      L^{ r }( \Omega; \R )
    }
\\ &\leq
  \left\| x - y \right\|
  \left\|
  \exp\!\left(
    \int_0^t
  \tfrac{ 
    (
      X^x_s - X^y_s 
    ) \,
    (
      \mu( X^x_s ) - 
      \mu( X^y_s )
    )
    +
    \frac{ p - 1 }{ 2 }
    \| 
      \sigma( X^x_s ) - \sigma( X^y_s ) 
    \|^2_{ \HS( \R^m, \R ) }
  }{
    |
      X^x_s - X^y_s
    |^2
  }
  \, ds
  \right)
  \right\|_{ L^{ q }( \Omega; \R ) }
\end{split}
\end{equation}
for all
$ p, q, r \in (0,\infty] $
with 
$
  \frac{ 1 }{ p } + \frac{ 1 }{ q } = \frac{ 1 }{ r }
$
in the case $ d = 1 $.
To analyze and to estimate 
the term appearing in the exponent 
in~\eqref{eq:powerful_est}
in Corollary~\ref{cor:powerful_est},
the following elementary proposition
can be quite useful 
(see, e.g., Subsection~\ref{ssec:cir_global} below).

\begin{prop}[On the global monotonicity of the dynamics and the global coercivity of its derivative]
\label{prop:ext_mean_value_theorem}
Assume the setting in Section~\ref{sec:setting},
assume that $ O $ is \sgc{}convex\cgs{} and
assume that $ \mu \in C^1(O,\R^d) $\sgc{},\cgs{} $ \sigma \in C^1( O, \R^{ d \times m } ) $.
Then 
for all 
$ p \in (0,\infty) $, $ x, y \in O $
with $ x \neq y $ 
it holds that
\allowdisplaybreaks
\begin{equation}
\begin{aligned}
\label{eq:ext_mean_value_theorem_1a}
&  
  \sup_{ 
    \substack{
      z \in 
      \{ 
        ( 1 - r ) x 
    \\
        + 
        r y \colon 
        r \in 
        (0,1] 
      \}
    }
  }
  \left[
    \tfrac{
      \langle x - y , \mu'( z ) ( x - y ) \rangle
      +
      \inv{2}
      \| \sigma'( z ) ( x - y ) \|_{
        \HS( \R^m , \R^d )
      }^2
    }{
      \| x - y \|^2
    }
    +
    \tfrac{     
      ( \frac{ p }{ 2 } - 1 )
      \,
      \| 
        ( \sigma'( z ) ( x - y ) )^* ( x - y )
      \|^2
    }{ 
      \| x - y \|^4 
    }
  \right]
\\ \leq
  &
  \sup_{ 
    \substack{
      z \in 
      \{ 
        ( 1 - r ) x 
    \\
        + 
        r y \colon 
        r \in 
        (0,1] 
      \}
    }
  }
  \left[
    \tfrac{ 
      \langle x - z , \mu(x) - \mu(z) \rangle
      +
      \inv{2}
      \| \sigma(x) - \sigma(z) \|_{\HS(\R^m,\R^d)}^2
    }{ 
      \| x - z \|^2 
    }
    +
    \tfrac{     
      ( \frac{ p}{ 2 } - 1 )
      \,
      \| ( \sigma(x) - \sigma(z) )^* ( x - z ) \|^2
    }{
      \| x - z \|^4
    }
 \right] 
 ,
\end{aligned}
\end{equation}
for all $ p \in [1,\infty) $, $ x, y \in O $ 
with $ x \neq y $ it holds that 
\begin{equation}
\begin{aligned}
\label{eq:ext_mean_value_theorem_1b}
&
  \sup_{ 
    \substack{
      z \in 
      \{ 
        ( 1 - r ) x 
    \\
        + 
        r y \colon 
        r \in 
        (0,1] 
      \}
    }
  }
  \left[
    \tfrac{
      \langle x - y , \mu'( z ) ( x - y ) \rangle
      +
      \inv{2}
      \| \sigma'( z ) ( x - y ) \|_{
        \HS( \R^m , \R^d )
      }^2
    }{
      \| x - y \|^2
    }
    +
    \tfrac{     
      ( \frac{ p }{ 2 } - 1 )
      \,
      \| 
        ( \sigma'( z ) ( x - y ) )^* ( x - y )
      \|^2
    }{ 
      \| x - y \|^4 
    }
  \right]
\\ & =
  \sup_{ 
    \substack{
      z \in 
      \{ 
        ( 1 - r ) x 
    \\
        + 
        r y \colon 
        r \in 
        (0,1] 
      \}
    }
  }
  \left[
    \tfrac{ 
      \langle x - z , \mu(x) - \mu(z) \rangle
      +
      \inv{2}
      \| \sigma(x) - \sigma(z) \|_{\HS(\R^m,\R^d)}^2
    }{ 
      \| x - z \|^2 
    }
    +
    \tfrac{     
      ( \frac{ p}{ 2 } - 1 )
      \,
      \| ( \sigma(x) - \sigma(z) )^* ( x - z ) \|^2
    }{
      \| x - z \|^4
    }
 \right] 
  ,
\end{aligned}
\end{equation}
for all $ p \in (0,\infty) $
it holds that
\begin{equation}\begin{aligned}
\label{eq:ext_mean_value_theorem_2B}
& 
 \sup_{ x \in O}
 \sup_{ v \in \R^d\backslash \{0\} }
   \left[
      \tfrac{
	\langle v , \mu'(x) v \rangle
	+
	\inv{2}
	\| \sigma'(x) v \|_{\HS(\R^m,\R^d)}^2
      }{
	\| v \|^2
      }
      +
      \tfrac{     
        ( \frac{ p}{ 2 } - 1 )
        \,
	\| 
	  (\sigma'(x)v)^* v 
	\|^2
      }{ 
        \| v \|^4 
      }
  \right]
\\ & \leq
  \sup_{
    \substack{ 
      x, y \in O,
    \\
      x \neq y
    }
  }\sgc{}
  \left[
    \tfrac{ 
      \langle x - y , \mu(x) - \mu(y) \rangle
      +
      \inv{2}
      \| \sigma(x) - \sigma(y) \|_{\HS(\R^m,\R^d)}^2
    }
    { \| x - y \|^2 }
    +
    \tfrac{     
      ( \frac{ p }{ 2 } - 1 )
      \,
      \| (\sigma(x) - \sigma(y))^* (x-y) \|^2
    }{
      \| x - y \|^4
    }
 \right]\cgs{}
\end{aligned}
\end{equation}
and for all $ p \in [1,\infty) $ it holds that
\begin{equation}
\begin{aligned}
\label{eq:ext_mean_value_theorem_2}
& 
 \sup_{x\in O}
 \sup_{v\in \R^d\backslash \{0\}}
   \left[
      \tfrac{
	\langle v , \mu'(x) v \rangle
	+
	\inv{2}
	\| \sigma'(x) v \|_{\HS(\R^m,\R^d)}^2
      }{
	\| v \|^2
      }
      +
      \tfrac{     
        ( \frac{ p }{ 2 } - 1 )
        \,
	\| 
	  (\sigma'(x)v)^* v 
	\|^2
      }
      { \| v \|^4 }
  \right]
\\ & =
  \sup_{
    \substack{ 
      x, y \in O,
    \\
      x \neq y
    } 
  }
  \left[
    \tfrac{ 
      \langle x - y , \mu(x) - \mu(y) \rangle
      +
      \inv{2}
      \| \sigma(x) - \sigma(y) \|_{\HS(\R^m,\R^d)}^2
    }
    { \| x - y \|^2 }
    +
    \tfrac{     
      (\frac{ p}{ 2 } -1) 
      \,
      \| (\sigma(x) - \sigma(y))^* (x-y) \|^2
    }{
      \| x - y \|^4
    }
 \right]
  .
\end{aligned}
\end{equation}
\end{prop}

\begin{proof}[Proof of Proposition~\ref{prop:ext_mean_value_theorem}]
First of all, note that the definition of $ \mu' $ and $ \sigma' $
ensures that for all $p\in (0,\infty)$, $ x \in O $,
$ v \in \R^d \backslash \{ 0 \} $
with $ x + v \in O $
it holds that
\allowdisplaybreaks
\begin{equation}
\label{eq:prove_ext_mean_value1a}
\begin{aligned}
& 
  \left[
      \tfrac{
	\langle v , \mu'(x) v \rangle
	+
	\inv{2}
	\| \sigma'(x) v \|_{\HS(\R^m,\R^d)}^2
      }{
	\| v \|^2
      }
      +
      \tfrac{     
        (\frac{ p}{ 2 } -1)
        \,
	\| 
	  (\sigma'(x)v)^* v
	\|^2
      }
      { \| v \|^4 }
   \right]
\\ & =  
  \lim_{ r \searrow 0 }
  \left[
    \tfrac{ 
      \langle -r v , \mu(x) - \mu(x+rv) \rangle
      +
      \inv{2}
      \| \sigma(x) - \sigma(x+rv) \|_{\HS(\R^m,\R^d)}^2
    }{ 
      \| r v \|^2 
    }
    +
    \tfrac{     
      (\frac{ p}{ 2 } -1)
      \,
      \| (\sigma(x) - \sigma(x+rv))^* ( r v ) \|^2
    }{
      \| r v \|^4
    }
 \right]
\\ & \leq
  \sup_{ 
    \substack{
        z \in \{ x + r v \colon 
      \\
        r \in (0,1] \} 
    }
  }
  \left[
    \tfrac{ 
      \langle x - z , \mu(x) - \mu( z ) \rangle
      +
      \inv{2}
      \| \sigma(x) - \sigma( z ) \|_{\HS(\R^m,\R^d)}^2
    }{ 
      \| x - z \|^2 
    }
    +
    \tfrac{     
      ( \frac{ p}{ 2 } - 1 )
      \| ( \sigma(x) - \sigma( z ) )^* ( x - z ) \|^2
    }{
      \| x - z \|^4
    }
 \right] 
 .
\end{aligned}
\end{equation}
This proves \eqref{eq:ext_mean_value_theorem_1a}.
Moreover, 
let 
$ 
  A 
$
be the set given by
\begin{equation}
  A = 
  \left\{ 
    (x,v) \in 
    O \times \left( \R^d \backslash \{ 0 \} \right)
    \colon
    x + v \in O
  \right\}
\end{equation}
and note that \eqref{eq:prove_ext_mean_value1a}
and the fact that $ O $ is an open set 
ensure that
for all 
$ p \in (0,\infty) $
it holds that
\allowdisplaybreaks
\begin{equation}
\label{eq:prove_ext_mean_value1b}
\begin{aligned}
&
  \sup_{ 
    x \in O
  }
  \sup_{
    v \in \R^d \backslash \{ 0 \}
  }
  \left[
      \tfrac{
	\langle v , \mu'(x) v \rangle
	+
	\inv{2}
	\| \sigma'(x) v \|_{\HS(\R^m,\R^d)}^2
      }{
	\| v \|^2
      }
      +
      \tfrac{     
      (\frac{ p}{ 2 } -1)
	\| 
	  (\sigma'(x)v)^* v
	\|^2
      }
      { \| v \|^4 }
   \right]
\\ &
  =
  \sup_{ 
    (x,v) \in A 
  }
  \left[
    \tfrac{
      \langle v , \mu'(x) v \rangle
      +
      \inv{2}
      \| \sigma'(x) v \|_{\HS(\R^m,\R^d)}^2
    }{
      \| v \|^2
    }
    +
    \tfrac{     
      (\frac{ p }{ 2 } - 1 )
      \,
      \| 
        ( \sigma'(x) v )^* v
      \|^2
    }{ 
      \| v \|^4 
    }
  \right]
\\ & \leq
  \sup_{ (x,v) \in A }
  \sup_{ 
    \substack{
        y \in \{ x + r v \colon 
      \\
        r \in (0,1] \} 
    }
  }
  \bigg[
    \tfrac{ 
      \langle x - y , \mu(x) - \mu(y) \rangle
      +
      \inv{2}
      \| \sigma(x) - \sigma(y) \|_{\HS(\R^m,\R^d)}^2
    }{ 
      \| x - y \|^2 
    }
\\ & \quad
    +
    \tfrac{     
      ( \frac{ p}{ 2 } - 1 )
      \| ( \sigma(x) - \sigma(y) )^* ( x - y ) \|^2
    }{
      \| x - y \|^4
    }
 \bigg] 
\\ & =
  \sup_{ 
    \substack{
      x, y \in O ,
    \\
      x \neq y
    }
  }
  \left[
    \tfrac{ 
      \langle x - y , \mu(x) - \mu(y) \rangle
      +
      \inv{2}
      \| \sigma(x) - \sigma(y) \|_{\HS(\R^m,\R^d)}^2
    }{ 
      \| x - y \|^2 
    }
    +
    \tfrac{     
      ( \frac{ p}{ 2 } - 1 )
      \| ( \sigma(x) - \sigma(y) )^* ( x - y ) \|^2
    }{
      \| x - y \|^4
    }
 \right] 
 .
\end{aligned}
\end{equation}
This proves \eqref{eq:ext_mean_value_theorem_2B}.
It thus remains to prove \eqref{eq:ext_mean_value_theorem_1b} 
and~\eqref{eq:ext_mean_value_theorem_2}.
To this end let
$ p \in [1,\infty) $
and let
$
  g_{ x, v } \colon [0,1] \rightarrow \R 
$,
$ (x,v) \in A $,
and
$ 
  h_{ x, v } \colon [0,1] \rightarrow \R 
$,
$ (x,v) \in A $,
be the functions 
with the property that
for all
$ r \in [0,1] $,
$ 
  ( x, v ) \in A
$
it holds that
\begin{equation}
\label{def:g}
\begin{split}
&
  g_{ x, v }(r) =
\\ & 
  \frac{
    \langle v , \mu'(x+rv) v \rangle
    +
    \inv{2}
    \| \sigma'(x+rv) v \|_{\HS(\R^m,\R^d)}^2
  }{
    \| v \|^2
  }
  +
  \frac{     
    (
      \frac{ p }{ 2 } - 1 
    )
    \,
    \| 
      ( \sigma'( x + r v ) v )^* v
    \|^2
  }{ 
    \| v \|^4 
  }
\end{split}
\end{equation}
and 
with the property that
for all
$ r \in (0,1] $,
$ (x,v) \in A $
it holds that
$ h_{ x, v }(0) = g_{ x, v }(0) $
and
\begin{equation}
\label{def:h}
\begin{split}
&
  h_{ x, v }(r)
\\ & =
  \tfrac{ 
    \langle r v , \mu(x) - \mu(x+rv) \rangle
    +
    \inv{2}
    \| \sigma(x) - \sigma(x+rv) \|_{\HS(\R^m,\R^d)}^2
  }{ 
    \| r v \|^2 
  }
  +
  \tfrac{     
    ( \frac{ p }{ 2 } - 1 )
    \,
    \| ( \sigma(x) - \sigma(x+rv) )^* ( r v ) \|^2
  }{
    \| r v \|^4
  }
\\ & =
    \tfrac{ 
      \langle v , \mu(x) - \mu(x+rv) \rangle
    }
    {
      -r \|v \|^2
    }
    +
    \tfrac{
      \inv{2}
      \| \sigma(x) - \sigma(x+rv) \|_{\HS(\R^m,\R^d)}^2
    }
    { r^2 \| v \|^2 }
    +
    \tfrac{     
      ( \frac{ p}{ 2 } - 1 )
      \,
      \| (\sigma(x) - \sigma(x+rv))^* v \|^2
    }{
      r^2 \| v \|^4
    }
    .
\end{split}
\end{equation}
Next we observe that for all 
$ r \in (0,1) $,
$ (x,v) \in A $
it holds that
$ 
  h_{ x, v }|_{ (0,1] } \colon (0,1] \to \R 
$ 
is continuously differentiable 
and that
\begin{equation}
\begin{aligned}  
  h_{ x, v }'(r)
& =
  \frac{
    \langle 
      v , - \mu'( x + r v ) v 
    \rangle 
  }{
    - r \| v \|^2
  }
  +
  \frac{
    \langle 
      v, \mu(x) - \mu(x + r v) 
    \rangle
  }{ 
    r^2 \| v \|^2 
  }
\\ & \quad 
  -
  \frac{ 
    \langle
      \sigma( x ) - \sigma( x + r v )
      ,
      \sigma'( x + r v ) v
    \rangle_{\HS(\R^m,\R^d)}
  }{ 
    r^2 \| v \|^2 
  }
\\ & \quad
  -
  \frac{
    \| \sigma(x) - \sigma( x + r v ) \|_{
      \HS(\R^m,\R^d)
    }^2
  }{
    r^3 \| v \|^2
  }
\\ & \quad 
  - 
  \frac{     
    (p-2)
    \left\langle 
      ( \sigma(x) - \sigma( x + r v ) )^* v
      ,
      ( \sigma'( x + r v ) v )^* v
    \right\rangle
  }{
    r^2 \| v \|^4
  }
\\ & \quad
    - 
    \frac{     
      ( p - 2 )
      \left\| 
	( \sigma(x) - \sigma( x + r v ) )^* v
      \right\|^2
    }{
      r^3 \| v \|^4
    }
    .
\end{aligned}
\end{equation}
This ensures that 
for all 
$ r \in (0,1) $,
$ (x,v) \in A $
it holds that
\begin{equation}
\begin{aligned}
&
  h_{ x, v }'(r)
\\ & =
  \frac{ 1 }{ r }
  \Bigg[
    \frac{
      \left\langle 
        v , \mu'( x + r v ) v
      \right\rangle 
    }{
      \| v \|^2
    }
    -
    \frac{
      \langle v, \mu(x) - \mu( x + r v ) 
      \rangle
    }{ 
      - r \| v \|^2 
    }
\\ & \qquad 
      -
      \frac{ 1 }{ 2 \, \| v \|^2 }
      \left\|
	\frac{
	  ( \sigma(x) - \sigma( x + r x ) )
	}{  
	  r  
	}
	+
	\sigma'(x+rv)v
      \right\|_{
        \HS(\R^m,\R^d)
      }^2
\\ & \qquad 
    +
    \frac{
      \inv{2}
      \left\| \sigma'(x+rv)v \right\|_{
        \HS(\R^m,\R^d)
      }^2
    }{
      \| v \|^2
    }
    -
    \frac{ 
      \inv{2}
      \left\| \sigma(x) - \sigma(x+rv) \right\|_{
        \HS(\R^m,\R^d)
      }^2
    }{
      r^2 \| v \|^2
    }
\\ & \qquad 
    - 
    \frac{
      ( \frac{ p }{ 2 } - 1 ) 
    }{
      \| v \|^4
    }
    \left\|
      \frac{
        ( \sigma(x) - \sigma( x + r v ) )^*v
      }{ r }
      +
      ( \sigma'(x+rv) v )^*v
    \right\|^2
\\ & \qquad 
    +
    \frac{
      ( \frac{ p }{ 2 } - 1 )
      \left\| 
        ( \sigma'(x+rv) v )^*v 
      \right\|^2
    }{
      \| v \|^4
    }
    -
    \frac{ 
      ( \frac{ p }{ 2 } - 1 )
      \left\| 
        ( \sigma(x) - \sigma( x + r v ) )^*v 
      \right\|^2
    }{
      r^2 \| v \|^4
    }
  \Bigg]
  .
\end{aligned}
\end{equation}
Hence, we obtain that for all
$ r \in (0,1) $ 
and all 
$ (x,v) \in A $
it holds that
\begin{equation}
\label{eq:h_bound}
\begin{aligned}
&
  h_{ x, v }'(r) 
\\ & =
  \frac{
    \left(
      g_{ x, v }(r) - h_{ x, v }(r)
    \right)
  }{ r }
  -
  \frac{ 1 }{
    2 r \| v \|^2
  }
  \left\|
    \frac{
      ( \sigma(x) - \sigma( x + r v ) )
    }{ r }
    +
    \sigma'(x+rv)v
  \right\|_{
    \HS( \R^m, \R^d ) 
  }^2
\\ & \quad 
  - 
  \frac{
    ( \frac{ p }{ 2 } - 1 )
  }{
    r \| v \|^4
  }
  \left\|
    \left[
      \frac{
        ( \sigma(x) - \sigma( x + r v ) )
      }{ r }
      +
      \sigma'( x + r v ) v 
    \right]^* v
  \right\|^2
\\ & \leq
  \frac{
    \left(
      g_{ x, v }(r) - h_{ x, v }(r)
    \right)
  }{ r }
  -
  \frac{ 1 }{
    2 r \| v \|^2
  }
  \left\|
    \frac{
      ( \sigma(x) - \sigma( x + r v ) )
    }{ r }
    +
    \sigma'(x+rv)v
  \right\|_{
    \HS( \R^m, \R^d ) 
  }^2
\\ & \quad 
  +
  \frac{
    1
  }{
    2 r \| v \|^4
  }
  \left\|
    \left[
      \frac{
        ( \sigma(x) - \sigma( x + r v ) )
      }{ r }
      +
      \sigma'( x + r v ) v 
    \right]^* v
  \right\|^2
\\ & \leq
  \frac{
    \left(
      g_{ x, v }(r) - h_{ x, v }(r)
    \right)
  }{ r }
  .
\end{aligned}
\end{equation}
Furthermore,
for every
$ (x,v) \in A $
let 
$ \tau_{ x, v } \in \R $
be a real number given by
\begin{equation}
  \tau_{ x, v }
  = 
  \sup\!\left( 
    \left\{ - 1 \right\} 
    \cup 
    \left\{ 
      r \in [0,1] \colon h_{ x, v }(r) = g_{ x, v }(r) 
    \right\} 
  \right) .
\end{equation}
The fact that 
$ 
  \forall \, 
  (x,v) \in A 
  \colon
  h_{ x, v }(0) = g_{ x, v }(0) 
$ 
ensures that 
\begin{equation}
  \forall \, 
  (x,v) \in A 
  \colon
  \tau_{ x, v } \in [0,1] 
  .
\end{equation}
In the next step
we intend to show that
for all 
$ (x,v) \in A $
it holds that
\begin{equation}
\label{eq:bound_for_h}
  h_{ x, v }(1) \leq \sup_{ r \in [0,1] } g_{ x, v }(r)
  .
\end{equation}
We prove \eqref{eq:bound_for_h} by contradiction.
Let $(x_0,v_0)\in A$ be such that
\begin{equation}
\label{eq:h_assumption_contr}
  h_{ x_0, v_0 }(1) > 
  \sup_{ r \in [0,1] } g_{ x_0, v_0 }(r)
  .
\end{equation}
This, in particular, implies that 
$
  h_{ x_0, v_0 }(1) > g_{ x_0, v_0 }( 1 )
$.
Hence, the fact that
$
  \forall \,
  (x,v) \in A 
  \colon
  h_{ x, v }, g_{ x, v } \in C( [0,1] , \R )
$
establishes that
$ 
  \tau_{ x_0, v_0 } < 1 
$
and that for all 
$ r \in ( \tau_{x_0,v_0}, 1 ] $
it holds that
\begin{equation}
  g_{ x_0, v_0 }(r) < h_{ x_0, v_0 }(r) 
  .
\end{equation}
This, the fundamental theorem 
of calculus and 
\eqref{eq:h_bound}
prove that for all $s\in (0,1) \cap [\tau_{x_0, v_0}, 1)$
it holds that
\begin{equation}
\begin{split}
&
  h_{ x_0, v_0 }( 1 )
=
  h_{ x_0, v_0 }( s )
  +
  \int_{ s }^1
  h_{ x_0, v_0 }'(r)
  \,
  dr
\\ & \leq
  h_{ x_0, v_0 }( s )
  +
  \int_{ s }^1
  \frac{
    \left(
      g_{ x_0, v_0 }(r) - h_{ x_0, v_0 }( r ) 
    \right)
  }{
    r
  }
  \, dr
  <
  h_{ x_0, v_0 }( s )
  .
\end{split}
\end{equation}
This and the fact that $h_{x_0,v_0}$ is continuous yields that $h_{x_0,v_0}(1) \leq h_{x_0,v_0}(\tau_{x_0,v_0}) = g_{x_0,v_0}(\tau_{x_0,v_0})$. This contradicts \eqref{eq:h_assumption_contr} 
and we hence obtain that
for all $ (x,v) \in A $
it holds that
\begin{equation}\label{eq:hlessthang}
  h_{ x, v }( 1 ) 
  \leq
  \sup_{ r \in [0,1] }
  g_{ x, v }( r )
  .
\end{equation}
This and \eqref{def:g}--\eqref{def:h} ensure that
for all $ r \in [0,1] $, $ (x,v) \in A $ it holds that
\begin{equation}
  h_{ x, v }( r )
  \leq
  \sup_{ s \in [0,r] }
  g_{ x, v }( s )
  .
\end{equation}
Therefore, we obtain that
for all $ (x,v) \in A $ it holds that
\begin{equation}
\label{eq:prove_ext_mean_value2}
  \sup_{ r \in [0,1] }
  h_{ x, v }( r )
  \leq
  \sup_{ r \in [0,1] }
  g_{ x, v }( r )
  .
\end{equation}
Combining this with \eqref{eq:hlessthang} proves \eqref{eq:ext_mean_value_theorem_1b}.
It thus remains to establish \eqref{eq:ext_mean_value_theorem_2}.
For this note that \eqref{eq:prove_ext_mean_value2}
ensures that

\begin{equation}
\begin{aligned}
& 
 \sup_{x\in O}
 \sup_{v\in \R^d\backslash \{0\}}
   \left[
      \tfrac{
	\langle v , \mu'(x) v \rangle
	+
	\inv{2}
	\| \sigma'(x) v \|_{\HS(\R^m,\R^d)}^2
      }{
	\| v \|^2
      }
      +
      \tfrac{     
        ( \frac{ p }{ 2 } - 1 )
        \,
	\| 
	  (\sigma'(x)v)^* v 
	\|^2
      }
      { \| v \|^4 }
  \right]
\\ & \geq
 \sup_{ (x,v) \in A }
   \left[
      \tfrac{
	\langle v , \mu'(x) v \rangle
	+
	\inv{2}
	\| \sigma'(x) v \|_{\HS(\R^m,\R^d)}^2
      }{
	\| v \|^2
      }
      +
      \tfrac{     
        ( \frac{ p }{ 2 } - 1 )
        \,
	\| 
	  (\sigma'(x)v)^* v 
	\|^2
      }
      { \| v \|^4 }
  \right]
\\ & \geq
  \sup_{ (x,v) \in A} 
  \sup_{r\in [0,1]} 
  g_{ x, v }( r )
 \geq 
  \sup_{ (x,v) \in A }
  h_{ x, v }( 1 )
\\ & =
  \sup_{
    \substack{ 
      x, y \in O,
    \\
      x \neq y
    } 
  }
  \left[
    \tfrac{ 
      \langle x - y , \mu(x) - \mu(y) \rangle
      +
      \inv{2}
      \| \sigma(x) - \sigma(y) \|_{\HS(\R^m,\R^d)}^2
    }
    { \| x - y \|^2 }
    +
    \tfrac{     
      (\frac{ p}{ 2 } -1) 
      \,
      \| (\sigma(x) - \sigma(y))^* (x-y) \|^2
    }{
      \| x - y \|^4
    }
 \right]
  .
\end{aligned}
\end{equation}
Combining \sgc{}this\cgs{} with estimate~\eqref{eq:prove_ext_mean_value1b}
proves~\eqref{eq:ext_mean_value_theorem_2}\sgc{}. The\cgs{} proof of Proposition~\ref{prop:ext_mean_value_theorem} is thus \sgc{}complete\cgs{}.
\end{proof}

\begin{remark}
The identities \eqref{eq:ext_mean_value_theorem_1b} 
and \eqref{eq:ext_mean_value_theorem_2},
in general, fail to hold for $ p \in (0,1) $. 
Indeed, let 
$ d = m = 1 $, \sgc{}let\cgs{} $ O = \sgc{}\R\cgs{} $ 
and let
$ \mu \colon \R \to \R $
and
$ g_p, \sigma_p \colon \R \rightarrow \R $,
$ p \in (0,1) $,
be the functions with the property that
for all $ x \in \R $, $ p \in (0,1) $ it holds that
\begin{equation}
 \mu(x)
 =
 \begin{cases}
  -x^3+3x+2 
  &
    \colon
    x \in [-1,1]
\\
  0
  &
    \colon
    x \in (-\infty,-1)
  \\
  4
  &
    \colon
    x \in (1,\infty) 
\end{cases}
  ,
\end{equation}
\begin{equation}
  g_p(x)
  =
  \begin{cases}
    \sqrt{6}
    \left( 1 - p \right)^{ - 1/ 2 }
    \sqrt{ 1 - x^2 }
  &
    \colon
    x \in [-1,1]
  \\
    - 
    \sqrt{6}
    \left( 1 - p \right)^{ - 1/2 }
    \sqrt{ 1 - ( x - 2 )^2 }
  &
    \colon
    x \in [1,3]
  \\
    0
  &
    \colon
    x \in ( - \infty , - 1 ) \cup ( 3 , \infty )
\end{cases}
\end{equation}
and
$
  \sigma_p(x) = \int_{-\infty}^x g_p(y) \, dy 
$.
Then note that 
for all $ p \in (0,1) $
it holds that
$ \mu, \sigma_p \in C^1( \R, \R ) $
and note that for all $ x \in \R $ it holds that
\begin{equation}
  \mu'(x)
  =
 \begin{cases}
  - 3 \left( x^2 - 1 \right) 
  &
    \colon
    x \in [-1,1]
  \\
    0
  & \colon
    x \in ( - \infty , - 1 ) \cup (1 , \infty ) 
 \end{cases}
  .
\end{equation}
Furthemore, observe that for all 
$ p \in (0,1) $,
$
  x \in ( - \infty , - 1 ] \cup [3,\infty)
$
it holds that $ \sigma_p(x) = 0 $.
In particular, 
for all $ p \in (0,1) $
it holds that
\begin{equation}
\label{eq:counterex_diff}
\begin{aligned}
&
  \frac{ 
    \left(
      \mu(3) - \mu( - 1 )
    \right)
  }{ 4 }
  +
  \frac{ 
    ( \frac{ p }{ 2 } - \frac{ 1 }{ 2 }  ) \left| \sigma_p(3) - \sigma_p( - 1 ) \right|^2 
  }{ 
    4^2 
  }
  = 
    1 - 0
  = 
    1 . 
\end{aligned}
\end{equation}
In addition, for all $ p \in (0,1) $
it holds that
\begin{equation}
\label{eq:counterex_d}
  \sup_{ x \in \R }
  \left(
    \mu'(x)
    +
    \left( \tfrac{ p }{ 2 } - \tfrac{ 1 }{ 2 } \right) 
    \left| \sigma_p'(x) \right|^2
 \right)
  =
  \sup_{ x \in \R }
  \left(
    \mu'(x)
    -
    \tfrac{
      ( 1 - p ) 
      \,
      | \sigma_p'(x) |^2
    }{ 
      2 
    }
 \right)
 = 0
 .
\end{equation}
Combining \eqref{eq:counterex_diff}
and \eqref{eq:counterex_d} proves, in particular,
that for all $ p \in (0,1) $ it
holds that
\begin{equation}
\begin{aligned}
& 
  \sup_{
    \substack{
      x,y \in \R , 
    \\
      x \neq y
    }
  }
  \left(
    \tfrac{ 
      \mu(x) - \mu(y)
    }{ 
      \left( x - y \right) 
    }
    +
    \frac{
      ( \frac{ p }{ 2 } - \frac{ 1 }{ 2 } )
      \,
      | \sigma_p(x)-\sigma_p(y) |^2
    }{
      \left| x - y \right|^2
    }
  \right)
  >
 \sup_{ 
   x \in \R
 }
 \left(
  \mu'(x)
  +
  ( \tfrac{ p }{ 2 } - \tfrac{ 1 }{ 2 } )
  \left| 
    \sigma_p'(x)
  \right|^2
 \right)
 .
\end{aligned}
\end{equation}
We have thus established that for all $ p \in (0,1) $ 
there exist 
$ \mu, \sigma \in C^1(\R,\R) $, 
$ x, y \in \R $
with $ x \neq y $
such that
\begin{equation}
\begin{split}
&
  \sup_{
    z \in \R
  }
  \sup_{
    v \in \R \backslash \{ 0 \}
  }
  \left[
    \tfrac{
      \langle v , \mu'( z ) v \rangle
      +
      \inv{2}
      \| \sigma'( z ) v \|_{
        \HS( \R , \R )
      }^2
    }{
      \| v \|^2
    }
    +
    \tfrac{     
      ( \frac{ p }{ 2 } - 1 )
      \,
      \| 
        ( \sigma'( z ) v )^* v
      \|^2
    }{ 
      \| v \|^4 
    }
  \right]
\\ & <
  \left[
    \tfrac{ 
      \langle x - y , \mu(x) - \mu(y) \rangle
      +
      \inv{2}
      \| \sigma(x) - \sigma(y) \|_{\HS( \R , \R )}^2
    }{ 
      \| x - y \|^2 
    }
    +
    \tfrac{     
      ( \frac{ p}{ 2 } - 1 )
      \,
      \| ( \sigma(x) - \sigma(y) )^* ( x - y ) \|^2
    }{
      \| x - y \|^4
    }
 \right]  
 .
\end{split}
\end{equation}
\end{remark}

Lemma~\ref{lem:multiple_exp}
below uses Corollary~\ref{cor:exp_mom} above
to estimate expectations of
certain exponential integrals.
Besides Corollary~\ref{cor:exp_mom},
the proof of Lemma~\ref{lem:multiple_exp} 
also uses the 
following well-known consequence of Jensen's inequality
in the next lemma
(see, e.g., inequality~(19) in 
Li~\cite{Li1994}).

\begin{lemma}
\label{lem:integral_sup}
Let $ T \in (0,\infty) $,
let
$
  \left( \Omega, \mathcal{F}, \P \right)
$
be a probability space
and
let 
$ A \colon [0,T] \times \Omega \to \R $
be a product measurable stochastic process 
satisfying
$$
  \min\left\{ 
    \int_0^T \max( 0, A_s ) \, ds,
    \int_0^T \max( 0, - A_s ) \, ds
  \right\}
  < \infty
$$
$ \P $-a.s. 
Then it holds for all $ p \in [1, \infty] $ that
\begin{equation}
  \left\|
  \exp\!\left(
    \smallint_0^T
    A_t \, dt
  \right)
  \right\|_{
    L^p( \Omega; \R )
  }
\leq
  \frac{ 1 }{ T }
  \int_0^T
  \left\|
      e^{
	T A_t
      }
  \right\|_{
    L^p( \Omega; \R )
  }
  dt
\leq
  \sup_{ t \in [0,T] }
  \left\| 
      e^{
	T A_t
      }
  \right\|_{
    L^p( \Omega; \R )
  }
  .
\end{equation}
\end{lemma}

Using Lemma~\ref{lem:integral_sup}
and Corollary~\ref{cor:exp_mom},
we are now ready to prove 
Lemma~\ref{lem:multiple_exp}.
Lemma~\ref{lem:multiple_exp}
is crucial for
Theorem~\ref{thm:UV}
and Theorem~\ref{thm:UV2}
below.

\begin{lemma}
\label{lem:multiple_exp}
Assume the setting in Section~\ref{sec:setting},
let $ k \in \N $, 
$ x, y \in O $,
$ \alpha_0,\,\alpha_1,$ $\beta_0,\,\beta_1\in\R^k$,
$p\in(0,\infty]$,
$q_0,\,q_1\in(0,\infty]^k$
with
$
  \sum_{ i = 0 }^1 \sum_{ l = 1 }^k \frac{ 1 }{ q_{ i, l } } = \frac{ 1 }{ p }	
,
$
\sgc{}let\cgs{}
$ 
  U_0 
  = 
  ( U_{ 0, l } )_{ l \in \{ 1, \dots, k \} } 
$,
$
  U_1 = ( U_{ 1, l } )_{ l \in \{ 1, \dots, k \} }
  \in C^{ 2 }( O, \R^k ) 
$,
$ 
  \overline{U}
  = ( \overline{U}_{ l } )_{ l \in \{ 1, \dots, k \} }
  \in C( O, \R^k ) 
$, 
$ V \in \mathcal{L}^0( O^2 ; \R ) $\sgc{},\cgs{}
$
  c \in \mathcal{L}^0( [0,T] ; \R )
$,
let $ \tau \colon \Omega \to [0,T] $ be a stopping time
and let	
$ X^z \colon [0,T] \times \Omega \to O $,
$ z \in \{ x, y \} $,
be adapted
stochastic processes
with continuous sample paths
satisfying
$
  \smallint\nolimits_0^{ \tau }
    \max( c(s),0 )
    +
    \| \mu( X^z_s ) \| 
    +
    \| 
      \sigma( X^z_s )
    \|^2
    +
    \max( V( X^x_s, X^y_s ) ,0)
    \,
  ds
  < \infty
$
$ \P $-a.s.\ and
$
  X^z_{ t \wedge \tau } = 
  z
  + \int_0^{ t \wedge \tau } \mu( X^z_s ) \, ds
  +
  \int_0^{ t \wedge \tau } \sigma( X^z_s ) \, dW_s
$
$ \P $-a.s.\ for all $ (t,z) \in [0,T] \times \{ x, y \} $
and
\begin{equation}
\label{eq:Vestimate}
  V( v, w )
\leq   
  c(t)
  +
  \smallsum_{ n = 1 }^k
  \big[
    \frac{ 
      U_{ 0, n }( v ) + U_{ 0, n }( w ) 
    }{ 2 q_{ 0, n } T e^{ \alpha_{ 0, n } t } } 
    +
    \frac{ 
      \overline{U}_{n}( v ) + \overline{U}_{n}( w ) 
    }{ 
      2 q_{ 1, n } 
      e^{ \alpha_{ 1, n } t }
    }
  \big]
\qquad 
  \text{and}
\end{equation}
\begin{equation}
\label{eq:multiple_exp_est2}
  ( \mathcal{G}_{ \mu, \sigma } U_{ i, l } )( u )
  +
  \tfrac{ 
    1
  }{ 
    2 
    e^{ 
      \alpha_{ i, l } t 
    }
  }
    \|
      \sigma( u )^* ( \nabla U_{ i, l } )( u )
    \|^2
  +
  \mathbbm{1}_{
    \{ 1 \}
  }(i)
  \cdot
  \overline{U}_{l}(u)
\leq
  \alpha_{ i, l } U_{ i, l }(u)
  +
  \beta_{ i, l }
\end{equation}
for all\sgc{}
$(v,w) \in  
\cup_{\omega \in \Omega}
    \{ 
        (X_s^x(\omega), X_s^y(\omega)) \in O \times O \colon s \in [0, \tau(\omega)] 
    \}
$,
$ i \in \{ 0, 1 \} $,
$ l \in \{ 1, \dots, k \} $,
$ u \in \{ v, w \} $,
$t\in [0,T]$\cgs{}.
Then
\begin{equation}
\begin{split}
&
    \left\|
      e^{
        \int_0^{ \tau }
        V( X^x_s , X^y_s )      
        \, ds
      }
    \right\|_{
      L^p( \Omega; \R )
    }
\leq
  \exp\!\left(
      \smallsum\limits_{ l = 1 }^k
       \smallsum\limits_{ i = 0 }^1
        \frac{
          U_{ i, l }( x ) + U_{ i, l }( y) 
        }{
          2 q_{ i, l }
        }
    \right)
\\ & 
  \cdot
\left\|
    \frac{ 
    \exp\!\left(
    {\scriptstyle
      \int\limits_0^{ \tau } c(s) \, ds
      +
      \sum\limits_{i=0}^{1}
      \sum\limits_{ l = 1 }^k
      \int\limits_{ 0 }^{ \tau }
      \frac{
	\beta_{ i, l }(1-\frac{s}{T})^{1-i}
      }{
	q_{ i, l } e^{ \alpha_{ i, l } s }
      }
      \, ds } \right)
}{
    \exp\!\left(
    {\scriptstyle
	\sum\limits_{ l = 1 }^k
	\frac{
	  U_{1,l}( X^x_{ \tau } )
	  +
	  U_{1,l}( X^y_{ \tau } )
	}{ 2 q_{ 1, l } e^{ \alpha_{ 1, l } \tau } }
	+
      \sum\limits_{ l = 1 }^k
      \int\limits_{ \tau }^T
	\frac{
	  U_{0,l}( X^x_{\tau} )
	  +
	  U_{0,l}( X^y_{\tau} )
	}{ 2 q_{ 0, l } T e^{ \alpha_{ 0, l }\tau } }
      \, ds
    }
    \right)
    }
  \right\|_{ 
    L^{ \infty }( \Omega; \R ) 
  }\!\!\!\!
  .
\end{split}
\end{equation}
\end{lemma}

\begin{proof}[Proof
of Lemma~\ref{lem:multiple_exp}]
\sgc{}Throughout this proof let $p_0, p_1 \in (0,\infty]$ satisfy for all $i\in \{0,1\}$ that\cgs{}
$ p_i = [\, \sum_{l=1}^{k} \frac{1}{q_{i,l}}\,]^{-1} $\sgc{} 
and let $R\in [0,\infty]$ satisfy\cgs{} 
\begin{align*}
& R = 
\left\|
    \frac{ 
    \exp\!\left(
    {\scriptstyle
      \int\limits_0^{ \tau } c(s) \, ds
      +
      \sum\limits_{i=0}^{1}
      \sum\limits_{ l = 1 }^k
      \int\limits_{ 0 }^{ \tau }
      \frac{
	\beta_{ i, l }(1-\frac{s}{T})^{1-i}
      }{
	q_{ i, l } e^{ \alpha_{ i, l } s }
      }
      \, ds } \right)
}{
    \exp\!\left(
    {\scriptstyle
	\sum\limits_{ l = 1 }^k
	\frac{
	  U_{1,l}( X^x_{ \tau } )
	  +
	  U_{1,l}( X^y_{ \tau } )
	}{ 2 q_{ 1, l } e^{ \alpha_{ 1, l } \tau } }
	+
      \sum\limits_{ l = 1 }^k
      \int\limits_{ \tau }^T
	\frac{
	  U_{0,l}( X^x_{\tau} )
	  +
	  U_{0,l}( X^y_{\tau} )
	}{ 2 q_{ 0, l } T e^{ \alpha_{ 0, l }\tau } }
      \, ds
    }
    \right)
    }
  \right\|_{ 
    L^{ \infty }( \Omega; \R ) 
  }\!\!\!\!
  .
\end{align*}
Observe that
H\"{o}lder's inequality proves that
\begin{equation*}
\begin{split}
&
    \left\|
      e^{
        \int_0^{ \tau }
        V( X^x_s , X^y_s )      
        \, ds
      }
    \right\|_{
      L^p( \Omega; \R )
    }
\\ & 
\leq 
  R
  \cdot 
  \left\| 
  \frac{
    \exp\!\left(
    {\scriptstyle
        \int\limits_0^{ \tau }
        V( X^x_s, X^y_s )
        \,ds
        + 
        \sum\limits_{ l = 1 }^k
        \int\limits_{ \tau }^T
        \frac{
            U_{0,l}( X^x_{\tau} )
            +
            U_{0,l}( X^y_{\tau} )
        }{ 2 q_{ 0, l } T e^{ \alpha_{ 0, l } {\tau} } }
        \, ds
    }\right)
  }{  
    \exp\!\left(
    {\scriptstyle
        \int\limits_0^{ \tau }
        c(s)
        \,ds
        + 
        \sum\limits_{ l = 1 }^k
        \int\limits_0^{ \tau }
        \left[
        \frac{
            \overline{U}_{l}( X^x_s )
            +
            \overline{U}_{l}( X^y_s )
            }{ 2 q_{ 1, l } e^{ \alpha_{ 1, l } s } }
            +
            \frac{
                \beta_{0,k}(1-\frac{s}{T})
            }{
                q_{0,l} e^{\alpha_{0,l} s}
            }
        \right]
        \, ds }
  \right)
  }
  \right\|_{ 
    L^{ 
      p_0
    }( \Omega; \R ) 
  } 
\\ & \quad \cdot 
  \left\|
  \exp\!\left(
  {\scriptstyle
    \sum\limits_{ l = 1 }^k
    \left[
      \frac{
        U_{1,l}( X^x_{ \tau } )
        +
        U_{1,l}( X^y_{ \tau } )
      }{ 2 q_{ 1, l } e^{ \alpha_{ 1, l } \tau } }
      +
      \int\limits_0^{ \tau }
      \frac{
          \overline{U}_{l}( X^x_s )
          +
          \overline{U}_{l}( X^y_s )
        }{ 2 q_{ 1, l } e^{ \alpha_{ 1, l } s } }
      \, ds
      -
    \int\limits_{ 0 }^{ \tau }
    \frac{
      \beta_{ 1, l }
    }{
      q_{ 1, l } e^{ \alpha_{ 1, l } s }
    }
    \, ds
    \right]
  }
  \right)
  \right\|_{ 
    L^{ 
      p_1
    }( \Omega; \R ) 
  } 
  \!\!\!\!.
\end{split}
\end{equation*}
Hence,~\eqref{eq:Vestimate} 
implies that
\begin{equation*}
\begin{split}
&
    \left\|
      e^{
        \int_0^{ \tau }
        V( X^x_s , X^y_s )      
        \, ds
      }
    \right\|_{
      L^p( \Omega; \R )
    }
\leq 
\\ & 
  R \cdot
  \left\|
  \exp\!\left(
    \smallsum\limits_{ l = 1 }^k
    \smallint\limits_0^{ T }
      \tfrac{
        U_{0,l}( X^x_{s\wedge \tau} )
        +
        U_{0,l}( X^y_{s\wedge \tau} )
      }{ 2 q_{ 0, l } T e^{ \alpha_{ 0, l } {s\wedge \tau} } }
    \,
    ds
    -
    \smallint\limits_{0}^{s\wedge\tau}
        \tfrac{
            2\beta_{0,k}
        }{
            2q_{0,l}e^{\alpha_{0,l}u }
        }
    \,du        
  \right)
  \right\|_{ 
    L^{ 
      p_0
    }( \Omega; \R ) 
  } 
\\ & \cdot 
  \left\|
  \exp\!\left(
    \smallsum\limits_{ l = 1 }^k
    \left[
      \tfrac{
        U_{1,l}( X^x_{ \tau } )
        +
        U_{1,l}( X^y_{ \tau } )
      }{ 2 q_{ 1, l } e^{ \alpha_{ 1, l } \tau } }
      +
      \smallint\limits_0^{ \tau }
      \tfrac{
          \overline{U}_{l}( X^x_s )
          +
          \overline{U}_{l}( X^y_s )
          - 2 \beta_{ 1, l }
        }{ 2 q_{ 1, l } e^{ \alpha_{ 1, l } s } }
      \, ds
    \right]
  \right)
  \right\|_{ 
    L^{ 
     p_1
    }( \Omega; \R ) 
  }.
\end{split}
\end{equation*}
\sgc{}H\"{o}lder's\cgs{} inequality
and \sgc{}the fact that for every $i\in \{0,1\}$ it holds that $ p_i = [\sum_{l=1}^{k}\frac{1}{ q_{i,l} }]^{-1} $\cgs{}
therefore prove that
\begin{equation}
\begin{split}
&   \left\|
      e^{
        \int_0^{ \tau }
        V( X^x_s , X^y_s )      
        \, ds
      }
    \right\|_{
      L^p( \Omega; \R )
    }
\\ & \leq 
  R\cdot
  \prod_{ l = 1 }^k
  \left\|
  \exp\!\left(
    \smallint\limits_0^{ T }
      \tfrac{
        U_{0,l}( X^x_{s\wedge \tau} )
      }{ 2 q_{ 0, l } T e^{ \alpha_{ 0, l } {s\wedge \tau} } }
    -
    \smallint_0^{s\wedge \tau}
    \tfrac{
      \beta_{ 0, l }
    }{
      2 q_{ 0, l } T e^{ \alpha_{ 0, l } u }
    }
    \, du
    \,
    ds
  \right)
  \right\|_{ 
    L^{ 
      2 q_{ 0, l }
    }( \Omega; \R ) 
  } 
\\ & \quad \cdot
  \prod_{ l = 1 }^k
  \left\|
  \exp\!\left(
    \smallint\limits_0^{ T }
      \tfrac{
        U_{0,l}( X^y_{s\wedge \tau} )
      }{ 2 q_{ 0, l } T e^{ \alpha_{ 0, l } {s\wedge \tau} } }
    -
    \smallint_0^{s\wedge \tau}
    \tfrac{
      \beta_{ 0, l }
    }{
      2 q_{ 0, l } T e^{ \alpha_{ 0, l } u }
    }
    \, du
    \,
    ds
  \right)
  \right\|_{ 
    L^{ 
      2 q_{ 0, l }
    }( \Omega; \R ) 
  } 
\\ &  \quad 
  \cdot
  \prod_{ l = 1 }^k
  \left\|
  \exp\!\left(
    \tfrac{
      U_{1,l}( X^x_{ \tau } )
    }{ 2 q_{ 1, l } e^{ \alpha_{ 1, l } \tau } }
    +
    \smallint\limits_0^{ \tau }
      \tfrac{
        \overline{U}_{l}( X^x_s )
        -
        \beta_{ 1, l }
      }{ 2 q_{ 1, l } e^{ \alpha_{ 1, l } s } }
    \,
    ds
  \right)
  \right\|_{ 
    L^{ 
      2 q_{ 1, l }
    }( \Omega; \R ) 
  } 
\\ & \quad  
  \cdot
  \prod_{ l = 1 }^k
  \left\|
  \exp\!\left(
    \tfrac{
      U_{1,l}( X^y_{ \tau } )
    }{ 2 q_{ 1, l } e^{ \alpha_{ 1, l } \tau } }
    +
    \smallint\limits_0^{ \tau }
      \tfrac{
        \overline{U}_{l}( X^y_s )
        -
        \beta_{ 1, l }
      }{ 2 q_{ 1, l } e^{ \alpha_{ 1, l } s } }
    \,
    ds
  \right)
  \right\|_{ 
    L^{ 
      2 q_{ 1, l }
    }( \Omega; \R ) 
  }
  .
\end{split}
\end{equation}
Lemma~\ref{lem:integral_sup}
hence implies that
\begin{equation}
\begin{split}
&
    \left\|
      e^{
        \int_0^{ \tau }
        V( X^x_s , X^y_s )      
        \, ds
      }
    \right\|_{
      L^p( \Omega; \R )
    }
\\ &  \leq
  R 
  \cdot
  \prod_{ l = 1 }^k
\left|
  \E\!\left[ 
  \exp\!\left(
  {\scriptstyle
    \frac{
      U_{ 1, l }( X^x_{ \tau } )
    }{
      e^{ \alpha_{ 1, l } \tau }
    }
    +
    \int\limits_0^{ \tau }
    \frac{ 
      \overline{U}_{ l }( X^x_s )
      - \beta_{ 1, l }
    }{ e^{ \alpha_{ 1, l } s } }
    \,
    ds
  }
  \right)
  \right]
  \E\!\left[ 
  \exp\!\left(
  {\scriptstyle
    \frac{
      U_{ 1, l }( X^y_{ \tau } )
    }{
      e^{ \alpha_{ 1, l } \tau }
    }
    +
    \int\limits_0^{ \tau }
    \frac{ 
      \overline{U}_{ l }( X^y_s )
      - \beta_{ 1, l }
    }{e^{ \alpha_{ 1, l } s } }
    \,
    ds
  }
  \right)
  \right]
\right|^{
  \frac{ 1 }{ 2 q_{ 1, l } }
}
\\ & 
\cdot
  \prod_{ l = 1 }^k
  \left|
    \sup_{ s \in [0,T] }
    \E\!\left[
    \exp\!\left(
    {\scriptstyle
      \frac{
	U_{ 0, l }( X^x_{s\wedge \tau} )
      }{ e^{ \alpha_{ 0, l } {s\wedge \tau} } }
      -
      \int\limits_0^{s\wedge \tau}
      \frac{
	\beta_{ 0, l }
      }{
	e^{ \alpha_{ 0, l } u }
      }
      \, du
    }
    \right)
    \right]
    \sup_{ s \in [0,T] }
    \E\!\left[
    \exp\!\left(
    {\scriptstyle
      \frac{
	U_{ 0, l }( X^y_s )
      }{ e^{ \alpha_{ 0, l } s } }
      -
      \int\limits_0^s
      \frac{
	\beta_{ 0, l }
      }{
	e^{ \alpha_{ 0, l } u }
      }
      \, du
    }
    \right) 
    \right]
  \right|^{
    \frac{ 1 }{ 2 q_{ 0, l } }
  }
  .
\end{split}
\end{equation}
Corollary~\ref{cor:exp_mom}
therefore proves that
\begin{equation}
\begin{split}
&  
    \left\|
      e^{
        \int_0^{ \tau }
        V( X^x_s , X^y_s )      
        \, ds
      }
    \right\|_{
      L^p( \Omega; \R )
    }
\\
& \leq
  R \cdot*
  \left[
  \prod_{ l = 1 }^k
    \exp\!\left(
    \tfrac{
      U_{ 0, l }( x )
      +
      U_{ 0, l }( y )
    }{ 2 q_{ 0, l } }
    \right) 
  \right]
  \cdot
  \left[
  \prod_{ l = 1 }^k
  \exp\!\left(
    \tfrac{
      U_{ 1, l }( x )
      +
      U_{ 1, l }( y )
    }{
      2 q_{ 1, l }
    }
  \right)
  \right]\sgc{}.\cgs{}
\end{split}
\end{equation}
\sgc{}This\cgs{} implies that
\begin{equation*}
\begin{split}
    \left\|
      e^{
        \int_0^{ \tau }
        V( X^x_s , X^y_s )      
        \, ds
      }
    \right\|_{
      L^p( \Omega; \R )
    }  
  & \leq 
  R \cdot
  \exp\!\left(
    \sum_{ l = 1 }^k
    \left[
    \tfrac{
      U_{ 0, l }( x )
      +
      U_{ 0, l }( y )
    }{ 2 q_{ 0, l } }
    +
    \tfrac{
      U_{ 1, l }( x )
      +
      U_{ 1, l }( x )
    }{
      2 q_{ 1, l }
    }
    \right]
  \right) 
  .
\end{split}
\end{equation*}
The proof of Lemma~\ref{lem:multiple_exp}
is thus completed.
\end{proof}

In the next step we present
Theorem~\ref{thm:UV} 
which is the main result of
this subsection.
It is an immediate consequence
of 
Proposition~\ref{prop:two_solution_supoutside}
and
Lemma~\ref{lem:multiple_exp}.
Theorem~\ref{thm:UV} appeared in the special
case $ c = 0 $ and $ U_0 \equiv 0 \equiv U_1 $ 
in Theorem 1.2 in Maslowski~\cite{Maslowski1986}
(cf.\ Ichikawa~\cite{Ichikawa1984}
and, e.g., Leha \&\ Ritter~\cite{LehaRitter1994, LehaRitter2003}\sgc{})\cgs{}.
In Sections~\ref{sec:examples_SODE}
and~\ref{sec:examples_SPDE}
below various examples of SDEs
are presented that fulfill the 
assumptions of Theorem~\ref{thm:UV}.

\begin{theorem}
\label{thm:UV}
Assume the setting in Section~\ref{sec:setting},
let $ x, y \in O $,
$ k \in \N $, 
$\alpha_0,\alpha_1,\beta_0,\beta_1\in\R^k$,
$ 
  c \in \mathcal{L}^0( [0,T] ; \R ) 
$,
$ V \in C^{ 2 }( O^2, [0,\infty) ) $,
$ U_0 \in C^2( O, \R^k ) $,
$ U_1 \in C^2( O, [0,\infty)^k ) $,
$ \overline{U} \in C( O, \R^k ) $,
$ r, p\in(0,\infty]$,
$q_0,q_1 \in (0,\infty]^k $
with
$
  \frac{ 1 }{ p } + \sum_{ i = 0 }^1 \sum_{ l = 1 }^k \frac{ 1 }{ q_{ i, l } } 
$
$
  = \frac{ 1 }{ r }	
$
and let	
$ X^z \colon [0,T] \times \Omega \to O $,
$ z \in \{ x, y \} $,
be adapted
stochastic processes
with continuous sample paths
satisfying
\begin{equation*}
  \tfrac{
    ( \overline{ \mathcal{G} }_{ \mu, \sigma } V)( v, w )
  }{
    V( v, w )
  }
  +
    \tfrac{
      ( p - 1 ) \,
      \|
          ( \overline{ G }_{ \sigma } V )( v, w )
      \|^2
    }{
      2 \,
      ( 
        V( v, w ) 
      )^2
    }
\leq   
  c(t) 
  +
  \smallsum\limits_{ n = 1 }^k
  \left[
    \tfrac{ 
      U_{ 0, n }( v ) + U_{ 0, n }( w ) 
    }{ 2 q_{ 0, n } T e^{ \alpha_{ 0, n } t } } 
    +
    \tfrac{ 
      \overline{U}_n( v ) + \overline{U}_n( w ) 
    }{ 
      2 q_{ 1, n } 
      e^{ \alpha_{ 1, n } t }
    }
  \right]
\end{equation*}
and
\begin{equation*}
  ( \mathcal{G}_{ \mu, \sigma } U_{ i, l } )( u )
  +
  \tfrac{ 
    1
  }{ 
    2 
    e^{ 
      \alpha_{ i, l } t 
    }
  }
    \|
      \sigma( u )^* ( \nabla U_{ i, l } )( u )
    \|^2
  +
  \mathbbm{1}_{
    \{ 1 \}
  }(i)
  \cdot
  \overline{U}_l(u)
\leq
  \alpha_{ i, l } U_{ i, l }(u)
  +
  \beta_{ i, l }
\end{equation*}
for all
$ i,j \in \{ 0, 1 \} $,
$ l \in \{ 1, \dots, k \} $,
$ u \in \{ v, w \} $,
$ (v,w) \in \operatorname{im}( (X^x, X^y) ) $,
$ t \in [0,T] $
and
\begin{equation*}
  \int\nolimits_0^{ T }
    | c(s) |
    +
    \| \mu( X^z_s ) \| 
    +
    \| 
      \sigma( X^z_s )
    \|^2
    +
    \tfrac{
      \|
          ( \overline{ G }_{ \sigma } V )( X^x_s, X^y_s )
      \|^2
    }{
      ( 
        V( X^x_s, X^y_s ) 
      )^2
    }
    \,
  ds
  < \infty
\quad \P \text{-a.s.}
\end{equation*} 
and
$
  X^z_{ t } = 
  z
  + \int_0^{ t } \mu( X^z_s ) \, ds
  +
  \int_0^{ t } \sigma( X^z_s ) \, dW_s
$
$ \P $-a.s.\ for all $ (t,z) \in [0,T] \times \{ x, y \} $.
Then
\begin{equation*}
\begin{split}
&
    \left\|
      V( X^x_{ T }, X^y_{ T } )
    \right\|_{
      L^r( \Omega; \R )
    }
\\ & \leq
    \exp\!\left(
      \smallint_0^T c(s) \, ds
      +
      \smallsum\limits_{ i = 0 }^1
      \smallsum\limits_{ l = 1 }^k
      \left[
        \smallint\limits_0^T
          \frac{
            \beta_{ i, l }
            \,
            ( 1 - \frac{ s }{ T } )^{ 
              ( 1 - i )
            }
          }{ 
            q_{ i, l } \,
            e^{ \alpha_{ i, l } s } 
          }
        \, ds
        +
        \frac{
          U_{ i, l }( x ) + U_{ i, l }( y) 
        }{
          2 q_{ i, l }
        }
      \right]
    \right)
  V( x, y )
  \, .
\end{split}
\end{equation*}
\end{theorem}

Corollary~\ref{cor:UV_squared_norm} below specializes
Theorem~\ref{thm:UV} to the case where $ \mu $ and $ \sigma $
are locally Lipschitz continuous functions
\sgc{}and\cgs{}
where
$ V \in C^2( O^2, [0,\infty) ) $
satisfies
$
  V( x, y ) = \| x - y \|^2
$
for all $ x, y \in O $.
For this the following result from the literature
(see, e.g., 
Theorem 2.1 in Yamada \&\ Ogura~\cite{YamadaOgura1981}) is needed.

\begin{lemma}
\label{lem:nicht_treffen}
Assume the setting in Section~\ref{sec:setting},
let
$ \mu \colon O \to \R^d $
and
$ \sigma \colon O \to \R^{ d \times m } $
be locally Lipschitz continuous,
let $ \tau \colon \Omega \to [0,T] $
be a stopping time
and let	
$ X^i \colon [0,T] \times \Omega \to O $,
$ i \in \{ 1, 2 \} $,
be adapted stochastic processes with continuous
sample paths satisfying
$
  X^i_{ t \wedge \tau } = 
  X^i_0
  + \int_0^{ t \wedge \tau } \mu( X^i_s ) \, ds
  +
  \int_0^{ t \wedge \tau } \sigma( X^i_s ) \, dW_s
$
$ \P $-a.s.\ for all 
$ (t,i) \in [0,T] \times \{ 1, 2 \} $.
Then 
$
  \P\!\left[ 
    \left\{
      \exists
      \,
      t \in [0,\tau] 
      \colon
      X^1_t = X^2_t
    \right\}
    \cap
    \left\{
      X^1_0 \neq X^2_0
    \right\}
  \right]
  = 0
$.
\end{lemma}

We are now ready to present the promised
Corollary~\ref{cor:UV_squared_norm}.
It follows immediately from
Theorem~\ref{thm:UV}
and Example~\ref{ex:phi_identity}.

\begin{corollary}
\label{cor:UV_squared_norm}
Assume the setting in Section~\ref{sec:setting},
assume that
$
  \mu \colon O \to \R^d
$
and
$
  \sigma \colon O \to \R^{ d \times m } 
$
are locally Lipschitz continuous,
let 
$ \alpha_0, \alpha_1, \beta_0, \beta_1, c \in \R $, 
$ r, p, q_0, q_1 \in ( 0, \infty ] $
with
$
  \frac{ 1 }{ p } + \frac{ 1 }{ q_0 } + \frac{ 1 }{ q_1 } = \frac{ 1 }{ r }	
$,
$ U_0 \in C^2( O, \R ) $, 
$ U_1 \in C^2( O, [0,\infty) ) $,
$ \overline{U} \in C( O, \R ) $
and let	
$ X^x \colon [0,T] \times \Omega \to O $,
$ x \in O $,
be adapted stochastic processes with continuous sample paths 
satisfying
$
  X^x_t = 
  x
  + \int_0^t \mu( X^x_s ) \, ds
  +
  \int_0^t \sigma( X^x_s ) \, dW_s
$
$ \P $-a.s.,
$ 
  ( \mathcal{G}_{ \mu, \sigma } U_i )( x )
  +
  \frac{ 1 }{ 2 }
  \|
    \sigma( x )^* ( \nabla U_i )( x )
  \|^2
  +
  \overline{U}( x )
  \sgc{}\mathbbm{1}_{ \{ 1 \} }( i )\cgs{}
\leq
  \alpha_i U_i( x )
  + \beta_i
$
and
\begin{equation*}
\begin{split}
&  \tfrac{ 
    \langle
      x - y ,
      \mu( x ) - 
      \mu( y )
    \rangle
    +
    \frac{ 1 }{ 2 }
    \| 
      \sigma( x ) - \sigma( y ) 
    \|^2_{ \HS( \R^m, \R^d ) }
  }{
    \|
      x - y
    \|^2
  }
  {
    +
    }
    \tfrac{
      ( \frac{ p }{ 2 } - 1 ) 
      \| 
        (
          \sigma( x ) - \sigma( y )
        )^*
        ( x - y )
      \|^2
    }{
      \| x - y \|^4
    }
\\ & \leq   
  c 
  +
  \tfrac{
      U_0( x ) + U_0( y )
  }{
    2 q_0 T e^{ \alpha_0 t }
  }
  +
  \tfrac{
    \overline{U}( x ) + \overline{U}( y )
  }{
    2 q_1 e^{ \alpha_1 t }
  }
\end{split}
\end{equation*}
for 
all 
$ i \in \{ 0, 1 \} $,
$ (t,x,y) \in [0,T] \times O^2 $
with $ x \neq y $.
Then it holds for all $ x, y \in O $
that
\begin{equation}
\begin{split}
&
    \left\|
      X^x_{ T } - X^y_{ T }
    \right\|_{
      L^r( \Omega; \R^d )
    }
\\ & \leq
    \exp\!\left(
      c T +
      \smallsum_{ i = 0 }^1
      \left[
      \smallint\nolimits_0^T
      \tfrac{
        \beta_i
        \,
        (
          1 - \frac{ s }{ T }
        )^{
          ( 1 - i )
        }
      }{
        q_i e^{ \alpha_i s } 
      }
      \, ds
      +
      \tfrac{
        U_i( x ) + U_i(y) 
      }{
        2 q_i
      }
      \right]
    \right)
  \left\| x - y \right\|
  .
\end{split}
\end{equation}
\end{corollary}

\subsection{Uniform strong stability analysis for solutions of SDEs}
\label{sec:uniform}

\begin{prop}
\label{prop:two_solution_supinside}
Assume the setting in Section~\ref{sec:setting},
let $ x, y \in O $,
$ V \in C^{ 2 }( O^2, [0,\infty) ) $,
let $ \tau \colon \Omega \to [0,T] $
be a stopping time
and let	
$ X^z \colon [0,T] \times \Omega \to O $,
$ z \in \{ x, y \} $,
be adapted stochastic processes with continuous sample paths satisfying
$
  \int_0^{ \tau }
    \| \mu( X^z_s ) \| 
    +
    \| 
      \sigma( X^z_s )
    \|^2
    +
    \max\!\big(
      \frac{
        ( \overline{\mathcal{G} }_{ \mu, \sigma } V 
        )( X_s^x, X_s^y )
      }{
        V(X_s^x,X_s^y)
      }
      ,
      0
    \big)
    +
    \tfrac{
      \|
          ( \overline{ G }_{ \sigma } V )( X^x_s, X^y_s )
      \|^2
    }{
      ( 
        V( X^x_s, X^y_s ) 
      )^2
    }
  \, ds
  < \infty
$
$ \P $-a.s.\ and
$
  X^z_{ t \wedge \tau } = 
  z
  + \int_0^{ t \wedge \tau } \mu( X^z_s ) \, ds
  +
  \int_0^{ t \wedge \tau } \sigma( X^z_s ) \, dW_s
$
$ \P $-a.s.\ for all $ (t,z) \in [0,T] \times \{ x, y \} $.
Then
\begin{equation*}
\begin{split}
&  \bigg\|
    \sup_{ t \in [ 0, \tau ] }
    V( X^x_t, X^y_t )
  \bigg\|_{
    L^r( \Omega; \R )
  }
\\ & \qquad \qquad 
  \leq  
  \frac{ 
    V( x, y ) 
  }{ 
    \big[ 1 - \frac{ \theta }{ p } \big]^{ \frac{1}{\theta} } 
  }
  \Big\|
  \exp\!\Big(
        \big[
          \tfrac{ 1 }{ ( \frac{ 1 }{ p } - \frac{ 1 }{ v } ) } - \theta
        \big]
    \smallint_0^{ \tau }
    \tfrac{
      \|
          ( \overline{ G }_{ \sigma } V )( X^x_s, X^y_s )
      \|^2
    }{
      2 \,
      ( 
        V( X^x_s, X^y_s )
      )^2
    }
    \,
    ds
  \Big)
  \Big\|_{
    L^v( \Omega; \R )
  }
\\ & \qquad \qquad \quad  
  \cdot 
  \bigg\|
  \exp\!\Big(
    \sup_{ t \in [0,\tau] }
    \smallint_0^t
     \tfrac{
       ( \overline{\mathcal{G}}_{ \mu, \sigma } V)( X^x_s, X^y_s )
     }{
      V( X^x_s, X^y_s )
     }
    +
    \tfrac{
      (
        \theta 
        - 1
      ) \,
      \|
          ( \overline{ G }_{ \sigma } V )( X^x_s, X^y_s )
      \|^2
    }{
      2 \,
      ( 
        V( X^x_s, X^y_s )
      )^2
    }
    \,
    ds
  \Big)
  \bigg\|_{
    L^q( \Omega; \R )
  }
\end{split}
\end{equation*}
for all
$ v \in [ p , \infty] $,
$ \theta \in (0, p) $
and all
$ p, q, r \in (0,\infty] $
with $ \frac{ 1 }{ p } + \frac{ 1 }{ q } = \frac{ 1 }{ r } $.
\end{prop}

\begin{proof}[Proof
of Proposition~\ref{prop:two_solution_supinside}]
Let 
$ p, q ,r \in ( 0, \infty] $
with $ \frac{ 1 }{ p } + \frac{ 1 }{ q } = \frac{ 1 }{ r } $
and
$ \theta \in (0,p) $.
Then H\"{o}lder's inequality 
proves that
\begin{align}
\label{eq:Vsupinside_est}
&
  \bigg\|
    \sup_{ t \in [0,\tau] }
  \exp\!\left(
    \smallint_0^t
    {\scriptstyle
      \frac{
	( \overline{\mathcal{G}}_{ \mu, \sigma } V)( X^x_s, X^y_s )
      }{
	V( X^x_s, X^y_s )
      }
      -
      \frac{
	\|
	    ( \overline{ G }_{ \sigma } V )( X^x_s, X^y_s )
	\|^2
      }{
	2 \,
	( 
	  V( X^x_s, X^y_s )
	)^2
      }
      \,
    }
    ds
    +
    \smallint_0^t
    {\scriptstyle
      \frac{
	( \overline{ G }_{ \sigma } V )( X^x_s, X^y_s )
      }{
	V( X^x_s, X^y_s )
      }
    }
    \,
    dW_s
  \right)
  \bigg\|_{
    L^r( \Omega; \R )
  }
\\ &
\nonumber
  \leq
  \bigg\|
    \sup_{ t \in [0,\tau] }
  \exp\!\left(
    \int_0^t
     \tfrac{
       ( \overline{\mathcal{G}}_{ \mu, \sigma } V)( X^x_s, X^y_s )
     }{
      V( X^x_s, X^y_s )
     }
    +
    \tfrac{
      (
        \theta
        - 1
      )
      \,
      \|
          ( \overline{ G }_{ \sigma } V )( X^x_s, X^y_s )
      \|^2
    }{
      2 \,
      ( 
        V( X^x_s, X^y_s )
      )^2
    }
    \,
    ds
  \right)
  \bigg\|_{
    L^q( \Omega; \R )
  }
\\ & \quad
\nonumber
  \cdot 
  \bigg\|
    \sup_{ t \in [0,\tau] }
  \exp\!\left(
    \int_0^t
    \tfrac{
      ( \overline{ G }_{ \sigma } V )( X^x_s, X^y_s )
    }{
      V( X^x_s, X^y_s )
    }
    \,
    dW_s
    -
    \int_0^t
    \tfrac{
      \theta \,
      \|
          ( \overline{ G }_{ \sigma } V )( X^x_s, X^y_s )
      \|^2
    }{
      2 \,
      ( 
        V( X^x_s, X^y_s )
      )^2
    }
    \,
    ds
  \right)
  \bigg\|_{
    L^p( \Omega; \R )
  }
\\ & 
\nonumber
  =
  \bigg\|
  \exp\!\left(
    \sup_{ t \in [0,\tau] }
    \int_0^t
     \tfrac{
       ( \overline{\mathcal{G}}_{ \mu, \sigma } V)( X^x_s, X^y_s )
     }{
      V( X^x_s, X^y_s )
     }
    +
    \tfrac{
      (
        \theta 
        - 1
      ) \,
      \|
          ( \overline{ G }_{ \sigma } V )( X^x_s, X^y_s )
      \|^2
    }{
      2 \,
      ( 
        V( X^x_s, X^y_s )
      )^2
    }
    \,
    ds
  \right)
  \bigg\|_{
    L^q( \Omega; \R )
  }
\\ & \quad 
 \nonumber
  \cdot 
  \bigg\|
    \sup_{ t \in [0,T] }
  \exp\!\left(
    \int_0^t
    \tfrac{
      \,
      \theta \,
      ( \overline{ G }_{ \sigma } V )( X^x_s, X^y_s )
    }{
      V( X^x_s, X^y_s )
    }
    \,
    dW_s
    -
    \int_0^t
    \tfrac{
      \|
        \theta 
          ( \overline{ G }_{ \sigma } V )( X^x_s, X^y_s )
      \|^2
    }{
      2 \,
      ( 
        V( X^x_s, X^y_s )
      )^2
    }
    \,
    ds
  \right)
  \bigg\|_{
    L^{ p / \theta }( \Omega; \R )
  }^{ \frac{1}{\theta} }
  .
\end{align}
In addition, Lemma~\ref{l:exponential.martingale}
with 
$
  A_s
  =
  \tfrac{
      \mathbbm{1}_{
        \{ s < \tau \}
      }
        \theta 
          ( \overline{ G }_{ \sigma } V )( X^x_s, X^y_s )
    }{
      V( X^x_s, X^y_s )
    }
$,
$s\in [0,T]$,
gives
\begin{equation}
\label{eq:lem_application}
\begin{split}
&
  \bigg\|
    \sup_{ t \in [0,T] }
  \exp\!\left(
    \int_0^{t}
    \tfrac{
      \,
      \theta \, 
      ( \overline{ G }_{ \sigma } V )( X^x_s, X^y_s )
    }{
      V( X^x_s, X^y_s )
    }
    \,
    dW_s
    -
    \int_0^t
    \tfrac{
      \|
        \theta  
          ( \overline{ G }_{ \sigma } V )( X^x_s, X^y_s )
      \|^2
    }{
      2 \,
      ( 
        V( X^x_s, X^y_s )
      )^2
    }
    \,
    ds
  \right)
  \bigg\|_{
    L^{ p / \theta }( \Omega; \R )
  }^{ \frac{1}{\theta} }
\\ & \leq
  \inf_{ 
    v \in 
    [ \frac{ p }{ \theta }, \infty ]
  }
  \left[
  \frac{ 1 }{ ( 1 - \frac{ \theta }{ p } ) }
  \left\|
  \exp\!\left(
        \tfrac{ 1 }{ 2 }
        \Big[
          \tfrac{ 1 }{ 
            ( 
              \frac{ \theta }{ p }
              -
              \frac{ 1 }{ v } 
            ) 
          }
          - 1
        \Big]
    \int_0^{ \tau }
    \tfrac{
      \theta^2 
      \|
          ( \overline{ G }_{ \sigma } V )( X^x_s, X^y_s )
      \|^2
    }{
      ( 
        V( X^x_s, X^y_s )
      )^2
    }
    \,
    ds
  \right)
  \right\|_{
    L^v( \Omega; \R )
  }
  \right]^{ \frac{1}{\theta} }
\\ & =
  \frac{ 1 }{ 
    \left[ 1 - \frac{ \theta }{ p } \right]^{ \frac{1}{\theta} } 
  }
  \inf_{ 
    v \in 
    [ \frac{ p }{ \theta }, \infty ]
  }
  \left\|
  \exp\!\left(
        \Big[
          \tfrac{ 1 }{ 
            ( 
              \frac{ \theta }{ p }
              -
              \frac{ 1 }{ v } 
            ) 
          }
          - 1
        \Big]
    \int_0^{ \tau }
    \tfrac{
      \theta \,
      \|
          ( \overline{ G }_{ \sigma } V )( X^x_s, X^y_s )
      \|^2
    }{
      2 \,
      ( 
        V( X^x_s, X^y_s )
      )^2
    }
    \,
    ds
  \right)
  \right\|_{
    L^{ v \theta }( \Omega; \R )
  }
\\ & =
  \frac{ 1 }{ 
    \left[ 1 - \frac{ \theta }{ p } \right]^{ \frac{1}{\theta} } 
  }
  \inf_{ 
    v \in 
    [ p, \infty ]
  }
  \bigg\|
  \exp\!\left(
        \Big[
          \tfrac{ 1 }{ 
            ( 
              \frac{ 1 }{ p }
              -
              \frac{ 1 }{ v } 
            ) 
          }
          - \theta
        \Big]
    \int_0^{ \tau }
    \tfrac{
      \|
          ( \overline{ G }_{ \sigma } V )( X^x_s, X^y_s )
      \|^2
    }{
      2 \,
      ( 
        V( X^x_s, X^y_s )
      )^2
    }
    \,
    ds
  \right)
  \bigg\|_{
    L^{ v }( \Omega; \R )
  } 
  .
\end{split}
\end{equation}
Combining~\eqref{eq:Vsupinside_est}, 
\eqref{eq:lem_application} and 
Proposition~\ref{prop:two_solution}
proves that
\begin{equation}
\begin{split}
&
 \bigg\|
    \sup_{ t \in [0,\tau] }
    | V( X^x_t, X^y_t ) |
  \bigg\|_{
    L^r( \Omega; \R )
  }
\\
& 
\leq  
  \frac{ 
    | V(x, y) |
  }{ 
    \big[ 
      1 - \frac{ \theta }{ p } 
    \big]^{ \frac{1}{\theta} } 
  }
  \bigg\|
  \exp\!\left(
        \Big[
          \tfrac{ 1 }{ 
            ( 
              \frac{ 1 }{ p }
              -
              \frac{ 1 }{ v } 
            ) 
          }
          - \theta
        \Big]
    \int_0^{ \tau }
    \tfrac{
      \|
          ( \overline{ G }_{ \sigma } V )( X^x_s, X^y_s )
      \|^2
    }{
      2 \,
      ( 
        V( X^x_s, X^y_s )
      )^2
    }
    \,
    ds
  \right)
  \bigg\|_{
    L^{ v }( \Omega; \R )
  }
\\ & \quad 
  \cdot 
  \bigg\|
  \exp\!\bigg(
    \sup_{ t \in [0,\tau] }
    \int_0^t
     \tfrac{
       ( \overline{\mathcal{G}}_{ \mu, \sigma } V)( X^x_s, X^y_s )
     }{
      V( X^x_s, X^y_s )
     }
    +
    \tfrac{
      (
        \theta
        - 1
      ) \,
      \|
          ( \overline{ G }_{ \sigma } V )( X^x_s, X^y_s )
      \|^2
    }{
      2 \,
      ( 
        V( X^x_s, X^y_s )
      )^2
    }
    \,
    ds
  \bigg)
  \bigg\|_{
    L^q( \Omega; \R )
  }
\end{split}
\end{equation}
for all 
$ v \in [ p, \infty ] $.
This completes the proof of Proposition~\ref{prop:two_solution_supinside}.
\end{proof}

The next remark to Proposition~\ref{prop:two_solution_supinside}
is the uniform counterpart to Remark~\ref{remark:not_extended_system1}.

\begin{remark}
Note in the setting of
Proposition~\ref{prop:two_solution_supinside}
that if 
$ \hat{V} \in C^2( O, [0,\infty) ) $, 
if
$ x = y $
and if 
$
 V(v,w) = \hat{V}( v )
$
for all $ v , w \in O $,
then 
Proposition~\ref{prop:two_solution_supinside}
reduces to an estimate for
$
 \| 
   \sup_{ t \in [0, \tau] }
   \hat{V}( X^x_t ) 
 \|_{
   L^r( \Omega; \R )
 }
$
for 
$ p, q, r \in (0,\infty] $,
$ v \in [ p , \infty] $,
$ \theta \in (0, p) $
with $ \frac{ 1 }{ p } + \frac{ 1 }{ q } = \frac{ 1 }{ r } $.
\end{remark}

The next corollary, Corollary~\ref{cor:two_solution_supinside}, specializes 
Proposition~\ref{prop:two_solution_supinside}
to the case where
the function
$ V \in C^2( O^2, [0,\infty) ) $
satisfies
$
  V( x, y )
  = \| x - y \|^2
$
for all $ x, y \in O $
(cf.\ Theorem~5.1 in Li~\cite{Li1994}).
Corollary~\ref{cor:two_solution_supinside}
follows immediately from 
Proposition~\ref{prop:two_solution_supinside}
and Example~\ref{ex:phi_identity}.

\begin{corollary}
\label{cor:two_solution_supinside}
Assume the setting in Section~\ref{sec:setting},
let $ x, y \in O $,
let 
$
  \tau \colon \Omega \to [0,T]
$
be a stopping time
and let	
$ X^z \colon [0,T] \times \Omega \to O $,
$ z \in \{ x, y \} $,
be adapted stochastic processes with continuous sample paths 
satisfying
\begin{equation*}
  \int_0^{ \tau }
    \| \mu( X^z_s ) \| 
    +
    \| 
      \sigma( X^z_s )
    \|^2
    +
    \tfrac{
      \max(
        \langle
          X^x_s - X^y_s ,
          \mu( X^x_s ) - 
          \mu( X^y_s )
        \rangle
      ,0)
      +
      \| 
        \sigma( X^x_s ) - \sigma( X^y_s ) 
      \|^2
    }{
      \| X^x_s - X^y_s \|^2
    }
  \,
  ds
  < \infty
\end{equation*}
$ \P $-a.s.\ 
and
$
  X^z_{ t \wedge \tau } = 
  z
  + \int_0^{ t \wedge \tau } \mu( X^z_s ) \, ds
  +
  \int_0^{ t \wedge \tau } \sigma( X^z_s ) \, dW_s
$
$ \P $-a.s.\ for all $ (t,z) \in [0,T] \times \{ x, y \} $.
Then
\begin{equation}
\begin{split}
&
  \left\|
    \sup_{ t \in [0,\tau] }
    \| X^x_t - X^y_t \|
  \right\|_{
    L^r( \Omega; \R )
  }
\\[1ex] & 
\leq  
  \frac{ 
    \| x - y \|
  }{ 
    \big[ 
      1 - \frac{ \theta }{ p } 
    \big]^{ \frac{1}{\theta} } 
  }
  \left\|
  \exp\!\left(
        \big[
          \tfrac{ 1 }{ ( \frac{ 1 }{ p } - \frac{ 1 }{ v } ) } - \theta 
        \big]
    \smallint\nolimits_0^{ \tau }
    \tfrac{
      \| 
        (
          \sigma( X^x_s ) - \sigma( X^y_s )
        )^*
        ( X^x_s - X^y_s )
      \|^2
    }{
      2 \, \| X^x_s - X^y_s \|^4
    }
    \,
    ds
  \right)
  \right\|_{
    L^v( \Omega; \R )
  }
\\ & \quad 
  \cdot 
  \left\|
  \exp\!\left(
    \sup_{ t \in [0,\tau] }
    \int_0^t
    \left[
    \begin{split}
    \tfrac{
      \langle
        X^x_s - X^y_s ,
        \mu( X^x_s ) - 
        \mu( X^y_s )
      \rangle
      +
      \frac{ 1 }{ 2 }
      \| 
        \sigma( X^x_s ) - \sigma( X^y_s ) 
      \|^2_{ \HS( \R^m, \R^d ) }
    }{
      \| X^x_s - X^y_s \|^2
    }
    \\
    +
    \tfrac{
      (
        \frac{\theta}{2}  
        - 1
      ) \,
      \| 
        (
          \sigma( X^x_s ) - \sigma( X^y_s )
        )^*
        ( X^x_s - X^y_s )
      \|^2
    }{
      \| X^x_s - X^y_s \|^4
    }
    \end{split}
    \right]
    ds
  \right)
  \right\|_{
    L^q( \Omega; \R )
  }
\end{split}
\end{equation}
for all
$ v \in [ p , \infty ] $,
$ \theta \in (0, p) $ and all
$ p, q, r \in (0,\infty] $
with $ \frac{ 1 }{ p } + \frac{ 1 }{ q } = \frac{ 1 }{ r } $.
\end{corollary}

The next theorem is the uniform
counterpart to Theorem~\ref{thm:UV}.
It follows directly from
Proposition~\ref{prop:two_solution_supinside}
and from Lemma~\ref{lem:multiple_exp}.

\begin{theorem}
\label{thm:UV2}
Assume the setting in Section~\ref{sec:setting},
let $ x, y \in O $, $ k \in \N $, 
$ r, p \in ( 0, \infty ] $, 
$ \theta \in (0, p) $,
$ \rho \in [p,\infty] $,
$ 
  ( 
    \alpha_{ i, j, l }
  )_{ i, j \in \{ 0, 1 \}, l \in \{ 1, \dots, k \} } 
  ,
  ( 
    \beta_{ i, j, l }
  )_{ i, j \in \{ 0, 1 \}, l \in \{ 1, \dots, k \} } 
  \subseteq \R
$, \sgc{}let\cgs{}
$ 
  ( 
    q_{ i, j, l }
  )_{
    i, j \in \{ 0, 1 \},
    l \in \{ 1, \dots, k \} 
  } 
  \subset
  [r,\infty]
$
\sgc{}be extended real numbers satisfying\cgs{}
$
  \frac{ 1 }{ p } + 
  \sum_{ i = 0 }^1 \sum_{ l = 1 }^k \frac{ 1 }{ q_{ 1, i, l } } 
  = \frac{ 1 }{ r }	
$
\sgc{}and\cgs{} 
$
  \sum_{ i = 0 }^1 \sum_{ l = 1 }^k \frac{ 1 }{ q_{ 0, i, l } } 
  =
  \frac{ 1 }{ \rho },
$
$
  \frac{ 1 }{ p } + 
  \sum_{ i = 0 }^1 \sum_{ l = 1 }^k \frac{ 1 }{ q_{ 1, i, l } } 
  = \frac{ 1 }{ r },	
$ 
$ c_0, c_1 \in \mathcal{L}^0( [0,T] ; \R ) $,
$ V \in C^2( O^2, [0,\infty) ) $,
$
  U_{ 0, 0 }, U_{ 1, 0 } \in
  C^2( O, \R^k )
$,
$
  U_{ 0, 1 }, U_{ 1, 1 } 
  \in C^2( O, [0,\infty)^k )
$,
$
  \overline{U}_{ 0, 1 },
  \overline{U}_{ 1, 1 }
  \in C( O, \R^k )
$
and let	
$ X^z \colon [0,T] \times \Omega \to O $,
$ z \in \{ x, y \} $,
be adapted stochastic processes with continuous sample paths satisfying
$
  \int_0^T
    \max(c_0(s),0)
    +
    \max(c_1(s),0)
    +
    \| \mu( X^z_s ) \| 
    +
    \| 
      \sigma( X^z_s )
    \|^2
  \, ds
  < \infty
$
$ \P $-a.s.\ and
$
  X^z_t = 
  z
  + \int_0^t \mu( X^z_s ) \, ds
  +
  \int_0^t \sigma( X^z_s ) \, dW_s
$
$ \P $-a.s.\ for all $ (t,z) \in [0,T] \times \{ x, y \} $
and
\begin{equation}
\begin{split}
&  ( \mathcal{G}_{ \mu, \sigma } U_{ i, j, l } )( u )
  +
  \tfrac{ 
    1
  }{ 2 e^{ \alpha_{ i, j, l } t }
  }
    \|
      \sigma( u )^* ( \nabla U_{ i, j, l } )( u )
    \|^2
  +
  \mathbbm{1}_{
    \{ 1 \}
  }( j )
  \cdot
  \overline{U}_{i,j,l}(u)
\\ & \leq
  \alpha_{ i, j, l } U_{ i, j, l }(u)
  +
  \beta_{ i, j, l } ,
\end{split}
\end{equation}
\begin{equation}
\begin{split}
\\
&    \left[
      \tfrac{ 1 }{ ( 2 / p - 2 / \rho ) } 
      - 
      \tfrac{ \theta }{ 2 }
    \right]
    \tfrac{
      \|
          ( \overline{ G }_{ \sigma } V )( v, w )
      \|^2
    }{
      ( 
        V( v, w )
      )^2
    }
\\ &  \leq
    c_0(t)
    +
    \smallsum_{ n = 1 }^k
    \left[
    \tfrac{
      U_{0,0,n}( v )
      +
      U_{0,0,n}( w )
    }{
      2 q_{ 0, 0, n }
      T
      e^{  
        \alpha_{ 0, 0, n } t
      }
    }
    +
    \tfrac{
      \overline{U}_{0,1,n}( v )
      +
      \overline{U}_{0,1,n}( w )
    }{
      2 q_{ 0, 1, n }
      e^{  
        \alpha_{ 0, 1, n } t  
      }
    }
    \right]
\end{split}
\end{equation}
and
\begin{equation}
\begin{split}
&
  0 \vee
  \left[
    \tfrac{
      ( \overline{\mathcal{G}}_{ \mu, \sigma } V)( v, w )
    }{
      V( v, w )
    }
    +
    \tfrac{
      (
        \theta 
        - 1
      ) \,
      \|
          ( \overline{ G }_{ \sigma } V )( v, w )
      \|^2
    }{
      2 \,
      ( 
        V( v, w )
      )^2
    }
  \right]
\\ & \leq
  c_1(t)
  +
  \smallsum_{ n = 1 }^k
  \left[
    \tfrac{
      U_{1,0,n}( v )
      +
      U_{1,0,n}( w )
    }{
      2 
      q_{ 1, 0, n }
      T
      e^{  
        \alpha_{ 1, 0, n } t
      }
    }
    +
    \tfrac{
      \overline{U}_{1,1,n}( v )
      +
      \overline{U}_{1,1,n}( w )
    }{
      2 q_{ 1, 1, n }
      e^{  
        \alpha_{ 1, 1, n } t  
      }
    }
  \right]
\end{split}
\end{equation}
for all
$ i \in \{ 0, 1 \} $,
$ l \in \{ 1, \dots, k \} $,
$ u \in \{ v, w \} $,
$ 
  (v,w) \in 
  \operatorname{im}( (X^x,X^y) ) 
$,
$ t \in [0,T] $.
Then
\begin{equation}
\begin{split}
&  \bigg\|
    \sup_{ t \in [ 0, T ] }
    V( X^x_t, X^y_t )
  \bigg\|_{
    L^r( \Omega; \R )
  }
\leq  
  \frac{ 
    V( x, y ) 
  }{ 
    \big[ 1 - \frac{ \theta }{ p } \big]^{ \frac{1}{\theta} } 
  } 
  \exp\!\left(
    \smallint_0^T c_0(s) + c_1(s) \, ds
  \right)
\\ & \qquad \qquad \cdot
  \exp\!\left(
    \smallsum\limits_{ i, j = 0 }^1
    \sum\limits_{ l = 1 }^k
    \left[
      \smallint\limits_0^T
      \tfrac{
        \beta_{ i, j, l }
        \,
        ( 1 - \frac{ s }{ T } )^{ ( 1 - j ) }
      }{
        q_{ i, j, l } 
        \,
        e^{
          \alpha_{ i, j, l } s
        }
      }
      \, ds
      +
      \tfrac{
        U_{ i, j, l }( x ) 
        +
        U_{ i, j, l }( y ) 
      }{
        2 q_{ i, j, l }
      }
    \right]
  \right)
  .
\end{split}
\end{equation}
\end{theorem}

The next corollary specializes
Theorem~\ref{thm:UV2} to the case
where $ k = 1 $,
where $ V $ satisfies
$ V(x,y) = \| x - y \|^2 $
for all $ x, y \in O $
and where
$ \mu $ and $ \sigma $
are continuous
(cf.\ Lemma~2.3 in Zhang~\cite{Zhang2010}).
It follows directly from
Theorem~\ref{thm:UV2}
and from Example~\ref{ex:phi_identity}.

\begin{corollary}
\label{cor:UV2}
Assume the setting in Section~\ref{sec:setting},
let 
$ x, y \in O $,
$ r, p \in ( 0, \infty ] $, 
$ \theta \in (0, p) $,
$ \rho \in [p,\infty] $,
$ 
  ( 
    \alpha_{ i, j }
  )_{ i, j \in \{ 0, 1 \} 
  }
  ,
  ( 
    \beta_{ i, j }
  )_{ i, j \in \{ 0, 1 \} } 
  \subseteq \R
$,
$ 
  ( 
    q_{ i, j }
  )_{
    i, j \in \{ 0, 1 \}
  } 
  \subset
  [r,\infty]
$
with
$
  \sum_{ i = 0 }^1 \frac{ 1 }{ q_{ 0, i } } 
  =
  \frac{ 1 }{ \rho }
$
and
$
  \frac{ 1 }{ p } + 
  \sum_{ i = 0 }^1 \frac{ 1 }{ q_{ 1, i } } 
  = \frac{ 1 }{ r }	
$,
$
  c_0, c_1 \in C( [0,T], \R )
$,
$
  U_{ 0, 0 }, U_{ 1, 0 } \in
  C^2( O, \R )
$,
$
  U_{ 0, 1 }, U_{ 1, 1 } 
  \in C^2( O, [0,\infty) )
$,
$
  \overline{U}_{ 0, 1 },
  \overline{U}_{ 1, 1 }
  \in C( O, \R )
$, assume that
$ \mu \in C( O, \R^d ) $,
$ \sigma \in C( O, \R^{ d \times m } ) $
and let	
$ X^z \colon [0,T] \times \Omega \to O $,
$ z \in \{ x, y \} $,
be adapted stochastic processes with continuous sample paths satisfying
$
  X^z_t = 
  z
  + \int_0^t \mu( X^z_s ) \, ds
  +
  \int_0^t \sigma( X^z_s ) \, dW_s
$
$ \P $-a.s.\ for all $ (t,z) \in [0,T] \times \{ x, y \} $
and
\begin{equation*}
  ( \mathcal{G}_{ \mu, \sigma } U_{ i, j } )( u )
  +
  \tfrac{ 
    1
  }{ 2 e^{ \alpha_{ i, j } t }
  }
    \|
      \sigma( u )^* ( \nabla U_{ i, j } )( u )
    \|^2
  +
  \mathbbm{1}_{
    \{ 1 \}
  }( j )
  \cdot
  \overline{U}_{i,j}(u)
\leq
  \alpha_{ i, j } U_{ i, j }(u)
  +
  \beta_{ i, j } ,
\end{equation*}
\begin{equation*}
\label{eq:UV2_est0}
    \left[
      \tfrac{ 1 }{ ( 2 / p - 2 / \rho ) } 
      - 
      \tfrac{ \theta }{ 2 }
    \right]
  \tfrac{
    \left\|
      (
        v - w 
      )^*
      (
        \sigma( v ) - 
        \sigma( w )
      )
    \right\|^2
  }{
    \left\| v - w \right\|^4
  }
  \leq
    c_0(t) 
    +
    \tfrac{
      U_{0,0}( v )
      +
      U_{0,0}( w )
    }{
      2
      q_{ 0, 0 }
      T
      e^{  
        \alpha_{ 0, 0 } t
      }
    }
    +
    \tfrac{
      \overline{U}_{0,1}( v )
      +
      \overline{U}_{0,1}( w )
    }{
      2
      q_{ 0, 1 }
      e^{  
        \alpha_{ 0, 1 } t  
      }
    }
\end{equation*}
and
\begin{equation}\label{eq:UV2_est_mu_sigma}
\begin{split}
&
  \max\!\left\{
    0,
  \tfrac{ 
    \left<
      v - w ,
      \mu( v ) - 
      \mu( w )
    \right>
    +
    \frac{ 1 }{ 2 }
    \left\| 
      \sigma( v ) - \sigma( w ) 
    \right\|^2_{ \HS( \R^m, \R^d ) }
  }{
    \left\|
      v - w
    \right\|^2
  }
    +
  \tfrac{
      (
        \theta / 2
        - 1
      ) \,
    \left\|
      (
        v - w 
      )^*
      (
        \sigma( v ) - 
        \sigma( w )
      )
    \right\|^2
  }{
    \left\| v - w \right\|^4
  }
  \right\}
\\ & \leq
  c_1(t) 
  +
    \tfrac{
      U_{1,0}( v )
      +
      U_{1,0}( w )
    }{
      2 
      q_{ 1, 0 }
      T
      e^{  
        \alpha_{ 1, 0 } t
      }
    }
    +
    \tfrac{
      \overline{U}_{1,1}( v )
      +
      \overline{U}_{1,1}( w )
    }{
      2 q_{ 1, 1 }
      e^{  
        \alpha_{ 1, 1 } t  
      }
    }
\end{split}
\end{equation}
for all
$ i, j \in \{ 0, 1 \} $,
$ u \in \{ v, w \} $,
$ 
  (v,w) \in 
  \operatorname{im}( (X^x,X^y) ) 
$,
$ t \in [0,T] $.
Then
\begin{equation*}
\begin{split}
&
  \left\|
    \sup\nolimits_{ t \in [ 0, T ] }
    \frac{
      \| X^x_t - X^y_t \|
    }{ 
      \| x - y \|
    }
  \right\|_{
    L^r( \Omega; \R )
  }
\\ &
\leq  
  \frac{ 
    e^{ \int_0^T c_0(s) + c_1(s) \, ds }
  }{ 
    [ 1 - \frac{ \theta }{ p } ]^{ \frac{1}{\theta} } 
  }
  \exp\!\bigg(
    \smallsum\limits_{ i, j = 0 }^1
    \left[
      \smallint\limits_0^T
      \tfrac{
        \beta_{ i, j }
        \,
        ( 1 - \frac{ s }{ T } )^{ ( 1 - j ) }
      }{
        q_{ i, j } 
        \,
        e^{
          \alpha_{ i, j } s
        }
      }
      \, ds
      +
      \tfrac{
        U_{ i, j }( x ) 
        +
        U_{ i, j }( y ) 
      }{
        2 q_{ i, j }
      }
    \right]
  \bigg) 
  .
\end{split}
\end{equation*}
\end{corollary}

The next corollary (Corollary~\ref{cor:UV3}) is the special case of
Corollary~\ref{cor:UV2} 
where 
$
  U_{0,0}(x)
  =
  U_{1,0}(x)
  =
  c \, (1+\|x\|^2)^{\eps}
$
and 
$
  U_{0,1}(x)=U_{1,1}(x)=\overline{U}_{0,1}(x)=\overline{U}_{1,1}(x)=0
$ for all $x\in\R^d$
and some $c\in(0,\infty)$, $\eps\in(0,1]$
and where 
$q_{0,1}=\infty=q_{1,1}$, $q_{1,0}=2r=p$, $q_{0,0}=4r=\rho$, $\theta=2$,
$\beta_{0,0}=\beta_{1,0}=\beta$, $\beta_{0,1}=\beta_{1,1}=0$,
$ r \in ( 1, \infty ] $.
Corollary~\ref{cor:UV3} is related to
Theorem~1.7 in Fang, Imkeller \&\ Zhang~\cite{FangImkellerZhang2007}
and to Corollary~6.3 in Li~\cite{Li1994}.

\begin{corollary}
\label{cor:UV3}
Assume the setting in Section~\ref{sec:setting},
let 
$ x, y \in O $, assume that
$ \mu \in C( O, \R^d ) $, $\sigma \in C(O,\R^{d\times m})$,
let
$ \alpha, \beta, c_0,c_1 \in [0,\infty) $,
$c\in(0,\infty)$,
$ r \in ( 1, \infty ] $,
$ \eps \in (0,1]$
and let	
$ X^z \colon [0,T] \times \Omega \to O $,
$ z \in \{ x, y \} $,
be adapted stochastic processes with continuous sample paths satisfying
$
  X^z_t = 
  z
  + \int_0^t \mu( X^z_s ) \, ds
  +
  \int_0^t \sigma( X^z_s ) \, dW_s
$
$ \P $-a.s.\ for all $ (t,z) \in [0,T] \times \{ x, y \} $
and
\begin{align}
\nonumber  
&  \tfrac{
    \left<
      v - w ,
      \mu( v ) - 
      \mu( w )
    \right>
    +
    \frac{1}{2}
    \|\sigma(v)-\sigma(w)\|_{\HS(\R^m,\R^d)}^2
   }{
     \|v-w\|^2
   } 
  \leq
    c_1
    +
    \tfrac{
      c
      \, (1+\| v \|^2)^{\eps}
      +
      c
      \, 
      (1+\| w \|^2)^{\eps}
    }{
      4r
      T
      e^{  
        \alpha T
      }
    }
    ,
 \\ 
&    \tfrac{
     ( 4 r - 2 )
     \,
     \| ( \sigma(v) - \sigma(w) )^{*} ( v - w ) 
     \|^2
    }{
      2 \, \| v - w \|^4
    } 
  \leq
    c_0
    +
    \tfrac{
      c 
      \,
      (1+\| v \|^2)^{\eps}
      +
      c
      \, 
      (1+\| w \|^2)^{\eps}
    }{
      8r
      T
      e^{  
        \alpha T
      }
    }
    ,
\\ \nonumber
&  2
  \left< u, \mu(u) \right>
  +
  \| \sigma(u) \|^2_{ \HS( \R^m, \R^d ) }
  +
  \tfrac{
    2 c \eps 
    \| \sigma( u )^* u \|^2
  }{
    ( 1 + \| u \|^2 )^{ ( 1 - \eps ) } 
  }
\\ \nonumber & \qquad 
  \leq
  \tfrac{ \alpha }{ \varepsilon }
  \left( 1 + \| u \|^2 \right)
  +
  \tfrac{
    \beta  
  }{
    c \varepsilon
  }
  \left(
    1 + \| u \|^2
  \right)^{ ( 1 - \eps ) }
\end{align}
for all
$ u \in \{ v, w \} $,
$ 
  (v,w) \in 
  \operatorname{im}( (X^x,X^y) )  
$.
Then
\begin{equation}
\begin{split}
&
  \left\|
    \sup\nolimits_{ t \in [ 0, T ] }
    \frac{
      \| X^x_t - X^y_t \|
    }{
      \| x - y \|
    }
  \right\|_{
    L^r( \Omega; \R )
  }
\\ &
\leq  
  \left[
      1 - \tfrac{ 1 }{ r } 
    \right]^{ - \frac{ 1 }{ 2 } }
  \exp\!\left(
    (c_0+c_1) T +
      \tfrac{
        \beta T 
        +
        c(1+\| x \|^2)^{\eps}
        +
        c(1+\| y \|^2)^{\eps}
      }{
        2 r
      }
  \right) 
  .
\end{split}
\end{equation}
\end{corollary}

\chapter{Strong completeness of SDEs}
\label{sec:strong_completeness}

The theory developed in Subsection~\ref{sec:two_solution} 
can be used to establish \emph{strong completeness} of SDEs
with non-globally Lipschitz continuous
nonlinearities by combining it with a suitable Kolmogorov
argument; see Theorem~\ref{thm:strong_completeness_uniform}
and Theorem~\ref{thm:strong_completeness_marginal}
below.

\section{Theorems of Kolmogorov-Chentsov type}
The proof of Kolmogorov-Chentsov's continuity theorem for stochastic processes indexed by a bounded subset of $\R^d$, 
i.e., Theorem~\ref{t:KolChen} below, is based 
on an extension result\sgc{};\cgs{} see Theorem~\ref{thm:extension_Holdercont_Rd} below, and Kolmogorov-Chentsov's continuity
theorem for stochastic fields indexed by a cube\sgc{};\cgs{} see Theorem~\ref{t:KolChen_rect} below. 
Theorem~\ref{thm:extension_Holdercont_Rd} below
for the case $E=\R$ is contained in~\cite[Theorem 3 in Section VI.2.2.1]{Stein1970},
the proof of Theorem~\ref{thm:extension_Holdercont_Rd} below follows the proof of~\cite[Theorem 3 in Section VI.2.2.1]{Stein1970}
mutatis mutandis. 
Kolmogorov-Chentsov's continuity criterion for stochastic fields indexed by a cube is well-known\sgc{};\cgs{} see e.g.~\cite[Theorem I.2.1]{RevuzYor1994}. The idea to combine Theorem~\ref{thm:extension_Holdercont_Rd} and Theorem~\ref{t:KolChen_rect} below in order to obtain Theorem~\ref{t:KolChen} below is due to~\cite{MittmannSteinwart2003}\sgc{};\cgs{} see~\cite[Theorem 2.1]{MittmannSteinwart2003} for a result very similar to~Theorem~\ref{t:KolChen} below. In particular, we do not claim originality of Theorem~\ref{thm:extension_Holdercont_Rd}, Theorem~\ref{t:KolChen_rect} or Theorem~\ref{t:KolChen} below, we only provide the proofs of these theorems for the readers' convenience.

\subsection{Preparatory lemmas}

In this section we collect some well-known results which are used in the proofs of 
Theorem~\ref{thm:extension_Holdercont_Rd}, Theorem~\ref{t:KolChen_rect} and Theorem~\ref{t:KolChen} below. 
Theorem~\ref{thm:Bochnerspace} below states
that the Bochner space equipped with the $L^p$-norm is a Banach space.

\begin{theorem}\label{thm:Bochnerspace}
Let $p\in [1,\infty)$, let $(S,\Sigma,\mu)$ be a measure space 
and let $(E,\left\| \cdot \right\|_E)$ be a separable $\R$-Banach space.
Then $(L^p(S;E),\left\| \cdot \right\|_{L^p(S;E)})$ is an $\R$-Banach space. 
\end{theorem}

\begin{proof}[Proof of Theorem~\ref{thm:Bochnerspace}]
Note that for all $f\in \calL^p(S;E)$, $(S_n)_{n\in\N} \in \calF^{\N}$ with $\sgc{}\forall\cgs{} \, n\in \N \colon S_n=\{s\in S\colon \|f(s)\|_E>\frac{1}{n}\}$ it holds for all $n\in \N$, $s\in S\setminus \cup_{m\in \N} S_m$ that $\mu(S_n)<\infty$ and $f(s)=0$.
This and, e.g., \cite[Theorem 1.1.6, Proposition 1.1.16 and the argument on Page 21]{HytonenEtAl:2016} 
imply that $(L^p(S;E),\left\| \cdot \right\|_{L^p(S;E)})$ is an $\R$-Banach space.
\end{proof}

Lemma~\ref{lem:simple_Hoelder} below states a well-known relation between the H\"older space and the H\"older semi-norm.

\begin{lemma}\label{lem:simple_Hoelder}
Let $d\in \N$, $\alpha \in (0,1]$, let
$(E,\left\|\cdot\right\|_{E})$ be an $\R$-Banach space, let $D\subseteq \R^d$ be
bounded and non-empty and let $f\colon D\rightarrow E$ be a function. 
Then  $f\in C^{\alpha}(D,E)$ if and only if $\| f \|_{\calC^{\alpha}(D,E)}<\infty$.
\end{lemma}

\begin{proof}[Proof of Lemma~\ref{lem:simple_Hoelder}]
Observe that for all $x\in D$ it holds that 
\begin{equation}
 \| f \|_{\calC^{\alpha}(D,E)}
 \leq 
 \| f \|_{C^{\alpha}(D,E)}
 \leq 
 \| f(x) \|_E 
 +
 (1+ \sup_{y\in D} \| x - y  \|^{\alpha} )\| f \|_{\calC^{\alpha}(D,E)}.
\end{equation}
This completes the proof of Lemma~\ref{lem:simple_Hoelder}.
\end{proof}

Lemma~\ref{lem:extension_Holdercont_closure} below states a simple extension result for H\"older continuous
functions.

\begin{lemma}\label{lem:extension_Holdercont_closure}
Let $d\in \N$, $\alpha \in (0,1]$, let $D\subseteq \R^d$ be non-empty, 
let $(E,\left\|\cdot\right\|_{E})$ be an $\R$-Banach space and let $f\colon D \rightarrow E$ satisfy 
$\| f \|_{\calC^{\alpha}(D,E)}<\infty$. 
Then there exists $\bar{f}\colon \overline{D} \rightarrow E$ such that
\begin{enumerate}
 \item\label{it:eHc_lim} it holds for all $x\in \overline{D}$ 
 that $\limsup_{D \ni y \rightarrow x} |\bar{f}(x)-f(y)|=0$,
 \item\label{it:eHc_restriction} it holds that $\bar{f}|_{D} = f$,
 \item\label{it:eHc_image} it holds that $\textup{im}(\bar{f}) \subseteq \overline{ \textup{im}(f)}^{\|\cdot \|_E}$ and
 \item\label{it:eHc_Holder} it holds that $ \| \bar{f} \|_{\calC^{\alpha}(\overline{D},E)} 
 =
 \| f \|_{\calC^{\alpha}(D,E)}$.
\end{enumerate}
\end{lemma}

\begin{proof}[Proof of Lemma~\ref{lem:extension_Holdercont_closure}]
First, note that for every Cauchy sequence $(x_n)_{n\in \N} \in D^{\N}$ and every $k,m\in \N$ it holds that 
$\| f(x_k) - f(x_m) \|_{E} 
\leq 
\| f \|_{\calC^{\alpha}(D,E)}\| x_k - x_m \|^{\alpha} $ and that 
$(f(x_n))_{n\in \N} \in E^{\N}$ is a Cauchy sequence in $(E,\left\| \cdot \right\|_E)$. 
This implies for all $x\in \overline{D}$ 
that there exists a unique $z\in \overline{\textup{im}(f)}^{\left\| \cdot \right\|_E}$ such that
$\lim_{D \ni y\rightarrow x}\| f(y) - z \|_E = 0$.
\sgc{}Hence, we obtain\cgs{} that there exists $\bar{f}\colon \overline{D} \rightarrow E$ with $\sgc{}\forall\cgs{} \, x\in \overline {D}\colon$
$ \bar{f}(x)
 =
 \lim_{
    D \ni y \rightarrow x
 }
 f(y).$
This \sgc{}establishes\cgs{}~\eqref{it:eHc_lim}--\eqref{it:eHc_image}. 
Finally, observe that if $\bar{f}\colon \overline{D}\rightarrow E$ satisfies~\eqref{it:eHc_lim}, then 
$\bar{f}$ satisfies~\eqref{it:eHc_Holder}. This completes the proof of Lemma~\ref{lem:extension_Holdercont_closure}.
\end{proof}

Lemma~\ref{lem:image_in_closed_set} below provides conditions on a closed set $F$ and a random field $X$
such that $X$ is indistinguishable from a random field taking values in $F$.

\begin{lemma}\label{lem:image_in_closed_set}
Let $d\in \N$, let $D\subseteq \R^d$ be non-empty and closed,
let $(E,\left\|\cdot\right\|_{E})$ be a separable $\R$-Banach space, let $F\subseteq E$ be non-empty and closed, 
let $(\Omega,\calF,\P)$ be a probability space\sgc{},\cgs{} let $X\in \calL^0(D \times \Omega;E)$, \sgc{}assume for 
all $x\in D$ that\cgs{} $\P\big[ X(z) \in F\big] = 1$ and \sgc{}assume for all $\omega \in \Omega$\cgs{} that $D \ni x \mapsto X(x,\omega) \in E$ is continuous.
Then there exists $Y\in \calL^0(D\times \Omega;F)$ such that
\begin{enumerate}
 \item\label{it:image_in_closed_set:cont} it holds for all $\omega \in \Omega$ that the function $D \ni x \mapsto Y(x,\omega) \in F$ is continuous and
 \item\label{it:image_in_closed_set:modification} it holds that $\P\big[\{\omega \in \Omega\colon (\sgc{}\forall\cgs{} \, x\in D \colon X(x,\omega)=Y(x,\omega) )\}\big]=1$.
\end{enumerate}
\end{lemma}

\begin{proof}[Proof of Lemma~\ref{lem:image_in_closed_set}]
Throughout this proof let $f_0\in F$.
First, observe that the fact that $\R^d$ is separable and the fact that $D$ is closed 
imply that there exist
$ (x_n)_{n\in \N}  \in D^{\N} $
such that
$
  \overline{
    \{ x_n \in D \colon n \in \N \}
  }
  = D
$.
This, the closedness of $F$
and the pathwise continuity of $X$ imply that there exists $ \tilde\Omega \in \mathcal{F} $
such that $\P\big[\tilde{\Omega}\big]=1$
and 
\begin{equation}
 \tilde{\Omega} 
  = 
  \cap_{ n \in \N }
  \{
    \omega \in \Omega 
    \colon
      X( x_n ,\omega ) 
      \in 
      F
  \}
  =
  \cap_{ x \in D }
  \{
    \omega \in \Omega 
    \colon
      X( x ,\omega ) 
      \in 
      F
  \}.
\end{equation}
Let $Y\sgc{}\in\cgs{} \calL^0(D\times \Omega;F)$
satisfy for all $x\in D, \omega \in \Omega$ that
\begin{equation}
 Y(x,\omega)
 =
 \begin{cases}
  \sgc{}X(x,\omega)\cgs{}
  &
  \sgc{}\colon\cgs{}\omega \in \sgc{}\tilde{\Omega}\cgs{}
  \\
  \sgc{}f_0\cgs{}
  &
  \sgc{}\colon\cgs{}\omega \notin \sgc{}\tilde{\Omega}\cgs{}
 \end{cases}\sgc{}.\cgs{}
\end{equation}
\sgc{}Note that\cgs{} $Y$ satisfies~\eqref{it:image_in_closed_set:cont} and ~\eqref{it:image_in_closed_set:modification}.
This completes the proof of Lemma~\ref{lem:image_in_closed_set}.
\end{proof}

We close this section with a simple result regarding closed subsets of $L^p(\Omega;E)$.

\begin{lemma}\label{lem:simple_closed_subset}
Let $p\in [1,\infty)$, let $(S,\Sigma,\mu)$ be a measure space, let $(E,\left\| \cdot \right\|_E)$
be a separable $\R$-Banach space and let $F\subseteq E$ be closed.
Then the set 
\begin{equation}
\begin{aligned}
\Big\{ \xi \in L^p(S;E) 
\colon 
\mu\big[ \xi \in E\backslash F \big] = 0
\Big\}
\end{aligned}
\end{equation}
is a closed subset of $L^p(S;E)$.
\end{lemma}

\begin{proof}[Proof of Lemma~\ref{lem:simple_closed_subset}]
Note that for all $\xi_{\infty} \in \calL^p(S;E)$, $(\xi_n)_{n\in \N} \in (\calL^p(S;E))^{\N}$ with $\sgc{}\forall\cgs{} \,n\in \N\colon \mu\big[ \xi_n \in E\backslash F \big] = 0$ and
$\lim_{m\rightarrow \infty}\| \xi_{\infty} - \xi_m \|_{L^p(S;E)} =0$ it holds that
\begin{equation}\begin{aligned}
& \mu\big[ \xi_{\infty} \in E\backslash F \big]
 = 
\lim_{j\rightarrow \infty}
\mu\big[ \{\omega \in \Omega \colon \inf_{z\in F}\| \xi_{\infty}(\omega) - z \|_{E} > \tfrac{1}{j} \} \big]
\\ &\leq 
\lim_{j\rightarrow \infty} 
\limsup_{n\rightarrow \infty}
\mu\big[ \| \xi_{\infty} - \xi_n \|_E > \tfrac{1}{j} \big]
\leq 
\lim_{j\rightarrow \infty} j^p \big[
\limsup_{n\rightarrow \infty} \| \xi_{\infty} - \xi_n \|_{L^p(S;E)}^p
\big]
=
0.
\end{aligned}\end{equation}
This completes the proof 
of Lemma~\ref{lem:simple_closed_subset}.
\end{proof}

\subsection{An extension result}

Theorem~\ref{thm:partition_of_unity} below is the key ingredient of 
the extension result provided by Theorem~\ref{thm:extension_Holdercont_Rd} below.

\begin{theorem}\label{thm:partition_of_unity}
Let $d\in \N$. Then there exists \sgc{}$C\in \R$\cgs{}
such that for all $\alpha \in (0,1]$, 
every non-empty and closed $D\subsetneq \R^d$, every $\R$-Banach space $(E,\left\| \cdot \right\|_E)$
and every function $f \colon D \rightarrow E$ 
with $\| f \|_{\calC^{\alpha}(D,E)}<\infty$ there 
exist 
$(\rho_n)_{n\in\N} \in (0,\infty)^{\N}$,
$(p_n)_{n\in \N}\in (\R^d)^{\N}$, 
$(x_n)_{n\in \N} =((x_{n,1},\ldots,x_{n,d}))_{n\in \N} \in (\R^d)^{\N}$, 
$(Q_n)_{n\in \N}, (Q_n^*)_{n\in \N} \in (\mathcal{P}(\R^d))^{\N}$, 
$(\sgc{}\varphi_n\cgs{})_{n\in \N} \in (C^{\infty}(\R^d,\R))^{\N}$
and functions
$r\colon \R^d\backslash D\rightarrow (0,\infty)$,
$I \colon \R^d\backslash D \rightarrow \mathcal{P}(\N)$ and
$\bar{f}\colon \R^d \rightarrow E$
such that
\begin{enumerate}
 \item\label{it:pou_def_Q} it holds for all $n\in \N$  that 
 \begin{equation}
    Q_n 
    = 
    [x_{n,1}- \rho_n, x_{n,1}+ \rho_n]
    \times \ldots \times 
    [x_{n,d} - \rho_n, x_{n,d}+\rho_n], 
 \end{equation}
 \item\label{it:pou_cubes_cover} it holds that $\cup_{n\in\N} Q_n = \R^d \backslash D$,
 \item\label{it:pou_dist_D_scales_size_cubes} it holds for all $n\in \N$ that 
 \begin{equation}
  2 \sqrt{d} \rho_n 
  \leq 
  \textup{dist}(Q_n,D) 
  \leq  
  8 \sqrt{d} \rho_n,
 \end{equation}
 \item\label{it:pou_touching_cubes_comparable_sizes} it holds for all $n,m\in \N$ that 
 \begin{equation} 
  \sgc{}\left(\left( 
    Q_n\cap Q_m\neq \emptyset
  \right)
  \Rightarrow 
  \left( 
    \tfrac{1}{4} \rho_n  
    \leq 
    \rho_m 
    \leq 
    4 \rho_n
  \right)\right)\cgs{}
  ,
 \end{equation}
 \item\label{it:pou_Qn_has_limited_neighbours} it holds for all $n\in \N$ that 
 \begin{equation}
  \#\{ 
    j\in \N 
    \colon
    Q_j \cap Q_n \neq \emptyset
  \}
  \leq 12^d,
 \end{equation}
 \item\label{it:pou_def_Qstar} it holds for all $n\in \N$ that
 \begin{equation}
    Q_n^*
    = 
    [x_{n,1}- \tfrac{9}{8} \rho_n, x_{n,1}+ \tfrac{9}{8} \rho_n]
    \times \ldots \times 
    [x_{n,d} - \tfrac{9}{8} \rho_n, x_{n,d}+\tfrac{9}{8} \rho_n], 
 \end{equation}
\item\label{it:pou_Qstarintersects_iff_Qtouches} it holds for all $n,m\in\N$ that
\begin{equation}
\sgc{}\left(\left( 
    Q_n^*\cap Q_m \neq \emptyset
\right)
\Rightarrow 
\left(
    Q_n \cap Q_m \neq \emptyset
\right)\right)\cgs{},
\end{equation}
\item\label{it:pou_Qstarsintersect_iff_Qtouches} it holds 
for all $n,i\in\N$ with $Q_n^*\cap Q_i^*\neq \emptyset$ that there exists $m\in \N$ 
such that $Q_n\cap Q_m\neq \emptyset$ and $Q_i\cap Q_m\neq \emptyset$,
\item\label{it:pou_phi_range} it holds for all $n\in \N$ that $\textup{im}(\sgc{}\varphi_n\cgs{})\subseteq [0,1]$,
\item\label{it:pou_phi_support} it holds for all $n\in \N$, $x\in \R^d\backslash Q_n^*$ that $\sgc{}\varphi_n\cgs{}(x)=0$,
\item\label{it:pou_pou} it holds for all $x\in \R^d\backslash D$ that 
\begin{equation}
 \sum_{n\in \N} \varphi_n(x) = 1,
\end{equation}
\item\label{it:pou_phi_derivest} it holds that $\sup_{n\in \N, x\in \R^d} \textup{dist}(x,D) \| \nabla \sgc{}\varphi_n\cgs{}(x) \| \leq C$,
 \item\label{it:pou_def_r} it holds for all $x\in \R^d\backslash D$ that $\inf_{n\in \N,\, Q_n \ni x} \rho_n \in (0,\infty)$
 and $r(x)=\frac{1}{8} \inf_{n\in \N,\, Q_n \ni x} \rho_n$,
\item\label{it:pou_Br_in_Q} it holds for all $x\in \R^d\backslash D$, $i,n\in \N$ with $x\in Q_n$ and $B_{r(x)}(x)\cap Q_i^* \neq \emptyset$ that there exists $m\in \N$
such that $Q_n\cap Q_m\neq \emptyset$ and $Q_m \cap Q_i\neq \emptyset$,
\item\label{it:pou_Ixdef} it holds for all $x\in \R^d\backslash D$ 
that $I(x)= \{i\in \N \colon B_{r(x)}(x)\cap Q_i^*\neq \emptyset \}$,
\item\label{it:pou_Ixsize} it holds for all $x\in\R^d\backslash D$ 
that $\#I(x)\in \{1,\ldots,12^{2d}\}$,
\item\label{it:pou_pdef} it holds for all $n\in \N$ that $p_n\in D$ and $\textup{dist}(p_n,Q_n) = \textup{dist}(D,Q_n)$, 
\item\label{it:pou_pproperty} it holds that $\sup_{x\in \R^d\backslash D} \sup_{y\in D,\, i\in I(x)} \frac{\| p_i - y \|}{\| x-y\|} \leq 86$,
\item\label{it:pou_def_barf} it holds for all $x\in \R^d$ that $\#\{ n\in \N \colon \sgc{}\varphi_n\cgs{}(x)>0\} < \infty$ and
\begin{align}
 \bar{f}(x) 
 =
 \begin{cases}
  \sgc{}f(x)\cgs{}
  &
  \sgc{}\colon\cgs{} x\in \sgc{}D\cgs{}
  \\
  \sum_{n\in \N} \sgc{}\varphi_n\cgs{}(x) \sgc{}f(p_n)\cgs{}
  &
  \sgc{}\colon\cgs{} x\in \R^d\backslash \sgc{}D\cgs{}
 \end{cases}\sgc{},\cgs{}
\end{align}
\item\label{it:pou_barf_finite_sum} it holds for all $x\in\R^d\backslash D, z\in B_{r(x)}(x)\cap(\R^d\backslash D)$
that
\begin{equation}
 \bar{f}(z) = \sum_{n\in I(x)} \sgc{}\varphi_n\cgs{}(z) f(p_n),
\end{equation}
\item\label{it:pou_barf_diffble} it holds for all $x\in \R^d\backslash D$ that $\bar{f}$ is differentiable in $x$
and 
\begin{equation}
 \| \nabla \bar{f} (x) \|_{L(\R^d,E)} 
 \leq 
  86 \sgc{}\cdot\cgs{} 12^{2d} \sgc{}C\cgs{} \| f \|_{\calC^{\alpha}(D,E)} \left|\textup{dist}(x,D)\right|^{-1+\alpha}
\end{equation} and
\item\label{it:pou_barf_hoelder} it holds that $\| \bar{f} \|_{\calC^{\alpha}(\R^d,E)} 
\leq 
86\cdot 12^{2d} \max\{4,C \}
\| f \|_{\calC^{\alpha}(D,E)}.$
\end{enumerate}
\end{theorem}

\begin{proof}[Proof of Theorem~\ref{thm:partition_of_unity}]
First,~\cite[Theorem 1 in Section VI.1.1]{Stein1970} implies that there exist $(x_n)_{n\in \N} = ((x_{n,1},\ldots,x_{n,d}))_{n\in \N} \in (\R^d)^{\N}$, $(\rho_n)_{n\in \N}\in (0,\infty)^{\N}$ and $(Q_n)_{n\in \N}\in (\mathcal{P}(\R^d))^{\N}$
such that ~\eqref{it:pou_def_Q}--\eqref{it:pou_dist_D_scales_size_cubes} hold. 
Moreover,~\cite[Proposition 1 and Proposition 2 in Section VI.1.3]{Stein1970} imply ~\eqref{it:pou_touching_cubes_comparable_sizes} and~\eqref{it:pou_Qn_has_limited_neighbours}.
In addition,~\cite[the proof of Proposition 3 in Section VI.1.3]{Stein1970} implies that there exist $(Q_n^*)_{n\in \N}\in (\mathcal{P}(\R^d))^{\N}$
such that~\eqref{it:pou_def_Qstar} and~\eqref{it:pou_Qstarintersects_iff_Qtouches} hold. Next, observe that~\eqref{it:pou_cubes_cover} and~\eqref{it:pou_Qstarintersects_iff_Qtouches}
imply~\eqref{it:pou_Qstarsintersect_iff_Qtouches}.
In addition,~\cite[Section VI.1.3, p.170]{Stein1970} implies that there exists $(\sgc{}\varphi_n\cgs{})_{n\in \N} \in (C^{\infty}(\R^d,\R))^{\N}$ 
such that
~\eqref{it:pou_phi_range}
--\eqref{it:pou_pou} hold. 
Furthermore,~\cite[inequality
$(13)$ on p.174 in Section VI.2.2.1]{Stein1970} implies~\eqref{it:pou_phi_derivest}. 
Moreover,~\eqref{it:pou_cubes_cover} and~\eqref{it:pou_Qn_has_limited_neighbours} imply that there 
exists $r\colon \R^d\backslash D \rightarrow (0,\infty)$ such that~\eqref{it:pou_def_r} holds.
Next, observe that~\eqref{it:pou_def_Q},~\eqref{it:pou_def_Qstar} and~\eqref{it:pou_def_r} imply 
for 
all $x\in \R^d\backslash D$ that $B_{r(x)}(x) \subseteq \cap_{n\in \N, Q_n \ni x} Q_n^*$. This,~\eqref{it:pou_def_Qstar},~\eqref{it:pou_Qstarsintersect_iff_Qtouches} and~\eqref{it:pou_def_r} imply~\eqref{it:pou_Br_in_Q}.
Moreover,~\eqref{it:pou_Qn_has_limited_neighbours} implies for all $n\in \N$ that 
\begin{align}
\#\left\{i\in \N \colon 
\left( \sgc{}\exists\cgs{} \, m\in \N \colon ( Q_n \cap Q_m \neq \emptyset) \wedge  (Q_m \cap Q_i \neq \emptyset) \right) \right\} \leq 12^{2d}.\end{align}
This,~\eqref{it:pou_cubes_cover} and~\eqref{it:pou_Br_in_Q} 
imply that there exists $I\colon \R^d\backslash D \rightarrow \mathcal{P}(\N)$ such that~\eqref{it:pou_Ixdef} and~\eqref{it:pou_Ixsize} hold.  
In addition, observe that if $A,B\subseteq \R^d$ are non-empty and $A$ is closed, 
then there exists $p\in A$ such that $\textup{dist}(A,B)=\textup{dist}(p,B)$.
This and~\eqref{it:pou_def_Q} imply that there exist $(p_n)_{n\in \N}\in (\R^d)^{\N}$ such that~\eqref{it:pou_pdef} holds.
Furthermore, note that~\eqref{it:pou_def_Q},~\eqref{it:pou_dist_D_scales_size_cubes},~\eqref{it:pou_touching_cubes_comparable_sizes},~\eqref{it:pou_Br_in_Q},~\eqref{it:pou_Ixdef} and~\eqref{it:pou_pdef} imply for all $x\in \R^d\backslash D$, $y\in D$, $i\in I(x)$ that
there exist $n,m\in \N$ such that $x\in Q_n$, $Q_n\cap Q_m\neq 0$, $Q_m\cap Q_i\neq 0$ and 
\begin{equation}
\begin{aligned}
 \| p_i - y \| 
 & \leq 
 \| p_i - x \| + \| x-y\|
 \\ 
 & \leq 
 \textup{dist}(D,Q_i) 
 + 2\sqrt{d}(\rho_i + \rho_m + \rho_n)
 + \| x - y\|
 \\
 & \leq 
 2\sqrt{d}(5\rho_i + \rho_m + \rho_n)
 + \| x - y\|
 \\
 & \leq 2\sqrt{d}(5\times 4^2+4+1)\rho_n + \| x - y \| \leq 85\, \textup{dist}(Q_n,D) + \| x - y \|
 \\ & \leq 86 \| x - y \|.
\end{aligned}
\end{equation}
This implies~\eqref{it:pou_pproperty}. Moreover,~\eqref{it:pou_phi_support},~\eqref{it:pou_Ixdef} and~\eqref{it:pou_Ixsize}
imply that there exists $\bar{f}\colon \R^d \rightarrow E$ such that~\eqref{it:pou_def_barf} and~\eqref{it:pou_barf_finite_sum} hold. Next, observe 
that~\eqref{it:pou_barf_finite_sum} and the fact that $\sgc{}\varphi_n\cgs{}\in C^{\infty}(\R^d,\R)$ 
imply for all $x\in \R^d\backslash D$ that $\bar{f}$ is differentiable in $x$. In addition,~\eqref{it:pou_pou},~\eqref{it:pou_def_barf} and~\eqref{it:pou_barf_finite_sum} imply 
for all $x\in \R^d\backslash D$, $z\in B_{r(x)}(x)\cap(\R^d\backslash D)$, $y\in D$ that 
\begin{equation}\label{eq:pou_barf_grad}
 \nabla \bar{f}(z) = \sum_{i\in I(x)} \nabla \sgc{}\varphi_i\cgs{}(z) ( f(p_i) - f(y)).
\end{equation}
This,~\eqref{it:pou_phi_derivest},~\eqref{it:pou_Ixsize} and~\eqref{it:pou_pproperty} imply  
for all $x\in \R^d\backslash D$, $y\in D$ that 
\begin{equation}\label{eq:pou_barf_gradest}
\begin{aligned}
 \left\| \nabla \bar{f}(x) \right\|_{L(\R^d,E) }
 & \leq 
 C \| f \|_{\calC^{\alpha}(D,E)} \left|\textup{dist}(x,D)\right|^{-1} 
 \sum_{i\in I(x)} \| p_i - y \|^{\alpha}
 \\ 
 &  \leq 
 86 C 12^{2d} \| f \|_{\calC^{\alpha}(D,E)} \left|\textup{dist}(x,D)\right|^{-1}  \| x - y \|^{\alpha}.
\end{aligned}
\end{equation}
This completes the proof of~\eqref{it:pou_barf_diffble}.
Next, observe that~\eqref{it:pou_def_barf} implies for all $x,y\in D$ that 
\begin{equation}\label{eq:pou_barf_hoelder_xyinD}
 \| \bar{f}(x) - \bar{f}(y) \|_E
 =
 \| f(x) - f(y) \|_{E}
 \leq 
 \| f \|_{\calC^{\alpha}(D,E)} \| x - y \|^{\alpha}.
\end{equation}
In addition,~\eqref{it:pou_phi_range},~\eqref{it:pou_pou},~\eqref{it:pou_Ixsize},~\eqref{it:pou_pproperty},~\eqref{it:pou_def_barf}
and~\eqref{it:pou_barf_finite_sum} imply for all $x\in \R^d\backslash D$, $y\in D$ that 
\begin{equation}\label{eq:pou_barf_hoelder_xnotinDyinD}
\begin{aligned}
 \| \bar{f}(x) - \bar{f}(y) \|_E
 & 
 = 
 \Big\|
    \sum_{i\in I(x)}
    \sgc{}\varphi_i\cgs{}(x)(f(p_i) - f(y))
 \Big\|_E 
 \\ 
 & \leq 
 \sum_{i\in I(x)} 
    \| f \|_{\calC^{\alpha}(D,E)}
    \| p_i - y \|^{\alpha}
 \leq 86 \cdot 12^{2d} 
    \| f \|_{\calC^{\alpha}(D,E)}
    \| x - y \|^{\alpha}.
\end{aligned}
\end{equation}
Next,~\eqref{it:pou_barf_diffble} implies for all $x,y\in \R^d\backslash D$
with $\textup{dist}(\{ x + s(y-x)\colon s\in [0,1] \}, D) > \| x - y\|$
that 
\begin{equation}\label{eq:pou_barf_hoelder_xynotinD1}
 \begin{aligned}
  \| \bar{f}(x) - \bar{f}(y) \|_E
  & =
  \left\|
    \int_{0}^{1} 
       \nabla{\bar{f}}(x+ s(y-x)) (x-y)
    \,ds
  \right\|_{E}
  \\ &
  \leq 
  86 C 12^{2d} \| f \|_{\calC^{\alpha}(D,E)} \| x - y \|^{\alpha}.
 \end{aligned}
\end{equation}
Moreover,~\eqref{eq:pou_barf_hoelder_xnotinDyinD} implies 
for all $x,y\in \R^d\backslash D$, $z\in D$
with $\textup{dist}(\{ x + s(y-x)\colon s\in [0,1] \}, D) \leq \| x - y\|$
and $\textup{dist}(\{ x + s(y-x)\colon s\in [0,1] \}, z) \leq \| x- y\|$ 
that 
\begin{equation}\label{eq:pou_barf_hoelder_xynotinD2}
 \begin{aligned}
  \| \bar{f}(x) - \bar{f}(y) \|_E
  & \leq  
  \| \bar{f}(x) - \bar{f}(z) \|_E
  +
  \| \bar{f}(z) - \bar{f}(y) \|_E
  \\ &
  \leq 
  86 \cdot 12^{2d} 
    \| f \|_{\calC^{\alpha}(D,E)}
  \left( 
    \| x - z \|^{\alpha}
    +
    \| z - y \|^{\alpha}
  \right)
  \\ &
  \leq  
  192 \cdot 2^{\alpha} 12^{2d} 
    \| f \|_{\calC^{\alpha}(D,E)} \| x - y \|^{\alpha}.
 \end{aligned}
\end{equation}
Estimates~\eqref{eq:pou_barf_hoelder_xyinD},~\eqref{eq:pou_barf_hoelder_xnotinDyinD},~\eqref{eq:pou_barf_hoelder_xynotinD1} 
and~\eqref{eq:pou_barf_hoelder_xynotinD2} imply~\eqref{it:pou_barf_hoelder}. This completes the proof of Theorem~\ref{thm:partition_of_unity}.
\end{proof}

\begin{theorem}\label{thm:extension_Holdercont_Rd}
Let $d\in \N$.
Then there exists $C\in [1,\infty)$ such that for every $\alpha\in (0,1]$, every non-empty 
$D\subseteq \R^d$, every $\R$-Banach space $(E,\left\|\cdot\right\|_{E})$ 
and every $f \colon D \rightarrow E$ with $\| f \|_{\calC^{\alpha}(D,E)}<\infty$ 
there exists $\bar{f}\colon \R^d \rightarrow E$ such that
\begin{enumerate}
 \item\label{it:barf_f} it holds that $\bar{f}|_{D} = f$,
 \item\label{it:barf_im} it holds for all $x\in \overline{D}$ that 
 $ \bar{f}(x)\in \overline{\textup{im}(f)}^{\left\| \cdot \right\|_{E}}$ and
 \item\label{it:barf_hoelder} it holds that   $ \| \bar{f} \|_{\calC^{\alpha}(\R^d,E)} 
 \leq C
 \| f \|_{\calC^{\alpha}(D,E)}$.
\end{enumerate}
\end{theorem}

\begin{proof}[Proof of Theorem~\ref{thm:extension_Holdercont_Rd}]
Note that Lemma~\ref{lem:extension_Holdercont_closure} and
Theorem~\ref{thm:partition_of_unity}~\eqref{it:pou_def_barf}\sgc{},\,\cgs{}~\eqref{it:pou_barf_hoelder}
imply that there exists $C\in [1,\infty)$ such that for every $\alpha\in (0,1]$, all non-empty
$D\subsetneq \R^d$, every $\R$-Banach space $(E,\left\|\cdot\right\|_{E})$  
and all $f \colon D \rightarrow E$ with $\| f \|_{\calC^{\alpha}(D,E)}<\infty$ 
there exists $\bar{f}\colon \R^d \rightarrow E$ such that~\eqref{it:barf_f}--\eqref{it:barf_hoelder} hold.
Moreover, for every $\alpha\in (0,1]$, every $\R$-Banach space $(E,\left\|\cdot\right\|_{E})$, all $f \colon \R^d \rightarrow E$ with $\| f \|_{\calC^{\alpha}(\R^d,E)}<\infty$ and every $\bar{f}\colon \R^d \rightarrow E$ with $\bar{f}=f$ we have that~\eqref{it:barf_f}--\eqref{it:barf_hoelder} hold.
This completes the proof of Theorem~\ref{thm:extension_Holdercont_Rd}.
\end{proof}

\subsection{Theorems of Kolmogorov-Chentsov type}
Theorem~\ref{t:KolChen_rect} below states the Kolmogorov-Chentsov
continuity result in the form that we need to prove its generalization, Theorem~\ref{t:KolChen} below.

\begin{theorem}\label{t:KolChen_rect}
There exists $\Theta = (\Theta_{d,M,p,\alpha,\beta})_{d,M,p,\alpha,\beta\in \R} \colon \R^5\rightarrow \R$ such that for all $d, M\in \N$, $p\in (d,\infty)$, $\beta \in (\frac{d}{p},1]$, every separable $\R$-Banach space $(E,\left\|\cdot\right\|_{E})$,  
every probability space $(\Omega,\calF,\P)$ and \sgc{}all\cgs{} $X\in C^{\beta}([-M,M]^d,L^p(\Omega;E))$
there exists 
$Y\in \calL^0([-M,M]^d \times \Omega;E)$ such that 
\begin{enumerate} 
 \item \label{it:KolChen_cont} it holds for all $\omega \in \Omega$ that the function 
 $[-M,M]^d\ni x \mapsto Y(x,\omega) \in  E$ is continuous,
 \item \label{it:KolChen_version} it holds for all $x\in [-M,M]^d$ that $[Y(x,\cdot)]_{\P}= X(x)$ and 
 \item \label{it:KolChen_estimate} it holds for all $\alpha \in (0,\beta-\frac{d}{p})$ that
\begin{equation}
\begin{aligned}
    \E \left[
        \| Y \|^{p}_{
        \calC^{\alpha}([-M,M]^d, E) 
        }
    \right]
 & \leq 
 \Theta_{d,M,p,\alpha,\beta}
 \| X \|_{\calC^{\beta}([-M,M]^d,L^p(\Omega;E))}^p.
 \end{aligned}
\end{equation}
\end{enumerate}
\end{theorem}

\begin{proof}[Proof of Theorem~\ref{t:KolChen_rect}]
For all $m\in \N$
let $D_m= \{ k2^{-m} \colon k\in \Z \}$ and let 
$D = \cup_{n\in \N} D_n$. 
For all $m\in \N$ let $\lfloor \cdot \rfloor_{D_m} \colon \R \rightarrow \R$ be the function such that for all $x\in \R$ it holds that $\lfloor x \rfloor_{D_m} = \sup\{ y\in D_m, y\leq x\}$. For all $m,d\in \N$ let 
$\lfloor \cdot \rfloor_{D_m^d} \colon \R^d \rightarrow \R^d$ be the function such that for all $x=(x_1,\ldots,x_d)\in \R^d$ 
it holds that $\lfloor x \rfloor_{D_m^d} = (\lfloor x_1 \rfloor_{D_m},\ldots,\lfloor x_d \rfloor_{D_m} )$.
First, observe that the proof of~\cite[Theorem I.2.1]{RevuzYor1994} implies that there exists 
$\Theta = (\Theta_{d,M,p,\alpha,\beta})_{d,M,p,\alpha,\beta\in \R} \colon \R^5\rightarrow \R$ such that for all $d\in \N$, $M\in \N$, $p\in (d,\infty)$, $\beta \in (\frac{d}{p},1]$, 
$\alpha\in (0,\beta-\frac{d}{p})$, every separable $\R$-Banach space $(E,\left\|\cdot\right\|_{E})$, 
every probability space $(\Omega,\calF,\P)$  
and all $X\in C^{\beta}([-M,M]^d, L^p(\Omega;E))$ 
there exists $\tilde{X}\in \mathcal{L}^{p}(([-M,M]\cap D)^d \times \Omega; E)$
with $\sgc{}\forall\cgs{} \,x\in ([-M,M]\cap D)^d\colon [\tilde{X}(x,\cdot)]_{\P} = X(x)$ 
and 
\begin{equation}\label{eq:X_KolExtest}
\begin{aligned}
  \E \left[
   \| \tilde{X} \|_{\calC^{\alpha}(([-M,M]\cap D)^d,E)}^p
 \right] 
 &\leq 
 \Theta_{d,M,p,\alpha,\beta}
 \| X \|_{\calC^{\beta}([-M,M]^d,L^p(\Omega;E))}^p.
\end{aligned}
\end{equation}
This and Lemma~\ref{lem:extension_Holdercont_closure} 
imply for all $d, M\in \N$, $p\in (d,\infty)$, $\beta \in (\frac{d}{p},1]$, every separable $\R$-Banach space $(E,\left\|\cdot\right\|_{E})$, 
every probability space $(\Omega,\calF,\P)$  
and all $X\in C^{\beta}([-M,M]^d, L^p(\Omega;E))$ 
that
there exist $\tilde{\Omega}\in \calF$, $\tilde{X}\in \calL^p((D\cap [-M,M])^d \times \Omega; E)$ and 
$Y \colon [-M,M]^d\times \Omega \rightarrow E$
such that 
\begin{enumerate}[(a)]
 \item it holds for all $x\in (D\cap [-M,M])^d$ that $[\tilde{X}(x, \cdot)]_{\P}= X(x)$,
 \item it holds that $\P\big[\tilde{\Omega}\big]=1$,
 \item it holds for all $\omega\in \tilde{\Omega}$, 
 $\alpha\in (0,\beta-\frac{d}{p})$ that 
 $\| \tilde{X}(\cdot, \omega) \|_{\calC^{\alpha}((D\cap [-M,M])^d,E)}<\infty$,
\item it holds for all $x\in [-M,M]^d$, $\omega \in \Omega$ that
\begin{equation}
\begin{aligned}
 Y(x,\omega) 
 & =
 \begin{cases}
  \lim_{(D\cap [-M,M])^d \ni y \rightarrow x} \sgc{}\tilde{X}(y,\omega)\cgs{}
  &
  \sgc{}\colon\cgs{} \omega\in \sgc{}\tilde{\Omega}\cgs{}
  \\
  \sgc{}0\cgs{}
  &
  \sgc{}\colon\cgs{} \omega \in \Omega\backslash \sgc{}\tilde{\Omega}\cgs{}
 \end{cases}\sgc{},\cgs{}
 \\ & = 
 \begin{cases}
  \lim_{m\rightarrow \infty} \sgc{}\tilde{X}(\lfloor x \rfloor_{D_m^d},\omega)\cgs{}
  &
  \sgc{}\colon\cgs{}\omega\in \sgc{}\tilde{\Omega}\cgs{}
  \\
  \sgc{}0\cgs{}
  &
  \sgc{}\colon\cgs{}\omega \in \Omega\backslash \sgc{}\tilde{\Omega}\cgs{}
 \end{cases}\sgc{},\cgs{}
\end{aligned}
\end{equation}
\item it holds that $Y \in \calL^0([-M,M]^d \times \Omega;E)$,
\item it holds for all $\omega \in \Omega$ that the function $[-M,M]^d \ni x \mapsto Y(x,\omega) \in E$ 
is continuous,
\item it holds for all $x\in [-M,M]^d$ that
\begin{equation}
\begin{aligned} 
& \E \left[ \| Y(x) - X(x) \|^p_{E}\right] 
\leq 
\liminf_{([-M,M]\cap D)^d \ni y \rightarrow x} 
  \E \left[ \| X(y) - X(x) \|^p_{E} \right]
\\ & \leq 
 \| X \|_{\calC^{\beta}([-M,M]^d,L^p(\Omega;E))}^p \liminf_{([-M,M]\cap D)^d \ni y \rightarrow x} \| x - y \|^{\beta p} = 0,
\end{aligned}
\end{equation} 
\item it holds for all $x\in [-M,M]^d$ that $[Y(x,\cdot)]_{\P} = X(x)$,
\item it holds for all $\omega \in \Omega$, $\alpha\in (0,\beta-\frac{d}{p})$ that
\begin{equation}\label{eq:Kol_pw_est}
\| Y(\cdot,\omega) \|_{\calC^{\alpha}([-M,M]^d,E)} \leq \| \tilde{X}(\cdot,\omega) \|_{\calC^{\alpha}(([-M,M]\cap D)^d,E)}
\end{equation}
and 
\item  it holds for all $\alpha \in (0,\beta - \frac{d}{p})$ that 
\begin{equation}\label{eq:Kolchen_est}
    \E \left[ 
        \| Y \|_{\calC^{\alpha}([-M,M]^d,E)}^p
    \right]
 \leq 
 \Theta_{d,M,p,\alpha,\beta} 
 \| X \|_{\calC^{\beta}([-M,M]^d,L^p(\Omega;E))}^p
 .
\end{equation}
\end{enumerate}
This completes the proof of Theorem~\ref{t:KolChen_rect}. 
\end{proof}

\begin{theorem}\label{t:KolChen}
There exists $\Theta = (\Theta_{d,M,p,\alpha,\beta})_{d,M,p,\alpha,\beta\in \R} \colon \R^5\rightarrow \R$ such that for all $d, M\in \N$, $p\in (d,\infty)$, $\beta \in (\frac{d}{p},1]$, every non-empty 
$D\subseteq [-M,M]^d$, every separable $\R$-Banach space $(E,\left\|\cdot\right\|_{E})$, every non-empty and closed $F\subseteq E$, 
every probability space $(\Omega,\calF,\P)$ and \sgc{}all\cgs{} $X\in C^{\beta}(D, L^p(\Omega;E))$
with $\sgc{}\forall\cgs{} \,x\in D \colon \P\big[X(x)\in F\big]=1$ there exists 
$Y\in \calL^0(\overline{D} \times \Omega;F)$ such that 
\begin{enumerate} 
 \item \label{it:t:KolChen_cont} it holds for all $\omega \in \Omega$ that the function $\overline{D}\ni x\mapsto Y(x,\omega) \in F$ is continuous,
 \item \label{it:t:KolChen_version} it holds for all $x\in D$ that $[Y(x,\cdot)]_{\P}= X(x)$ and
 \item \label{it:t:KolChen_Holder} it holds for all $\alpha \in (0,\beta-\frac{d}{p})$ that
\begin{equation}
\begin{aligned}
    \E \left[ 
        \| Y \|^{p}_{
            \calC^{\alpha}(\overline{D},E)
        }
    \right]
 & \leq 
 \Theta_{d,M,p,\alpha,\beta} 
 \left\| 
    X
 \right\|_{\calC^{\beta}(D,L^p(\Omega;E))}^p.
 \end{aligned}
\end{equation}
\end{enumerate}
\end{theorem}

\begin{proof}[Proof of Theorem~\ref{t:KolChen}]
First, Theorem~\ref{thm:Bochnerspace}, 
Theorem~\ref{thm:extension_Holdercont_Rd}
and Theorem~\ref{t:KolChen_rect} 
imply that there exist $C=(C_d)_{d\in\N} \sgc{}\colon \N \rightarrow \R\cgs{}$ and
$\Theta = (\Theta_{d,M,p,\alpha,\beta})_{d,M,p,\alpha,\beta\in \R} \colon \R^5\rightarrow \R$
such that for every $d, M\in \N$, $p\in (d,\infty)$, 
$\beta \in (\frac{d}{p},1]$, every non-empty 
$D\subseteq [-M,M]^d$, every separable $\R$-Banach space $(E,\left\|\cdot\right\|_{E})$, every non-empty and closed $F\subseteq E$, 
every probability space $(\Omega,\calF,\P)$ and every $X\in C^{\beta}(D, L^p(\Omega;E))$
with~$\sgc{}\forall\cgs{} \,x\in D \colon \P\big[X(x)\in F\big]=1$ there exist
$\bar{X} \in C^{\beta}([-M,M]^d, L^p(\Omega;E))$ and
$\bar{Y}\in \calL^0([-M,M]^d \times \Omega;E)$ such that 
\begin{enumerate}[(a)]
 \item it holds that $\bar{X}|_{D} = X$,
 \item \label{it:t:Kolchen_image} it holds for all $x\in \overline{D}$ 
 that $\bar{X}(x) \in \overline{\operatorname{im}(X)}^{\left\|\cdot\right\|_{L^p(\Omega;E)}}$,
 \item it holds that $ \| \bar{X} \|_{\calC^{\beta}([-M,M]^d,L^p(\Omega;E))} 
 \leq C_d
 \| X \|_{\calC^{\beta}(D, L^p(\Omega;E))}
 $,
 \item \label{it:t:KolChen_cont_version} it holds for all $\omega \in \Omega$ that the function $[-M,M]^d\ni x\mapsto \bar{Y}(x,\omega) \in E$ is continuous,
 \item \label{it:t:KolChen_version_version1} it holds for all $x\in \overline{D}$ that $[\bar{Y}(x,\cdot)]_{\P}= \bar{X}(x)$, 
 \item \label{it:t:KolChen_version_version2} it holds for all $x\in D$  that $[\bar{Y}(x,\cdot)]_{\P}= X(x)$,
 \item \label{it:t:KolChen_Holder_version1} it holds for all $\alpha \in (0,\beta-\frac{d}{p})$ that
\begin{equation}
\begin{aligned}
    \E \left[ 
        \| \bar{Y} \|^{p}_{
            \calC^{\alpha}([-M,M]^d,E)
        }
    \right]
 & \leq 
 \Theta_{d,M,p,\alpha,\beta} 
 \left\| 
    \bar{X}
 \right\|_{\calC^{\beta}([-M,M]^d,L^p(\Omega;E))}^p
 \end{aligned}
\end{equation}
and
 \item \label{it:t:KolChen_Holder_version2} it holds for all $\alpha \in (0,\beta-\frac{d}{p})$ that
\begin{equation}
\begin{aligned}
    \E \left[ 
        \| \bar{Y} \|^{p}_{
            \calC^{\alpha}([-M,M]^d,E)
        }
    \right]
 & \leq 
 C_d \Theta_{d,M,p,\alpha,\beta} 
 \left\| 
    X
 \right\|_{\calC^{\beta}(D,L^p(\Omega;E))}^p.
 \end{aligned}
\end{equation}
\end{enumerate}
Lemma~\ref{lem:simple_closed_subset} and~\eqref{it:t:Kolchen_image} yield for all $x\in \overline{D}$ that $\P[\bar{X}(x) \in F]=1$. This,~\eqref{it:t:KolChen_cont_version},~\eqref{it:t:KolChen_version_version1},~\eqref{it:t:KolChen_version_version2},~\eqref{it:t:KolChen_Holder_version2} and Lemma~\ref{lem:image_in_closed_set} complete
the proof of Theorem~\ref{t:KolChen}.
\end{proof}

From Theorem~\ref{t:KolChen} we obtain Corollary~\ref{cor:cont_via_KolChen} below, which provides conditions under which a random field indexed by an unbounded domain allows for a continuous modification.

\begin{corollary}\label{cor:cont_via_KolChen}
Let $d\in \N$, $p\in (d,\infty)$, $\beta \in (\frac{d}{p},1]$, let
$D\subseteq \R^d$ be non-empty, let $(E,\left\|\cdot\right\|_{E})$ be a separable $\R$-Banach space, let $F\subseteq E$ be non-empty and closed, let $(\Omega,\calF,\P)$ be a probability space and let $X\colon D \times \Omega \rightarrow E$ be a random field \sgc{}satisfying for all $n\in \N$, $z\in D$ that\cgs{} $\sup_{x,y\in D\cap B_{n}(0)}\frac{\E\left[ \| X(x) - X(y) \|_E^p \right]}{\| x-y \|^{\beta p}}<\infty$\sgc{},\cgs{} $\E\big[\| X(\sgc{}z\cgs{})\|_E^p\big]<\infty$ \sgc{}and\cgs{} $\P\big[X(\sgc{}z\cgs{})\in F\big]=1$.
Then there exists 
$Y\in \calL^0(\overline{D} \times \Omega;F)$ such that  
\begin{enumerate} 
 \item \label{it:t:cont_via_KolChen:cont} it holds for all $\omega\in \Omega$ that the function $\overline{D}\ni x\mapsto Y(x,\omega) \in F$ is continuous and
 \item \label{it:t:cont_via_KolChen:YX} it holds for all $x\in D$ that $\P\big[X(x)=Y(x)\big]=1$.
\end{enumerate}
\end{corollary}

\begin{proof}[Proof of Corollary~\ref{cor:cont_via_KolChen}]
Throughout this proof let $f_0\in F$, $n_0\in \N$ satisfy $D \cap B_{n_0}(0)\neq \emptyset$ and let $(d_n)_{n\in \N}\in D^{\N}$ be dense in $D$. For \sgc{}every\cgs{} $n\in \{n_0,n_0+1,\ldots\}$ let $X_n \colon D\cap B_n(0) \rightarrow L^p(\Omega;E)$
satisfy for all $x\in D\cap B_n(0)$ that $X_n(x)=[X(x)]_{\P}$. Note that Lemma~\ref{lem:simple_Hoelder} implies
for all $n\in \{n_0,n_0+1,\ldots\}$ that $X_n \in C^{\beta}(D\cap B_n(0),L^p(\Omega;E))$.
This and Theorem~\ref{t:KolChen} imply for all $n\in \{n_0,n_0+1,\ldots\}$ that there exists
$Y_n\in \calL^0(\overline{D\cap B_{n}(0)}\times \Omega ;F)$ such that 
\begin{enumerate}[(a)] 
 \item \label{it:t:cont_via_KolChen:cont_n} it holds for all $\omega\in \Omega$ that the function 
 $\overline{D\cap B_{n}(0)} \ni x\mapsto Y_n(x,\omega) \in F$ is continuous and
 \item \label{it:t:cont_via_KolChen:YX_n} it holds  for all $x\in D\cap B_{n}(0)$  that $\P\big[X(x)=Y_n(x)\big]=1$.
 \end{enumerate}
This implies for all $n\in \{n_0,n_0+1,\ldots\}$ that
there exists a non-empty $\Omega_n\in \calF$ with $\sgc{}\forall\cgs{} \,\omega \in \Omega_n, k\in \N, x\in D\cap B_{n}(0), m\in \{n,n+1,\ldots\}\colon\P\big[\Omega_n\big]=1$, $Y_m(d_k,\omega) = X(d_k,\omega)$ and $Y_m(x,\omega)=Y_n(x,\omega)$. This
implies that there exist a non-empty $\tilde{\Omega}\in \calF$ and $Y\in \calL^0(\overline{D}\times \Omega;F)$ with $\sgc{}\forall\cgs{} \,n\in \{n_0,n_0+1,\ldots\}, x\in \overline{D}, \omega \in \tilde{\Omega}\colon\P\big[\tilde{\Omega}\big]=1$ and
\begin{equation}\label{eq:cont_via_KolChen:defY}
\begin{aligned}
 Y(x,\omega) 
 &=
 \begin{cases}
  \sgc{}Y_n(x,\omega)\cgs{}
  &
  \sgc{}\colon\cgs{} x\in \overline{D}\cap B_n(0), \omega \in \sgc{}\tilde{\Omega}\cgs{}
  \\
  f_0 
  & \sgc{}\colon\cgs{} x\in \overline{D}, \omega \in \Omega\backslash \sgc{}\tilde{\Omega}\cgs{}
 \end{cases}\sgc{}.\cgs{}
 \end{aligned}
\end{equation}
This and~\eqref{it:t:cont_via_KolChen:cont_n}--\eqref{it:t:cont_via_KolChen:YX_n}
imply~\eqref{it:t:cont_via_KolChen:cont}--\eqref{it:t:cont_via_KolChen:YX}.
This completes the proof of Corollary~\ref{cor:cont_via_KolChen}.
\end{proof}

\section{Theorems on strong completeness of SDEs}

Theorem~\ref{thm:strong_completeness_uniform} and Theorem~\ref{thm:strong_completeness_marginal} below 
imply strong completeness results for SDEs. 
In order to prove these theorems we use Corollary~\ref{cor:cont_via_KolChen}.
Roughly speaking, Theorem~\ref{thm:strong_completeness_uniform} and Theorem~\ref{thm:strong_completeness_marginal} below provide conditions under which \sgc{}there exists a $\tau \in (0,\infty)$ such that the solution\cgs{} to an SDE \sgc{}restricted to the time interval $[0,\tau]$ allows\cgs{} for a modification that is continuous with respect to the initial value. With some additional effort, as, e.g., in the proof of Theorem~2.4 in Zhang~\cite{Zhang2010}, \sgc{}it might be possible to use\cgs{} the existence of weak solutions, pathwise uniqueness of solutions, the Yamada-Watanabe theorem \sgc{}(cf., e.g., Theorem 21.14 in~\cite{Kallenberg2002})\cgs{} and 
local Lipschitz continuity estimates
on a time interval of positive length to establish
strong completeness on the whole time interval $ [0,\infty) $.\par 

Theorem~\ref{thm:strong_completeness_uniform} below implies a strong completeness result for 
SDEs that 
allow for certain strong stability estimates with respect to the supremum-norm on a time interval.

\begin{theorem}[Strong completeness based on uniform strong stability estimates]
\label{thm:strong_completeness_uniform}
Let $ d,m \in \N $, $T\in (0,\infty)$,
let $ D \subseteq \R^d $
be a non-empty set,
let 
$
  \mu \in \calL^0( \overline{D} ; \R^d )
$,
$
  \sigma \in \calL^0( \overline{D} ; \R^{ d \times m } )
$,
let 
$
  ( \Omega, \mathcal{F}, \P, ( \mathcal{F}_t )_{ t \in [0,T] } )  
$
be a filtered probability space satisfying the usual conditions, 
let 
$
  W \colon [0,T] \times \Omega \to \R^m
$
be a standard 
$ (  \mathcal{F}_t  )_{ t \in [0,T] } $-Brownian motion,
\sgc{}for every $ x \in D $ let 
$  X^x  \colon [0,T] \times \Omega \to \overline{D}  $
be an adapted stochastic process with continuous sample 
paths satisfying for all $t\in [0,T]$ that $\P\text{-a.s.}$ it holds that\cgs{} 
$\int_0^{T} \| \mu( X^x_s ) \| + \|\sigma( X^x_s )\|^2 \, ds < \infty$
and
\begin{align}\label{eq:SDEsol_fixed_initial_data1}
  X^x_t = 
  x
  + \int_0^t \mu( X^x_s ) \, ds
  +
  \int_0^t \sigma( X^x_s ) \, dW_s
\end{align}
and \sgc{}let\cgs{}
$ 
  p \in (d,\infty),
$
$
  \alpha \in (\frac{d}{p},1]
$
\sgc{}satisfy\cgs{} for \sgc{}all $z\in D$, $n\in \N$\cgs{} that
$ 
  (X^z_t)_{t\in[0,T]}
  \in 
  \calL^p(\Omega;C([0,T],\R^d))
$ 
and
\begin{equation}\label{eq:loc_lip_cont_SDEsol} 
  \sup_{x,y\in D\cap B_n(0)}
  \frac{ 
    \E\left[ 
     \sup_{t\in [0,T]}
      \| X^x_t - X^y_t \|^p 
    \right]
  }
  {
    \| x - y \|^{\alpha p}
  }
  < \infty.
\end{equation}
Then there exists a measurable function $Y\colon [0,T] \times \overline{D} \times \Omega \rightarrow \overline{D}$ such that 
\begin{enumerate}
 \item\label{it:thm:strong_completeness_uniform:cont} it holds for all $\omega \in \Omega$
 \sgc{}that $Y(\cdot,\cdot,\omega)$ is continuous,\cgs{}
 \item\label{it:thm:strong_completeness_uniform:adapted} it holds for all $x\in \overline{D}$ that $Y(\cdot,x,\cdot)$ is $(\mathcal{F}_t)_{t\in [0,T]}$-adapted and 
 \item\label{it:thm:strong_completeness_uniform:sol} it holds for all $t\in [0,T]$, $ x\in D$ that it holds \sgc{}$\P$-a.s.\ \cgs{}that
 \begin{equation}
 \int_0^T \| \mu(Y(s,x)) \|+\| \sigma(Y(s,x) )\|^2 \,ds < \infty
 \end{equation} 
 and 
 \begin{equation}
  Y(t,x) = x + \int_{0}^{t} \mu(Y(s,x))\,ds + \int_{0}^{t} \sigma(Y(s,x))\,dW_s.
 \end{equation}
\end{enumerate}

\end{theorem}

\begin{proof}[Proof of Theorem~\ref{thm:strong_completeness_uniform}]
Note that Corollary~\ref{cor:cont_via_KolChen} with $d=d$, $p=p$, $\beta=\alpha$, $D=D$, $E=C([0,T],\R^d)$, $F=C([0,T],\overline{D})$, $(\Omega,\calF,\P)=(\Omega,\calF,\P)$, $X=(D\times \Omega \ni (x,\omega) \mapsto X^x(\omega) \in C([0,T],\R^d\sgc{}))\cgs{}$  and~\eqref{eq:loc_lip_cont_SDEsol} imply that there exists a measurable $Y\colon [0,T] \times \overline{D} \times \Omega \rightarrow \overline{D}$ such that~\eqref{it:thm:strong_completeness_uniform:cont} holds and such that for all $(s,x)\in [0,T]\times D$ it holds that $\P[Y(s,x)=X^x_{s}]=1$. The \sgc{}fact that $(\Omega,\calF,\P,(\calF_t)_{t\in [0,T]})$ satisfies the usual conditions and~\eqref{eq:SDEsol_fixed_initial_data1} hence establish\cgs{}~\eqref{it:thm:strong_completeness_uniform:adapted} and~\eqref{it:thm:strong_completeness_uniform:sol}. This completes the proof of Theorem~\ref{thm:strong_completeness_uniform}.
\end{proof}

Theorem~\ref{thm:strong_completeness_marginal} below implies a strong completeness result for 
SDEs that 
allow for certain marginal strong stability estimates.

\begin{theorem}[Strong completeness based on marginal strong stability estimates]
\label{thm:strong_completeness_marginal}
Let $ d,m \in \N $, $T\in (0,\infty)$,
let $ D \subseteq \R^d $
be a non-empty set,
let 
$
  \mu \in \calL^0( \overline{D} ; \R^d )
$,
$
  \sigma \in \calL^0( \overline{D} ; \R^{ d \times m } )
$,
let 
$
  ( \Omega, \mathcal{F}, \P, ( \mathcal{F}_t )_{ t \in [0,T] } )  
$
be a filtered probability space satisfying the usual conditions,
let 
$
  W \colon [0,T] \times \Omega \to \R^m
$
be a standard 
$ (  \mathcal{F}_t  )_{ t \in [0,T] } $-Brownian motion,
\sgc{}for every $x\in D$ let 
$  X^x  \colon [0,T] \times \Omega \to \overline{D}  $
be a 
adapted
stochastic process\cgs{}
with continuous sample 
paths \sgc{}satisfying that for all $t\in [0,T]$ it holds {}$\P\text{-a.s.}$ that\cgs{}
$\int_0^{T} \| \mu( X^x_s ) \| + \|\sigma( X^x_s )\|^2 \, ds < \infty$
and
\begin{align}\label{eq:SDEsol_fixed_initial_data2}
  X^x_t = 
  x
  + \int_0^t \mu( X^x_s ) \, ds
  +
  \int_0^t \sigma( X^x_s ) \, dW_s
\end{align}
and \sgc{}let\cgs{}
$
  p \in (d + 1,\infty),
$ 
$
  \alpha \in (\frac{d+1}{p},1]
$
\sgc{}satisfy for all $r \in [0,T]$, $z\in D$, $n \in \N$\cgs{} that
$ 
  X^{z}_{r}
  \in 
  \calL^p(\Omega;\R^d)
$ 
and
\begin{equation}\label{eq:loc_lip_cont_SDEsol_marginal} 
  \sup_{(s,x),(t,y) \in [0,T]\times D\cap B_n(0)}
  \frac{ 
    \E\left[
      \| X^x_s - X^y_t \|^p 
    \right]
  }
  {
    \left| \| x - y \|^{2} + | s - t |^2 \right|^{\alpha p/2}
  }
  < \infty.
\end{equation}
Then there exists a measurable $Y\colon [0,T] \times \overline{D} \times \Omega \rightarrow \overline{D}$ such that 
\begin{enumerate}
 \item\label{it:thm:strong_completeness_marginal:cont} it holds for all $\omega \in \Omega$
 \sgc{}that $Y(\cdot,\cdot,\omega)$ is continuous,\cgs{}
 \item\label{it:thm:strong_completeness_marginal:adapted} it holds for all $x\in \overline{D}$ that $Y(\cdot,x,\cdot)$ is $(\mathcal{F}_t)_{t\in [0,T]}$-adapted and 
 \item\label{it:thm:strong_completeness_marginal:sol} it holds for all $t\in [0,T]$, $ x\in D$ that it holds \sgc{}$\P$-a.s.\ \cgs{}that 
 \begin{equation}\int_0^T \| \mu(Y(s,x)) \| + \| \sigma(Y(s,x) )\|^2 \,ds < \infty\end{equation} and 
 \begin{equation}
  Y(t,x) = x + \int_{0}^{t} \mu(Y(s,x))\,ds + \int_{0}^{t} \sigma(Y(s,x))\,dW_s.
 \end{equation}
\end{enumerate}

\end{theorem}

\begin{proof}[Proof of Theorem~\ref{thm:strong_completeness_marginal}]
Note that Corollary~\ref{cor:cont_via_KolChen} with $d=d+1$, $p=p$, $\beta=\alpha$, $D=[0,T]\times D$, $E=\R^d$, $F=\overline{D}$, $(\Omega,\calF,\P)=(\Omega,\calF,\P)$, 
$X=([0,T]\times D \ni (t,x)\mapsto X_t^x\in \overline{D})$ and~\eqref{eq:loc_lip_cont_SDEsol_marginal} 
imply that there exists a measurable $Y\colon [0,T] \times \overline{D} \times \Omega \rightarrow \overline{D}$ such that~\eqref{it:thm:strong_completeness_marginal:cont} holds and such that for all $(s,x)\in [0,T]\times D$ it holds that $\P[Y(s,x)=X^x_s]=1$. \sgc{}The\cgs{} fact that $(\Omega,\calF,\P,(\calF_t)_{t\in [0,T]})$ satisfies the usual conditions and~\eqref{eq:SDEsol_fixed_initial_data2} \sgc{}hence 
establishes\cgs{}~\eqref{it:thm:strong_completeness_uniform:adapted} and~\eqref{it:thm:strong_completeness_marginal:sol}. This completes the proof of Theorem~\ref{thm:strong_completeness_marginal}.
\end{proof}

\section{Strong completeness for SDEs
with additive noise}
\label{sec:strong_completeness_additive}

Theorem~\ref{thm:UV} and Theorem~\ref{thm:UV2} together with 
Theorem~\ref{thm:strong_completeness_marginal}
and Theorem~\ref{thm:strong_completeness_uniform} can be used
to prove strong completeness for SDEs.
In the case of additive noise, another well-known 
possibility for proving strong completeness is to 
subtract the driving noise process from the SDE and 
then to try to solve the resulting random ordinary 
differential equation (RODE) globally for every continuous 
trajectory of the driving noise process. This approach 
works, for instance, if the drift coefficient grows 
at most linearly. However, if the drift coefficient grows
super-linearly then it might happen that an SDE is strongly complete even though
the resulting
RODE can not be solved globally 
for every continuous trajectory of the driving
noise process. This is illustrated by Lemma~\ref{lemma:Drehbeispiel_stronglycomplete} and  Lemma~\ref{lemma:Drehbeispiel_cannotsubtractnoise} below. More specifically, Lemma~\ref{lemma:Drehbeispiel_stronglycomplete} below uses Theorem~\ref{thm:strong_completeness_marginal} to establish that the SODE~\eqref{eq:Drehbeispiel} below is strongly complete, whereas Lemma~\ref{lemma:Drehbeispiel_cannotsubtractnoise} shows for every $T\in (0,\infty)$ that the RODE associated to the SODE~\eqref{eq:Drehbeispiel} cannot be solved on $[0,T]$ for every 
continuous trajectory of the driving noise process.

\begin{lemma}\label{lemma:Drehbeispiel_stronglycomplete} 
Let $(\Omega,\calF,\P,(\mathcal{F}_t)_{t\in [0,\infty)})$ 
be a filtered probability space satisfying the usual conditions,
let
$
  W \colon [0,\infty) \times \Omega \to \R^2
$
be a two-dimensional standard Brownian motion and
let 
$ R = $
{\tiny $
  \left(
    \begin{array}{cc}
      0 & 1 \\
      -1 & 0
    \end{array}
  \right)
$}
$ \in \R^{ 2 \times 2 } $. 
Then there exist $T\in (0,\infty)$ and $Y\colon [0,T]\times \R^2\times \Omega \rightarrow \R^2$ such \sgc{}that\cgs{} 
\begin{enumerate} 
\item for all $\omega \in \Omega$ \sgc{}it holds that $Y(\cdot,\cdot,\omega)$\cgs{} is continuous, 
\item for all $x\in \R^2$ it holds that $Y(\cdot,x,\cdot)$ is $(\mathcal{F}_t)_{t\in [0,T]}$-adapted and
\item for all $(t,x)\in [0,T]\times \R^2$ it holds $\P$-a.s.\ that 
 $\int_0^T \| Y(s,x) \|^2 \| R Y(s,x) \| \,ds < \infty$ and 
 \begin{equation}
  Y(t,x) = x + \int_{0}^{t} \| Y(s,x) \|^2 R Y(s,x) \,ds + W_t.
 \end{equation}
\end{enumerate}
 
\end{lemma}

\begin{proof}
First, observe for every $ \rho \in [0,\infty) $
and every $ x \in \R^2 $ that
\begin{equation}
\begin{split}
&
  2 \rho \left< x , \| x\|^2 Rx \right>
  +
  \tfrac{ 1 }{ 2 }
  \operatorname{tr}\!\big(
    2 \rho I
  \big)
  +
  \tfrac{ 1 }{ 2 }
  \left\|
    2 \rho x
  \right\|^2
=
  2 \rho
  +
  2 \rho^2
  \left\|
    x
  \right\|^2
=
  2 \rho
  +
  2 \rho
  \left[
    \rho \| x \|^2
  \right]
  .
\end{split}
\end{equation}
This and, e.g.,~\cite[Lemma 2.2]{gk96b} imply that there exist $(\mathcal{F}_t)_{t\in [0,\infty)}$-adapted stochastic processes $ X^x \colon [0,\infty) \times \Omega \to \R^2 $,
$ x \in \R^2 $, with continuous sample paths such that
\begin{equation}
\label{eq:Drehbeispiel}
  X^x_t
  =
  x
  +
  \int_0^t
  \| X^x_s \|^2 R X^x_s \, ds
  +
  W_t
\end{equation}
$ \P $-a.s.\ for every
$ (t,x) \in [0,\infty) \times \R^2 $.
\sgc{}Corollary~\ref{cor:exp_mom}\cgs{} 
with $U(x) = 1+\rho \| x \|^2$
and \sgc{}$\overline{U}(x)\equiv 0$\cgs{}
hence implies for 
every $ \rho, t \in [0,\infty) $
and every $ x \in \R^2 $ that
\begin{equation}
  \E\!\left[
    \exp\!\left(
      \frac{ \rho \| X^x_t \|^2 }{ 
        e^{ 2 \rho t }
      }
    \right)
  \right]
\leq
  \exp\!\left(
    1
    - 
      e^{-2\rho t}
    +
      \rho \, \| x \|^2 
  \right)
\leq
  \exp\!\left(
    1
    +
    \rho \, \| x \|^2 
  \right)
  .
\end{equation}
This \sgc{}demonstrates\cgs{} for all $(t,x)\in [0,\infty)\times \R^2$, $\rho\in (0,\infty)$, 
$r\in [1,\infty)$
that 
\begin{equation} 
 \E\! \left[ \| X_t^x \|^r \right]
 \leq \rho^{-r/2} 
 (\lceil \tfrac{r}{2}\rceil)!
 \exp(1+r \rho t +\rho \| x \|^2).   
\end{equation}
This and~\eqref{eq:Drehbeispiel} \sgc{}ensure\cgs{} for every $(T,x)\in (0,\infty)\times \R^2$, $\rho\in (0,\infty)$, $r\in [1,\infty)$ that 
\begin{equation}\label{eq:Drehbeispiel:Holderintime}
\begin{aligned}
& \sup_{s,u\in [0,T]} 
 \left| 
  \frac{ \| X_s^x - X_u^x \|_{L^r(\Omega;\R^{2})} }{|s-u|^{1/2}}
 \right|
\\ & \leq 
 \rho^{-3/2} 
 \left|(\lceil \tfrac{3r}{2}\rceil)!\right|^{1/r}
 \exp\big(\tfrac{1}{r}(1+\rho \| x \|^2)+ 3\rho T \big)
 +
 \| W_1 \|_{L^r(\Omega;\R^{2})}.
\end{aligned}
\end{equation}
In addition, note 
for every 
$ x, y\in \R^2 $ 
with $ x \neq y $ that
\begin{equation}
\begin{split}
&	
  \frac{
    \left< x - y, \| x \|^2 R x - \| y \|^2 R y \right>
  }{
    \| x - y \|^2
  }
=
  \frac{
    \left(
      \| x \|^2 
      - 
      \| y \|^2 
    \right)
    \left< x - y, R (x + y) \right>
  }{
    2 \, \| x - y \|^2
  }
\\ & \leq 
  \frac{
    \left(
      \| x \|
      +
      \| y \|
    \right)
    \left|\left< x - y, R (x + y) \right>\right|
  }{
    2 \, \| x - y \|
  }
\leq
  \frac{
    \left(
      \| x \|
      +
      \| y \|
    \right)
    \left\| R ( x + y ) \right\|
  }{ 2 }
\leq
  \| x \|^2 + \| y \|^2.
\end{split}
\end{equation}
\sgc{}Corollary~\ref{cor:UV_squared_norm}\cgs{}
with $d=m=2$, $O=\R^2$, $\mu(x)=\| x \|^2 Rx$, $\sigma = \textup{Id}_{\R^2}$, $\alpha_0=2\rho$, $\alpha_1=\beta_0=\beta_1=c=0$, $q_1=p=\infty$, $q_0=r$, $U_0= 1+\rho\| \cdot \|^{2}$ and
$U_1=\overline{U}\equiv \sgc{}0\cgs{}$
hence implies
for every $ x, y \in \R^2 $,
$ \rho, T \in ( 0, \infty ) $
and every
$ 
  r \in 
  \big( 
    0, 
    \frac{
      \rho \, e^{ - 2 \rho T }
    }{ 
     2T
    } 
  \big)
$ 
that
\begin{equation}
\sup_{t\in[0,T]}
  \bigg|
    \frac{
      \| X^x_t - X^y_t \|_{ L^r( \Omega; \R^{2} ) }
    }{
      \| x - y \|
    }
  \bigg|
\leq
  \exp\!\left(
    \frac{ 
      2 
      +
      \rho
      \, \| x \|^2
      +
      \rho
      \, \| y \|^2
    }{ 2 r }
  \right)
  .
\end{equation}
This and~\eqref{eq:Drehbeispiel:Holderintime} imply
for every $ \rho, T, n\in ( 0, \infty ) $
and every
$ 
  r \in 
  \big( 
    0, 
    \frac{
      \rho \, e^{ - 2 \rho T }
    }{ 
     2T
    } 
  \big)
$ 
that
\begin{equation}  
 \sup_{(s,x),(u,y)\in [0,T]\times B_n(0)} 
 \Bigg|
  \frac{ \| X_s^x - X_u^y \|_{L^r(\Omega;\R^{2})} }
  {\big||s-u|^{2}+\| x - y \|^2\big|^{1/4}}
 \Bigg|
 < \infty.
\end{equation}
This,
the fact that for every $ \rho \in ( 0, \infty ) $
it holds that
$
  \lim_{ T \searrow 0 }
  \frac{
    \rho \, e^{ - 2 \rho T }
  }{
     2 T 
  }
  = \infty
$ and Theorem~\ref{thm:strong_completeness_marginal}
show that there exist $T\in (0,\infty)$ and $Y\colon [0,T]\times \R^2\times \Omega \rightarrow \R^2$ such \sgc{}that\cgs{} 
\begin{enumerate}
 \item 
for all $\omega \in \Omega$ \sgc{}it holds that $Y(\cdot,\cdot,\omega)$\cgs{} is continuous, 
\item for all $x\in \R^2$ it holds that $Y(\cdot,x,\cdot)$ is $(\mathcal{F}_t)_{t\in [0,T]}$-adapted and 
\item for all $(t,x)\in [0,T]\times \R^2$ it holds $\P$-a.s.\ that 
 $\int_0^T \| Y(s,x) \|^2 \| R Y(s,x) \| \,ds < \infty$ and 
 \begin{equation}
  Y(t,x) = x + \int_{0}^{t} \| Y(s,x) \|^2 R Y(s,x) \,ds + W_t.
 \end{equation}
\end{enumerate}
This completes the proof of Lemma~\ref{lemma:Drehbeispiel_stronglycomplete}.
\end{proof}

\begin{remark}\label{rem:Drehbeispiel_stronglycomplete} 
Let $(\Omega,\calF,\P,(\mathcal{F}_t)_{t\in [0,\infty)})$ 
be a filtered probability space satisfying the usual conditions,
let
$
  W \colon [0,\infty) \times \Omega \to \R^2
$
be a two-dimensional standard Brownian motion and
let 
$ R = $
{\tiny $
  \left(
    \begin{array}{cc}
      0 & 1 \\
      -1 & 0
    \end{array}
  \right)
$}
$ \in \R^{ 2 \times 2 } $. 
Then by adapting the arguments in the proof of Theorem~2.4 in Zhang~\cite{Zhang2010} it is possible to establish the existence of a $Y\colon [0,\infty)\times \R^2\times \Omega \rightarrow \R^2$ such \sgc{}that\cgs{}
\begin{enumerate} 
\item for all $\omega \in \Omega$ \sgc{}it holds that $ Y(\cdot,\cdot,\omega)$\cgs{} is continuous, 
\item for all $x\in \R^2$ it holds that $Y(\cdot,x,\cdot)$ is $(\mathcal{F}_t)_{t\in [0,\infty)}$-adapted and 
\item for all $(t,x)\in [0,\infty)\times \R^2$ it holds $\P$-a.s.\ that  $\int_0^t \| Y(s,x) \|^2 \| R Y(s,x) \| \,ds < \infty$ and 
 \begin{equation}
  Y(t,x) = x + \int_{0}^{t} \| Y(s,x) \|^2 R Y(s,x) \,ds + W_t.
 \end{equation}
\end{enumerate}
\end{remark}

\begin{lemma}\label{lemma:Drehbeispiel_cannotsubtractnoise} 
Let $T\in (0,\infty)$, let $x \in \R^2$ satisfy $\| x \| > 2^8 T^{-8}$, let 
$
  \Omega
  =
  \big\{
    f \in
    C( [0,T], \R^2 )
    \colon
    f(0) = 0
  \big\}
$,
let
$
  \mathcal{F}
  =
  \mathcal{B}( 
    \Omega
  )
$,
let 
$ 
  \mathbb{P} \colon \mathcal{F} \to [0,1] 
$
be the Wiener measure on 
$ ( \Omega, \mathcal{F} ) $,
let
$
  W \colon [0,T] \times \Omega \to \R^2
$
be given by
$ W_t( \omega ) = \omega(t) $
for every  
$ (t,\omega) \in [0,T]\times \Omega $,
let $(\mathcal{F}_t)_{t\in [0,T]}$
be given by $\mathcal{F}_t = \sigma(\{W_s \colon s\in [0,t]\})$
for every $t\in [0,T]$,
let 
$ R = $
{\tiny $
  \left(
    \begin{array}{cc}
      0 & 1 \\
      -1 & 0
    \end{array}
  \right)
$}
$ \in \R^{ 2 \times 2 } $ and  let $y\in C([0,T],\R^2)$.
Then there 
exists $(t,\omega) \in  [0,T]\times \Omega$ 
such that 
\begin{equation}\label{eq:Drehbeispiel_RODE}
    y(t) 
    \neq 
    x + \int_{0}^t \| y(s) + W_s(\omega) \|^2 R (y(s) + W_s(\omega)) \,ds.
\end{equation}
\end{lemma}

\begin{proof}
Let $ \tau \in ( 0 , \infty ] $
be the unique maximal extended real number
such that there exists a unique 
continuously differentiable function
$ z \colon [0,\tau) \to \R^2 $
\sgc{}which satisfies for all $t \in [0,\tau)$ that\cgs{}
\begin{equation}
\label{eq:ODE_example_blowup}
  z(0) = x
\qquad
  \text{and}
\qquad
  z'(t)
  =
  \bigg\n 
    z(t) 
    - 
    \tfrac
      { t  R z(t)}
      {\n z(t) \n^{ 3/2 } }
  \bigg\n^2
  R
  \bigg(
    z(t)
    -
    \tfrac
      { t  R z(t)}
      {\n z(t) \n^{ 3/2 } }
  \bigg)
  .
\end{equation}
Note that this ensures 
for every $ t \in [0,\tau) $
that
\begin{equation}
\begin{split}
&
  \tfrac{ \partial }{ \partial t }
  \| z(t) \|^2
=
  2 
  \left< 
    z(t), 
    z'(t)
  \right>
=
  2
  \left\n 
    z(t) 
    - 
    \tfrac
      { t  R z(t)}
      {\n z(t) \n^{ 3/2 } }
  \right\n^2
  \left< 
    z(t), 
    R
    \left(
      z(t)
      -
      \tfrac
        { t  R z(t)}
        {\n z(t) \n^{ 3/2 } }
    \right)
  \right>
\\ & =
  -
  2
      \tfrac
        { t  }
        {\n z(t) \n^{ 3/2 } }
  \left\n 
    z(t) 
    - 
    \tfrac
      { t  R z(t)}
      {\n z(t) \n^{ 3/2} }
  \right\n^2
  \left< 
    z(t), R^2 z(t)
  \right>
\\ & =
  2
  t 
  \n z(t) \n^{ 1/2 }
  \left\n 
    z(t) 
    - 
    \tfrac
      { t  R z(t)}
      {\n z(t) \n^{ 3/2 } }
  \right\n^2
\\ & =
  2
  t  
  \n z(t) \n^{ 1/2 }
  \left[
    \n z(t) \n^2
    +
    \tfrac{t^2}{\| z(t) \|}
  \right]
\geq
  2
  t  
  \n z(t) \n^{ 5 / 2 }
=
  2
  t  
  \left[
    \n z(t) \n^2 
  \right]^{ 5 / 4 }
  .
\end{split}
\end{equation}
This implies that $\tau \leq 2 \|x\|^{-1/8} < T$. 
Next let $ \omega \in \Omega $ be given by
\begin{equation}
  \omega( t )
  =
  \begin{cases}
    -
    \frac{
      t R z(t)
    }{
      \| z(t) \|^{ \frac{3}{2} }
    }
    & \colon t < \tau
  \\
    0
    & \colon t \geq \tau
  \end{cases}
\end{equation}
for every $ t \in [0,T] $
and note that~\eqref{eq:ODE_example_blowup}
implies that $z(0) = x$, $\lim_{t\nearrow \tau} \| z(t) \| = \infty$ and 
\begin{equation}
  \forall \, t \in [0,\tau) \colon
\quad
  z'(t)
  =
  \left\n 
    z(t) 
    +
    W_t( \omega )
  \right\n^2
  R
  \left(
    z(t)
    +
    W_t( \omega )
  \right)
  .
\end{equation}
Let $T_{\textup{max}}\in [0,T]$ satisfy
\begin{equation}\begin{aligned}
 &  T_{\textup{max}} = 
 \\
 &  \inf\left( 
    \left\{ 
        t\in [0,T]
        \colon 
        y(t) 
        \neq 
        x + \int_{0}^t \| y(s) + W_s(\omega) \|^2 R (y(s) + W_s(\omega)) \,ds
    \right\} 
    \cup 
    \{ T \}
   \right).\end{aligned}
\end{equation}
\sgc{}Note\cgs{} that for all $t \in [ 0, T_{\textup{max}})$
it holds that $y(t)=z(t)$. \sgc{}The\cgs{} fact that $\lim_{t \nearrow \tau} \| z(t) \| = \infty$
\sgc{}hence establishes\cgs{} that $T_{\textup{max}}\leq \tau < T$. 
This completes the proof of Lemma~\ref{lemma:Drehbeispiel_cannotsubtractnoise}.
\end{proof}

\chapter{Examples of SODEs}
\label{sec:examples_SODE}

In this section we apply
Theorem~\ref{thm:UV}
and Theorem~\ref{thm:UV2}
to several example SDEs from the literature.

\section{Setting}\label{sec:setting_SDE}
Throughout Section~\ref{sec:examples_SODE} we shall frequently use the 
following 
setting.
Let $ m, d \in \N $, let $D \subseteq 
\R^d$ be a closed set,
let 
$ ( \Omega, \mathcal{F}, \P, ( \mathcal{F}_t )_{ t \in [0,\infty) } ) $
be a filtered probability space satisfying the usual conditions, 
let
$ W \colon [0,\infty) \times \Omega \to \R^m $
be a standard 
$ ( \mathcal{F}_t )_{ t \in [0,\infty) } $-Brownian \sgc{}motion\cgs{}
and let
$
  \mu \colon D \rightarrow \R^d
$
and 
$
  \sigma \colon D \rightarrow \R^{ d \times m }
$
be locally Lipschitz continuous functions.

Finally, let
$ X^x \colon [0,\infty) \times \Omega \to D $,
$ x \in D $,
be adapted stochastic processes with 
continuous sample paths satisfying
\begin{equation}\label{eq:SDE_examples}
  X^x_t = x + \int_0^t \mu( X_s ) \, ds + \int_0^t 
\sigma( X^x_s ) \, dW_s
\end{equation}
$ \P $-a.s.\ for all $ (t,x) \in [0,\infty) \times D $.

\section{Stochastic van der Pol oscillator}
\label{ssec:stochastic.van.der.Pol.oscillator}

The van der Pol oscillator was proposed to describe stable
oscillations; see van der Pol~\cite{VanDerPol1926}.
Timmer et~al.~\cite{TimmerEtAl2000} considered a stochastic version
with additive noise acting on the velocity. Here we consider a more general
version thereof.

Assume the setting of Section~\ref{sec:setting_SDE}\sgc{}, assume $d=2$, assume $D=\R^2$,\cgs{} let 
$
  \alpha \in ( 0, \infty )
$,
$ 
  \gamma,\, \delta,\, \eta_0,\, \eta_1
  \in [0,\infty)
$,
let $ g \colon \R \to \R^{ 1 \times m } $
be a globally Lipschitz continuous function
with
$
  \| g(y) \|^2
  \leq
  \eta_0 +
  \eta_1 y^2
$
for all $ y \in \R $,
let
$
  \mu \colon \R^2 \rightarrow \R^2
$
and 
$
  \sigma \colon \R^2 \rightarrow \R^{ 2 \times m }
$
be given by
$
  \mu( x )
=
  \left(
    x_2 ,
    \left( \gamma - \alpha ( x_1 )^2 \right)
    x_2
    - \delta x_1
  \right) 
$
and
$
  \sigma( x ) u
=
  \left(
    0 ,
    g( x_1 ) u
  \right)
$
for all
$
  x = (x_1, x_2)
  \in \mathbb{R}^2
$,
$ u \in \R^m $.
Let
$ X^x = (X^{x,1},X^{x,2}) \colon [0,\infty) \times \Omega \to \R^d $,
$ x = (x_1,x_2) \in \R^2 $,
be adapted stochastic processes with 
continuous sample paths satisfying
\eqref{eq:SDE_examples}
$ \P $-a.s.\ for all $ (t,x) \in [0,\infty) \times \R^d $, i.e.,
$ X^{x,1} $ is the solution process to the SDE known as
the \emph{stochastic van der Pol oscillator}:
\begin{equation}\label{eq:van_der_Pol}
\begin{split}
&  \ddot{X}^{x,1}_t 
  =
    \big(
        \gamma
        - 
        \alpha \big(X^{x,1}_t \big)^2 
    \big) 
    \dot{X}^{x,1}_t
    - 
    \delta X_t^{x,1}
    +
    g(X_t^{x,1})\dot{W}_t,
    \quad t\in [0,\infty),
\\ & 
    X^{x,1}_0=x_1,\, \dot{X}^{x,1}_0=x_2.
\end{split}
\end{equation}

Next we define a function
$ \vartheta \colon (0,\infty) \to [0,\infty) $
by
$
  \vartheta( \rho )
:=
  \min_{ r \in ( 0, \infty ) }
  \big(
    \big[
      \tfrac{
        | \delta - 1 |
      }{
        r
      }
      +
      \eta_1 
    \big] 
    \vee
    \big[
      r \, 
      | \delta - 1 |
      +
      2 \gamma
      +
      4 \eta_0 \rho
    \big]
  \big)
$
for all $ \rho \in ( 0, \infty ) $.
If $ \rho \in (0,\infty) $
and if $ U, \overline{U} \colon \R^2 \to \R $ are given by
$
  U( x ) = \rho \, \| x \|^2
$
and
$
  \overline{U}( x )
  =
  2 \rho
  \left[
    \alpha
    -
    \rho \eta_1
  \right]
  ( x_1 x_2 )^2
$
for all $ x = ( x_1, x_2 ) \in \R^2 $,
then it holds for every
$ x = ( x_1, x_2 ) \in \R^2 $ 
that
\begin{align}
& \nonumber
  ( \mathcal{G}_{ \mu, \sigma } U)( x )
  +
  \tfrac{ 1 }{ 2 }
  \|
    \sigma( x )^* ( \nabla U )(x)
  \|^2
  +
  \overline{U}( x )
\\ \nonumber & =
  2 \rho 
  \left[ 
    \left( 1 - \delta \right) x_1 x_2
    +
    \gamma ( x_2 )^2
    - \alpha ( x_1 x_2 )^2
    +
    \tfrac{ 1 }{ 2 }
    \| g( x_1 )^* \|^2
  \right]
\\ \nonumber & \quad
  +
  2 ( \rho x_2 )^2
  \| g( x_1 )^* \|^2
  +
  \overline{U}( x )
\\ & 
\leq 
\label{eq:Van_der_Pol_nonlinearity_Uest}
  \rho \eta_0
  +
  2 \rho
  \left[
    \left( 1 - \delta \right) x_1 x_2
    +
    \tfrac{ \eta_1 }{ 2 } ( x_1 )^2
    +
    \left[
      \gamma
      +
      2 \eta_0 \rho
    \right]
    ( x_2 )^2
  \right]
\\ \nonumber & \quad 
  +
  2 \rho
  \left[
    \rho \eta_1
    - \alpha
  \right]
  ( x_1 x_2 )^2
  +
  \overline{U}( x )
\\ \nonumber & \leq
  \rho \eta_0
  +
  2 \rho
  \inf_{ r \in ( 0, \infty ) }
  \left[
    \big[
      \tfrac{
        | \delta - 1 |
      }{
        2 r
      }
      +
      \tfrac{ \eta_1 }{ 2 } 
    \big] 
    ( x_1 )^2
    +
    \big[
      \tfrac{
        r \, | \delta - 1 |
      }{
        2 
      }
      +
      \gamma
      +
      2 \eta_0 \rho
    \big]
    ( x_2 )^2
  \right]
\\ \nonumber & 
\leq
  \rho \eta_0
  +
  \vartheta( \rho ) \,
  U(x)
  .
\end{align}
Corollary~\ref{cor:exp_mom}
hence proves 
for every $ x \in \R^2 $,
$ t \in [0,\infty) $, 
$
  \rho \in (0, \frac{ \alpha }{ \eta_1 } ] \cap \R
$
that
\begin{equation}
\begin{split}
&  \E\!\left[
    \exp\!\left(
      \tfrac{ 
        \rho 
      }{
        e^{ \vartheta( \rho ) t }
      }
      \| X^x_t \|^2
      +
      \int_0^t
      \tfrac{
        2 \rho
        (
          \alpha - \rho \eta_1
        )
      }{
        e^{ \vartheta( \rho ) s }
      }
        \left|
          X^{ 1, x }_s
          X^{ 2, x }_s
        \right|^2
      ds
    \right)
  \right]
\\ 
& \leq
 \exp\!\left(
    \int_0^t
      \frac{
        \rho \eta_0
      }{ e^{ \vartheta( \rho ) s } 
      } 
    \, ds
    +
    \rho \| x \|^2
 \right)
\leq
  e^{
    \frac{ 1 }{ 4 }
    +
    \rho \| x \|^2
  }
  .
\end{split}
\end{equation}
In the next step we
observe
for every 
$ 
  x = (x_1, x_2), y = ( y_1, y_2 ) \in \R^2 
$ 
with 
$
  \sgc{}x_2\cgs{} y_2 < 0
$
that 
$
  \sgc{}( x_2 - y_2 ) \cgs{}
  ( ( x_1 )^2 x_2 - ( y_1 )^2 y_2 )
  \geq 0
$.
Consequently,
we get
for every $ x = (x_1, x_2), y = ( y_1, y_2 ) \in \R^2 $ 
with $x\neq y$
that
\begin{equation}
\label{eq:Van_der_Pol_nonlinearity}
\begin{split}
&
  -
  \frac{ 
    \alpha
    \left( x_2 - y_2 \right)
    \left[  
      ( x_1 )^2 x_2
        -
      ( y_1 )^2 y_2
    \right]
  }{
    \|
      x - y
    \|^2
  }
\\ & 
\leq
  -
  \1_{[0,\infty)}( x_2 y_2 )
  \cdot
  \frac{ 
    \alpha
    \left( x_2 - y_2 \right)
    \left[  
      ( x_1 )^2 x_2
        -
      ( y_1 )^2 y_2
    \right]
  }{
    \|
      x - y
    \|^2
  }
\\ & =
  -
  \1_{[0,\infty)}(x_2y_2)
  \cdot
  \frac{ 
    \alpha
    \left( |x_2| - |y_2| \right)
    \left[  
      ( x_1 )^2 |x_2|
        -
      ( y_1 )^2 |y_2|
    \right]
  }{
    \|
      x - y
    \|^2
  }
\\ & \leq
  -
  \1_{ [0,\infty) }( x_2 y_2 )
  \cdot
  \frac{ 
    \alpha
    \left( |x_2| - |y_2| \right)
    \left(  
      ( x_1 )^2
        -
      ( y_1 )^2 
    \right) \min(|x_2|,|y_2|)
  }{
    \|
      x - y
    \|^2
  }
\\ & \leq
  \frac{ 
    \alpha
    \left| x_2 - y_2 \right|
    \left|  
      ( x_1 )^2
        -
      ( y_1 )^2 
    \right| \min(|x_2|,|y_2|)
  }{
    \|
      x - y
    \|^2
  }
\\ & \quad
  \cdot
  \1_{ ( - \infty, 0 ] }\!\left(
    ( | x_2 | - | y_2 | )
    ( ( x_1 )^2 - ( y_1 )^2 )
  \right)
\\ & \leq
    \tfrac{\alpha}{2}
    \left( |x_1| + |y_1| \right)
    \min(|x_2|,|y_2|)
    \cdot
  \1_{(-\infty,0]}\!\left((|x_2|-|y_2|)(|x_1|-|y_1|)\right)
\\ & \leq
    \tfrac{\alpha}{2}
    \left( |x_1| + |y_1| \right)
    \min(|x_2|,|y_2|)
\leq
    \tfrac{ \alpha }{ 2 }
    \left[
      |x_1 x_2| + |y_1 y_2|
    \right] .
\end{split}
\end{equation}
This implies 
for every $ t, q, \theta \in ( 0, \infty ) $,
$ \rho \in (0, \frac{ \alpha }{ \eta_1 } ) $,
$ x = (x_1, x_2), y = ( y_1, y_2 ) \in \R^2 $
with $ x \neq y $ that
\begin{equation}
\begin{split}
&
  \tfrac{ 
    \langle
      x - y ,
      \mu( x ) - 
      \mu( y )
    \rangle
    +
    \frac{ 1 }{ 2 }
    \| 
      \sigma( x ) - \sigma( y ) 
    \|^2_{ \HS( \R^m, \R^2 ) }
  }{
    \|
      x - y
    \|^2
  }
    +
    \tfrac{
      ( \frac{ \theta }{ 2 } - 1 ) \,
      \| 
        (
          \sigma( x ) - \sigma( y )
        )^*
        ( x - y )
      \|^2
    }{
      \| x - y \|^4
    }
\\ & =
  \tfrac{ 
    \langle
      x - y ,
      \mu( x ) - 
      \mu( y )
    \rangle
    +
    \frac{ 1 }{ 2 }
    \| 
      g( x_1 )^* - g( y_1 )^* 
    \|^2
  }{
    \|
      x - y
    \|^2
  }
    +
    \tfrac{
      ( \frac{ \theta }{ 2 } - 1 ) \, ( x_2 - y_2 )^2 \,
      \| 
        g( x_1 )^* - g( y_1 )^*
      \|^2
    }{
      \| x - y \|^4
    }
\\ & \leq
  \tfrac{
    \gamma 
    +
    \sqrt{
      \gamma^2
      +
      ( \delta - 1 )^2
    }
  }{
    2
  }
  +
  \tfrac{ \alpha }{ 2 }
  \left(
    | x_1 |
    | x_2 | 
    +
    | y_1 |
    | y_2 |
  \right)
  +
  \tfrac{ 
    \frac{ 1 }{ 2 }
    | x_1 - y_1 |^2
    \| g^* \|_{ \operatorname{Lip}( \R, \R^m ) }^2
  }{
    \|
      x - y
    \|^2
  }
\\ & \quad
    +
    \tfrac{
      \max( \frac{ \theta }{ 2 } - 1 , 0 )\,
      ( x_1 - y_1 )^2 \,
      ( x_2 - y_2 )^2 \,
      \| g^* \|_{
        \operatorname{Lip}( \R, \R^m )
      }^2
    }{
      \| x - y \|^4
    }
\\ & \leq
  \tfrac{
    \gamma 
    +
    \sqrt{
      \gamma^2
      +
      ( \delta - 1 )^2
    }
  }{
    2
  }
  +
  \left[
    \tfrac{ 1 }{ 2 }
    +
    \tfrac{ 1 }{ 4 }
    \max( \tfrac{ \theta }{ 2 } - 1 , 0 )
  \right]
  \| g^* \|_{ \operatorname{Lip}( \R, \R^m ) }^2
  +
  \tfrac{ 
    q 
    \alpha^2 
    e^{ \vartheta( \rho ) t }
  }{ 
    8 \rho [ \alpha - \rho \eta_1 ]
  }
\\ & \quad
  +
  \tfrac{
    2 \rho 
    \left[ 
      \alpha - \rho \eta_1
    \right]
    \left[
      ( x_1 x_2 )^2
      +
      ( y_1 y_2 )^2
    \right]
  }{
    2 q e^{ \vartheta( \rho ) t }
  }
  .
\end{split}
\end{equation}
Combining this and~\eqref{eq:Van_der_Pol_nonlinearity_Uest}
with Corollary~\ref{cor:UV2} \sgc{}with\cgs{}
$ U_{0,0}=U_{1,0}=U_{0,1}=\overline{U}_{0,1} \equiv 0 $,
$ U_{1,1}(x) = \rho \| x \|^2$ for all $x\in \R^2 $,
$ \overline{U}_{1,1}(x_1, x_2) = 2\rho (\alpha - \rho \eta_1)(x_1 x_2)^2 $
for all $ (x_1,x_2) \in \R^2 $,
$
\alpha_{1,1} = \vartheta(\rho)
$,
$
\beta_{1,1} = \rho \eta_0
$,
$ \rho = q_{0,0} = q_{0,1} = q_{1,0} = \infty $,
$ q_{1,1} = q $, 
$ 
  c_0(t) 
  = 
  \inv{8}( p - \theta) 
    \| g^* \|_{ \operatorname{Lip}( \R, \R^m ) }^2
$,
$
  c_1(t)
  =
   \tfrac{
      \gamma 
      +
      \sqrt{
        \gamma^2
        +
        ( \delta - 1 )^2
      }
  }{
    2
  }
  +
  \tfrac{ 
    \left[
      \theta
      + 
      2 \vee ( 4- \theta )
    \right]
    \,
    \| g^* \|_{ \operatorname{Lip}( \R, \R^m ) }^2
  }{ 8 }
  +
  \tfrac{ 
      q 
      \alpha^2 
      e^{ \vartheta( \rho ) t }
  }{ 
    8 \rho [ \alpha - \rho \eta_1 ]
  }
$
for all \sgc{}$t\in [0,T]$\cgs{}
proves 
for every $ T \in ( 0, \infty ) $,
$ x = ( x_1 , x_2 ) $, 
$ y = ( y_1 , y_2 ) \in \R^2 $,
$ \rho \in ( 0, \frac{ \alpha }{ \eta_1 } ) $,
$ r, p, q \in (0,\infty] $,
$ \theta \in ( 0, p ) $
with
$
  \frac{ 1 }{ p } + \frac{ 1 }{ q } = \sgc{}\frac{ 1 }{ r }\cgs{}
$
that
\begin{equation}
\begin{split}
\label{eq:Van_der_Pol_estimate}
&
  \left\|
    \sup\nolimits_{ t \in [ 0, T ] }
      \| X^x_t - X^y_t \|
  \right\|_{
    L^r( \Omega; \R )
  }
\leq  
  \frac{ 
    \| x - y \|
  }{
    \left[ 
      1 - \theta / p 
    \right]^{
      \frac{ 1 }{ \theta }
    }
  }
  \exp\!\left(
    \tfrac{
      \left[
	\gamma 
	+
	\sqrt{
	  \gamma^2
	  +
	  ( \delta - 1 )^2
	}
      \right] T
    }{
      2
    }
  \right)
\\ & \quad \cdot
  \exp\!\left(
    \tfrac{ 
      \left[
      p
      + 
      2 \vee ( 4- \theta )
      \right]
      \,
      T
      \,
      \| g^* \|_{ \operatorname{Lip}( \R, \R^m ) }^2
    }{ 8 }
    +
    \tfrac{ 
      \int_0^T
	q 
	\alpha^2 
	e^{ \vartheta( \rho ) s }
	\,
      ds
    }{ 
      8 \rho [ \alpha - \rho \eta_1 ]
    }
	+
	\tfrac{
	  \frac{ 1 }{ 2 } 
	  +
	  \rho \| x \|^2
	  +
	  \rho \| y \|^2
	}{
	  2 q
	}
  \right)
  .
\end{split}
\end{equation}
In particular, in the case of additive noise,
i.e., $ g(y) = g(0) $ for all $ y \in \R $,
this shows
for every $ T \in ( 0, \infty ) $,
$ x = ( x_1 , x_2 ) $, 
$ y = ( y_1 , y_2 ) \in \R^2 $,
$ \rho \in ( 0, \frac{ \alpha }{ \eta_1 } ) $,
$ r \in (0,\infty] $
that
\begin{equation}
\begin{split}
&
  \left\|
    \sup\nolimits_{ t \in [ 0, T ] }
      \| X^x_t - X^y_t \|
  \right\|_{
    L^r( \Omega; \R )
  }
\\ &
\leq 
  \| x - y \|
    \exp\!\left(
  \tfrac{
    \left[
      \gamma 
      +
      \sqrt{
        \gamma^2
        +
        ( \delta - 1 )^2
      }
    \right] T
  }{
    2
  }
  +
  \tfrac{ 
    \int_0^T
      r 
      \alpha^2 
      e^{ \vartheta( \rho ) s }
      \,
    ds
  }{ 
    8 \rho [ \alpha - \rho \eta_1 ]
  }
      +
      \tfrac{
        \frac{ 1 }{ 2 } 
        +
        \rho \| x \|^2
        +
        \rho \| y \|^2
      }{
        2 r
      }
    \right)
  .
\end{split}
\end{equation}
In addition, combining~\eqref{eq:Van_der_Pol_estimate}
with Theorem~\ref{thm:strong_completeness_uniform}
proves that the stochastic Van der Pol oscillator~\eqref{eq:van_der_Pol}
is strongly complete.
Strong completeness for the SDE~\eqref{eq:van_der_Pol}
\sgc{}in the case where $g$ is globally bounded and globally Lipschitz continuous follows also from\cgs{}
Theorem 2.4 in Zhang~\cite{Zhang2010}
with $ \mathcal{W}(x) = \|x\|^2 $ 
for all $ x \in \R^2 $ and $ \alpha = 1 
$\sgc{}. Strong completeness for the SDE~\eqref{eq:van_der_Pol}
in the case where $ g $ is twice continuously differentiable with a globally 
bounded first derivative follows also with\cgs{} the method of Theorem 3.5 in 
Schenk-Hopp\'{e}~\cite{SchenkHoppe1996Deterministic}
by showing for every $ x \in \R^2 $ 
that 
$
  [0,\infty) \ni 
  t \mapsto 
  \big(
    X_t^{x,1} , X_t^{x,2} - g( X_t^{x,1} ) W_t
  \big)
  \in \R^2
$ 
is the solution of an appropriate random ordinary differential equation (RODE).

\section{Stochastic Duffing-van
der Pol oscillator}
\label{ssec:stochastic.Duffing.van.der.Pol.oscillator}

The Duffing-van der Pol equation unifies both the Duffing equation and
the van der Pol equation and has been used\sgc{}, for example,\cgs{} in 
certain flow-induced structural vibration problems
(see Holmes \&\ Rand~\cite{hr80})
and the references therein.
Schenk-Hopp\'{e}~\cite{SchenkHoppe1996Bifurcation} studied a stochastic version 
with affine-linear
noise acting on the velocity
(see also the references 
in~\cite{SchenkHoppe1996Bifurcation}).
Here we consider a more general version thereof.

Assume the setting of Section~\ref{sec:setting_SDE}, \sgc{}assume $d=2$, assume $D=\R^2$,\cgs{}
let
$
  \eta_0,\, \eta_1,\, \alpha_1 \in [0,\infty)
$,
$
  \alpha_2,\, \alpha_3
  \in (0,\infty)
$,
let $ g \colon \R \to \R^{ 1 \times m } $
be a globally Lipschitz continuous function
with
$
  \| g(y) \|^2
  \leq
  \eta_0 +
  \eta_1 y^2
$
for all $ y \in \R $,
let
$
  \mu \colon \R^2 \rightarrow \R^2
$
and 
$
  \sigma \colon \R^2 \rightarrow \R^{ 2 \times m }
$
be given by
$
  \mu( x )
=
  \left(
    x_2 ,
    \alpha_2 x_2 - \alpha_1 x_1 
    - \alpha_3 ( x_1 )^2 x_2
    - ( x_1 )^3
  \right) 
$
and
$
  \sigma( x ) u
=
  \left(
    0 ,
    g( x_1 ) u
  \right)
$
for all
$
  x = (x_1, x_2)
  \in \mathbb{R}^2
$,
$ u \in \R^m $
and let
$ X^x = ( X^{ x, 1 }, X^{ x, 2 } ) \colon [0,\infty) \times \Omega \to \R^2 $,
$ x = (x_1,x_2) \in \R^2 $,
be adapted stochastic processes with 
continuous sample paths satisfying~\eqref{eq:SDE_examples}
$ \P $-a.s.\ for all $ (t,x) \in [0,\infty) \times \R^2 $,
i.e., $X^{x,1}$ is the solution process to the SDE know as the 
\emph{stochastic Duffing-van der Pol oscillator}:
\begin{equation}\label{eq:Duffing}
\begin{split}
 \ddot{X}^{x,1}_t
 &=
 \alpha_2 \dot{X}^{x,1}_t
 -
 \alpha_1 X^{x,1}_t
 -
 \alpha_3 \big( X^{x,1}_t \big)^2 \dot{X}^{x,1}_t
 -
 \big(X^{x,1}_t\big)^3
 +
 g(X^{x,1}_t) \dot{W}_t, 
 \quad t\in [0,\infty), 
\\
 X^{x,1}_0 &= x_1,\,
 \dot{X}^{x,1}_0 = \sgc{}x_2.\cgs{}
\end{split}
\end{equation}
If $ \rho \in (0,\infty) $
and if
$
  U \colon \R^2 \to \R
$
is given by
$
  U(x_1,x_2) 
=
  \rho
  \big[
    \tfrac{ \left( x_1 \right)^4 }{ 2 }
    +
    \alpha_1
    \left( x_1 \right)^2
    +
    \left( x_2 \right)^2
  \big]
$
for all
$ x = (x_1, x_2) \in \mathbb{R}^2 $
(cf., e.g., (8) in
Holmes \&\ Rand~\cite{hr80}),
then it holds for every 
$ x = ( x_1, x_2 ) \in \R^2 $
that
\begin{equation}
\begin{split}
&
  ( \mathcal{G}_{ \mu, \sigma } U)( x )
  +
  \tfrac{ 1 }{ 2 }
  \|
    \sigma(x)^* 
    ( \nabla U)( x )
  \|^2
\\ & =
  2 \rho \alpha_1 x_1 x_2
  +
  2 \rho
  x_2 
  \left[ 
    \alpha_2 x_2 
    -
    \alpha_1 x_1 
    - \alpha_3 x_2 ( x_1 )^2 
  \right]
  +
  \rho \, \| g( x_1 ) \|^2
  +
  2 ( \rho x_2 )^2  
  \|
    g(x_1) 
  \|^2
\\ & \leq
  \rho \eta_0 
  +
  \rho 
  \left[ 
    \eta_1 ( x_1 )^{ 2 } 
    +
    2 
    \left[
      \rho \eta_0
      +
      \alpha_2 
    \right]
    ( x_2 )^2
  \right]
  +
  2 \rho 
  \left[
    \rho
    \eta_1 
    -
    \alpha_3
  \right] 
  ( x_1 x_2 )^2
\\ & \leq
  \rho \eta_0 
  +
  \rho 
    \left[ 
      \eta_1 
      -
      2 \alpha_1
      (
        \rho \eta_0
        +
        \alpha_2 
      )
    \right] 
    ( x_1 )^{ 2 } 
    +
      2 \rho
      \left(
        \rho \eta_0
        +
        \alpha_2 
      \right)
    \left[
      \alpha_1 ( x_1 )^2
      +
      ( x_2 )^2
    \right]
\\ & \quad
  +
  2 \rho 
  \left[
    \rho
    \eta_1 
    -
    \alpha_3
  \right] 
  ( x_1 x_2 )^2
\\ & \leq
  \rho \eta_0 
  +
    \tfrac{
      \rho
    \left|
      0 \vee 
      (
        \eta_1 
        -
        2 \alpha_1
        (
          \rho \eta_0
          +
          \alpha_2 
        )
      )
    \right|^2
    }{
      4
      \left(
        \rho \eta_0
        +
        \alpha_2 
      \right)
    }
    +
      2 
      \left(
        \rho \eta_0
        +
        \alpha_2 
      \right)
      U(x)
  +
  2 \rho 
  \left[
    \rho
    \eta_1 
    -
    \alpha_3
  \right] 
  ( x_1 x_2 )^2
  .
\end{split} 
\end{equation}
Corollary~\ref{cor:exp_mom}
hence proves for every 
$ t \in (0,\infty) $,
$ \rho \in [0, \frac{ \alpha_3 }{ \eta_1 } ] \cap \R $,
$ x = ( x_1, x_2) \in \R^2 $ that
\begin{equation}
\begin{split}
&
  \E\!\left[
    \exp\!\left(
    \tfrac{
      \rho \,
      \big(
        \frac{ 
          1
        }{ 2 }
        [ X^{ x, 1 }_t ]^4
        +
        \alpha_1
        [ X^{ x, 1 }_t ]^2
        +
        [ X^{ x, 2 }_t ]^2
      \big)
    }{
      \exp\!\left(
        2 t
        \left[
          \rho \eta_0
          +
          \alpha_2 
        \right]
      \right)
    }
    +
    \smallint_0^t
    \tfrac{
      2 \rho 
      \left[ 
        \alpha_3 - \rho \eta_1
      \right]
        \left[
          X^{ x, 1 }_s 
          X^{ x, 2 }_s 
        \right]^2
    }{
      \exp\!\left(
        2 s
        \left[
          \rho \eta_0
          +
          \alpha_2 
        \right]
      \right)
    }
    \, ds
    \right)
  \right]
\\ & \qquad \leq
    \exp\!\left(
      \smallint_0^t
      \tfrac{
  \rho \eta_0 
  +
    \tfrac{
      \rho
    \left|
      0 \vee 
      (
        \eta_1 
        -
        2 \alpha_1
        \left[
          \rho \eta_0
          +
          \alpha_2 
        \right]
      )
    \right|^2
    }{
      4
      \left[
        \rho \eta_0
        +
        \alpha_2 
      \right]
    }
      }{
      \exp\!\left(
        2 s
        \left[ 
          \rho \eta_0
          +
          \alpha_2 
        \right]
      \right)
      }
      \, ds
      +
      \rho 
      \Big[
        \tfrac{ 
          ( x_1 )^4
        }{ 2 }
        +
        \alpha_1 ( x_1 )^2
        +
        ( x_2 )^2
      \Big]
    \right)
\\ & \qquad \leq
   \exp\!\left(
      \tfrac{
        1 +
        \rho 
        ( \alpha_1 )^2
      }{ 2 }
        +
    \tfrac{
      t
      \rho 
      \left( 
        \eta_1 
      \right)^2
    }{
      4
      \left(
        \rho \eta_0
        +
        \alpha_2 
      \right)
    }
        +
      \rho 
        ( x_1 )^4
        +
      \rho 
        ( x_2 )^2
   \right)
  .
\end{split}
\end{equation}
In the next step we observe for every 
$ x = (x_1, x_2) , y = (y_1, y_2) \in \R^2 $
with $ x \neq y $
that
\begin{equation}
\label{eq:Duffing_nonlinearity}
\begin{split}
&
  -
  \tfrac{
    \left[ ( x_1 )^3 - ( y_1 )^3 \right]
    \left[ x_2 - y_2 \right]
  }{
    \| x - y \|^2
  }
=
  -
  \tfrac{
    \left[ ( x_1 )^2 + x_1 y_1 + ( y_1 )^2 \right]
    \left[ x_1 - y_1 \right]
    \left[ x_2 - y_2 \right]
  }{
    \left[ x_1 - y_1 \right]^2
    +
    \left[ x_2 - y_2 \right]^2
  }
\leq
  \tfrac{ 
    ( x_1 )^2 + x_1 y_1 + ( y_1 )^2 
  }{ 2 }
  .
\end{split}
\end{equation}
Combining~\eqref{eq:Van_der_Pol_nonlinearity}
and~\eqref{eq:Duffing_nonlinearity}
implies 
for every
$ \varepsilon, \theta \in ( 0, \infty ) $,
$ x = ( x_1, x_2 ) $, $ y = ( y_1, y_2 ) \in \R^2 $
with $ x \neq y $ that
\begin{equation}
\begin{split}
&
  \tfrac{ 
    \langle
      x - y ,
      \mu( x ) - 
      \mu( y )
    \rangle
    +
    \frac{ 1 }{ 2 }
    \| 
      \sigma( x ) - \sigma( y ) 
    \|^2_{ \HS( \R^m, \R^2 ) }
  }{
    \|
      x - y
    \|^2
  }
    +
    \tfrac{
      ( \frac{ \theta }{ 2 } - 1 ) \,
      \| 
        (
          \sigma( x ) - \sigma( y )
        )^*
        ( x - y )
      \|^2
    }{
      \| x - y \|^4
    }
\\ & \leq
  \tfrac{
    \alpha_2 + \sqrt{
      ( \alpha_1 - 1 )^2
      +
      ( \alpha_2 )^2
    }
  }{
    2
  }
  -
  \tfrac{
    \alpha_3
    \left[ x_2 ( x_1 )^2 - y_2 ( y_1 )^2 \right]
    \left[ x_2 - y_2 \right]
    + \left[ ( x_1 )^3 - ( y_1 )^3 \right]
    \left[ x_2 - y_2 \right]
  }{
    \| x - y \|^2
  }
\\ & \quad
  +
  \tfrac{
    \| 
      g( x_1 ) - g( y_1 ) 
    \|^2
  }{
    2 \,
    \| x - y \|^2
  }
    +
    \tfrac{
      ( \frac{ \theta }{ 2 } - 1 ) \,
      | x_2 - y_2 |^2 \,
      \| 
        g( x_1 ) - g( y_1 ) 
      \|^2
    }{
      \| x - y \|^4
    }
\\ & \leq
  \tfrac{
    \alpha_2 + \sqrt{
      ( \alpha_1 - 1 )^2
      +
      ( \alpha_2 )^2
    }
  }{
    2
  }
  +
  \tfrac{
    \max( \theta + 2 , 4 )
  }{
    8
  }
  \,
  \| 
    g^*
  \|^2_{ \operatorname{Lip}( \R, \R^m ) }
\\ & \quad 
  +
    \tfrac{ 
      \alpha_3
    }{
      2
    }
    \left[ 
      | x_1 x_2 | + | y_1 y_2 |
    \right]
  +
  \tfrac{ 
    3
    [
      ( x_1 )^2 + ( y_1 )^2 
    ]
  }{ 4 }
  \,
  .
\end{split}
\end{equation}
Corollary~\ref{cor:UV2} \sgc{}with\cgs{}
$ U_{0,0}=U_{0,1}=\overline{U}_{0,1} \equiv 0$,
$ U_{1,0}(x_1,x_2) = \rho_0 (\inv{2} x_1^4 + \alpha_1 x_1^2 + x_2^2)$,
$ U_{1,1}(x_1,x_2) = \rho_1 (\inv{2} x_1^4 + \alpha_1 x_1^2 + x_2^2)$,
$ \overline{U}_{1,1}(x_1, x_2) = 2\rho_1 (\alpha_3 - \rho_1 \eta_1)(x_1 x_2)^2$
for all $(x_1,x_2)\in \R^2$,
$
\alpha_{1,0} = 2(\rho_0 \eta_0 + \alpha_2),
$
$
\alpha_{1,1} = 2(\rho_1 \eta_0 + \alpha_2),
$
$
  \beta_{1,0} 
  = 
  \rho_0 \eta_0 
  + 
  \tfrac{
    |
      0 \vee
      (
	\eta_1 -
	2 \alpha_1 ( \rho_0 \eta_0 + \alpha_2 )
      )
    |^2
  }{
    4 ( \rho_0 \eta_0 + \alpha_2 )
  },
$
$
  \beta_{1,1} 
  = 
  \rho_1 \eta_0 
  + 
  \tfrac{
    |
      0 \vee
      (
	\eta_1 -
	2 \alpha_1 ( \rho_1 \eta_0 + \alpha_2 )
      )
    |^2
  }{
    4 ( \rho_1 \eta_0 + \alpha_2 )
  },
$
$ \rho = q_{0,0} = q_{0,1}= \infty$,
$ q_{1,0} = q_0$,
$ q_{1,1} = q_1$, 
$ 
  c_0(t) 
  = 
  \inv{8}( p - \theta) 
    \| g^* \|_{ \operatorname{Lip}( \R, \R^m ) }^2,
$
\begin{equation}
\begin{split}
  c_1(t)
  & =
    \tfrac{
	\alpha_2 + \sqrt{
	( \alpha_1 - 1 )^2
	+
	( \alpha_2 )^2
	}
    }{
      2
    }
  +
    \tfrac{
      \theta + 2 \vee ( 4 - \theta ) 
      \| 
	g^*
      \|^2_{ \operatorname{Lip}( \R, \R^m ) }
    }{
      8
    }
\\ & \quad
  +
    \tfrac{ 
      q_1 ( \alpha_3 )^2
      e^{
        2 t [ \rho_1 \eta_0 + \alpha_2 ]
      }
    }{
      8 \rho_1 
      \left[ 
   \alpha_3   -    \rho_1 \eta_1
      \right]
    }
  +
    \tfrac{ 
      9 q_0 T 
      e^{
        2 t [ \rho_0 \eta_0 + \alpha_2 ] 
      }
    }{
      16 \rho_0
    }
\end{split}
\end{equation}
for all $t\in [0,T]$
\sgc{}hence shows\cgs{} for every 
$ T \in ( 0, \infty ) $,
$ \rho_0 \in (0, \frac{ \alpha_3 }{ \eta_1 } ] \cap \R $,
$ \rho_1 \in (0, \frac{ \alpha_3 }{ \eta_1 } ) $,
$ x = ( x_1, x_2 ) $, $ y = ( y_1, y_2 ) \in \R^2 $,
$ r, p, q_0, q_1 \in ( 0, \infty ] $,
$ \theta \in (0,p) $
with
$
  \frac{ 1 }{ p } + \frac{ 1 }{ q_0 } + \frac{ 1 }{ q_1 } = \frac{ 1 }{ r }	
$
that
\begin{equation}
\begin{split}
&
  \left\|
    \sup\nolimits_{ t \in [0,T] }
      \| 
	X^x_t - X^y_t
      \|
  \right\|_{
    L^r( \Omega; \R^2 )
  }
\\ & \leq
  \tfrac{
    \left\| x - y \right\|
  }{
    [ 
      1 - 
      \theta / p
    ]^{ 
      1 / \theta 
    }
  }
  \exp\!\left(
  \tfrac{
    \big[
      \alpha_2 + \sqrt{
      ( \alpha_1 - 1 )^2
      +
      ( \alpha_2 )^2
      }
    \big] 
    T
  }{
    2
  }
  +
  \tfrac{
    [ p + 2 \vee ( 4 - \theta ) ] T
    \| 
      g^*
    \|^2_{ \operatorname{Lip}( \R, \R^m ) }
  }{
    8
  }
 \right)
\\ & \quad \cdot 
  \exp\!\left(
    \int_0^T
      \tfrac{ 
	9 q_0 T 
	e^{
	  2 s [ \rho_0 \eta_0 + \alpha_2 ] 
	}
      }{
	16 \rho_0
      }
      \,
    ds
    +
    \int_0^T
    \tfrac{ 
      q_1 ( \alpha_3 )^2
      e^{
        2 s [ \rho_1 \eta_0 + \alpha_2 ]
      }
    }{
      8 \rho_1 
      \left[ 
        \alpha_3 - \rho_1 \eta_1 
      \right]
    }
    \, ds
  \right)
\\ & \quad  \cdot
  \exp\!\left(
  \int_0^T
    \tfrac{1}{
      q_0
    }
    \left[ 
      \rho_0 \eta_0 
      + 
      \tfrac{
	|
	  0 \vee
	  (
	    \eta_1 -
	    2 \alpha_1 ( \rho_0 \eta_0 + \alpha_2 )
	  )
	|^2
      }{
	4 ( \rho_0 \eta_0 + \alpha_2 )
      }
    \right]
    ( 1 - \tfrac{ s }{ T } )
      e^{
        - 2 s [ \rho_0 \eta_0 + \alpha_2 ]
      }
    \,ds
  \right)
\\ & \quad \cdot
  \exp\!\left(
    \int_0^T
      \tfrac{1}{ q_1 }
      \left[ 
        \rho_1 \eta_0 
        + 
        \tfrac{
          |
            0 \vee
            (
              \eta_1 -
              2 \alpha_1 ( \rho_1 \eta_0 + \alpha_2 )
            )
          |^2
        }{
          4 ( \rho_1 \eta_0 + \alpha_2 )
        }
      \right]
      e^{
        - 2 s [ \rho_1 \eta_0 + \alpha_2 ]
      }
    \,
  ds
  \right)
\\ & \quad \cdot
  \exp\!\left(
    \left[ 
      \tfrac{ \rho_0 }{ q_0 }
      +
      \tfrac{ \rho_1 }{ q_1 }
    \right]
    \left[
      \tfrac{ 
          ( x_1 )^4 
          +
          ( y_1 )^4 
      }{ 4 }
      +
      \tfrac{
        \alpha_1 
        [ 
          ( x_1 )^2
          +
          ( y_1 )^2
        ]
      }{ 2 }
      +
      \tfrac{
        ( x_2 )^2
        +
        ( y_2 )^2
      }{ 2 }
    \right]
  \right)
  .
\end{split}
\end{equation}
This implies
for every 
$ T \in ( 0, \infty ) $,
$ \rho_0, \rho_1 \in (0, \frac{ \alpha_3 }{ \eta_1 } ) $,
$ x = ( x_1, x_2 ) $, $ y = ( y_1, y_2 ) \in \R^2 $,
$ r, q_0, q_1 \in ( 0, \infty ] $,
$ p \in (2,\infty] $
with
$
  \frac{ 1 }{ p } + \frac{ 1 }{ q_0 } + \frac{ 1 }{ q_1 } = \frac{ 1 }{ r }	
$
that
\begin{equation}
\begin{split}
&
    \left\|
      \sup\nolimits_{ t \in [0,T] }
        \| 
          X^x_t - X^y_t
        \|
    \right\|_{
      L^r( \Omega; \R^2 )
    }
\\ & \leq
  \tfrac{
    \left\| x - y \right\|
  }{
    \sqrt{ 
      1 - 
      2 / p
    }
  }
  \exp\!\left(
    \tfrac{ 1 }{ 2 q_0 }
    +
    \tfrac{ 1 }{ 2 q_1 }
    +
    \alpha_2 T
    +
    \tfrac{ ( \alpha_1 + 1 ) T }{ 2 }
    +
  \tfrac{
    ( p + 2 ) T
    \| 
      g^*
    \|^2_{ \operatorname{Lip}( \R, \R^m ) }
  }{
    8
  }
  \right)
\\ & \quad \cdot
  \exp\!\left(
    \smallsum\limits_{ i = 0 }^1
    \tfrac{
        2^i
        (
          \eta_1 
        )^2
        T
    }{
      8 
      q_i
      ( \rho_i \eta_0 + \alpha_2 )
    }
    +
    \tfrac{ 
      q_0 T^2 
      e^{
        2 T [ \rho_0 \eta_0 + \alpha_2 ] 
      }
    }{
      \rho_0
    }
    +
    \tfrac{ 
      q_1 T ( \alpha_3 )^2
      e^{
        2 T [ \rho_1 \eta_0 + \alpha_2 ]
      }
    }{
      8 \rho_1 
      \left[ 
        \alpha_3 - \rho_1 \eta_1 
      \right]
    }
  \right)
\\ & \quad \cdot
  \exp\!\left(
    \left[ 
      \tfrac{ \rho_0 }{ q_0 }
      +
      \tfrac{ \rho_1 }{ q_1 }
    \right]
    \left[
      \tfrac{ 
          ( x_1 )^4 
          +
          ( y_1 )^4 
      }{ 4 }
      +
      \tfrac{
        \alpha_1 
        [ 
          ( x_1 )^2
          +
          ( y_1 )^2
        ]
      }{ 2 }
      +
      \tfrac{
        ( x_2 )^2
        +
        ( y_2 )^2
      }{ 2 }
    \right]
  \right)
  .
\end{split}
\end{equation}
Combining this with Theorem~\ref{thm:strong_completeness_uniform}
proves that the 
stochastic Duffing-van der Pol oscillator~\eqref{eq:Duffing}
is strongly complete.
Strong completeness for \sgc{}the Duffing oscillator, i.e., \cgs{}the SDE~\eqref{eq:Duffing} 
\sgc{}with $\alpha_3 =0$, follows also from \cgs{}Theorem 
2.4 in \sgc{}Zhang~\cite{Zhang2010}\cgs{} with $\mathcal{W}(x_1,x_2)= \sgc{}1 + \inv{2} (x_1)^4 + 
(x_2)^2\cgs{}$ for all $x_1,x_2\in\R^2$ and $\alpha=1$. \sgc{}
Strong completeness for the SDE~\eqref{eq:Duffing} in the case where $g$ is \cgs{}a twice continuously differentiable function with globally bounded first 
\sgc{}derivative\cgs{} follows also with the method of Theorem 3.5 in 
Schenk-Hopp\'{e}~\cite{SchenkHoppe1996Deterministic}
by showing that $[0,\infty)\ni t\mapsto 
(X_t^{x,1},X_t^{x,2}-g(X_t^{x,1})W_t)\in\R^2$ is the 
solution of a random ordinary differential equation for every $x\in\R^2$.

\section{Stochastic Lorenz equation with additive noise}
\label{ssec:stochastic.Lorenz.equation}

Lorenz~\cite{Lorenz1963} suggested a three-dimensional ordinary differential 
equation
as a simplified model of convection rolls in the atmosphere.
As, for instance, in 
Zhou \&\ E~\cite{ZhouE2010},
we consider a stochastic version thereof with additive noise.

Assume the setting of Section~\ref{sec:setting_SDE}\sgc{}, assume $d=m=3$, assume $D=\R^3$,\cgs{} let
$
  \alpha_1,\, \alpha_2,\, \alpha_3,\, \beta \in [0,\infty)
$,
and
let
$ 
  A \in \R^{ 3 \times 3 } 
$,
$
  \mu \colon \R^3 \to \R^3
$
and
$
  \sigma \colon \R^3 \to \R^{ 3 \times 3 }
$
be given by
\begin{equation}
  A 
  =
  \left(
    \begin{array}{ccc}
      - \alpha_1 & \alpha_1 & 0
      \\
      \alpha_2 & - 1 & 0
      \\
      0 & 0 & - \alpha_3
    \end{array}
  \right)
  ,
\qquad
  \mu\!\left(
    \begin{array}{c}
      x_1
      \\
      x_2
      \\
      x_3
    \end{array}
  \right)
=
  A x
  +
  \left(
    \begin{array}{c}
      0
      \\
      - x_1 x_3
      \\
      x_1 x_2 
    \end{array}
  \right)
\end{equation}
and $ \sigma(x) = \sqrt{ \beta } I_{ \R^3 } $
for all
$
  x = (x_1, x_2, x_3)
  \in \mathbb{R}^3
$.
Moreover, let
\begin{equation}
 X^x = (X^{x,1},X^{x,2},X^{x,3})\colon [0,\infty) \times \Omega \to \R^3 ,
\quad x \in \R^3 ,
\end{equation}
be adapted stochastic processes with 
continuous sample paths satisfying~\eqref{eq:SDE_examples}
$ \P $-a.s.\ for all $ (t,x) \in [0,\infty) \times \R^3 $,
i.e.,
\begin{equation}\label{eq:Lorenz}
  X_t^x 
 = 
  x 
  + 
  \int_0^t 
      A X_s^x
      +
      \left(
	\begin{array}{ccc}
	0 & 0 &  0 \\
	0 & 0 & -X_s^{x,1} \\
	0 & X_s^{x,1} &  0 
	\end{array}
      \right)
      X_s^x
  \,ds
  +
  \sqrt{\beta} W_t
\end{equation}
$\P$-a.s.\ for all $(t,x)\in [0,\infty)\times \R^3$.
Thus the processes $ X^x $, $ x \in \R^d $, are 
solution processes of the
stochastic Lorenz equation
in 
Zhou \sgc{}\&\cgs{} E~\cite{ZhouE2010}.

In the next step we define a real number 
$ \vartheta \in [0,\infty) $ by
\begin{equation}
  \vartheta := 
    \min_{ r \in (0,\infty) }
    \left[
      \big[
      \tfrac{
        ( \alpha_1 + \alpha_2 )^2
      }{
        r
      }
      -
      2 \alpha_1
      \big]
    \vee
      \left[
        r - 1
      \right]
    \vee
      0
    \right]
  .
\end{equation}
If $ \rho \in [0,\infty) $ and
if $ U \colon \R^3 \to \R $ is given by
$ U( x ) = \rho \, \| x \|^2 $
for all $ x \in \R^3 $,
then it holds 
for every $ x = ( x_1, x_2, x_3 ) \in \R^3 $ that
\begin{equation}
\begin{split}
&
  ( \mathcal{G}_{ \mu, \sigma } U)( x )
  +
  \tfrac{ 1 }{ 2 }
  \| \sigma(x)^* ( \nabla U )(x) \|^2
\\ & =
  2 \rho
  \left< x, \mu( x ) \right>
  +
  3 \rho \beta 
  +
  2 \rho^2 \beta
  \| x \|^2
\\ & =
  2 \rho \alpha_1 x_1 ( x_2 - x_1 )
  +
  2 \rho x_2 ( \alpha_2 x_1 - x_2 )
  -
  2 \rho \alpha_3 ( x_3 )^2
  +
  3 \rho \beta 
  +
  2 \rho \beta U(x)
\\ & =
  2 \rho ( \alpha_1 + \alpha_2 ) x_1 x_2 
  -
  2 \rho 
  \left[ \alpha_1 ( x_1 )^2 + ( x_2 )^2 + \alpha_3 ( x_3 )^2 
  \right]
  +
  3 \rho \beta 
  +
  2 \rho \beta U(x)
\\ & \leq
  \rho 
  \cdot
  \inf_{ r \in ( 0, \infty ) }
  \left[
    \big[ 
      \tfrac{
        ( \alpha_1 + \alpha_2 )^2
      }{
        r
      }
      -
      2 \alpha_1 
    \big] 
    ( x_1 )^2 
    + 
    \left(
      r - 1
    \right) 
    ( x_2 )^2 
    - 
    2 \alpha_3 ( x_3 )^2 
  \right]
  +
  3 \rho \beta 
  +
  2 \rho \beta U(x)
\\ & \leq
  3 \rho \beta 
  +
  \left[ 
    2 \rho \beta 
    +
    \inf_{ r \in (0,\infty) }
    \left[
      \big[
      \tfrac{
        ( \alpha_1 + \alpha_2 )^2
      }{
        r
      }
      -
      2 \alpha_1
      \big]
    \vee
      \left[
        r - 1
      \right]
    \vee
      \left[
        - 2 \alpha_3
      \right]
    \right]
  \right] 
  U(x)
\\ & \leq
  3 \rho \beta
  +
  \left[ 2 \rho \beta + \vartheta \right] U(x) .
\end{split}
\end{equation}
Hence, Corollary~\ref{cor:exp_mom}
implies for \sgc{}all $ x \in \R^3 $, $ t, \rho \in [0,\infty) $\cgs{}
that
\begin{equation}
\begin{split}
&
  \E\!\left[
    \exp\!\left(
      \tfrac{ 
        \rho 
      }{
        e^{ ( 2 \rho \beta + \vartheta ) t }
      }
      \| X^x_t \|^2
    \right)
  \right]
\leq
  \exp\!\left(
    \smallint\nolimits_0^t
      \tfrac{
        3 \rho \beta 
      }{ e^{ ( 2 \rho \beta + \vartheta ) s } 
      } 
    \, ds
    +
    \rho \, \| x \|^2
  \right)
\leq
  \exp\!\left(
    \tfrac{ 3 }{ 2 }
    +
    \rho \, \| x \|^2
  \right)
  .
\end{split}
\end{equation}
Next we apply Corollary~\ref{cor:UV2}.
For this observe 
for every 
$ \theta \in (0,\infty] $
and every
$ x = (x_1, x_2, x_3), y = (y_1, y_2, y_3) \in \R^3 $ 
with $ x \neq y $ 
that
\begin{align}
\nonumber
&
  \tfrac{ 
    \langle
      x - y ,
      \mu( x ) - 
      \mu( y )
    \rangle
    +
    \frac{ 1 }{ 2 }
    \| 
      \sigma( x ) - \sigma( y ) 
    \|^2_{ \HS( \R^3 ) }
  }{
    \|
      x - y
    \|^2
  }
    +
    \tfrac{
      ( \frac{ \theta }{ 2 } - 1 ) \,
      \| 
        (
          \sigma( x ) - \sigma( y )
        )^*
        ( x - y )
      \|^2
    }{
      \| x - y \|^4
    }
\\ \nonumber & =
  \tfrac{ 
    \langle
      x - y ,
      \mu( x ) - 
      \mu( y )
    \rangle
  }{
    \|
      x - y
    \|^2
  }
\\ \nonumber & \leq
  \tfrac{
    \max( \text{spectrum}( A + A^* ) )
  }{ 
    2 
  }
  +
  \tfrac{ 
    \left( x_3 - y_3 \right)
    \left( 
      x_1 x_2 - y_1 y_2
    \right)
    -
    \left( x_2 - y_2 \right)
    \left( 
      x_1 x_3 - y_1 y_3
    \right)
  }{
    \|
      x - y
    \|^2
  }
\\ & =
  \tfrac{
    \max( \text{spectrum}( A + A^* ) )
  }{ 
    2 
  }
  +
  \tfrac{ 
    \left( x_1 - y_1 \right) 
    \left( 
      y_2 x_3 - x_2 y_3 
    \right)
  }{
    \|
      x - y
    \|^2
  }
\\ \nonumber & =
  \tfrac{
    \max( \text{spectrum}( A + A^* ) )
  }{ 
    2 
  }
  +
  \tfrac{ 
    \left( x_1 - y_1 \right) 
    \left[
      \left(
        y_2 + x_2
      \right)
      \left( x_3 - y_3 \right) 
      -
      \left( x_2 - y_2 \right)
      \left(
        y_3 + x_3
      \right)
    \right]
  }{
    2 \,
    \|
      x - y
    \|^2
  }
\\ \nonumber & \leq
  \tfrac{
    \max( \text{spectrum}( A + A^* ) )
  }{ 
    2 
  }
  +
  \tfrac{
    | x_2 | 
    +
    | x_3 |
    +
    | y_2 | 
    +
    | y_3 |
  }{ 
    4
  }
  .
\end{align}
The estimate 
$
  a \leq \frac{ \delta }{ 4 } + \frac{ a^2 }{ \delta }
$
for all $ a \in \R $,
$ \delta \in (0,\infty) $
hence proves
for every 
$ x = (x_1, x_2, x_3), y = (y_1, y_2, y_3) \in \R^3 $ 
with $ x \neq y $ 
and every $ r, t, T, \rho \in ( 0, \infty ) $,
$ \theta \in (0,\infty] $
that
\begin{equation}
\begin{split}
\label{eq:estimate.Lorenz.equation}
&
  \tfrac{ 
    \langle
      x - y ,
      \mu( x ) - 
      \mu( y )
    \rangle
    +
    \frac{ 1 }{ 2 }
    \| 
      \sigma( x ) - \sigma( y ) 
    \|^2_{ \HS( \R^3 ) }
  }{
    \|
      x - y
    \|^2
  }
    +
    \tfrac{
      ( \frac{ \theta }{ 2 } - 1 ) \,
      \| 
        (
          \sigma( x ) - \sigma( y )
        )^*
        ( x - y )
      \|^2
    }{
      \| x - y \|^4
    }
\\ & \leq
  \tfrac{
    \max( \text{spectrum}( A + A^* ) )
  }{ 
    2 
  }
  +
  \tfrac{ 1 }{ 4 }
  \cdot
  \tfrac{ 
    r T e^{ ( 2 \rho \beta + \vartheta ) t } 
  }{ 
    8 \rho  
  }
  +
  \tfrac{
    8 \rho
  }{
    r T e^{ ( 2 \rho \beta + \vartheta ) t } 
  }
  \cdot
  \tfrac{
    | x_2 |^2 
    +
    | x_3 |^2
    +
    | y_2 |^2
    +
    | y_3 |^2
  }{ 
    16
  }
\\ & \leq
  \tfrac{
    \max( \text{spectrum}( A + A^* ) )
  }{ 
    2 
  }
  +
  \tfrac{ 
    r T e^{ ( 2 \rho \beta + \vartheta ) t } 
  }{ 
    32 \rho
  }
  +
  \tfrac{
    \rho \left[ \| x \|^2 + \| y \|^2 \right]
  }{ 
    2 r T e^{ ( 2 \rho \beta + \vartheta ) t } 
  }
  .
\end{split}
\end{equation}
Corollary~\ref{cor:UV2}
hence implies 
for every 
$ T, r, \rho \in ( 0, \infty) $
and every $ x, y \in \R^3 $ that
\begin{equation}
\label{eq:Lorenz_stability}
\begin{split}
&
  \left\|
    \sup\nolimits_{ t \in [0,T] }
    \|
      X^x_t - X^y_t
    \|
  \right\|_{
      L^r( \Omega; \R^3 )
    }
\\ & 
\leq
  \| x - y \|
  \exp\!\left(
    \tfrac{
      \max( \text{spectrum}( A + A^* ) ) T
    }{ 
      2 
    }
    +
    \tfrac{ 
      r T^2 e^{ ( 2 \rho \beta + \vartheta ) T } 
    }{ 
      32 \rho
    }
    +
    \tfrac{
      3 +
      \rho \,
      \| x \|^2 + 
      \rho \, \| y \|^2 
    }{
      2 r
    }
  \right)
  .
\end{split}
\end{equation}
Combining this with Theorem~\ref{thm:strong_completeness_uniform}
ensures that the stochastic Lorenz equation~\eqref{eq:Lorenz} is strongly 
complete.
In the same way as above strong stability estimates of the form 
\eqref{eq:Lorenz_stability} and strong completeness can be proved
if the diffusion coefficient is not necessarily constant as in~\eqref{eq:Lorenz}
but globally bounded and globally Lipschitz continuous.
Strong completeness
for the SDE~\eqref{eq:Lorenz}
follows also from\sgc{}
Theorem 2.4
in Zhang~\cite{Zhang2010} and inequality~\eqref{eq:estimate.Lorenz.equation}\cgs{}.
If the diffusion coefficient is linear and if $m=1$, then strong completeness
follows in the case $\alpha_2=\alpha_1$ from Theorem 4.1 in 
Schmalfu\ss~\cite{Schmalfuss1997}.
If the diffusion coefficient is merely globally Lipschitz continuous but
not globally bounded, then it is still an \emph{open question}
whether strong stability estimates of the form~\eqref{eq:Lorenz_stability}
do hold 
(see also Section~2 in Hairer et al.~\cite{hhj12} for a counterexample with a 
related drift coefficient
function and a linear diffusion coefficient function)
and also
whether the SDE \eqref{eq:Lorenz} is strongly complete
if $\sigma$ is non-linear or if 
$m>1$.
Another way for establishing
strong completeness for the stochastic Lorenz equation~\eqref{eq:Lorenz}
in the case of additive noise is to subtract the noise
process and then to solve the resulting random
ordinary differential equations for 
every continuous trajectory of the driving noise process.

\section{Langevin dynamics}
\label{ssec:Langevin.dynamics}

Assume the setting of Section~\ref{sec:setting_SDE}\sgc{}, assume $d=2m$, assume $D=\R^{2m}$,\cgs{} let
$ \gamma,\, \varepsilon \in (0,\infty) $,
$ U \in C^2( \R^m, \R ) $,
let $ \mu \colon \R^{ 2 m } \to \R^{ 2 m } $
and
$ \sigma \colon \R^{ 2 m } \to \R^{ ( 2 m ) \times m } $
be given by
$
  \mu( x )
  =
  ( x_2, - ( \nabla U )( x_1 ) - \gamma x_2 )
$
and
$
  \sigma( x ) u
  = ( 0, \sqrt{ \varepsilon } u )
$
for all $ x = ( x_1, x_2 ) \in \R^{ 2 m } $, $ u \in \R^m $
and let
$ X^x \colon [0,\infty) \times \Omega \to \R^{ 2 m } $,
$ x = (x_1,x_2) \in \R^{ 2 m } $, be adapted stochastic processes 
with continuous sample paths satisfying~\eqref{eq:SDE_examples}
$ \P $-a.s.\ for all $ (t, x) \in [0,\infty) \times \R^{ 2 m } $,
i.e., $X^{x,1}$ is a solution process to the SDE known as the 
\emph{stochastic Langevin equation}, 
a well-known model for the dynamics of a molecular system:
\begin{equation}\label{eq:Langevin_dynamics}
  \ddot{X}^{x,1}_t 
 =
  - 
  (\nabla U)(X^{x,1}_t)
  -
  \gamma \dot{X}^{x,1}_t
  +
  \sqrt{\varepsilon} \dot{W}_t,
  \quad t \in [0,\infty),
  \qquad X^{x,1}_0 = x_1,\, \dot{X}^{x,1}_0 = x_2.
\end{equation}

Next observe that if 
$ \rho \in [0,\infty) $
and if
$ U_0 \colon \R^{ 2 m } \to \R $
is given by
$ U_0( x ) = \rho \, U( x_1 ) + \tfrac{ \rho }{ 2 } \, \| x_2 \|^2 $
for all $ x = ( x_1, x_2 ) \in \R^{ 2 m } $,
then it holds for every
$ x = ( x_1, x_2 ) \in \R^{ 2 m } $
that
\begin{equation}
\label{eq:Langevin_dynamics_expmom}
\begin{split}
&
  ( \mathcal{G}_{ \mu, \sigma } U_0 )( x )
  +
  \tfrac{ 
    1
  }{ 2 }
    \| 
      \sigma( x )^*
      ( \nabla U_0 )( x )	
    \|^2
  =
  \tfrac{ \rho \varepsilon m }{ 2 }
  +
  \rho
  \left[
    \tfrac{ \rho \varepsilon }{ 2 }
    -
    \gamma
  \right]
  \| x_2 \|^2
  .
\end{split}
\end{equation}
Corollary~\ref{cor:exp_mom}
hence implies for every 
$ x = ( x_1 , x_2 ) \in \R^{ 2 m } $,
$ t, \rho \in [0,\infty) $
that
\begin{equation}
\begin{split}
&  \E\!\left[
    \exp\!\left(
      \rho \, U( X^{ x, 1 }_t )
      +
      \tfrac{ 
        \rho 
        \,
        \|
          X^{ x, 2 }_t
        \|^2
      }{ 2 }
      +
      \int_0^t
        \rho 
        \left[
          \gamma 
          -
          \tfrac{ \rho \varepsilon }{ 2 }
        \right]
        \left\|
          X^{ x, 2 }_s
        \right\|^2
      ds
    \right)
  \right]
\\ & 
\leq
 \exp\!\left(
      \tfrac{
        \rho \varepsilon m t
      }{ 
        2
      }
      +
      \rho \, U( x_1 )
      +
      \tfrac{ \rho }{ 2 } 
      \left\| x_2 \right\|^2
 \right)
  \, .
\end{split}
\end{equation}
Combining~\eqref{eq:Langevin_dynamics_expmom}
and Corollary~\ref{cor:UV2} shows 
that if there exist
$ c\in[0,\infty)$,
$ \rho \in [0, \tfrac{ 2 \gamma }{ \varepsilon } ]$\sgc{},\cgs{} $r,\, T \in 
(0,\infty)$ such that for all $x,\,y\in \R^m$ \sgc{}with\cgs{} $x\neq y$ it holds that
\begin{equation}
\label{eq:assumption_Langevin}
  \tfrac{ 
    \|
      ( \nabla U )( x )
      -
      ( \nabla U )( y )
    \|^2
  }{
    2 
    \left\|
      x - y
    \right\|^2
  }
\leq
  c
  +
    \tfrac{
      \rho U( x )
      +
      \rho U( y )
    }{
      2 r
      T
    }
  ,
\end{equation}
then it holds for every
$ x = ( x_1, x_2 ), y = ( y_1, y_2 ) \in \R^{ 2 m } $
that
\begin{equation}  \begin{split}
  &\left\|
    \sup\nolimits_{ t \in [ 0, T ] }
      \| X^x_t - X^y_t \|
  \right\|_{
    L^r( \Omega; \R )
  }
  \\&
\leq  
  \| x - y \|
  \exp\!\left(
      \left[ 
        c + 1 +
        \tfrac{ \rho \varepsilon m }{ 4 r }
      \right] T 
      +
      \tfrac{
        \rho 
        U( x_1 ) 
        +
        \rho U( y_1 ) 
      }{
        2 r
      }
      +
      \tfrac{
        \rho 
        \|x_2\|^2
        +
        \rho \|y_2\|^2
      }{
        4 r
      }
 \right) 
  .
\end{split}     \end{equation}
This and Theorem~\ref{thm:strong_completeness_uniform} imply that
if
\begin{equation}
  \exists \, c \in [0,\infty)
  \colon
\qquad
  \sup\nolimits_{ x, y \in \R^m }
  \left[
  \tfrac{ 
    \|
      ( \nabla U )( x )
      -
      ( \nabla U )( y )
    \|
  }{
    2 
    \left\|
      x - y
    \right\|^2
  }
  -
    c
      U( x )
      -
    c
      U( y )
  \right]
  < \infty
  ,
\end{equation}
then the SDE~\eqref{eq:Langevin_dynamics} is
strongly complete.
Strong completeness
for the SDE~\eqref{eq:Langevin_dynamics}
follows \sgc{}also\cgs{} from 
\sgc{}Theorem 2.4
in Zhang~\cite{Zhang2010} and  inequality~\eqref{eq:Langevin_dynamics_expmom}\cgs{}.
Let us point out that
even in the case of SDEs with additive noise
such as~\eqref{eq:Langevin_dynamics}
strong completeness is not clear in general;
see Subsection~\ref{sec:strong_completeness_additive}
above for details.

\section{Brownian dynamics (Over-damped Langevin dynamics)}
\label{sec:overdamped_Langevin}
\label{ssec:overdamped.Langevin.dynamics}

\emph{Brownian dynamics} is a simplified version of Langevin dynamics
in the limit of no average acceleration, and models the positions
of molecules in a potential
(see, for instance,
Section 2.1 in Beskos \&\ Stuart~\cite{BeskosStuart2009}).

Assume the setting of Section~\ref{sec:setting_SDE}\sgc{}, assume $d=m$, assume $D=\R^d$,\cgs{} and let $ \varepsilon \in (0,\infty) $,
$ \eta_0 \in [0,\infty) $, $\eta_1\in\R$,
$ \eta_2 \in [ 0, \frac{ 2 }{ \varepsilon } ] $,
$ U \in C^2( \R^d, [0,\infty) ) $ 
satisfy
\begin{equation}
\label{eq:eta0eta1eta2}
  \forall \, x \in \R^d
  \colon
\qquad
  ( \Delta U)( x )
\leq
  \eta_0
  +
  2 \eta_1
  U(x)
  +
  \eta_2
  \left\|
    ( \nabla U )( x )
  \right\|^2
  ,
\end{equation}
and let
$ X^x \colon [0,\infty) \times \Omega \to \R^d $,
$ x \in \R^d $,
be adapted stochastic processes with continuous sample paths
satisfying 
\begin{equation}
\label{eq:Langevin}
  X^x_t
  =
  x - \int_0^t ( \nabla U )( X^x_s ) \, ds
  +
  \sqrt{ \varepsilon } W_t
\end{equation}
$ \P $-a.s.\ for all $ (t,x) \in [0,\infty) \times \R^d $.
If $ \rho \in [0,\infty) $
and $ U_1,\, \overline{U} \colon \R^d \to \R $
satisfy for every $x\in\R^d$ that
\begin{equation}
  U_1(x)
  =
  \rho \, U(x)
  \quad\text{and}\quad
  \overline{U}(x)
  =
  \rho 
  \left(
    1 
    - 
    \tfrac{ \varepsilon }{ 2 }
    ( \eta_2 + \rho ) 
  \right)
  \left\| ( \nabla U )( x ) \right\|^2
  ,
\end{equation}
then it holds for every
$ x \in \R^d $ that
\begin{equation}
\begin{split}
\label{eq:estimate.Brownian.dynamics}
&
  ( \mathcal{G}_{ - \nabla U, \sqrt{ \varepsilon } I } U_1 )( x )
  +
  \tfrac{ 1 }{ 2 }
  \left\| 
    \sqrt{ \varepsilon } ( \nabla U_1 )( x )
  \right\|^2
  +
  \overline{U}(x)
\\ & =
  - \rho
  \left\| ( \nabla U )( x ) \right\|^2
  +
  \tfrac{ \varepsilon \rho }{ 2 }
  \operatorname{tr}(
    ( \operatorname{Hess} U)( x )
  )
  +
  \tfrac{
    \varepsilon \rho^2
  }{
    2
  }
  \left\| ( \nabla U )( x ) \right\|^2
  +
  \overline{U}( x )
\\ & \leq
  \tfrac{\eps\rho 2\eta_1}{2}U(x)
  +
  \rho
  \left[
    \tfrac{
      \varepsilon ( \eta_2 + \rho )
    }{
      2
    }
    -
    1
  \right]
  \left\| ( \nabla U )( x ) \right\|^2
  +
  \overline{U}( x )
  +
  \tfrac{ \varepsilon \rho \eta_0 }{ 2 }
\\ & 
=
  \tfrac{\eps\rho\eta_0}{2}
  +
  \eps\eta_1U_1(x)
  \,
  .
\end{split}
\end{equation}
Hence, Corollary~\ref{cor:exp_mom} implies
for every $ t \in [0,\infty) $,
$ \rho \in [ 0, \frac{ 2 }{ \varepsilon } - \eta_2 ] $\sgc{}, \cgs{}$ x \in \R^d $ that
\begin{equation}
  \E\!\left[
    \exp\!\left(
      \tfrac{\rho \, U( X^x_t )}{e^{\eps\eta_1t}}
      +
      \int_0^t
      \tfrac{ 
        \rho ( 1 - \frac{ \varepsilon }{ 2 } ( \eta_2 + \rho ) )
      }{
        e^{ \eps \eta_1 s }
      }
      \left\|
        ( \nabla U )( X^x_s ) 
      \right\|^2
      -\tfrac{\eps\rho\eta_0
      }{2e^{\eps\eta_1s}}
      ds
    \right)
  \right]
\leq
  e^{ \rho U( x ) }
  .
\end{equation}
This exponential moment estimate generalizes 
Lemma 2.5 in the work of
Bou-Rabee \&\ Hairer~\cite{BouRabeeHairer2013}
in the case where the function
$ \Theta $ appearing in Lemma 2.5 in~\cite{BouRabeeHairer2013}
satisfies
$
  \Theta(u) = \exp( \rho u )
$ 
for all $ u \in \R $ and some $ \rho \in [0, \tfrac{ 2 }{ \eps } - \eta_2 ] $.
If 
$ T \in ( 0, \infty ) $,
$ c \in \R $, $ \rho_0, \rho_1 \in [ 0, \frac{ 2 }{ \varepsilon } - \eta_2 ] $,
$ r \in (0,\infty] $,
$ q_0, q_1 \in [r,\infty] $
satisfy 
for every $ x, y \in \R^d $ that
\sgc{}$ \frac{ 1 }{ q_0 } + \frac{ 1 }{ q_1 } = \frac{ 1 }{ r } $ and\cgs{}
\begin{equation}
\begin{split}
\label{eq:condition.Brownian.dynamics}
&
  \tfrac{ 
    \left\langle x-y,
      ( \nabla U )( y )
      -
      ( \nabla U )( x )
    \right\rangle
  }{
    \left\|
      x - y
    \right\|^2
  }
\\ &
\leq
  c
  +
  \tfrac{
    \rho_0
  }{
    2 q_0 T e^{\eps\eta_1 T}
  }
    \left[
      U(x) + U(y)
    \right]
  +
  \tfrac{
    \rho_1
    (
      1 - 
      \frac{ \varepsilon }{ 2 }
      ( \eta_2 + \rho_1 )
    )
  }{
    2 q_1 e^{\eps\eta_1 T}
  }
    \left[
      \| ( \nabla U )( x ) \|^2
      +
      \| ( \nabla U )( y ) \|^2
    \right]\sgc{},\cgs{}
\end{split}
\end{equation}
then Corollary~\ref{cor:UV2} \sgc{}with \cgs{}$U_{0,0}\equiv U_{0,1}\equiv\overline{U}_{0,1}\equiv 0$, 
$U_{1,0}=\rho_0 U$,
$U_{1,1}=\rho_1 U_1$ and $c_0\equiv 0$
together with inequalities~\eqref{eq:estimate.Brownian.dynamics}
and~\eqref{eq:condition.Brownian.dynamics}
shows 
for every $ x, y \in \R^d $ that
\begin{equation}
\label{eq:Langevin_Estimate}
\begin{split}
&
  \left\|
    \sup\nolimits_{ t \in [0,T] }
      \|
        X^x_t - X^y_t
      \|
  \right\|_{
    L^r( \Omega; \R )
  }
\\ &
\leq
  \left\| x - y \right\|
  \exp\!\left(
    c T 
    +
    \big( 
      \tfrac{ \rho_0 }{ q_0 } + \tfrac{ \rho_1 }{ q_1 } 
    \big) 
    \tfrac{ 
      \varepsilon 
      \eta_0 T 
    }{ 2 }
    +
    \tfrac{
      \rho_0
      \left(
        U(x) + U(y)
      \right)
    }{
      2 q_0
    }
    +
    \tfrac{
      \rho_1
      \left(
        U(x) + U(y)
      \right)
    }{
      2 q_1
    }
  \right)
  .
\end{split}
\end{equation}
Hence,
if there exist 
$ 
  \overline{\rho}_0 \in (0,\infty) 
$ 
and 
$
  \overline{\rho}_1
  \in
  \big( 
    0,
    \tfrac{ 
      ( 1 - \frac{ \varepsilon }{ 2 } \eta_2 )^2
    }{ 
      4 \varepsilon d
    }
  \big)
$ 
such that
\begin{equation}  \begin{split}
\label{eq:Langevin_strongcomplete}
&  \sup_{ x , y \in \R^d }\!
  \Big[
  \tfrac{ 
    \left\langle x-y,
      ( \nabla U )( y )
      -
      ( \nabla U )( x )
    \right\rangle
  }{
    \left\|
      x - y
    \right\|^2
  }
  -\overline{\rho}_0\left[U(x)+U(y)\right]
  -
    \overline{\rho}_1 
    \left[
      \| ( \nabla U )( x ) \|^2
  +
      \| ( \nabla U )( y ) \|^2
      \right]\!
  \Big]\!\!
\\ &
  < \infty
  ,
\end{split}     \end{equation}
then there exist 
$ T,c \in ( 0, \infty ) $,
$ \rho_0, \rho_1 \in [ 0, \frac{ 2 }{ \varepsilon } - \eta_2 ] $,
$ r \in (d,\infty) $,
$ q_0, q_1 \in (r,\infty) $
with $ \frac{ 1 }{ q_0 } + \frac{ 1 }{ q_1 } = \frac{ 1 }{ r } $
such that inequality~\eqref{eq:Langevin_Estimate} holds
and then
Theorem~\ref{thm:strong_completeness_uniform}
proves that
the SDE~\eqref{eq:Langevin} is strongly complete.
Strong completeness
for the SDE~\eqref{eq:Langevin}
\sgc{}in the case where~\eqref{eq:eta0eta1eta2} and~\eqref{eq:Langevin_strongcomplete} hold with $\eta_2=0$
and $\overline{\rho}_1=0$\cgs{}
follows \sgc{}also\cgs{} from
Theorem 2.4
in Zhang~\cite{Zhang2010}.

\section{Stochastic SIR model}
\label{ssec:stochastic.SIR.model}

The SIR model from epidemiology for the total number
of \underline{\textbf{s}}usceptible, \underline{\textbf{i}}nfected and 
\underline{\textbf{r}}ecovered individuals
has been introduced by Anderson \&\ May~\cite{AndersonMay1979}. 
This section establishes strong stability estimates for the 
stochastic SIR model 
studied in Tornatore, Buccellato \&\ Vetro~\cite{TornatoreBuccellatoVetro2005}.

Assume the setting of Section~\ref{sec:setting_SDE}\sgc{}, assume $d=3$, assume $m=1$, assume $D=[0,\infty)^3$,\cgs{}
let
$ 
  \alpha,\, \beta,\, \gamma,\, 
  \delta \in (0,\infty) 
$,
let 
$ \mu \colon [ 0, \infty )^3 \to \R^3 $
and
$
  \sigma \colon [ 0, \infty )^3 \to \R^3 
$
be given by
\begin{equation}
  \mu\!\left( 
    \begin{array}{c}
	x_1 \\
	x_2 \\
	x_3
    \end{array}
  \right)
=
  \left( 
    \begin{array}{c}
	- \alpha x_1 x_2
	- \delta x_1
	+ \delta
	\\ 
	\alpha x_1 x_2
	- 
	( \gamma + \delta ) x_2 
	\\ 
	\gamma x_2 - \delta x_3
    \end{array}
  \right) ,
\qquad
  \sigma\!\left( 
    \begin{array}{c}
	x_1
	\\ 
	x_2 
	\\ 
	x_3
    \end{array}
  \right)
=
  \left(
    \begin{array}{ccc}
	- \beta x_1 x_2 
	\\ 
	\phantom{-}\beta x_1 x_2
	\\ 
	0
    \end{array}
  \right)
\end{equation}
for all
$ 
  x = (x_1, x_2, x_3)
  \in [ 0, \infty )^3
$
and let
$
  X^x = 
  ( X^{ x, 1 }, X^{ x, 2 }, X^{ x, 3 } )
  \colon [0,\infty) \times \Omega \to [0,\infty)^3
$,
$ x \in [ 0, \infty )^3 $,
be adapted stochastic processes with continuous sample paths
satisfying~\eqref{eq:SDE_examples}
$ \P $-a.s.\ for all
$ (t,x) \in [0,\infty) \times [ 0, \infty )^3 $, i.e.,
\begin{equation}
\label{eq:SIR}
  X^x_t 
 = 
  x
  +
  \int_0^t
    \left(
      \begin{array}{c}
       {\scriptstyle
	-\alpha X_s^{x,1}X_s^{x,2} - \delta X_s^{x,1} + \delta }\\ 
       {\scriptstyle
	\alpha X_s^{x,1}X_s^{x,2} - (\gamma+\delta) X_s^{x,2} }\\
       {\scriptstyle
	\gamma X_s^{x,2} - \delta X_s^{x,3}    
       }
      \end{array}
    \right)
  \, ds
  + 
  \int_0^t
    \left(
      \begin{array}{c}
       {\scriptstyle
       - \beta X_s^{x,1}X_s^{x,2}  }\\ 
       {\scriptstyle
       \phantom{-}\beta X_s^{x,1}X_s^{x,2}  }\\ 
       {\scriptstyle 
       0
       }
      \end{array}
    \right) 
  \, dW_t
\end{equation}
$\P$-a.s.\ for all $(t,x)\in [0,\infty)\times [0,\infty)^3$.

For the stochastic SIR model 
it is well known that 
the sum of the first two coordinates serves as a Lyapunov-type function
(cf., e.g., Tornatore, Buccellato \&\ Vetro~\cite{TornatoreBuccellatoVetro2005}).
We use this to construct an exponential Lyapunov-type function
in the sense of Corollary~\ref{cor:exp_mom}.
More formally, 
if $ \rho, \kappa \in [0,\infty) $
and if
$ U \colon \R^3 \to \R $
is given by
$
  U( x )
  =
  \rho
  \left( 
    x_1 + x_2 
    -
    \kappa
  \right)
$
for all 
$ x = ( x_1, x_2, x_3 ) \in \R^3 $,
then it holds 
for every
$ x = ( x_1, x_2, x_3 ) \in [0,\infty)^3 $, $t\in[0,\infty)$
that
\begin{equation}
\label{eq:SIR_exp_mom_cond0}
\begin{split}
&
  ( \mathcal{G}_{ \mu, \sigma } U)( x )
  +
  \tfrac{ 1 }{ 2 e^{ - \delta t } }
  \|
    \sigma( x )^*
    ( \nabla U )(x)
  \|^2
\\ & =
  \rho
  \left[
    - \alpha x_1 x_2 - \delta x_1 + \delta 
    + \alpha x_1 x_2 - ( \gamma + \delta ) x_2
  \right]
\\ & =
  \rho
  \left[
    - \delta x_1 + \delta 
    - \delta x_2
  \right]
  - \rho \gamma x_2
=
  - \delta \rho
  \left[
    x_1 + x_2 - 1
  \right]
  - \rho \gamma x_2
\\ & 
=
  - \delta U(x)
  - \delta \rho \left[ \kappa - 1 \right]
  - \rho \gamma x_2
\leq
  - \delta U(x)
  - \delta \rho \left[ \kappa - 1 \right]
\leq 
  \delta \rho
  .
\end{split}
\end{equation}
Corollary~\ref{cor:exp_mom}
hence implies 
for every
$ x = ( x_1, x_2, x_3 ) \in [ 0, \infty )^3 $,
$
  \rho, t \in [ 0, \infty )
$
that
\begin{equation}
\begin{split}
  \E\!\left[
    \exp\!\left(
      \rho
      \,
      e^{ \delta t }
      (
        X^{ x, 1 }_t
        +
        X^{ x, 2 }_t
        -
        1
      )
    \right)
  \right]
& \leq
  e^{
      \rho
      (
        x_1 + x_2
        -
        1
      )
  }
  ,
\\ 
\quad
  \E\!\left[
    \exp\!\left(
      \rho
      \,
      (
        X^{ x, 1 }_t
        +
        X^{ x, 2 }_t
      )
      +
      \delta \rho 
      \int_0^t
      (
        X^{ x, 1 }_s
        +
        X^{ x, 2 }_s
      )
      \, ds
    \right)
  \right]
& \leq
  e^{
    \rho
    (
      \delta t 
      +
      x_1 + x_2
    )
  }
  .
\end{split}
\end{equation}
Moreover, if $ \rho \in [0,\infty) $
and if
$ U \colon \R^3 \to \R $
is given by
$
  U( x )
  =
  \rho
  \left( x_1 + x_2 - 1 \right)^2
$
for all 
$ x = ( x_1, x_2, x_3 ) \in \R^3 $,
then it holds 
for every
$ x = ( x_1, x_2, x_3 ) \in [0,\infty)^3 $, $t\in[0,\infty)$
that
\begin{equation}
\label{eq:SIR_exp_mom_cond}
\begin{split}
&
  ( \mathcal{G}_{ \mu, \sigma } U)( x )
  +
  \tfrac{ 1 }{ 2e^{ -  2 \delta t } }
  \|
    \sigma( x )^*
    ( \nabla U )(x)
  \|^2
\\ & =
  2 \rho \left( x_1 + x_2 - 1 \right)
  \left[
    - \alpha x_1 x_2 - \delta x_1 + \delta 
    + \alpha x_1 x_2 - ( \gamma + \delta ) x_2
  \right]
\\ & =
  2 \rho  
  \left( x_1 + x_2 - 1 \right)
  \left[ - \delta x_1 + \delta - \delta x_2 \right]
  -
  2 \rho \gamma \left( x_1 + x_2 - 1 \right) x_2
\\ & =
  -
  2 \delta U(x)
  -
  2 \rho \gamma \left( x_1 + x_2 - 1 \right) x_2
\leq
  -
  2 \delta U(x)
  -
  2 \rho \gamma \left( x_2 - 1 \right) x_2
\\ & =
  -
  2 \delta U(x)
  -
  2 \rho \gamma \left( ( x_2 )^2 - x_2 + \tfrac{ 1 }{ 4 } \right) 
  +
  \tfrac{ \rho \gamma }{ 2 }
=
  -
  2 \delta U(x)
  -
  2 \rho \gamma \left( x_2 - \tfrac{ 1 }{ 2 } \right)^2
  +
  \tfrac{ \rho \gamma }{ 2 }
\\ & \leq
  \tfrac{ \rho \gamma }{ 2 }
  -
  2 \delta U(x)
\leq
  \tfrac{ \rho \gamma }{ 2 }
  \,
  .
\end{split}
\end{equation}
Corollary~\ref{cor:exp_mom}
hence ensures 
for every $ x = ( x_1, x_2, x_3) \in [0,\infty)^3 $,
$ \rho, t \in [0,\infty) $
that
\begin{equation}
  \E\!\left[
    \exp\!\left(
      \rho
      \,
      e^{ 
        2 \delta t
      }
      \left[
        X^{ x, 1 }_t
        +
        X^{ x, 2 }_t
        - 1
      \right]^2
    \right)
  \right]
\leq
  \exp\!\left( 
    \tfrac{ \gamma }{ 4 }
    \left[ 
      e^{ 2 \delta t } - 1
    \right]
    +
    \rho
    \left[
      x_1
      +
      x_2 
      - 1
    \right]^2
  \right)
\end{equation}
and
$
  \E\big[
    e^{
      \rho
      [
        X^{x,1}_t
        +
        X^{x,2}_t
        - 1
      ]^2
      +
      2 \rho \delta
      \int_0^t
      [
        X^{x,1}_s
        +
        X^{x,2}_s
        - 1
      ]^2
      \, ds
    }
  \big]
\leq
  e^{
    \frac{ \rho \gamma t }{ 2 }
    +
    \rho
    [
      x_1
      +
      x_2 
      - 1
    ]^2
  }
$.
In the next step note 
for every 
$ x = (x_1, x_2, x_3) , y = (y_1, y_2, y_3) \in [ 0, \infty )^3 $
with $ x \neq y $
that
\begin{align}\nonumber
&
    \frac{
      \| 
        (
          \sigma( x ) - \sigma( y )
        )^*
        ( x - y )
      \|^2
    }{
      \| x - y \|^4
    }
  \leq
    \frac{
      \| 
        \sigma( x ) - \sigma( y )
      \|^2
    }{
      \| x - y \|^2
    }
  =
    \frac{
      \beta^2
      \left(
        x_1 x_2
        -
        y_1 y_2
      \right)^2
    }{
      \| x - y \|^2
    }
\\ & =\label{eq:gagaga}
  \frac{
    \beta^2
    \left[
      ( x_1 - y_1 ) ( x_2 + y_2 )
      +
      ( x_1 + y_1 ) ( x_2 - y_2 )
    \right]^2
  }{
    4 \,
    \|
      x - y
    \|^2
  }
\\ \nonumber & \leq
  \tfrac{ \beta^2 }{ 4 }
  \max( [ x_1 + y_1 ]^2 , [ x_2 + y_2 ]^2 )
  +
  \tfrac{ \beta^2 }{ 8 }
  [ x_1 + y_1 ] [ x_2 + y_2 ]
\leq
  \tfrac{ \beta^2 }{ 4 }
  \left[ x_1 + y_1 + x_2 + y_2 \right]^2
\\ \nonumber & \leq
  \beta^2
  \left[ x_1 + x_2 - 1 \right]^2 
  + 
  \beta^2
  \left[ y_1 + y_2 - 1 \right]^2 
  +
  2 \beta^2
\end{align}
\label{eq:SIR_diffusion}
and that
\begin{equation}
\label{eq:SIR_drift}
\begin{split}
&
  \frac{
    \alpha
    \left[
      ( x_2 - y_2 )
      -
      ( x_1 - y_1 )
    \right]
    \left(
      x_1 x_2 - y_1 y_2
    \right)
  }{
    \|
      x - y
    \|^2
  }
\\ & =
  \frac{
    \alpha
    \left( 
      x_2 - y_2 
      -
      ( 
        x_1
        - y_1
      )
    \right)
    \left[
      ( x_1 - y_1 ) ( x_2 + y_2 )
      +
      ( x_1 + y_1 ) ( x_2 - y_2 )
    \right]
  }{
    2 \,
    \|
      x - y
    \|^2
  }
\\ & \leq
  \tfrac{ 3 \alpha }{ 4 }
  \left( x_1 + y_1 \right)
  +
  \tfrac{ 3 \alpha }{ 4 }
  \left( x_2 + y_2 \right)
\leq
  \tfrac{ 3 \alpha }{ 2 }
  +
  \tfrac{ 3 \alpha }{ 4 }
  \left( x_1 + x_2 - 1 \right)
  +
  \tfrac{ 3 \alpha }{ 4 }
  \left( y_1 + y_2 - 1 \right)
\end{split}
\end{equation}
This implies
for every 
$ \theta \in [0,\infty) $,
$ x = (x_1, x_2, x_3) , y = (y_1, y_2, y_3) \in [ 0, \infty )^3 $
with $ x \neq y $
that
\begin{align}
\label{eq:SIR_UV2_cond2}
\nonumber 
&
  \tfrac{ 
    \langle
      x - y ,
      \mu( x ) - 
      \mu( y )
    \rangle
    +
    \frac{ 1 }{ 2 }
    \| 
      \sigma( x ) - \sigma( y ) 
    \|^2_{ \HS( \R, \R^3 ) }
  }{
    \|
      x - y
    \|^2
  }
    +
    \tfrac{
      ( \frac{ \theta }{ 2 } - 1 ) \,
      \| 
        (
          \sigma( x ) - \sigma( y )
        )^*
        ( x - y )
      \|^2
    }{
      \| x - y \|^4
    }
\\ \nonumber & \leq
  \tfrac{
    ( \sqrt{ 2 } - 1 ) \, \gamma
  }{ 2 }
  - \delta
  +
  \tfrac{ 3 \alpha }{ 4 }
  \left( x_1 + x_2 \right)
  +
  \tfrac{ 3 \alpha }{ 4 }
  \left( y_1 + y_2 \right)
  +
  \tfrac{ 
    [
      1
      +
      | \theta - 2 | 
    ]
    \,
    \| 
      \sigma( x ) - \sigma( y ) 
    \|^2
  }{
    2 \,
    \|
      x - y
    \|^2
  }
\\ & \leq
  \gamma
  +
  \tfrac{ 3 \alpha }{ 4 }
  \left( x_1 + x_2 \right)
  +
  \tfrac{ 3 \alpha }{ 4 }
  \left( y_1 + y_2 \right)
\\ \nonumber & \quad
  +
  \beta^2
  \big[
    \tfrac{ 1 }{ 2 }
    +
    | \tfrac{ \theta }{ 2 } - 1 | 
  \big]
  \left(
    \left[ x_1 + x_2 - 1 \right]^2 
    + 
    \left[ y_1 + y_2 - 1 \right]^2 
    +
    2 
  \right)
  .
\end{align}
Combining 
\eqref{eq:SIR_exp_mom_cond0},
\eqref{eq:SIR_exp_mom_cond}, \eqref{eq:gagaga},~\eqref{eq:SIR_diffusion}
and~\eqref{eq:SIR_UV2_cond2} with
Corollary~\ref{cor:UV2}
then shows 
for every
$ T \in [ 0, \infty) $,
$ r, p, q_{ 1, 0, 1 }, q_{ 1, 0, 2 } \in [2,\infty) $,
$ \eta \in [p,\infty] $,
$ \theta \in (0,p) $
with
$
  \frac{ 1 }{ p }
  +
  \frac{ 1 }{ q_{ 1, 0, 1 } }
  +
  \frac{ 1 }{ q_{ 1, 0, 2 } }
  =
  \frac{ 1 }{ r } 
$
and every
$ x = ( x_1, x_2, x_3 ), y = ( y_1, y_2, y_3 ) \in [ 0, \infty)^3 $
that
\begin{equation}
\begin{split}
&
  \left\|
    \sup\nolimits_{ t \in [0,T] }
      \|
        X^x_t - X^y_t
      \|
  \right\|_{
    L^r( \Omega; \R )
  }
\\ & \leq
  \frac{ 
    \| x - y \|
  }{
    [ 
      1 - 
      \tfrac{ \theta }{ p }
    ]^{ \frac{1}{\theta} }
  }
  \cdot
    \exp\!\left(
      \left(
        2 \beta^2 
        \left[ 
          \tfrac{ 1 }{ 
            ( 2 / p - 2 / \eta )
          }
          - 
          \tfrac{ \theta }{ 2 } 
        \right]
        +
        \gamma 
        +\beta^2\big[1+|\theta-2|\big]
      \right)
      T
  \right)
\\ & \quad
  \cdot
  \exp\!\left(
      \beta^2 
      \gamma
      T^2
      \left[ 
        \tfrac{ 1 }{ 
          ( 2 / p - 2 / \eta )
        }
        - 
        \tfrac{ \theta }{ 2 } 
      \right]
      +
      \tfrac{ 3 \alpha  T }{ 2 }
  \right)
\\ & \quad
  \cdot
  \exp\!\left(
      \beta^2 T
      \left[ 
        \tfrac{ 1 }{ 
          ( 2 / p - 2 / \eta )
        }
        - 
        \tfrac{ \theta }{ 2 } 
      \right]
      \big[
        ( x_1 + x_2 - 1 )^2 
        +
        ( y_1 + y_2 - 1 )^2
      \big]
  \right)
\\ & \quad 
  \cdot
  \exp\!\left(
      \tfrac{ 3 \alpha  }{ 4\delta } ( x_1 + x_2 + y_1 + y_2 )
      +
      \beta^2 \gamma T^2
      \left[
        \tfrac{ 1 }{ 2 }
        +
        |
          \tfrac{ \theta }{ 2 } - 1
        |
      \right]
  \right)
\\ & \quad 
  \cdot
  \exp\!\left(
      \beta^2 T
      \left[
        \tfrac{ 1 }{ 2 }
        +
        |
          \tfrac{ \theta }{ 2 } - 1
        |
      \right]
      \left[
        ( x_1 + x_2 - 1 )^2 
        +
        ( y_1 + y_2 - 1 )^2
      \right]
    \right)
  .
\end{split}
\end{equation}
Combining this with Theorem~\ref{thm:strong_completeness_uniform}
proves that the
stochastic SIR model~\eqref{eq:SIR} is 
strongly complete.

\section{Experimental psychology model}
\label{ssec:experimental.psychology}

Assume the setting of Section~\ref{sec:setting_SDE}\sgc{}, assume $d=2$, assume $m=1$, assume $D=\R^2$,\cgs{}
let
$ \alpha,\, \delta \in (0,\infty) $,
$ \beta \in \mathbb{R} $,
let 
$ \mu \colon \R^2 \to \R^2 $
and
$
  \sigma \colon \R^2 \to \R^2
$
be given by
\begin{equation}
  \mu\!\left( 
    \begin{array}{c}
      x_1
      \\ 
      x_2 
    \end{array}
  \right)
=
  \left( 
    \begin{array}{c}
      \left( x_2 \right)^2
      \left( \delta + 4 \alpha x_1 \right)
      - \frac{ \beta^2 x_1 }{ 2 }
      \\[1ex] 
      - x_1 x_2 
      \left( \delta + 4 \alpha x_1 \right)
      - \frac{ \beta^2 x_2 }{ 2 }
    \end{array}
  \right) ,
\qquad
  \sigma\!\left( 
    \begin{array}{c}
      x_1
      \\ 
      x_2 
    \end{array}
  \right)
=
  \left(
    \begin{array}{ccc}
      - \beta x_2
    \\
      \phantom{-}\beta x_1
    \end{array}
  \right)
\end{equation}
for all
$ 
  x = (x_1, x_2)
  \in \mathbb{R}^2
$
and let
$ 
  X^x = (X^{x,1}, X^{x,2})\colon 
  [0,\infty) \times \Omega
  \to \R^2
$,
$ x \in \R^2 $,
be adapted stochastic processes with continuous 
sample paths satisfying~\eqref{eq:SDE_examples}
$ \P $-a.s.\ for all  \sgc{}$t \in [0,\infty)$, $x \in \R^2 $\cgs{}, i.e.,
\begin{equation}
\label{eq:Experimental}
\begin{split}
  \left(
    \begin{array}{c}
      \scriptstyle{
      X^{x,1}_t
      }
      \\
      \scriptstyle{
      X^{x,2}_t
      }
    \end{array}
  \right)
 &=
  x
  +
  \int_0^t
    \left(
      \begin{array}{c}
	{\scriptstyle
	\big( X^{x,2}_s \big)^2
	\big( 
	  \delta 
	  +
	  4 \alpha  X^{x,1}_s 
	\big)
	-
	\frac12 \beta^2 X^{x,1}_s 
	}
      \\[1ex]
	{\scriptstyle
	- X^{x,1}_t X^{x,2}_s 
	\big( 
	  \delta 
	  +
	  4 \alpha  X^{x,1}_s 
	\big)
	-
	\frac12 \beta^2 X^{x,2}_s 
	}
    \end{array}
  \right)
  \, ds
  +
  \int_0^t
      \left(
	\begin{array}{c}
	  {\scriptstyle
	  - \beta X^{x,2}_s 
	  }
	  \\
	  {\scriptstyle
	  \phantom{-}\beta X^{x,1}_s
	  }
	\end{array}
      \right) 
  \, dW_s
\end{split}
\end{equation}
$\P$-a.s.\ for all \sgc{}$t \in [0,\infty)$, $x \in \R^2 $\cgs{}.
The SDE~\eqref{eq:Experimental} 
is a suitable transformed version of a model
proposed in Haken, Kelso \&\ Bunz~\cite{Hakenetal1985} in the deterministic case 
and 
in Sch\"{o}ner, Haken \&\ Kelso~\cite{Schoeneretal1985} in the stochastic case
(see Section~7.2 in Kloeden
\&\ Platen~\cite{kp92} for details).

If 
$ p \in [1,\infty) $,
$ \rho \in (0,\infty) $
and if
$ U \colon \R^2 \to \R $
satisfies
$
  U( x )
  =
  \rho \,
  \| x \|^{ 2 p }
$
for all 
$
  x \in \R^2
$,
then it holds for every
$ x \in \R^2 $ that
\begin{equation}
\begin{split}
&
  ( \mathcal{G}_{ \mu, \sigma } U )( x )
  +
  \tfrac{ 1 }{ 2 }
  \| 
    \sigma( x )^* \,
    ( \nabla U )( x )
  \|^2
=
  2 p \rho
  \left\| x \right\|^{ ( 2 p - 2 ) }
  \left[
    \left< x, \mu( x ) \right>
    +
    \tfrac{ 1 }{ 2 }
    \| \sigma(x) \|^2
  \right]
  = 0
  .
\end{split}
\label{eq:Experimental_Exp}
\end{equation}
In addition, we get from It\^{o}'s formula 
that 
for every $ x \in \R^2 $,
$ t \in [0,\infty) $ it holds \sgc{}$\P $-a.s.\ \cgs{}that
$
  \| X^x_t \| = \| x \|
$.
Next note 
for every 
$ \theta \in (0,\infty) $,
$ x, y \in \R^2 $
with $ x \neq y $
that
\begin{equation}
\begin{split}
&
  \tfrac{ 
    \langle
      x - y ,
      \mu( x ) - 
      \mu( y )
    \rangle
    +
    \frac{ 1 }{ 2 }
    \| 
      \sigma( x ) - \sigma( y ) 
    \|^2_{ \HS( \R, \R^2 ) }
  }{
    \|
      x - y
    \|^2
  }
    +
    \tfrac{
      ( \frac{ \theta }{ 2 } - 1 ) \,
      \| 
        (
          \sigma( x ) - \sigma( y )
        )^*
        ( x - y )
      \|^2
    }{
      \| x - y \|^4
    }
\\ & =
  \frac{ 
    \langle
      x - y ,
      \mu( x ) - 
      \mu( y )
    \rangle
    +
    \frac{ 1 }{ 2 }
    \| 
      \sigma( x ) - \sigma( y ) 
    \|^2	
  }{
    \|
      x - y
    \|^2
  }
\\ & =
  \frac{ 
    \left( x_1 - y_1 \right)
    \left[
      ( x_2 )^2
      ( \delta + 4 \alpha x_1 )
      - 
      \tfrac{ \beta^2 x_1 }{ 2 }
      -
      ( y_2 )^2
      ( \delta + 4 \alpha y_1 )
      +
      \tfrac{ \beta^2 y_1 }{ 2 }
    \right]
  }{
    \| x - y \|^2
  }
\\ &
  +
  \frac{
    \left( x_2 - y_2 \right)
    \left[
      - x_1 x_2
      ( \delta + 4 \alpha x_1 )
      - 
      \tfrac{ \beta^2 x_2 }{ 2 }
      + y_1 y_2
      ( \delta + 4 \alpha y_1 )
      +
      \tfrac{ \beta^2 y_2 }{ 2 }
    \right]
  }{
    \|
      x - y
    \|^2
  }
\\ & +
  \frac{
    \frac{ 1 }{ 2 }
    \beta^2
    \left[
      (
        x_1 - y_1
      )^2
      +
      (
        x_2 - y_2
      )^2
    \right]
  }{
    \|
      x - y
    \|^2
  }
  .
\end{split}
\end{equation}
This implies 
for every 
$ \theta, \varepsilon \in (0,\infty) $,
$ x, y \in \R^2 $
with $ x \neq y $
that
\begin{align} 
\nonumber 
\label{eq:Experimental_Nonl}
&
  \tfrac{ 
    \langle
      x - y ,
      \mu( x ) - 
      \mu( y )
    \rangle
    +
    \frac{ 1 }{ 2 }
    \| 
      \sigma( x ) - \sigma( y ) 
    \|^2_{ \HS( \R, \R^2 ) }
  }{
    \|
      x - y
    \|^2
  }
    +
    \tfrac{
      ( \frac{ \theta }{ 2 } - 1 ) \,
      \| 
        (
          \sigma( x ) - \sigma( y )
        )^*
        ( x - y )
      \|^2
    }{
      \| x - y \|^4
    }
\\ \nonumber & =
  \tfrac{ 
    \left( x_1 - y_1 \right)
    \left[
      ( x_2 )^2
      ( \delta + 4 \alpha x_1 )
      -
      ( y_2 )^2
      ( \delta + 4 \alpha y_1 )
    \right]
  }{
    \| x - y \|^2
  }
  -
  \tfrac{
    \left( x_2 - y_2 \right)
    \left[
      x_1 x_2
      ( \delta + 4 \alpha x_1 )
      - y_1 y_2
      ( \delta + 4 \alpha y_1 )
    \right]
  }{
    \|
      x - y
    \|^2
  }
\\ \nonumber & \leq
  \tfrac{
    \delta
    \left( x_1 - y_1 \right)
    \left( x_2 - y_2 \right)
    \left( x_2 + y_2 \right)
  }{
    \| x - y \|^2
  }
  -
  \tfrac{
    \delta
    \left( x_2 - y_2 \right)
    \left[
      ( x_2 - y_2 ) ( x_1 + y_1 )
      +
      ( x_1 - y_1 ) ( x_2 + y_2 )
    \right]
  }{
    2 \, \| x - y \|^2
  }
\\ \nonumber & \quad
  +
  \tfrac{ 
    4 \alpha 
    \left( x_1 - y_1 \right)
    \left[
      ( x_2 )^2
      x_1 
      -
      ( y_2 )^2
      y_1 
    \right]
  }{
    \| x - y \|^2
  }
  -
  \tfrac{
    4 \alpha
    \left( x_2 - y_2 \right)
    \left[
      ( x_1 )^2 x_2
      - ( y_1 )^2 y_2
    \right]
  }{
    \|
      x - y
    \|^2
  }
\\ & \leq
  \tfrac{
    \delta
    \left( x_1 - y_1 \right)
    \left( x_2 - y_2 \right)
    \left( x_2 + y_2 \right)
  }{
    2 \, \| x - y \|^2
  }
  -
  \tfrac{
    \delta
    \left( x_2 - y_2 \right)^2
    ( x_1 + y_1 )
  }{
    2 \, \| x - y \|^2
  }
\\ \nonumber & \quad
  +
  \tfrac{ 
    2 \alpha 
    \left( x_1 - y_1 \right)
    \left[
      ( x_2 - y_2 )
      ( x_2 + y_2 )
      ( x_1 + y_1 )
      +
      (
        ( x_2 )^2
        +
        ( y_2 )^2
      )
      ( x_1 - y_1 )
    \right]
  }{
    \| x - y \|^2
  }
\\ \nonumber & \quad
  -
  \tfrac{
    2 \alpha
    \left( x_2 - y_2 \right)
    \left[
      ( x_1 - y_1 ) ( x_1 + y_1 )
      ( x_2 + y_2 )
      +
      ( 
        ( y_1 )^2 
        +
        ( x_1 )^2
      )
      ( x_2 - y_2 ) 
    \right]
  }{
    \|
      x - y
    \|^2
  }
\\ \nonumber & \leq   
  \tfrac{ \delta }{ 4 }
  | x_2 + y_2 |
  +
  \tfrac{ \delta }{ 2 }
  | x_1 + y_1 |
  +
  \tfrac{ 
    2 \alpha 
    \left( x_1 - y_1 \right)^2
      (
        ( x_2 )^2
        +
        ( y_2 )^2
      )
    -
    2 \alpha 
    \left( x_2 - y_2 \right)^2
      ( 
        ( x_1 )^2
        +
        ( y_1 )^2 
      )
  }{
    \| x - y \|^2
  }
\\ \nonumber & \leq   
  \tfrac{ \delta }{ 4 }
  | x_2 + y_2 |
  +
  \tfrac{ \delta }{ 2 }
  | x_1 + y_1 |
  +
    2 \alpha 
      \left[
        ( x_2 )^2
        +
        ( y_2 )^2
      \right]
\\ \nonumber & \leq 
  \tfrac{ \delta^2 }{ 32 \varepsilon }
  +
  \tfrac{ \delta^2 }{ 8 ( 2 \alpha + \varepsilon ) }
  +
  [ 2 \alpha + \varepsilon ]
  \left[
    \| x \|^2
    +
    \| y \|^2
  \right]
  .
\end{align}
Combining~\eqref{eq:Experimental_Exp}
and~\eqref{eq:Experimental_Nonl}
with Corollary~\ref{cor:UV2}
proves 
for every 
$ \varepsilon, r, T \in ( 0, \infty ) $,
$ x, y \in \R^2 $
with $ x \neq y $
that
\begin{equation}
\begin{split}
&
  \bigg\|
    \sup_{ t \in [ 0, T ] }
    \left\| X^x_t - X^y_t 
    \right\|
  \bigg\|_{
    L^r( \Omega; \R )
  }
\\ &
\leq  
  \exp\!
  \left(
    \tfrac{ \delta^2 T }{ 32 \varepsilon }
    +
    \tfrac{ \delta^2 T }{ 8 ( 2 \alpha + \varepsilon ) }
    +
    \left[ 2 \alpha + \varepsilon \right] T
    \Big(
      \left\| x \right\|^2
      +
      \left\| y \right\|^2
    \Big)
  \right)
      \| x - y \|
  .
\end{split}
\end{equation}
Combining this with Theorem~\ref{thm:strong_completeness_uniform}
proves that the SDE~\eqref{eq:Experimental}
is strongly complete.
Strong completeness for the SDE~\eqref{eq:Experimental}
follows also from 
Theorem 2.4
in Zhang~\cite{Zhang2010}
\sgc{}together with\cgs{} 
the inequalities~\eqref{eq:Experimental_Exp}
and~\eqref{eq:Experimental_Nonl}.

\section{Stochastic Brusselator in the well-stirred case}
\label{ssec:Brusselator}

The Brusselator
is a model for a trimolecular chemical reaction
and has been studied in Prigogine \&\ Lefever~\cite{PrigogineLefever1968} 
and by other scientists from Brussels (cf.\ Tyson~\cite{Tyson1973}).
A stochastic version thereof
was proposed by Dawson~\cite{Dawson1981} 
(see also Scheutzow~\cite{Scheutzow1986}).

Assume the setting of Section~\ref{sec:setting_SDE}\sgc{}, assume $d=2$, assume $D=[0,\infty)^2$,\cgs{} and let
$ \alpha, \delta \in (0,\infty) $,
let
$ 
  \sigma 
  = ( \sigma_{ i, j } )_{    
    (i,j) \in \{ 1, 2 \} \times \{ 1, \dots, m \}
  }
  \sgc{}\colon\cgs{}
  [0,\infty)^2
  \rightarrow \R^{ 2 \times m }
$
be a globally Lipschitz continuous
function with 
$ 
  \forall \,
  y \in (0,\infty)
\colon
  \sigma( 0, y ) 
= 
  \sigma( y, 0 ) 
= 0 
$
(cf.\ the last sentence in Section~1 in Scheutzow~\cite{Scheutzow1986}),
let $ \mu = ( \mu_1, \mu_2 ) \colon [0,\infty)^2 \to \R^2 $
be given by
\begin{equation}
  \mu\!\left( 
    \begin{array}{c}
      x_1
      \\ 
      x_2 
    \end{array}
  \right)
=
  \left( 
    \begin{array}{c}
      \delta - 
      \left( \alpha + 1 \right) x_1
      +
      x_2  ( x_1 )^2
      \\
      \alpha x_1
      - x_2 
      ( x_1 )^2
    \end{array}
  \right) 
\end{equation}
for all
$ 
  x = (x_1, x_2)
  \in [0,\infty)^2
$.
Let 
$ X^x \colon [0,\infty) \times \Omega \to [0,\infty)^2 $,
$ x \in [0,\infty)^2 $,
be adapted stochastic processes with continuous sample paths
satisfying~\eqref{eq:SDE_examples} $\P$-a.s.\ for all
\sgc{}$t \in [0,\infty)$, $x \in [0,\infty)^2$\cgs{}, i.e.,
\begin{equation}  
\label{eq:ex_Brusselator}
  X_t^x
  =
  x
  +
  \int_0^t
  \left( 
    \begin{array}{c}
      \delta - 
      \left( \alpha + 1 \right) X^{x,1}_s
      +
      X^{x,2}_s ( X^{x,1}_s )^2
      \\
      \alpha X^{x,1}_s
      - X^{x,2}_s 
      \big( X^{x,1}_s \big)^2
    \end{array}
  \right) 
  \, ds
  +
  \int_0^t
  \sigma( X^x_s ) \, dW_s
\end{equation}
$ \P $-a.s.\ for all 
\sgc{}$t \in [0,\infty)$, $x \in [0,\infty)^2$\cgs{}.
Moreover, we assume \sgc{}that\cgs{}
\begin{equation}\label{eq:def_eta_Brusselator}
  \eta :=
  \sup_{ y \in [0,\infty)^2 }
  \|
    \sigma( y )^* ( 1, 1 )
  \|
  \in [0,\infty)
\end{equation}
(cf.\ the last sentence in Section~1 in Scheutzow~\cite{Scheutzow1986}).
If $ \rho \in ( 0, \infty ) $
and if 
$ U \colon \R^2 \to \R $
is given by
$
  U( x ) 
  =
  \rho \left(
    x_1 + x_2
  \right)^2
$
for all 
$ x = ( x_1 , x_2 ) \in \R^2 $,
then it holds 
for every
$ x = ( x_1 , x_2 ) \in [0,\infty)^2 $,
$ \varepsilon \in (0,\infty) $
that
\begin{equation}
\begin{split}
&
  ( \mathcal{G}_{ \mu, \sigma } U)( x ) 
  +
  \tfrac{ 1 }{ 2 }
  \| 
    \sigma( x )^* ( \nabla U )( x )
  \|^2
\\ & =
  2 \rho \left( x_1 + x_2 \right)
  \left(
    \delta -  x_1
  \right)
  +
  \rho
  \operatorname{tr}\!\left(
    \sigma(x)
    \sigma(x)^*
    \left(
      \begin{array}{cc}
        1 & 1
      \\
        1 & 1
      \end{array}
    \right)
  \right)
  +
  2 \rho U(x)
  \| 
    \sigma( x )^* ( 1, 1 )
  \|^2
\\ & \leq
  2 \rho \delta \left( x_1 + x_2 \right)
  +
  \rho \eta^2
  +
  2 \rho \eta^2 
  U(x)
\leq
  \tfrac{ \delta^2 }{ 2 \varepsilon }
  +
  \rho \eta^2
  +
  2 \rho 
  \left[ 
    \eta^2 
    + 
    \varepsilon
  \right]
  U(x)
  \,
  .
\end{split}
\label{eq:Brusselator_expmom}
\end{equation}
In addition, note 
for every
$ x = ( x_1 , x_2 )$\sgc{},\cgs{} $y = ( y_1, y_2 ) \in [0,\infty)^2 $,
$ \theta \in (0,\infty) $
that
$x_2(x_1)^2-y_2(y_1)^2=\frac{(x_1)^2+(y_1)^2}{2}(x_2-y_2)
+\frac{(x_1+y_1)(x_2+y_2)}{2}(x_1-y_1)$
and that
\begin{align}
\nonumber 
\label{eq:Brusselator2}
&
  \tfrac{ 
    \langle
      x - y ,
      \mu( x ) - 
      \mu( y )
    \rangle
    +
    \tfrac{ 1 }{ 2 }
    \| 
      \sigma( x ) - \sigma( y ) 
    \|^2_{ \HS( \R^m, \R^2 ) }
  }{
    \|
      x - y
    \|^2
  }
    +
    \tfrac{
      ( \frac{ \theta }{ 2 } - 1 ) \,
      \| 
        (
          \sigma( x ) - \sigma( y )
        )^*
        ( x - y )
      \|^2
    }{
      \| x - y \|^4
    }
\\ &
 =\tfrac{-(\alpha+1)(x_1-y_1)^2+
 \big(
x_2(x_1)^2-y_2(y_1)^2\big)\big((x_1-y_1)-(x_2-y_2)\big)
+\alpha(x_2-y_2)(x_1-y_1)}{\|x-y\|^2}\nonumber
\\&\quad
    +\tfrac{ 1 }{ 2 }
  \tfrac{ 
    \| 
      \sigma( x ) - \sigma( y ) 
    \|^2_{ \HS( \R^m, \R^2 ) }
  }{
    \|
      x - y
    \|^2
  }
    +
    \tfrac{
      ( \frac{ \theta }{ 2 } - 1 ) \,
      \| 
        (
          \sigma( x ) - \sigma( y )
        )^*
        ( x - y )
      \|^2
    }{
      \| x - y \|^4
    }\nonumber
\\ & \leq
  \tfrac{ 
    ( x_1 )^2 + ( y_1 )^2
  }{ 4 }
  +
  \tfrac{
    3 ( x_1 + y_1 ) ( x_2 + y_2 )
  }{
    4
  }
  +\tfrac{\alpha}{2}
  +
  \big[
    \tfrac{ 1 }{ 2 }
    +
    |
      \tfrac{ \theta }{ 2 }
      - 1
    |
  \big]
  \| 
    \sigma
  \|_{
    \operatorname{Lip}( \R^2, \HS( \R^2 ) )
  }^2
\\ \nonumber & =
  \tfrac{ 
    ( x_1 )^2 + ( y_1 )^2
  }{ 4 }
  +
  \tfrac{
    3 ( x_1 x_2 + x_1 y_2 + y_1 x_2 + y_1 y_2 ) 
  }{
    4
  }
  +\tfrac{\alpha}{2}
  +
  \big[
    \tfrac{ 1 }{ 2 }
    +
    |
      \tfrac{ \theta }{ 2 }
      - 1
    |
  \big]
  \| 
    \sigma
  \|_{
    \operatorname{Lip}( \R^2, \HS( \R^2 ) )
  }^2
\\ \nonumber & \leq
  ( x_1 + x_2 )^2 + ( y_1 + y_2 )^2
  +\tfrac{\alpha}{2}
  +
  \big[
    \tfrac{ 1 }{ 2 }
    +
    |
      \tfrac{ \theta }{ 2 }
      - 1
    |
  \big]
  \| 
    \sigma
  \|_{
    \operatorname{Lip}( \R^2, \HS( \R^2 ) )
  }^2
  .
\end{align}
Combining\sgc{}~\eqref{eq:def_eta_Brusselator}\cgs{},~\eqref{eq:Brusselator_expmom}
and~\eqref{eq:Brusselator2} with 
Corollary~\ref{cor:UV2} hence
implies \sgc{}that for every\cgs{}
$ x = ( x_1, x_2 ), y = ( y_1, y_2 ) \in [0,\infty)^2 $,
$ r, p, q \in (2,\infty) $,
$ T, \varepsilon, \rho \in (0,\infty) $
with
$
  \frac{ 1 }{ p } +
  \frac{ 1 }{ q } = 
  \frac{ 1 }{ r }
$
and
$ 
  \exp\!\big(
    2 \rho T
    [
      \eta^2 + \varepsilon
    ]
  \big)
\leq 
  \frac{ \rho }{ 2 q T } 
$
\sgc{}it holds that\cgs{}
\begin{equation}
\begin{split}
&
    \left\|
      \sup\nolimits_{ t \in [0,T] } 
      \| X^x_t - X^y_t \|
    \right\|_{
      L^r( \Omega; \R )
    }
\\ & \leq
  \frac{
    \left\| x - y \right\|
  }{
    \sqrt{
      1 - 2 / p
    }
  }
    \exp\!\left(
    \tfrac{\alpha T}{2}+
      \tfrac{ ( p - 1 ) T }{ 2 }
      \,
      \| 
        \sigma
      \|_{
        \operatorname{Lip}( \R^2, \HS( \R^2 ) )
      }^2
      +
      \tfrac{
        \delta^2 T / \varepsilon
        +
        1 +
        \rho ( x_1 + x_2 )^2 + 
        \rho ( y_1 + y_2 )^2 
      }{
        2 q
      }
    \right)
  .
\end{split}
\end{equation}
This together with Theorem~\ref{thm:strong_completeness_uniform}
implies \sgc{}in the case where $\sigma$ is globally Lipschitz continuous and where\cgs{}
$
  \sup_{ y \in [0,\infty)^2 }
  \|
    \sigma( y )^* (1,1)
  \|^2
  < \infty
$ 
\sgc{}that\cgs{} the SDE~\eqref{eq:ex_Brusselator}
is strongly complete.

\section[Stochastic volatility processes]{Stochastic volatility processes and 
interest rate models (CIR, Ait-Sahalia, \texorpdfstring{$3/2$}{three halves}-model)}
\label{ssec:cir}

There are a number of models in the finance literature which
generalize the Black-Scholes model 
by the use 
of 
a stochastic process for the squared volatility.
The following SDE includes a number of these models
for the squared volatility.
Let 
$ ( \Omega, \mathcal{F}, \P, ( \mathcal{F}_t )_{ t \in [0,\infty) } ) $
be a filtered probability space satisfying the usual conditions, 
let
$ W \colon [0,\infty) \times \Omega \to \R $
be a standard 
$ ( \mathcal{F}_t )_{ t \in [0,\infty) } $-Brownian motion, let
$
  a \in (1,\infty)
$, 
$
  b \in [ \tfrac{ 1 }{ 2 } , \infty ) 
$, 
$
  c,\, \beta \in (0, \infty)
$, 
$
  \alpha,\, \kappa \in [0,\infty)
$,
$
  \gamma,\, \delta \in \R
$,
and let
$
  X^x\colon[0,\infty)\times\Omega\to[0,\infty)
$, 
$
  x \in (0,\infty)
$,
be adapted stochastic processes
with continuous sample paths satisfying
\begin{equation}  
\label{eq:SDE.volatility.processes}
  X_t^x=x
  +
  \int_0^t
  \left[
  \frac{ \kappa }{ ( X_s^x )^c } 
  + 
  \delta
  +
  \gamma X_s^x
  -
  \alpha \, ( X_s^x )^a 
  \right] 
  ds
  +
  \int_0^t 
  \beta \, ( X_s^x )^b \, dW_s
\end{equation}
$ \P $-a.s.\ for all $ t \in [0,\infty) $\sgc{},\cgs{} $ x \in (0,\infty) $.
The class~\eqref{eq:SDE.volatility.processes} of processes
includes Cox-Ingersoll-Ross processes ($b=0.5$, $\gamma<0<\delta$, 
$\alpha=\kappa=0$),
Ait-Sahalia interest rate models, 
the volatility processes in
Heston's $3/2$-models
($b=1.5$, $a=2$, $\delta=\kappa=0\leq \gamma$, $\alpha,\beta>0$)
and constant elasticity of variance processes
($b\in[0.5,1]$, $\alpha=\delta=\kappa=0\leq\gamma$, $\beta>0$).
Let\sgc{}
$
  \mu \colon (0,\infty) \rightarrow \R
$ and
$
  \sigma \colon (0,\infty) \rightarrow \R
$\cgs{}
be the functions \sgc{}satisfying\cgs{} 
for all $ x \in (0,\infty) $ \sgc{}that\cgs{}
\begin{equation}\label{eq:def_mu_sigma_IRM}
  \mu(x) = \kappa x^{-c} + \delta +\gamma x - \alpha x^a
  \qquad
  \text{and}
  \qquad
  \sigma(x) = \beta x^b
  \sgc{}.\cgs{}
\end{equation}
\sgc{}For every $x\in (0,\infty)$, $y\in [0,\infty)$ let\cgs{}
$
  \tau_y^{x}\colon\Omega\to[0,\infty]
$
\sgc{}be given by\cgs{}
$
  \tau_y^x =
  \sgc{}\inf\!\big( 
    \{ t \in [0,\infty) \colon X_t^x = y \} \cup \{ \infty \} 
  \big)\cgs{}
$.

For the rest of Subsection~\ref{ssec:cir},
we assume that
the boundary point $ 0 $ is inaccessible, that is, that
$
  \P\big[
    \tau_0^x = \infty
  \big] = 1 
$
for all $x\in(0,\infty)$. 
According to
Feller's boundary classification
(see, e.g.,~Theorem V.51.2 in~\cite{RogersWilliams2000b}),
the boundary point $ 0 $ is inaccessible
if and only if
$ (i) $ $ \kappa > 0 $
or $ (ii) $ $ b = \frac{ 1 }{ 2 } $ and $ 2 \delta \geq \beta^2 $
or $ (iii) $ $ b > \frac{ 1 }{ 2 } $ and $ \delta > 0 $
or $ (iv) $ $ b \geq 1 $ and $ \delta \geq 0 $.

\subsection{Local Lipschitz continuity in the initial value}

In the case $ b \neq 1 $, 
the processes 
$ ( X^x )^{ (1 - b) } $, $ x \in (0,\infty) $, satisfy
an SDE with constant diffusion function 
(see, e.g., Alfonsi~\cite{Alfonsi2005})
and globally one-sided
Lipschitz continuous drift function
(see, e.g.,
Dereich, Neuenkirch \&\ Szpruch~\cite{DereichNeuenkirchSzpruch2012} \sgc{}and\cgs{}
Neuenkirch \&\ Szpruch~\cite{NeuenkirchSzpruch2014}).
In the following calculation we exploit this observation
together with the results in Section~\ref{sec:strong_stability}
to derive an estimate for the Lyapunov-type function
$ V \colon ( 0, \infty )^2 \to [0,\infty) $
given by 
$ 
  V( x, y ) = 
  | 
    x^{ (1 - b) } - y^{ (1 - b) } 
  |^2
$
for all $ x, y \in ( 0, \infty ) $.
\sgc{}Example\cgs{}~\ref{ex:vola}
implies that
for all 
$ x, y \in (0,\infty) $
it holds that
\begin{equation}
\begin{split}
&
  \tfrac{
    \|
      ( \overline{G}_{ \sigma } V)( x, y )
    \|^2
  }{
    | V(x,y) |^2
  }
=
  \tfrac{
    4 
    ( 1 - b )^2
    \left[
      x^{ - b }
      \beta 
      x^b
      -
      y^{ - b }
      \beta y^b
    \right]^2
  }{
    \left[
      x^{ ( 1 - b ) }
      -
      y^{ ( 1 - b ) }
    \right]^2
  }
=
  0
\end{split}
\end{equation}
and in the case $ b \neq 1 $ that
for all $ x, y \in (0,\infty) $ with $ x \neq y $
it holds that
\begin{equation}
\begin{split}
\label{eq:GenBarVola}
&
  \tfrac{
    ( \overline{ \mathcal{G} }_{ \mu, \sigma } V)( x, y )
  }{
    V( x, y )
  }
\\ &
=
  \tfrac{ 
    2 
    \left( 1 - b \right)
    \left(
      x^{ - b } 
      \left[
        \kappa 
        x^{ - c }
        +
        \delta
        +
        \gamma x
        -
        \alpha x^a
      \right]
      -
      y^{ - b }
      \left[
        \kappa y^{ - c }
        +
        \delta
        +
        \gamma y
        -
        \alpha y^a
      \right]
      -
      \frac{ b }{ 2 }
      \left[ 
        x^{ 1 - b - 2 }
        \beta^2 x^{ 2 b }
        -
        y^{ 1 - b - 2 }
        \beta^2 y^{ 2 b }
      \right]
    \right)
  }{
      x^{ 1 - b } - y^{ 1 - b }
  }
\\&
=
  2 
  \left( 1 - b \right)
  \gamma
  +
  \tfrac{ 
    2 
    \left( 1 - b \right)
    \int_y^x
    \left[ 
      - \delta b z^{ ( - b - 1 ) }
      - \alpha ( a - b ) z^{ ( a - b - 1 ) }
      - \kappa ( c + b ) z^{ ( - c - b - 1 ) }
      - \frac{ b ( b - 1 ) }{ 2 } 
      \beta^2 z^{ ( b - 2 ) }
    \right]
    \,
    dz
  }{
      x^{ 1 - b } - y^{ 1 - b }
  }
\\&
\leq
  2 
  \left( 1 - b \right)
  \gamma
\\ & \quad
  +
  \tfrac{ 
    2 
    \left( 1 - b \right)
    \left[ 
      \int_y^x
       z^{ - b }
       \,
      dz
    \right]
    \left[
    \sup\nolimits_{ u \in (0,\infty) }
    \left(
      - \delta b u^{-1}
      - \alpha ( a - b ) u^{ ( a - 1 ) }
      - \kappa (c+b) u^{ ( - c - 1 ) }
      - \frac{ b ( b - 1 ) }{ 2 } 
      \beta^2 
      u^{ ( 2 b - 2 ) }
    \right)
    \right]
  }{
      x^{ (1-b) } - y^{ (1-b) }
  }
\\ &
=
  2
  \left[
    \sup\nolimits_{ u \in (0,\infty) }
    \left(
      ( 1 - b ) \gamma
      - \tfrac{ \delta b }{ u }
      - \alpha ( a - b ) u^{ (a - 1) }
      - \tfrac{ \kappa ( c + b ) }{ u^{ (c + 1) } }
      + \tfrac{ b ( 1 - b ) \beta^2 }{ 2 } u^{ ( 2 b - 2 ) }
    \right)
  \right]
  .
\end{split}
\end{equation}
Proposition~\ref{prop:two_solution_supoutside} 
with $ r = \infty = p = q $
together with Lemma~\ref{lem:nicht_treffen} 
hence yields
for all $ t \in [0,\infty) $\sgc{},\cgs{} $ x, y \in (0,\infty) $
that
\begin{equation}
\begin{split}
\label{eq:estimate.cir}
&
  \big\|
    ( X_t^x )^{ (1 - b) }
    -
    ( X_t^y )^{ (1 - b) }
  \big\|_{
    L^{ \infty }( \Omega; \R ) 
  }
\\ & 
\leq
  \left|
    x^{ ( 1 - b ) } 
    -
    y^{ ( 1 - b ) }
  \right|
  \exp\!\left(
    t 
    \left[
    \sup\nolimits_{ 
      u, v \in (0,\infty),
      u \neq v
    }
    \tfrac{
      ( \overline{ \mathcal{G} }_{ \mu, \sigma } V)( u, v )
    }{
      V( u, v )
    }
    \right]
  \right)
\\ & \leq
  \left|
    x^{ ( 1 - b ) } 
    -
    y^{ ( 1 - b ) }
  \right|
\\ & \ \  
  \cdot
  \exp\!\left(
    t \,
      \sup\nolimits_{ u \in (0,\infty) }
      \!
      \big(
        (1-b)\gamma
        - \tfrac{ \delta b }{ u }
        - \alpha ( a - b ) u^{ (a - 1) }
        - 
        \tfrac{ 
          \kappa ( b + c ) 
        }{ 
          u^{ (c + 1) } 
        }
        +
        \tfrac{ b ( 1 - b ) }{ 2 }
        \beta^2
        u^{ (2 b - 2 ) }
      \big)
  \!\right)\!
  .
\end{split}
\end{equation}
Moreover,
Proposition~\ref{prop:two_solution_supinside} 
with $ v = \infty = p = q = r $
together with Lemma~\ref{lem:nicht_treffen}
yields
for all 
$
  t \in [0,\infty)
$\sgc{},\cgs{} 
$ x, y \in (0,\infty) $
that
\begin{equation}
\begin{split}
\label{eq:estimate.sup.cir}
&
  \left\|
    \sup\nolimits_{ s \in [0,t] }
    \left|
      ( X_s^x )^{ (1 - b) }
      -
      ( X_s^y )^{ (1 - b) }
    \right|
  \right\|_{
    L^{ \infty }( \Omega; \R )
  }
\\ & 
\leq
  \left|
    x^{ ( 1 - b ) } 
    -
    y^{ ( 1 - b ) }
  \right|
  \exp\!\left(
    t 
    \max\left\{ 0,
      \sup\nolimits_{ 
        u, v \in (0,\infty)
      }
      \tfrac{
        ( \overline{ \mathcal{G} }_{ \mu, \sigma } V)( u, v )
      }{
        V( u, v )
      }
    \right\}
  \right)
\\ & \leq
  \left|
    x^{ (1 - b) } 
    -
    y^{ (1 - b) }
  \right|
\\ & \quad
  \cdot
  \exp\!\left(
  {\scriptstyle
    t
    \max\left\{
      0,
      \, 
      \sup\nolimits_{ u \in (0,\infty) }
      \big(
        (1-b)\gamma
        - \tfrac{ \delta b }{ u }
        - \alpha ( a - b ) u^{ (a - 1) }
        - 
        \tfrac{ 
          \kappa ( b + c ) 
        }{ 
          u^{ (c + 1) } 
        }
        +
        \tfrac{ b ( 1 - b ) }{ 2 }
        \beta^2
        u^{ (2 b - 2 ) }
      \big)
    \right\}
  }
  \right)
  .
\end{split}
\end{equation}
The right-hand sides
of~\eqref{eq:GenBarVola},
of~\eqref{eq:estimate.cir}
and of~\eqref{eq:estimate.sup.cir}
are finite
if
$ (i) $ $ \kappa > 0 $
or $ (ii) $ $ b = \frac{ 1 }{ 2 } $ and $ 2 \delta \geq \beta^2 $
or $ (iii) $ $ b > \frac{ 1 }{ 2 } $ and $ \delta > 0 $
or $ (iv) $ $ b \geq 1 $ and $ \delta \geq 0 $.

Next, for the convenience of the reader, we derive well-known moment estimates.
For this let
$ \eta \in ( 0, \infty ) $ be a fixed real number and let
$ U_p \colon (0,\infty) \to ( 0, \infty ) $,
$ p \in \R $,
be given by $ U_p( x ) = \eta + x^p $
for all $ x \in (0,\infty) $, $ p \in \R $.
Then
Lemma~\ref{lem:Lyapunov} implies 
for all $ t \in [0,\infty) $, $ x \in ( 0, \infty ) $, $ p \in \R $ 
that
\begin{equation}
\label{eq:moments.cir}
\begin{split}
&
  \E\big[
    \eta + 
    ( X_t^x )^p
  \big]
\leq 
  ( \eta + x^p )
  \exp\!\Big(
      \sup\nolimits_{
        u \in (0,\infty)
      }
      \tfrac{
        t \,
        ( \mathcal{G}_{ \mu, \sigma } U_p )( u )
      }{
        U_p(u)
      }
  \Big)
\\ & =
  \left( 
    \eta + x^p 
  \right)
\\ & \quad 
  \cdot
  \exp\!\Big(
  \sup\nolimits_{
    u \in (0, \infty)
  }
    \tfrac{
      t \, p \,
      \left[
      \kappa u^{ (p - 1 - c) }
      +
      \delta u^{ (p - 1) }
      +
      \gamma u^p
      - 
      \alpha u^{ (p + a - 1) }
      +
      \frac{ 1 }{ 2 } ( p - 1 ) \beta^2 u^{ (p + 2 b - 2) }
      \right]
    }{ \eta + u^p }
  \Big)
  .
\end{split}
\end{equation}
Next observe
for all 
$ 
  p \in [1,\infty) 
  \cap 
  \big[
    ( c + 1 ) \1_{ (0,\infty) }( \kappa )
    ,
    \infty
  \big)
$ 
that
if (i)  $ b \leq 1 $
or (ii) $ b < \frac{ a + 1 }{ 2 } $ and $ \alpha > 0 $
or (iii) $ b = \frac{ a + 1 }{ 2 } $ and $ 2 \alpha \geq (p - 1 ) \beta^2 $
or (iv) $ p = 1 $,
then
\begin{equation}
\begin{split}
\label{eq:normal_Gen_cir}
&     \sup\nolimits_{u\in(0,\infty)}
  \left[
    \tfrac{
      (
        \mathcal{G}_{ \mu, \sigma } U_p 
      )( u )
    }{
      U_p( u )
    }
  \right]
\\ &  =
  p
  \sup\nolimits_{
    u \in (0, \infty)
  }
  \left[
    \tfrac{
      \kappa u^{ (p - 1 - c) }
      +
      \delta u^{ (p - 1) }
      +
      \gamma u^p
      - 
      \alpha u^{ (p + a - 1) }
      +
      \frac{ 1 }{ 2 } ( p - 1 ) \beta^2 u^{ (p + 2 b - 2) }
    }{ \eta + u^p }
  \right]
  < \infty
  .
\end{split}
\end{equation}
Moreover, note that It\^o's formula shows that
\begin{equation}  \begin{split}
\label{eq:cir.Ito}
  U_p(X_t^x)
  &=U_p(x)+\int_0^t (\mathcal{G}_{\mu,\sigma}U_p)(X_r^x)\,dr
                                   +\int_0^t (G_{\sigma}U_p)(X_r^x)\,dW_r
  \\&
  \leq 
    U_p(x)+
    \bigg[
      \sup_{
	u \in (0,\infty)
      }
      \tfrac{
	( \mathcal{G}_{ \mu, \sigma } U_p )( u )
      }{
	U_p(u)
      }
    \bigg]
    \int_0^t  
      U_p(X_r^x)
      \,
    dr
  +
  \int_0^t 
  ( G_{ \sigma } U_p )( X_r^x )
  \, dW_r
\end{split}     \end{equation}
$\P$-a.s.~for all $ t \in [0,\infty) $, $ x \in (0,\infty) $, $ p \in \R $.
In addition,
Jensen's inequality and Doob's martingale inequality 
yield for all $ t \in [0,\infty) $, $ x \in (0,\infty) $, $ p \in \R $ that
\begin{equation}  \begin{split}
\label{eq:CIR.after.BGD}
&
  \bigg\|
    \sup_{ s \in [0,t] }
    \bigg|
      \int_0^s 
        ( G_{ \sigma } U_p )( X_r^x )
        \,
      dW_r
    \bigg|
  \bigg\|_{
    L^1( \Omega; \R )
  }
\leq
  \bigg\|
    \sup_{ s \in [0,t] }
    \bigg|
      \int_0^s 
        ( G_{ \sigma } U_p )( X_r^x )
        \,
      dW_r
    \bigg|
  \bigg\|_{
    L^2( \Omega; \R )
  }
\\&
  \leq
  2
  \,
  \left\|
    \int_0^t 
      \left|
        (G_{\sigma}U_p)(X_r^x)
      \right|^2 
    dr
  \right\|_{
    L^1( \Omega; \R )
  }^{ \frac{1}{2} }
  =
  2 | p | \beta
  \left[ 
  \int_0^t
  \E\!\left[
    ( X_r^x )^{ (2 p - 2 + 2 b) }
  \right]
  dr
  \right]^{ \frac{1}{2} }
  .
\end{split}     \end{equation}
Then taking supremum over the time interval 
$
  [0, t \wedge \tau_n^x ] 
$
for $ t \in [0,\infty) $
in inequality~\eqref{eq:cir.Ito}, taking the expectation
and applying inequality~\eqref{eq:CIR.after.BGD} 
shows
for all 
$
  t \in [0,\infty)
$, 
$ x \in (0,\infty) $, 
$ p \in \R $, $ n \in \N $
that
\begin{align}
\label{eq:CIR.after.taking.expectation}
&
  \bigg\|
    \sup_{ s \in [0, t\wedge\tau_n^x ] 
    }
    U_p( X_s^x )
  \bigg\|_{
    L^1( \Omega; \R )
  }
\nonumber
\\ & 
\leq
     U_p(x)
     +
  \bigg\|
     \int_0^{ t \wedge \tau_n^x }  
     U_p(X_r^x)
     \, dr
  \bigg\|_{L^1(\Omega;\R)}
     \max\!\bigg\{ 
       0, 
       \sup_{
         u \in (0,\infty)
       }
       \tfrac{
         ( \mathcal{G}_{ \mu, \sigma } U_p )( u )
       }{
         U_p(u)
       }
     \bigg\}
\nonumber
\\ & \quad 
  +
  \bigg\|
    \sup_{ s \in [0,t] }
    \bigg|
      \int_0^s 
        ( G_{ \sigma } U_p )( X_r^x )
        \,
      dW_r
    \bigg|
  \bigg\|_{
    L^1( \Omega; \R )
  }
\\ & \leq
  U_p( x )
  +
  2 | p | \beta
  \left[ 
  \int_0^t
  \E\!\left[
    ( X_r^x )^{ (2 p - 2 + 2 b) }
  \right]
  dr
  \right]^{ \frac{1}{2} }
\nonumber
\\ & \quad +
  \bigg[
  \int_0^t
    \bigg\|
      \sup_{
        r \in [0, s \wedge \tau_n^x ] 
      }
      U_p( X_r^x )
    \bigg\|_{
      L^1( \Omega; \R )
    }
  \, ds
  \bigg]
     \max\!\bigg\{ 
       0, 
       \sup_{
         u \in (0,\infty)
       }
       \tfrac{
         ( \mathcal{G}_{ \mu, \sigma } U_p )( u )
       }{
         U_p(u)
       }
     \bigg\}
  .
\nonumber
\end{align}
Now the monotone convergence theorem, Gronwall's inequality,
inequality~\eqref{eq:CIR.after.taking.expectation}
and
inequality~\eqref{eq:moments.cir}
yield 
for all 
$ t \in [0,\infty) $, $ x \in ( 0, \infty) $, $ p \in \R $
that
\begin{align}
\label{eq:uniform.moments}
\nonumber 
  &
  \big\|
    \sup\nolimits_{ s \in [0,t] }
    U_p( X_s^x )
  \big\|_{
    L^1( \Omega; \R )
  }
  = 
  \lim_{ n \to \infty }
  \big\|
    \sup\nolimits_{ 
      s \in [ 0, t \wedge \tau_n^x ]
    }
    U_p( X_s^x )
  \big\|_{
    L^1( \Omega; \R )
  }
\\ \nonumber &
  \leq
  \left[
    U_p(x)
    + 
    2 | p | \beta
    \left|
      \int_0^t
        \E\!\left[
          ( X_r^x )^{ (2 p - 2 + 2 b) }
        \right]
      dr
    \right|^{ \frac{1}{2} }
  \right]
\\ & \quad 
\cdot
  \exp\!\Big(
    t
    \max\!\Big\{
      0 ,
      \sup\nolimits_{
        u \in (0,\infty)
      }\!
      \tfrac{
        ( \mathcal{G}_{ \mu, \sigma } U_p )( u )
      }{
        U_p(u)
      }
    \Big\}
  \Big)
\\ \nonumber & \leq
  \left[
  U_p(x)
  + 
  2 | p | \beta
    \left|
    U_{ 2 p - 2 + 2 b }(x)
    \int_0^t
      \exp\!\bigg(
        \sup_{ u \in (0, \infty) }
        \tfrac{ 
          r \,
          ( \mathcal{G}_{ \mu, \sigma } U_{ 2 p - 2 + 2 b } )( u )
        }{
          U_{ 2 p - 2 + 2 b }( u )
        }
      \bigg)
    \, dr
    \right|^{ \frac{1}{2} }
  \right]
\\ \nonumber & \quad 
\cdot
  \exp\!\bigg(
    \max\bigg\{
      0 ,
      \sup_{
        u \in (0,\infty)
      }\!
      \tfrac{
        t \, 
        ( \mathcal{G}_{ \mu, \sigma } U_p )( u )
      }{
        U_p(u)
      }
    \bigg\}
  \bigg)
  .
\end{align}
In the next step we observe that in the case $ b \neq 1 $ it holds 
for all $ x, y \in ( 0, \infty) $ that
\begin{equation}
\begin{aligned} 
\label{eq:x1b1.locally.Lipschitz}
  | x - y |
&  =
  \left|
    \left[
      x^{ 1 - b }
    \right]^{
      \frac{ 1 }{ ( 1 - b ) }
    }
    -
    \left[
      y^{ 1 - b }
    \right]^{
      \frac{ 1 }{ ( 1 - b ) }
    }
  \right|
\\ &  
  = 
  \bigg|
    \int_{ y^{ 1 - b } }^{ x^{ 1 - b } }
      \tfrac{ 1 }{ (1 - b) }
      \,
      z^{ 
        \left[ 
          \frac{ 1 }{ 1 - b } - 1 
        \right]
      }
    \, dz
  \bigg|
\leq 
  \tfrac{
    \max( x^b , y^b )
  }{
    | 1 - b |
  }
  \sgc{}\left|
    x^{ 1 - b } 
    - 
    y^{ 1 - b }
  \right|\cgs{}
\end{aligned}
\end{equation}
\sgc{}and\cgs{}
\begin{equation}
\begin{aligned}\label{eq:x1b2.locally.Lipschitz}
  \left|
    x^{ 1 - b } - y^{ 1 - b } 
  \right|
&  =
  \left|
    (1-b)
    \int_y^x z^{ - b } \, dz
  \right|
  \leq
  \tfrac{
    |1 - b|
  }{
    \min( x^b, y^b) 
  }
  | x - y |
  .   
\end{aligned}
\end{equation}
The estimate~\eqref{eq:estimate.sup.cir}
together \sgc{}with~\eqref{eq:x1b1.locally.Lipschitz},~\eqref{eq:x1b2.locally.Lipschitz}\cgs{}
and monotonicity of solutions of 
\eqref{eq:SDE.volatility.processes}
with respect to the initial values
implies in the case $b\neq 1$ that
for all $ x, y, t, p \in (0,\infty) $
it holds that
\begin{equation}  \begin{split}
  &\bigg\|
    \sup_{ s \in [0,t] }
    | X^x_s - X^y_s |
  \bigg\|_{
    L^p( \Omega; \R )
  }
  \\&
  \leq
  \frac{1}{|1-b|}
  \bigg\|
    \sup_{ s \in [0,t] }
    \max\!\left\{ 
      ( X_s^x )^b ,
      ( X_s^y )^b
    \right\}
  \bigg\|_{
    L^p( \Omega; \R )
  }
  \bigg\|
    \sup_{ s \in [0,t] }
    \left|
      \left( X_s^x \right)^{ 1 - b }
      -
      \left( X_s^y \right)^{ 1 - b }
    \right|
  \bigg\|_{
    L^{ \infty }( \Omega; \R )
  }
\!\!\!
\\ & \leq
  \frac{ 1 }{ | 1 - b | }
  \left[
    \max_{ z \in \{ x, y \} }
    \bigg\|
      \sup_{ s \in [0, t ] }
      \!
      \left(
        X_s^z
      \right)^b
    \bigg\|_{
      L^p( \Omega; \R )
    }
  \right]
  \bigg\|
    \sup_{ s \in [0,t] }
    \left|
      \left( X_s^x \right)^{ 1 - b }
      -
      \left( X_s^y \right)^{ 1 - b }
    \right|
  \bigg\|_{
    L^{ \infty }( \Omega; \R )
  }
\\ & \leq
  \frac{ 
    \left|
      x^{ 1 - b } - y^{ 1 - b }
    \right|
  }{ | 1 - b | }
  \!
  \left[
    \max_{ z \in \{ x, y \} }
    \bigg\|
      \sup_{ s \in [0, t ] }
      \!
      X_s^z
    \bigg\|_{
      L^{ p b }( \Omega; \R )
    }^b
  \right]
  \exp\!\bigg(\!
    t 
    \max\bigg\{ 0,\!
      \sup_{ 
        u, v \in (0,\infty)
      }
      \!
      \tfrac{
        ( \overline{ \mathcal{G} }_{ \mu, \sigma } V)( u, v )
      }{
        V( u, v )
      }
    \bigg\}\!
  \bigg)
\!\!\!
\\ & \leq
  \frac{ 
    \left|
      x - y
    \right|
  }{ 
    \min( x^b, y^b )
  }
  \!
  \left[
    \max_{ z \in \{ x, y \} }
    \bigg\|
      \sup_{ s \in [0, t ] }
      \!
      X_s^z
    \bigg\|_{
      L^{ p b }( \Omega; \R )
    }^b
  \right]
  \exp\!\bigg(\!
    t 
    \max\bigg\{ 0,\!
      \sup_{ 
        u, v \in (0,\infty)
      }
      \!
      \tfrac{
        ( \overline{ \mathcal{G} }_{ \mu, \sigma } V)( u, v )
      }{
        V( u, v )
      }
    \bigg\}\!
  \bigg)
  \sgc{}.\cgs{}
\end{split}     \end{equation}
\sgc{}Inserting~\eqref{eq:uniform.moments}\cgs{}
hence shows
in the case $ b \neq 1 $ that
for all $ x, y, t, p \in (0,\infty) $
it holds that
\begin{equation}
\begin{split}
\label{eq:Lipschitz.estimate.cir.uniform}
&
  \bigg\|
    \sup_{ s \in [0,t] }
    | X^x_s - X^y_s |
  \bigg\|_{ 
    L^p( \Omega; \R ) 
  }
\leq
  \frac{ 
    \left|
      x - y
    \right|
  }{ 
    \min( x^b, y^b )
  }
\\ & \quad 
  \cdot
  \exp\!\bigg(
    t 
    \bigg[
    \max\bigg\{ 0,
      \sup_{ 
        u, v \in (0,\infty)
      }
      \tfrac{
        ( \overline{ \mathcal{G} }_{ \mu, \sigma } V)( u, v )
      }{
        V( u, v )
      }
    \bigg\}
    +
    \max\bigg\{
      0 ,
      \sup_{
        u \in (0,\infty)
      }
      \!\!
      \tfrac{
        ( \mathcal{G}_{ \mu, \sigma } U_{pb} )( u )
      }{
        p \, U_{pb}(u)
      }
    \bigg\}
    \bigg]
  \bigg)
\\ & \quad
  \cdot
  \max_{ z \in \{ x , y \} }
  \bigg[
      U_{ p b }(z)
      + 
      2 p b \beta
      \bigg[
      U_{ 2 ( p b + b - 1 ) }(z)
      \!
      \int_0^t
      \!
      \exp\!\bigg(
          \sup_{ u \in (0, \infty) } \!\!
          \tfrac{ 
            r \,
            ( \mathcal{G}_{ \mu, \sigma } U_{ 2 p b - 2 + 2 b } )( u )
          }{
            U_{ 2 p b - 2 + 2 b }( u )
          }
      \bigg)
      \, dr
    \bigg]^{ \frac{1}{2} }
  \bigg]^{ \frac{1}{p} }
  \!\!\!\!
  .
\end{split}
\end{equation}
This implies in the case $ b \neq 1 $,
$ \eta \geq 1 $
that for all $ x, y, t \in (0,\infty) $, $ p \in [1,\infty) $
it holds that
\begin{equation}
\begin{split}
\label{eq:Lipschitz.estimate.cir.uniform2}
&  \bigg\|
    \sup_{ s \in [0,t] }
    | X^x_s - X^y_s |
  \bigg\|_{ 
    L^p( \Omega; \R ) 
  } 
\leq
    \left|
      x - y
    \right|
  \tfrac{ 
      \left[ 
        1
        +
        \eta 
        + 
        \max( x^{\frac{b(p+1)}{p}}, y^{\frac{b(p+1)}{p}} )
      \right]
  }{ 
    \min( x^b, y^b )
  }
    \left[ 
      1
      +
      \left[ 
        2 p b \beta \sqrt{t}
      \right]^{ \frac{ 1 }{ p } }
    \right]
\\ & \quad
\cdot
  \exp\!\bigg(
    t 
    \bigg[
      \max\bigg\{ 0,
	\sup_{ 
	  u, v \in (0,\infty)
	}
	\tfrac{
	  ( \overline{ \mathcal{G} }_{ \mu, \sigma } V)( u, v )
	}{
	  V( u, v )
	}
      \bigg\}
      +
      \max\bigg\{
	0 ,
	\sup_{
	  u \in (0,\infty)
	}
	\tfrac{
	  ( \mathcal{G}_{ \mu, \sigma } U_{pb} )( u )
	}{
	  p \, U_{pb}(u)
	}
      \bigg\}
    \bigg]
  \bigg)
\\ & \quad
\cdot
  \exp\!\bigg(
        t \max\bigg\{
          0
          ,
          \sup_{ u \in (0, \infty) }
          \tfrac{ 
            t \,
            ( \mathcal{G}_{ \mu, \sigma } U_{ 2 p b - 2 + 2 b } )( u )
          }{
            2 p \, U_{ 2 p b - 2 + 2 b }( u )
          }
        \bigg\}
  \bigg)
  .
\end{split}
\end{equation}
Combining the statement below~\eqref{eq:estimate.sup.cir}
and~\eqref{eq:normal_Gen_cir}
shows that
the right-hand sides
of~\eqref{eq:Lipschitz.estimate.cir.uniform}
and of~\eqref{eq:Lipschitz.estimate.cir.uniform2}
are finite 
for all 
\sgc{}$ x, y, t \in ( 0, \infty ) $\cgs{} and \sgc{}all\cgs{}
\begin{equation}
  p \in 
  \big[
    \max\!\big\{ 
      1
      ,
      \tfrac{1}{b}
      ,
      \tfrac{( c + 1 ) }{b}
      \1_{ (0,\infty) }( \kappa )
      ,
      ( \tfrac{ c + 3 }{ 2 b } - 1 )
      \1_{ (0,\infty) }( \kappa )
    \big\}
    ,
    \infty
  \big)
\end{equation}
in either of the four following cases:
\begin{enumerate}
 \item
  $ b = \frac{ 1 }{ 2 } $ and ($ 2 
  \delta \geq \beta^2 $ or $ \kappa 
  > 
  0 $),
 \item $ \frac{ 1 }{ 2 } < b \leq 1 $ 
  and 
  $ \kappa + \max( 0, \delta ) > 0 $,
 \item $ 1 \leq b < \frac{ a + 1 }{ 2 } $, $ \alpha > 0 $ 
  and ($ \delta \geq 0 $ or $ \kappa > 0 $),
 \item $ b = \frac{ a + 1 }{ 2 } $, 
  $ 
    \max\{ \frac{ p - 1 }{ 2 } , p b + b - \frac{ 3 }{ 2 } \} 
    \leq \frac{ \alpha }{ \beta^2 } $ 
    and ($ \delta \geq 0 $ or $ \kappa > 0 $).
\end{enumerate}
Analogously, the inequalities~\eqref{eq:estimate.cir},
\eqref{eq:moments.cir}\sgc{},~\eqref{eq:x1b1.locally.Lipschitz} 
and~\eqref{eq:x1b2.locally.Lipschitz}\cgs{}
show in the case $ b \neq 1 $ that
for all $ x, y, t, p \in (0,\infty) $
it holds that
\begin{equation}
\begin{split}
\label{eq:Lipschitz.estimate.cir}
  \left\|
    X^x_t - X^y_t
  \right\|_{ 
    L^p( \Omega; \R ) 
  }
& \leq
    \left|
      x - y
  \right|
  \frac{ 
    \left[
      \eta
      +
      \max( 
        x^{ p b }, y^{ p b }
      )
    \right]^{ 
      \frac{ 1 }{ p }
    }
  }{ 
    \min( x^b, y^b )
  }
\\ & \quad
  \cdot
  \exp\!\bigg(
    t 
    \bigg[
      \sup_{ 
        u, v \in (0,\infty), u \neq v
      }
      \tfrac{
        ( \overline{ \mathcal{G} }_{ \mu, \sigma } V)( u, v )
      }{
        V( u, v )
      }
    +
      \sup_{
        u \in (0,\infty)
      }
      \tfrac{
        ( \mathcal{G}_{ \mu, \sigma } U_{pb} )( u )
      }{
        p \, U_{pb}(u)
      }
    \bigg]
  \bigg)
  .\!\!\!\!\!
\end{split}
\end{equation}
Combining the statement below~\eqref{eq:estimate.sup.cir}
with equation~\eqref{eq:normal_Gen_cir}
shows that
the right-hand side
of~\eqref{eq:Lipschitz.estimate.cir}
is finite 
for all 
$ x, y, t \in ( 0, \infty ) $ \sgc{}and all\cgs{}
\begin{equation}
  p \in 
  [1,\infty)
  \cap 
  [\tfrac{1}{b},\infty)
  \cap 
  \big[
    \tfrac{( c + 1 ) }{b}
    \1_{ (0,\infty) }( \kappa )
    ,
    \infty
  \big)
\end{equation}
in either of the five following cases:
\begin{enumerate}
 \item 
  $ b = \frac{ 1 }{ 2 } $ and 
  ($ 2\delta \geq \beta^2 $ or $ \kappa > 0 $),
 \item $ \frac{ 1 }{ 2 } < b \leq 1 $ 
  and $ \kappa + \max( 0, \delta ) > 0 $,
 \item $ 1 \leq b < \frac{ a + 1 }{ 2 } $, 
  $ \alpha > 0 $ and ($ \delta \geq 0 $ or $ \kappa > 0 )$,
 \item $ b = \frac{ a + 1 }{ 2 } $, $ ( pb - 1 ) \beta^2 \leq 2 \alpha $
   and ($ \delta \geq 0 $ or $ \kappa > 0 )$,
 \item $ b = 1 $, $ \delta \geq 0 $ and $ p = 1 $.
\end{enumerate}
We observe that in the case of the Cox-Ingersoll-Ross processes ($b=0.5$, $\gamma < 0 < \delta$, $\alpha=\kappa=0$)
the statement below~\eqref{eq:estimate.sup.cir} and equations~\eqref{eq:normal_Gen_cir},~\eqref{eq:Lipschitz.estimate.cir.uniform2} \sgc{}and~\eqref{eq:Lipschitz.estimate.cir}\cgs{}
imply that if $2\delta \geq \beta^2$ then for all $x,y,t \in (0,\infty)$, $p\in [1,\infty)$ it holds that
\begin{equation}
\begin{split}
\label{eq:Lipschitz.estimate.cir.uniform2.expl}
&  \bigg\|
    \sup_{ s \in [0,t] }
    | X^x_s - X^y_s |
  \bigg\|_{ 
    L^p( \Omega; \R ) 
  } 
\leq
    \left|
      x - y
    \right|
  \tfrac{ 
      \left[ 
        1
        +
        \eta 
        + 
        \max( x^{\frac{p+1}{2p}}, y^{\frac{p+1}{2p}} )
      \right]
  }{ 
    \min( x^{\frac{1}{2}}, y^{\frac{1}{2}} )
  }
    \left[ 
      1
      +
      \left[ 
        p \beta \sqrt{t}
      \right]^{ \frac{ 1 }{ p } }
    \right]
\\ & \quad
\cdot
  \exp\!\bigg(
    t 
    \bigg[
      \max\bigg\{ 0,
	\sup_{ 
	  u, v \in (0,\infty)
	}
	\tfrac{
	  ( \overline{ \mathcal{G} }_{ \mu, \sigma } V)( u, v )
	}{
	  V( u, v )
	}
      \bigg\}
      +
      \max\bigg\{
	0 ,
	\sup_{
	  u \in (0,\infty)
	}
	\tfrac{
	  ( \mathcal{G}_{ \mu, \sigma } U_{\frac{p}{2}} )( u )
	}{
	  p \, U_{\frac{p}{2}}(u)
	}
      \bigg\}
    \bigg]
  \bigg)
\\ & \quad
\cdot
  \exp\!\bigg(
        t \max\bigg\{
          0
          ,
          \sup_{ u \in (0, \infty) }
          \tfrac{ 
            t \,
            ( \mathcal{G}_{ \mu, \sigma } U_{  p -1 } )( u )
          }{
            2 p \, U_{ p - 1 }( u )
          }
        \bigg\}
  \bigg)
< \infty,
\end{split}
\end{equation}
and
\begin{equation}
\begin{split}
\label{eq:Lipschitz.estimate.cir.expl}
  \left\|
    X^x_t - X^y_t
  \right\|_{ 
    L^p( \Omega; \R ) 
  }
& \leq
    \left|
      x - y
  \right|
  \frac{ 
    \left[
      \eta
      +
      \max\{ 
        x^{ \frac{p}{2} }, y^{ \frac{p}{2}}
      \}
    \right]^{ 
      \frac{ 1 }{ p }
    }
  }{ 
    \min( x^{\frac{1}{2}}, y^{\frac{1}{2}} )
  }
\\ & \quad
  \cdot
  \exp\!\bigg(
    t 
    \bigg[
      \sup_{ 
        u, v \in (0,\infty), u \neq v
      }
      \tfrac{
        ( \overline{ \mathcal{G} }_{ \mu, \sigma } V)( u, v )
      }{
        V( u, v )
      }
    +
      \sup_{
        u \in (0,\infty)
      }
      \tfrac{
        ( \mathcal{G}_{ \mu, \sigma } U_{\frac{p}{2}} )( u )
      }{
        p \, U_{\frac{p}{2}}(u)
      }
    \bigg]
  \bigg)
  \!\!\!\!\!
\\ & < \infty.
\end{split}
\end{equation}
In~\eqref{eq:Lipschitz.estimate.cir.uniform2},~\eqref{eq:Lipschitz.estimate.cir},~\eqref{eq:Lipschitz.estimate.cir.uniform2.expl} \sgc{}and~\eqref{eq:Lipschitz.estimate.cir.expl}\cgs{}
we establish local Lipschitz estimates for solutions to one-dimensional stochastic volatility processes.
A key idea in our proof is to specify a suitable Lyapunov-type function and, thereafter, to apply Proposition~\ref{prop:two_solution_supoutside} and Proposition~\ref{prop:two_solution_supinside} with this specific Lyapunov function. If such a suitable Lyapunov-type function could also be found in the case of multi-dimensional volatility processes, then the above approach could also be applied to multi-dimensional volatility processes such as Wishart processes; see, e.g.,~\cite{Bru:1991}. We leave this for future research.

\subsection{Global Lipschitz continuity in the initial value}
\label{ssec:cir_global}
The estimates provided
by equations~\eqref{eq:Lipschitz.estimate.cir.uniform2}
and~\eqref{eq:Lipschitz.estimate.cir} imply under suitable assumptions
(see the statements below~\eqref{eq:Lipschitz.estimate.cir.uniform2}
and~\eqref{eq:Lipschitz.estimate.cir} for details)
that the solutions process of the SDE~\eqref{eq:SDE.volatility.processes}
is \emph{locally} Lipschitz continuous in the initial value in some sense. 
By applying Corollary~\ref{cor:powerful_est}
and Proposition~\ref{prop:ext_mean_value_theorem}
one can obtain global Lipschitz continuity in the initial
value for a smaller class of problems, as we will now
demonstrate.
\sgc{}Note that~\eqref{eq:def_mu_sigma_IRM} ensures that\cgs{}
for all $ x, y \in (0,\infty) $ with $ x \neq y $
it holds that
\begin{equation}
\begin{split}
&
  \tfrac{
    (x-y) (\mu(x)-\mu(y))
    +
    \frac{1}{2}(p-1)
      (\sigma(x) - \sigma(y))^2
  }
  {
    (x-y)^2
  }
\\ & =
  \gamma
  +
  \tfrac{
    \kappa (x^{-c} - y^{-c})
    - \alpha (x^a - y^a)
  }
  {
    (x-y)
  }
  +
  \tfrac{
    \beta^2 (p-1)
  }{2} 
  \left(
    \tfrac{
      x^b - y^b
    }
    {
      x-y
    }
  \right)^{ \! 2 }
  .
\end{split}
\end{equation}
This together with
Lemma~\ref{lem:nicht_treffen} and
Corollary~\ref{cor:powerful_est}
(see~\eqref{eq:powerful_est_one_dimensional})
implies that 
for all 
$ 
  p, x, y \in (0,\infty)
$,
$ t \in [0,\infty) $
it holds that
\begin{equation}
\label{eq:global_Lip_volatility}
\begin{split}
&
 \|
  X_t^x - X_t^y
 \|_{L^p(\Omega;\R)}
\\ & 
\leq
 | x - y |
 \exp\!
  \left(
    t \, \gamma
    +
    t 
    \sup_{
      \substack{
        u,v \in (0,\infty), 
      \\ 
        u\neq v
      }
    }
    \!
    \bigg[
      \tfrac{
	\kappa (u^{-c} - v^{-c})
	- \alpha(u^a - v^a)
      }
      {
	(u-v)
      }
      +
      \tfrac{ 
        \beta^2 ( p - 1 )
      }{ 2 } 
      \left(
	\tfrac{
	  u^b - v^b
	}
	{
	  u-v
	}
      \right)^{ \! 2 }
    \bigg]
  \right).
\end{split}
\end{equation}
It remains to identify the values of 
$\alpha$, $\beta$, $\kappa$, $a$, $b$, $c$ 
and $p$ for which the right-hand side of~\eqref{eq:global_Lip_volatility} is 
finite. 
First of all, observe that the right-hand side 
of~\eqref{eq:global_Lip_volatility}
is finite if $ p \in (0,1] $. Moreover,  
Proposition~\ref{prop:ext_mean_value_theorem} 
shows that for all $ p \in [1,\infty) $
it holds that
\begin{equation}
\label{eq:application_prop_ext}
\begin{aligned}
&
  \sup_{
    \substack{
      u, v \in (0,\infty), 
    \\
      u \neq v
    }
  }
    \!
    \bigg[
      \tfrac{
	\kappa (u^{-c} - v^{-c})
	- \alpha(u^a - v^a)
      }
      {
	(u-v)
      }
      +
      \tfrac{ 
        \beta^2 ( p - 1 )
      }{ 2 } 
      \left(
	\tfrac{
	  u^b - v^b
	}
	{
	  u-v
	}
      \right)^{ \! 2 }
    \bigg]                
\\ & =
  \sup_{
    x \in (0,\infty)
  }
  \left[
    \mu'(x)
    -
    \gamma
    +
    \tfrac{ p - 1 }{ 2 }
    \left| \sigma'(x) \right|^2
  \right]              
\\ & =
  \sup_{ x \in (0,\infty) }
  \left[
    - c \kappa x^{ - c - 1 }
    - a \alpha x^{a-1}
    + \tfrac{ ( p - 1 ) b^2 \beta^2 }{ 2 } x^{ 2 b - 2 }
  \right]
  .
\end{aligned}\end{equation}
This shows that for every $ p \in [1,\infty) $ it holds that the right-hand side of \eqref{eq:application_prop_ext}
is finite 
if one of the following three
cases is satisfied:
\begin{enumerate}
 \item $ 1 \leq b < \tfrac{a+1}{2}$
and $\alpha>0$,
 \item $ b = \tfrac{a+1}{2} $
 and $p \leq 1 + \tfrac{8\alpha a}{\beta^2 (a+1)^2}$,
 \item $\frac{1}{2}\leq b\leq 1$ and $c,\kappa>0$.
\end{enumerate}
In conclusion, for every $ p \in (0,\infty) $
it holds that the right-hand side of~\eqref{eq:global_Lip_volatility}
is finite for all $ x, y \in (0,\infty) $, $ t \in [0,\infty) $
if at least one of the following three
cases is satisfied:
\begin{enumerate}
 \item $ p \leq 1 $,
 \item $ 1 \leq b < \tfrac{ a + 1 }{ 2 } $ and $ \alpha > 0 $,
 \item $ b = \tfrac{a+1}{2} $
 and $p \leq 1 + \tfrac{8\alpha a}{\beta^2 (a+1)^2}$,
 \item $\frac{1}{2}\leq b\leq 1$ and $c,\kappa>0$.
\end{enumerate}
In particular for all $p \in (0,1]$, $b=\frac{1}{2}$, $\gamma<0<\delta$, $\alpha=\kappa = 0$ (i.e., in the setting of the CIR process) we obtain that the right-hand side of~\eqref{eq:global_Lip_volatility} is finite. \par 
Recall that 
estimate~\eqref{eq:global_Lip_volatility}
holds if the boundary point~$0$ is inaccessible. Note that 
in the appendix of the work Sabanis~\cite{Sabanis2013ECP} 
it is demonstrated 
for the specific case that 
$ \kappa = \delta = 0 $, $ a = 2 $, $ b = \frac{ 3 }{ 2 } $
that estimate~\eqref{eq:global_Lip_volatility}
has a finite right-hand side
for all 
$ x, \, y \in (0,\infty) $,
$ t \in [0,\infty) $,
$
  p \in \big( 0, 1 + \frac{ \alpha }{ \beta^2 } \big]
$.
The arguments above demonstrate 
in the specific case 
$ \kappa = \delta = 0 $, $ a = 2 $, $ b = \frac{ 3 }{ 2 } $
that estimate~\eqref{eq:global_Lip_volatility}
has a finite right-hand side
even for all 
$ x, \, y \in (0,\infty) $,
$ t \in [0,\infty) $,
$
  p \in \big( 0, 1 + \frac{ 16 \alpha }{ 9 \beta^2 } \big]
$.

\section{Wright-Fisher diffusion and Jacobi processes}
\label{ssec:Wright.Fisher.diffusion}

In biology, the Wright-Fisher diffusion is a model
for the relative frequency of type `A' in a panmictic population
of constant population size with two types `A' and `a';
see, e.g., ~Chapter 7 in Durrett~\cite{Durrett2008}.
Let $s\in\R$ denote the relative fitness advantage of
type `A',
let
$\rho_0\in[0,\infty)$ be the mutation rate
from type `a' to type `A',
let
$\rho_1\in[0,\infty)$ be the mutation rate
from type `A' to type `a'
and let $\beta\in(0,\infty)$ be the inverse of the  (effective) number of 
haploid individuals.
Let 
$ ( \Omega, \mathcal{F}, \P, ( \mathcal{F}_t )_{ t \in [0,\infty) } ) $
be a filtered probability space satisfying the usual conditions, 
let
$ W \colon [0,\infty) \times \Omega \to \R $
be a standard 
$ ( \mathcal{F}_t )_{ t \in [0,\infty) } $-Brownian motion,
and let
$X^x\colon[0,\infty)\times\Omega\to[0,1]$, $x\in(0,1)$,
be adapted stochastic processes
with continuous sample paths
satisfying
\begin{equation}  \label{eq:SDE.Wright.Fisher}
  X_{t}^x=x
  +\int_0^{t} \big[\rho_0(1-X_r^x)-\rho_1X_r^x+sX_r^x(1-X_r^x)\big]\,dr
  +\int_0^{t} \sqrt{\beta X_r^x(1-X_r^x)}\,dW_r
\end{equation}
$\P$-a.s.\ for all $t\in[0,\infty)$ and all $x\in(0,1)$.
In addition, define stopping times
$\tau_h^{x}\colon\Omega\to[0,\infty]$,
$x\in(0,1)$, 
$
  h \in [0,1]
$, 
\sgc{}by\cgs{}
\begin{equation}\label{eq:def_tauhx_WF} 
\tau_h^x:=\inf(\{t\in[0,\infty)\colon X_t^x=h\}
\cup\{\infty\})
\end{equation}
for all $x\in(0,1)$ and all $h\in[0,1]$.
Feller's boundary classification
(e.g., Theorem V.51.2 in~\cite{RogersWilliams2000b})
implies \sgc{}that\cgs{}
\begin{enumerate}
\item \sgc{} $\big[\,$\cgs{}$\P[\tau_0^x=\infty]=1$ for all $x\in(0,1)$
if and only if
$2\rho_0\geq \beta$\sgc{}$\,\big]$\cgs{} and
\item \sgc{} $\big[\,$\cgs{}$\P[\tau_1^x=\infty]=1$ for all $x\in(0,1)$
if and only if
$2\rho_1\geq \beta$\sgc{}$\big]\,$\cgs{}.
\end{enumerate}
Note that the processes in~\eqref{eq:SDE.Wright.Fisher} are also called Jacobi processes\sgc{};\cgs{} see, e.g.,~\cite{DelbaenShirakawa:2002}.

In the case 
$ \rho_0, \eta_1 \in [ \frac{ \beta }{ 4 }, \infty) $
the processes $\arcsin(\sqrt{X^x})$, $x\in(0,\infty)$, satisfy
an SDE with constant diffusion function
and non-increasing
drift function;
see, e.g.,
Neuenkirch \&\ Szpruch~\cite{NeuenkirchSzpruch2014}.
In the following calculation, we exploit this observation
together with the results in Section~\ref{sec:strong_stability}
to derive an estimate for the Lyapunov-type function
$ V \colon ( 0, 1 )^2 \to \R $
given \sgc{}by\cgs{} 
\begin{equation}\label{eq:WF-defV} 
  V( x, y ) 
  = 
  | \arcsin(\sqrt{x}) - \arcsin(\sqrt{y}) |^2
\end{equation}
for all $ x, y \in ( 0, 1 ) $.
Let
\sgc{}$ \mu \colon (0,1) \to \R $ and $ \sigma \colon (0,1) \to \R$\cgs{}
be given by
$
  \mu( x ) = \rho_0(1-x)-\rho_1x+sx(1-x)
$
and
$
  \sigma( x ) 
  =
  \sqrt{\beta x(1-x)}
$
for all $ x \in ( 0, 1 ) $
and let
$
  f\colon(0,\pi/2)\to\R
$ 
be given by
\begin{equation}  \begin{split}
  f(x)&=
  \left[ 
    \tfrac{
      \rho_0(1-y)
      -
      \rho_1 y
      +
      sy(1-y)
      -
      \frac{\beta}{4}(1-2y)
    }{
      \sqrt{y(1-y)}
    }
  \right]_{y=(\sin(x))^2}
  \\&
  =\tfrac{
      \rho_0(\cos(x))^2-\rho_1 (\sin(x))^2
      +
      s(\sin(x))^2(\cos(x))^2
      -
      \frac{\beta}{4}\left((\cos(x))^2-(\sin(x))^2\right)
    }
    {
      \sin(x)\cos(x)
    }
  \\&
  =\left(\rho_0-\tfrac{\beta}{4}\right)
   \tfrac{1}{\tan(x)}
   -\left(\rho_1-\tfrac{\beta}{4}\right)
    \tan(x)
   +\tfrac{s}{2}\sin(2x)
\end{split}     \end{equation}
for all 
$
  x \in(0,\pi/2)
$.
Next we infer from  $(0,\pi/2)\ni x\mapsto \tan(x)\in(0,\infty)$ being an 
increasing function
that 
\begin{equation}  \begin{split}
\label{eq:f.globally.one-sided}
  \tfrac{f(x)-f(y)}{x-y}
  \leq \tfrac{s}{2}\tfrac{\sin(2x)-\sin(2y)}{x-y}
  \leq |s|
\end{split}     \end{equation}
for all $x,y\in(0,\pi/2)$
in the case 
$\rho_0,\rho_1\in[\tfrac{\beta}{4},\infty)$.
Now we apply Lemma~\ref{lem:extended_drift}
and inequality~\eqref{eq:f.globally.one-sided}
to obtain that for all $ x, y \in (0,1) $ 
it holds that 
\begin{equation}
\begin{split}
&
  \tfrac{
    \|
      ( \overline{G}_{ \sigma } V)( x, y )
    \|^2
  }{
    | V(x,y) |^2
  }
=
  \tfrac{
    4
    \left[
      \frac{
	\sqrt{\beta x(1-x)}
      }{
	2 \sqrt{x(1-x)}
      }
     -
     \frac{
	\sqrt{\beta y(1-y)}
     }
     {
	2 \sqrt{y(1-y)}
     }
     \right]^2
  }{
    \left|
      \arcsin( \sqrt{x} ) 
      -
      \arcsin( \sqrt{y} )
    \right|^2
  }
=
  0
\end{split}
\end{equation}
and
\begin{equation}  \begin{split}
&  \tfrac{
    ( \overline{ \mathcal{G} }_{ \mu, \sigma } V)( x, y )
  }{
    V( x, y )
  }
=
  \tfrac{
   2
  }{
      \arcsin( \sqrt{x} ) 
      - 
      \arcsin( \sqrt{y} ) 
  }
  \left[
   \tfrac{ 
      \rho_0(1-x)-\rho_1 x+sx(1-x)
    }{
      2 \sqrt{ x (1-x) } 
    }    
    -
    \tfrac{ 
      \rho_0 ( 1 - y )
      -
      \rho_1 y + s y ( 1 - y )
    }{
      2 \sqrt{ y (1 - y ) }
    }
  \right]
\\ &
  \phantom{
  \tfrac{
      ( \overline{ \mathcal{G} }_{ \mu, \sigma } V)( x, y )
    }{
      V( x, y )
    }
  } 
  \quad
  +
  \tfrac{
    1
  }{
      \arcsin(\sqrt{x}) - \arcsin(\sqrt{y}) 
  }
    \left[
      \tfrac{ 
        \beta x ( 1 - x )
        ( 2 x - 1 )
      }{
        4 
        [ x ( 1 - x ) ]^{
          3 / 2
        }
      }
      -
      \tfrac{ 
	\beta y ( 1 - y )
        ( 2 y - 1 )
      }{ 
        4 
        [ y ( 1 - y ) ]^{ 3 / 2
        }
      }
    \right]
\\
&
=
\tfrac{
  \frac{1}{
    \sqrt{x(1-x)}
  }
  \left(
    \rho_0 ( 1 - x )
    -
    \rho_1 x 
    + 
    s x ( 1 - x )
    - 
    \frac{ \beta }{ 4 }( 1 - 2 x )
  \right)
  -
  \frac{ 1 }{ 
    \sqrt{ y ( 1 - y ) }
  }
  \left(
    \rho_0(1-y)
    -
    \rho_1 y
    +
    sy(1-y)
    -
    \frac{\beta}{4}(1-2y)
  \right)
}{ 
         \arcsin( \sqrt{ x } ) 
         - 
         \arcsin( \sqrt{ y } ) 
}
\\
&
=
\tfrac{f\left(\arcsin(\sqrt{x})\right)
      -f\left(\arcsin(\sqrt{y})\right)
      }
     {  \arcsin(\sqrt{x}) - \arcsin(\sqrt{y}) }
\leq |s|
  .
\end{split}     \end{equation}
Hence,
Proposition~\ref{prop:two_solution_supinside} 
with $O=(0,1)$ and
with $v=\infty=p=q=r$ shows
that in the case
$
  \rho_0, \rho_1 \in [\tfrac{\beta}{2},\infty)
$
it holds
for all $ t \in [0,\infty) $\sgc{}, \cgs{}$ x, y \in (0,1) $
that
\begin{equation}  \begin{split}
  \bigg\|
  \!
    \sup_{ r \in [0,t] }
     \left|
       \arcsin\!\big(
         \sqrt{X_{r}^x}
       \big)
       -
       \arcsin\!\big(
         \sqrt{ X_r^y }
       \big)
     \right|^2
  \bigg\|_{L^{\infty}(\Omega;\R)}
  \!\!\leq
  \left|
    \arcsin( \sqrt{x} ) 
    -
    \arcsin( \sqrt{y} )
  \right|^2
  e^{
    t |s|
  } 
  .
  \!
\end{split}     \end{equation}
Clearly, this implies that if
$
  \rho_0, \rho_1 \in [\tfrac{\beta}{2},\infty)
$,
then it holds
for all $ t \in [0,\infty) $, $ x, y \in (0,1) $
that
\begin{equation}  \begin{split}
\label{eq:estimate.sup.wright}
  \bigg\|
    \sup_{ r \in [0,t] }
     \left|
       \arcsin\!\big(
         \sqrt{X_{r}^x}
       \big)
       -
       \arcsin\!\big(
         \sqrt{ X_r^y }
       \big)
     \right|
  \bigg\|_{L^{\infty}(\Omega;\R)}
  \leq
  e^{
    \frac{ t |s| }{ 2 }
  } 
  \left|
    \arcsin( \sqrt{x} ) 
    -
    \arcsin( \sqrt{y} )
  \right|
  .
\end{split}     \end{equation}
This together with the estimates
\begin{equation}  \begin{split}
  |\arcsin(\sqrt{x})-\arcsin(\sqrt{y})|
  & =
  \left|
    \int_y^x
    \tfrac{ 1 }{ \sqrt{ 4 z (1 - z ) }
    }
    \,dz
  \right|
\leq
  \left| x - y \right|
  \left[ 
    \max_{z\in\{x,y\}}\tfrac{1}{\sqrt{4z(1-z)}}
  \right]
  ,
  \\
  \left|x-y\right|
  & =
  \left|
    \left[
      \sin( \arcsin( \sqrt{ x } ) )
    \right]^2
    -
    \left[
      \sin( \arcsin( \sqrt{ y } ) )
    \right]^2
  \right|
\\ &
  =
  \left|\int_{\arcsin(\sqrt{y})}^{\arcsin(\sqrt{x})}
          2\sin(z)\cos(z)\,dz\right|
\\ &
  \leq\left|\arcsin(\sqrt{x})-\arcsin(\sqrt{y})
       \right|
\end{split}     \end{equation}
for all $x,y\in(0,1)$
implies that if
$\rho_0,\rho_1\in[\tfrac{\beta}{2},\infty)$,
then it holds
for all $t\in[0,\infty)$, $x,y\in(0,1)$
that
\begin{equation}  \begin{split}
\label{eq:estimate.sup.wright.final}
  \bigg\|
    \sup_{r\in[0,t]}
     \left|X_{r}^x -X_{r}^y\right|
  \bigg\|_{L^{\infty}(\Omega;\R)}
  \leq
  \left[
    e^{ 
      \frac{ t | s | }{ 2 } 
    }
    \max_{z\in\{x,y\}}
    \tfrac{
      1
    }{
      \sqrt{ 4 z ( 1 - z ) }
    }
  \right]
  \left|x-y\right|
  .
\end{split}     \end{equation}

\chapter{Examples of SPDEs}
\label{sec:examples_SPDE}
The theory developed in Chapter~\ref{sec:strong_stability}
for proving regularity with respect to the 
initial value of the solution to a stochastic differential
equation in $\R^d$, $d\in \N$, can also be applied to
stochastic differential equations in an infinite-dimensional 
Hilbert space $H$ (i.e., an equation that typically describes 
a stochastic \emph{partial} differential equation). 
To be precise, one first verifies that Theorem~\ref{thm:UV2}
(or Corollary~\ref{cor:UV2}) may be applied to
the stochastic differential equations 
in $\R^n$, $n\in \N$, arising by considering Galerkin
projections of the stochastic differential equation in $H$.
If the solutions to the finite-dimensional Galerkin projections 
converge to the solution of the 
original SDE in the right sense, then regularity with respect to the 
initial value of the solution to the SDE in $H$ is obtained.
In this section we demonstrate this 
principle based on two examples, namely the stochastic Burgers
equation (Section~\ref{ssec:stochastic.Burgers.equation})
and the stochastic Cahn-Hilliard-Cook 
equation (Section~\ref{ssec:Cahn_Hilliard}). In addition,
in Section~\ref{ssec:stochastic.wave.equation}
we demonstrate that Theorem~\ref{thm:UV2} can be applied to 
the finite-dimensional
SDEs obtained by taking the Galerkin projection of a certain 
stochastic non-linear wave equation. 
\sgc{}Proving\cgs{}
convergence of these finite-dimensional processes
to the solution of the non-linear wave-equation in a suitable \sgc{}sense\cgs{}
is beyond the scope of this article.

\section{Setting}
\label{sec:setting_SPDE}

Throughout this section the following setting is used.
Let $ T \in (0,\infty) $,
let 
$ 
  ( 
    H, 
    \left< \cdot, \cdot \right>_{H} , 
    \left\| \cdot \right\|_{H} 
  )
$
and 
$ 
  ( 
    \mathcal{H}, 
    \left< \cdot, \cdot \right>_{\mathcal{H}} , 
    \left\| \cdot \right\|_{\mathcal{H}} 
  ) 
$
be non-trivial $\R$-Hilbert spaces,
let $(\HS(\mathcal{H},H),$ 
$\langle \cdot , \cdot \rangle_{\HS(\mathcal{H},H)}, \left\| \cdot \right\|_{\HS(\mathcal{H},H)})$ be the Hilbert space of 
Hilbert-Schmidt operators from $\mathcal{H}$ to $H$, 
let 
$
  ( V_{F,1}, \left\| \cdot \right\|_{V_{F,1}} ) 
$,
$
  ( V_{F,2}, \left\| \cdot \right\|_{V_{F,2}} ) 
$
be $\R$-Banach spaces,
\sgc{}assume
($D(A)\subseteq V_{F,1}\subseteq H$ continuously) 
or
($D(A)\subseteq H\subseteq V_{F,1}$ continuously),
assume
$H\subseteq V_{F,2}$ continuously,\cgs{}
let
$ A \colon D(A) \subseteq H \to H $ be a closed \sgc{}and\cgs{} densely defined 
linear operator,
let 
$
  F \colon V_{F,1} \rightarrow V_{F,2}
$
be Lipschitz continuous on bounded sets, let 
$
  B \colon H \rightarrow \HS(\mathcal{H},H)
$
be globally Lipschitz continuous, let $\xi \in \R$ satisfy
\begin{equation}\label{eq:spde_LipB}
  \varsigma \sgc{}=\cgs{}
  \sup_{ 
    x,y \in \sgc{}H,\, \cgs{}x \neq y
  }
  \sgc{}\frac{
    \| B(x) - B(y) \|_{ \HS( \mathcal{H}, H ) }
  }{
    \| x - y \|_H
  }\cgs{},
\end{equation} 
let 
$ ( \Omega, \mathcal{F}, \P, ( \mathcal{F}_t )_{ t \in [0,T] } ) $
be a filtered probability space satisfying the usual conditions, 
let $ ( W_t )_{ t \in [0,T] } $ be a cylindrical
$ \operatorname{Id}_{\mathcal{H}} $-Wiener process 
with respect to $ ( \mathcal{F}_t )_{ t \in [0,T] } $,
for all $n\in \N$
let $H_n \subseteq D(A)\cap V_{F,1}$ and $\calH_n\subseteq \calH$
be finite-dimensional subspaces
such that for all $v\in H_n$ it holds that
$F(v) \in H$, for all $n\in \N$ let $P_n \in L(H)$,
$Q_n\in L(\calH)$ 
be such that $P_n(H) = H_n$
and $Q_n(\calH) = \calH_n$, 
for all $n\in \N$ define
$\mu_n \colon H_n \rightarrow H_n$,
$\sigma_n \colon H_n \rightarrow \HS(\calH_n,H_n)$,
$W^{n} \colon [0,T]\times \Omega \rightarrow \calH_n$
by 
$
  \mu_n( v ) 
  = P_n( A v + F(v) )
$, 
$
  \sigma_n(v) u 
  = 
  P_n\big( 
    B(v) 
    u 
  \big)
$
for all 
$ v \in H_n $,
$ u \in \calH_n $
and $W^n_t = Q_n W_t $ for 
all $t\in [0,T]$
and for all $(x,n)\in H\times \N$ let
$ X^{x,n} \colon [0,T] \times \Omega \to H_n $
be an adapted stochastic process with 
continuous sample paths such that for all $t\in [0,T]$
it holds \sgc{}$\P$-a.s.\ \cgs{}that
\begin{equation}
 \int_{0}^{T}
  \| \mu_n( X^{x,n}_{s} ) \|_{H}
  +
  \| \sigma_n( X^{x,n}_{s} ) \|_{\HS(\mathcal{H}_n,H)}^{2}
 \,ds
 < \sgc{}\infty\cgs{}
\end{equation}
and
\begin{equation}
\label{eq:Galerkin}
  X_t^{x,n}
  =
  P_n x 
  +
  \int_0^t
  \mu_n( X^{x,n}_s ) \, ds
  +
  \int_0^t
  \sigma_n( X^{x,n}_s ) \, dW^{n}_s
  \sgc{}.\cgs{}
\end{equation}

\section[Stochastic Burgers equation]{Stochastic Burgers equation with a 
globally bounded diffusion coefficient and trace class noise}
\label{ssec:stochastic.Burgers.equation}

Assume the setting of Section~\ref{sec:setting_SPDE}\sgc{}, assume $H=L^2((0,1);\R)$, assume 
$D(A)=W^{2,2}((0,1);\R)\cap W^{1,2}_0((0,1);\R)$, assume for all $v\in D(A)$ that $Av = v''$, 
assume\cgs{}
$V_{F,1} = H$,
$
  V_{F,2} = 
  W^{-1,1}((0,1);\R)
  \subseteq
  W^{-\frac{3}{2},2}((0,1);\R)
$,
let $c\in \R$\sgc{},
assume\cgs{}
$
  F \colon V_{F,1}
  \to V_{F,2}
$
is given by
$
  F( v ) = -\tfrac{ c }{ 2 } \, ( v^2 )'
$
for all $ v \in H $,
\sgc{}let $\eta \in (0,\infty)$ satisfy\cgs{}
$
  \eta 
  \sgc{}=\cgs{}
  \sup_{ x \in H }
  \| B( x ) \|^2_{ \HS( \calH, H ) }
  \sgc{}\cgs{}
$
\sgc{}and let $ e_k \in H $, $ k \in \N $,
be given by\cgs{}
$ 
  e_k( y ) = \sqrt{2} \sin( k \pi y )
$
for all $ y \in ( 0, 1 ) $,
$ k \in \N $ (we follow the convention of 
identifying an element of $v\in L^2((0,1);\R)$
that admits a continuous version with 
this continuous function).
\sgc{}Note that for all $v\in D(A)$ it holds that 
$A v = -\sum_{k\in \N} (k\pi)^2  \langle v, e_k \rangle_{H} e_k$.\cgs{} 
For all $n\in \N$
\sgc{}assume\cgs{}
$
  (
    H_n, 
    \langle \cdot, \cdot \rangle_{H_n},
    \| \cdot \|_{H_n}
  )
  = 
  (
    \textrm{span}_{\R}(\{e_1,\ldots,e_n\}),
    \langle \cdot, \cdot \rangle_{H},
    \| \cdot \|_{H}
  )
$
and 
\sgc{}assume that $ P_n$ is\cgs{} the $H$-orthogonal projection 
on $H_n$, i.e., for all $(v,n) \in H\times \N$
it holds that
$
  P_n v 
  =
  \sum_{ k = 1 }^n
  \left< e_k, v \right>_H
  e_k.
$

Results on the existence of processes
$(X_t^{x,n})_{t\in [0,T]}$\sgc{}, $x \in H$, $n\in \N$,\cgs{}  satisfying~\eqref{eq:Galerkin} $\P$-a.s.\ for all $t\in [0,T]$
can be found in, \sgc{}e.g.,\cgs{}
\cite[Theorem 1.1 and Remark 3.1]{LiuRoeckner2010}. Note that the processes 
$(X_t^{x,n})_{t\in [0,T]}$\sgc{}, $x \in H$, $n\in \N$, \cgs{}are
Galerkin approximations of a solution to the stochastic Burgers equation
with a certain type of 
multiplicative trace-class noise\sgc{};\cgs{} see also Remark~\ref{rem:Burgers_Holder} below.

As for all $n\in \N$ it holds that 
$H_n\subseteq D(A)$, for all 
$n\in \N$ and all $x\in H_n$
it holds that $\langle x , F(x) \rangle_H = 0$
and $ \| x \|_H \leq \frac{1}{\pi} \| x' \|_{H}$.
It follows that for all
$ (n,\rho) \in \N \times [0,\infty) $
and for $ U \colon H_n \to \R $ given by
$ U( x ) = \rho \, \| x \|^2_H $
for all $ x \in H_n $,
one has that for every
$ x \in H_n $ it holds that
\begin{equation}
\label{eq:Burgers_exp_mom}
\begin{split}
&
  U'(x) \, \mu_n( x )
  +
  \tfrac{ 1 }{ 2 }
  \operatorname{tr}\!\big(
    \sigma_n( x ) \,
    \sigma_n( x )^* \,
    ( \operatorname{Hess} U)( x )
  \big)
  +
  \tfrac{ 1 }{ 2 }
  \,
  \|
    \sigma_n( x )^*
    ( \nabla U)( x )
  \|^2_{\calH_n}
\\ & =
  2 \rho
  \left< x, A x + F( x ) \right>_H
  +
  \rho \,
  \| \sigma_n( x ) \|^2_{ \HS( \calH, H_n ) }
  +
  2 \rho^2
  \|
    \sigma_n( x )^*
    x
  \|^2_{\calH_n}
\\ & \leq
  - 2 \rho \,
  \| 
    x'
  \|_{H}^2
  +
  \rho \,
  \| B( x ) \|^2_{ \HS( \calH, H ) }
  +
  2 \rho^2 \,
  \|
    B( x )^*
    x
  \|^2_{\calH}
\\ & \leq
  \rho \eta
  - 2 \rho \,
  \| 
    x'
  \|_{H}^2
  +
  2 \rho^2 \eta 
  \|
    x
  \|^2_H
\leq
  \rho \eta
  +
  2 \rho 
  \left[
    \tfrac{
      \rho \eta 
    }{ \pi }
    - 1
  \right]
  \| 
    x'
  \|_{H}^2
  .
\end{split}
\end{equation}
As for every $ n \in \N $ it holds that
$ 
  ( 
    H_n, 
    \left< \cdot, \cdot \right>_{ H_n }, 
    \left\| \cdot \right\|_{ H_n } 
  ) 
$ 
is isometrically isomorphic to 
$ 
  ( \R^n , \left< \cdot , \cdot \right>, \left\| \cdot \right\| )
$,
it follows from~\eqref{eq:Burgers_exp_mom}
with $ \rho = \frac{ \pi }{ 2 \eta } $
and Corollary~\ref{cor:exp_mom}
(with 
$
 \overline{U}(t,x)
 = 
 \frac{\pi}{2\eta}\| x' \|_H^2 - \frac{\pi}{2}
$
for all $(t,x)\in [0,T]\times H_n$)
that for all
$ (n,t) \in \N \times  [0,T]$
and all 
$ x \in H_n $
it holds that
\begin{equation}
\begin{split}
&
  \E\!\left[
    \exp\!\left(
      \tfrac{ \pi }{ 2 \eta }
      \| X^{ x, n }_t \|^2_H
      +
      \int_0^t
      \tfrac{ \pi }{ 2 \eta }
      \| 
	(X^{ x, n }_s)'
      \|_{H}^2 
      ds
    \right)
  \right]
\leq 
  e^{
    \frac{ \pi t }{ 2 }
    +
    \frac{ \pi }{ 2 \eta }
    \left\| x \right\|_H^2
  }
  .
\end{split}
\end{equation}
In the next step we note \sgc{}that\cgs{}
for all 
\sgc{}
$
  n \in \N
$,
$ 
    p,\varepsilon
  \in
  (0,\infty)
$, 
$ x, y \in H_n $
with $ x \neq y $
it holds that\cgs{}
\begin{equation}
\label{eq:Burgers_diffusion_est}
    \left[
      \tfrac{ 1 }{ ( 2 / p ) } 
      - 
      1
    \right]
  \tfrac{
    \left\|
      (
        \sigma_n( x ) - 
        \sigma_n( y )
      )^*
      ( x - y )
    \right\|^2_{\calH_n}
  }{
    \left\| x - y \right\|^4_H
  }
  \leq
    \tfrac{ \max\{0, p - 2 \} }{ 2 } 
    \varsigma^2
\end{equation}
where $\varsigma$ is as defined 
in~\eqref{eq:spde_LipB}\sgc{} and \cgs{}
\begin{equation}\label{eq:Burgers_h1}
\begin{split}
&
  \max\!\left\{
    0,
  \tfrac{ 
    \left<
      x - y ,
      \mu_n( x ) - 
      \mu_n( y )
    \right>_H
    +
    \frac{ 1 }{ 2 }
    \left\| 
      \sigma_n( x ) - \sigma_n( y ) 
    \right\|^2_{ \HS(  \calH_n, H ) }
  }{
    \left\|
      x - y
    \right\|^2_H
  }
  \right\}
\\ & \leq
  \max\!\left\{
    0,
  \tfrac{ 
    -
    \frac{ c }{ 4 }
    \langle
      ( x - y )^2 ,
      ( x + y )'
    \rangle_H
    -
    \|
       x' - y' 
    \|^2_{H}
  }{
    \left\|
      x - y
    \right\|^2_H
  }
    +
    \tfrac{ 1 }{ 2 }
    \varsigma^2
  \right\}
\\ & \leq
  \max\!\left\{
    0,
  \tfrac{ 
    | c |
    \|
      x' + y'
    \|_H
    \left\|
      x - y 
    \right\|_H
    \left\|
      x - y 
    \right\|_{ L^{ \infty }( (0,1); \R ) }
    -
    \|
      x' - y' 
    \|^2_{H}
  }{
    4
    \,
    \left\|
      x - y
    \right\|^2_H
  }
    +
    \tfrac{ 1 }{ 2 }
    \varsigma^2
  \right\}
\\ & \leq
  \max\!\left\{
    0,
    \tfrac{ \varepsilon 
    \|
      x' + y' 
    \|^2_{H}
    }{ 4 }
    +
  \tfrac{ 
    \frac{ c^2 }{ 16 \varepsilon }
    \left\|
      x - y 
    \right\|_{ L^{ \infty }( (0,1); \R ) }^2
    -
    \|
      ( x' - y' )
    \|^2_{H}
  }{
    \left\|
      x - y
    \right\|^2_H
  }
    +
    \tfrac{ 
      \varsigma^2
    }{ 2 }
  \right\}
  .
\end{split}
\end{equation}
By the Sobolev inequalities, 
\cite[Theorem 4.36]{Lunardi2009} and 
\cite[Section 8]{Grisvard1967},
and the interpolation property
for fractional powers of an operator
(see e.g.~\cite[Theorem 37.6]{sy02})
there exists a function 
$ \kappa \colon (0,\infty) \to (0,\infty) $
such that
for all $ r \in ( 0, \infty ) $ and all
$ u \in W^{1,2}_0((0,1)) = D((-A)^{\frac{1}{2}})$
it holds that
\begin{equation}\label{eq:Laplace_interpolation}
  \left\| u \right\|_{ L^{ \infty }( (0,1); \R ) }^2
\leq
  \kappa(r)
  \left\| u \right\|^2_H
  +
  r
    \|
      u' 
    \|^2_{H}.
\end{equation}
It follows that
for all
$ (n,q) \in \N \times (0,\infty)$
and all 
$ (x, y) \in H_n^2 $
with $ x \neq y $,
by substituting
$\varepsilon = \frac{\pi}{2\eta q}$
and $r= \frac{16 \varepsilon}{c^2} = \frac{8 \pi}{c^2\eta q}$
in \eqref{eq:Burgers_h1} and \eqref{eq:Laplace_interpolation}, 
that
\begin{equation}
\label{eq:Burgers_drift_est}
\begin{split}
&
  \max\!\left\{
    0,
  \tfrac{ 
    \left<
      x - y ,
      \mu_n( x ) - 
      \mu_n( y )
    \right>_H
    +
    \frac{ 1 }{ 2 }
    \left\| 
      \sigma_n( x ) - \sigma_n( y ) 
    \right\|^2_{ \HS(\calH_n, H ) }
  }{
    \left\|
      x-y
    \right\|^2_H
  }
  \right\}
\\ & \leq
    \tfrac{
      \frac{ \pi }{ 2 \eta }
      \|
	x'
      \|^2_{H}
      +
      \frac{ \pi }{ 2 \eta } 
    \|
      y'
    \|^2_{H}
    }{ 2 q
    }
    +
    \tfrac{ 
      c^2 \eta q
      \kappa(
        8 \pi / ( c^2 \eta q )
      )
    }{ 8 \pi }
    +
    \tfrac{ 
      \varsigma^2
    }{ 2 }
  .
\end{split}
\end{equation}
Combining 
\eqref{eq:Burgers_exp_mom},
\eqref{eq:Burgers_diffusion_est}
and~\eqref{eq:Burgers_drift_est}
and Corollary~\ref{cor:UV2} 
(for all $n\in \N$ with, in the setting of that corollary,
$\theta = 2$, $\rho = \infty$,
$q_{1,1} = q$,
$q_{0,0}=q_{0,1}=q_{1,0}=\infty$,
$p=p$, $r=r$, $\alpha_{i,j}=0$
for $(i,j)\in \{0,1\}^2$, 
$\beta_{1,1}=\frac{\pi}{2}$,
$\beta_{0,0}=\beta_{0,1}=\beta_{1,0}=0$,
$c_0 \equiv \frac{p-2}{2}\varsigma^2$,
\begin{equation}
  c_1 
  \equiv
  \tfrac{
    c^2 \eta q 
    \kappa\left( 
      \frac{8\pi}{c^2 \eta q}
    \right)
  }
  {8 \pi}
  +
  \tfrac{\varsigma^2}{2},
\end{equation}
$U_{0,0} = U_{0,1} = U_{1,0} = \overline{U}_{0,1} \equiv 0$
and
$ U_{1,1}(x) = \frac{\pi}{2\eta} \| x \|_H^2$,
$ \overline{U}_{1,1}(x) = \frac{\pi}{2\eta} \| x' \|_{H}^2$
for all $x\in H_n$)
implies that
for all 
$ (n,T,r) \in \N \times (0,\infty) \times (2,\infty)$,
all $p,q\in (r,\infty)$
such that 
$ 
  \frac{ 1 }{ p } + \frac{ 1 }{ q }
  =
  \frac{ 1 }{ r } 
$, 
and all $( x,y ) \in H_n^2$
it holds that
\begin{equation}
\label{eq:Burger_Galerkin}
\begin{split}
&
  \left\|
    \sup\nolimits_{ t \in [ 0, T ] }
      \left\| X^{ x, n }_t - X^{ y, n }_t 
      \right\|_H
  \right\|_{
    L^r( \Omega; \R )
  }
\\ & \leq  
  \frac{ 
      \| x - y \|_H
  }{ 
    \sqrt{
      1 - 2 / p
    }
  }
  \exp\!\left(
    \tfrac{ 
      c^2 \eta q T
      \kappa(
        8 \pi / ( c^2 \eta q )
      )
    }{ 8 \pi }
    +
    \tfrac{ 
      ( p - 1 ) T
      \varsigma^2
    }{ 2 }
    +
    \tfrac{
      \pi T
    }{
      2 q 
    }
    +
    \tfrac{ 
      \pi 
      \left\|
        x
      \right\|_H^2
      +
      \pi 
      \left\|
        y
      \right\|_H^2
    }{ 
      4 \eta q
    }
  \right) 
  .
\end{split}
\end{equation}

\begin{remark}\label{rem:Burgers_Holder}
Theorem 1.1 and Remark 3.1 \sgc{}in\cgs{}~\cite{LiuRoeckner2010} guarantee 
for all $x\in H$ that there exists an adapted stochastic process
$ X^x \colon [0,T] \times \Omega \to H $ with 
continuous sample paths 
such that $\operatorname{im}(X^x|_{(0,T]\times \Omega})\subseteq V_{F,1}$
and such that for all 
$t\in  [0,T]$
it holds \sgc{}$\P$-a.s.\ \cgs{}that
\begin{equation}
  \int_{0}^{t}
    \| e^{A(t-s)}F(X_s^x) \|_{V_{F,2}} 
    +
    \| e^{A(t-s)}B(X_s^x) \|_{\HS(\mathcal{H},H)}^2
  \,ds
  <\sgc{}\infty\cgs{}
\end{equation}
and
\begin{equation}
\label{eq:Burgers_SPDE}
  X_t^x
  =
  e^{ A t } x
  +
  \int_0^t
  e^{ A ( t - s ) } F( X_s^x ) \, ds
  +
  \int_0^t
  e^{ A ( t - s ) } B( X_s^x ) \, dW_s
  \sgc{}.\cgs{}
\end{equation}
Note that the processes 
$(X_t^x)_{t\in [0,T]}$, $x\in H$, \sgc{}provide\cgs{}
mild \sgc{}solutions\cgs{} to the stochastic Burgers equation
with a certain type of 
multiplicative trace-class noise.
One can use~\eqref{eq:Burger_Galerkin}, Fatou's lemma and a standard localization argument
(see\sgc{}, e.g.,~\cite[Corollary 4.4]{Coxetal2020}\cgs{} for specific
types of noise) to show that
that
for all
\sgc{}$x,y\in H $, 
$ T \in(0,\infty)$,
$r\in (2,\infty)$,
\cgs{}$p,q\in (r,\infty)$
\sgc{}with\cgs{}
$ 
  \frac{ 1 }{ p } + \frac{ 1 }{ q }
  =
  \frac{ 1 }{ r } 
$
it holds that
\begin{equation}
\begin{split}
&
  \left\|
    \sup\nolimits_{ t \in [ 0, T ] }
      \left\| 
        X^x_t - X^y_t 
      \right\|_H
  \right\|_{
    L^r( \Omega; \R )
  }
\\ & 
\leq  
  \frac{ 
      \| x - y \|_H
  }{ 
    \sqrt{
      1 - 2 / p
    }
  }
  \exp\!\left(
    \tfrac{ 
      c^2 \eta q T
      \kappa(
        8 \pi / ( c^2 \eta q )
      )
    }{ 8 \pi }
    +
    \tfrac{ 
      ( p - 1 ) T
      \varsigma^2
    }{ 2 }
    +
    \tfrac{
      \pi T
    }{
      2 q 
    }
    +
    \tfrac{ 
      \pi 
      \left\|
        x
      \right\|_H^2
      +
      \pi 
      \left\|
        y
      \right\|_H^2
    }{ 
      4 \eta q
    }
  \right) 
  .
\end{split}
\end{equation}
\end{remark}

\section{Cahn-Hilliard Cook equation with trace class noise}
\label{ssec:Cahn_Hilliard}

Assume the setting 
of Section~\ref{sec:setting_SPDE}\sgc{}, assume $H = L^2((0,1);\R)$,\cgs{}
let
$ L \colon D(L) \subseteq H \to H $
be the Laplacian with homogeneous
Neumann boundary conditions on $ (0,1) $,
\sgc{}in particular,\cgs{}
$ 
  D( L ) = \{ v \in W^{2,2}( (0,1), \R ) \colon v'(0)=v'(1)=0 \},
$
and assume $ A = - L^2$, \sgc{}in particular\cgs{},
$ D( A ) = D( L^2 ) $,
assume $V_{F,1} = W^{\frac{1}{3},2}((0,1);\R)$
and assume $V_{F,2}=W^{-2,2}((0,1);\R)$,
let $c\in (0,\infty)$ and 
assume $F\colon V_{F,1}\rightarrow V_{F,2}$
is given by $F(v)=cL(v^3-v)$ for all $v\in V_{F,1}$ 
\sgc{}and let $ e_k \in H $, $ k \in \N $,
be given by\cgs{}
$
  e_1( x ) = 1
$
and
$
  e_{ k + 1 }( x ) = \sqrt{ 2 } \cos( k \pi x )
$
for all 
$ x \in (0,1) $,
$ k \in \N $.
\sgc{}Note that for all $v\in D(A)$ it holds that 
$A v = -\sum_{k\in \N} (k\pi)^2  \langle v, e_{k+1} \rangle_{H} e_{k+1}$.\cgs{} For all $n\in \N$
\sgc{}assume\cgs{}
$
  (
    H_n, 
    \langle \cdot, \cdot \rangle_{H_n},
    \| \cdot \|_{H_n}
  )
  = 
  (
    \textrm{span}_{\R}(\{e_1,\ldots,e_n\}),
    \langle \cdot, \cdot \rangle_{H},
    \| \cdot \|_{H}
  )
$
and for all \sgc{}$v \in H$, $n \in \N$\cgs{} \sgc{}assume\cgs{}
$
  P_n v 
  =
  \sum_{ k = 1 }^n
  \left< e_k, v \right>_H
  e_k.
$
\sgc{}Let $\eta_{\varepsilon}\in \R$, $\varepsilon \in (0,\infty)$ satisfy for all $\varepsilon \in (0,\infty)$ that\cgs{}
\begin{equation}
\begin{split}
& \label{eq:CHC_B_assumption}
  \eta_{ \varepsilon } \sgc{}=\cgs{}
  \sup_{ v \in H }
  \left[
    \|
      ( I - P_1 ) B( v ) 
    \|^2_{ \HS( \mathcal{H}, H ) }
    -
    \varepsilon 
    \left(
      \| 
	( ( I - P_1 ) v )^2 
      \|_H^2
      -
      \|
	( I - P_1 ) v
      \|_H^2
      \left\| v \right\|^2_H
    \right)
  \right]
  \sgc{}.\cgs{}
\end{split}
\end{equation}

\sgc{}If $\calH = H$ and there exists a strictly positive trace class operator $Q\in L(H)$ that commutes with $L$ (see, e.g., Appendix~B
in Pr\'{e}v\^{o}t \&\R\"{o}ckner~\cite{PrevotRoeckner2007}) such that $ ( B( v ) u )( x ) = ( \sqrt{ Q } u )( x ) $ for all $ x \in (0,1)$, $u,v \in H$,
then it holds that
\begin{enumerate}
\item 
$
  \sup_{ 
    v, w \in H , v \neq w
  }
  \frac{
    \| B(v) - B(w) \|_{ \HS(  \calH, H ) }
  }{
    \| v - w \|_H
  }
  = 0
$ and
\item there exists an $\eta_{\varepsilon} \in (0,\infty)$ such that \eqref{eq:CHC_B_assumption} is satisfied.
\end{enumerate}
In that case results on the existence and uniqueness
of processes 
$(X^{x,n}_t)_{t\in [0,T]}$, \sgc{}$x\in H$, $n\in \N$,\cgs{}
satisfying~\eqref{eq:Galerkin} $\P$-a.s.\ for all $t\in [0,T]$ 
can be found in, e.g., Da Prato \&\
Debussche~\cite[Theorem 2.2 and Remark 2.2]{DaPratoDebussche1996}.\cgs{}
Note that the processes $(X^{x,n}_t)_{t\in [0,T]}$, \sgc{}$x\in H$, $n\in \N$,\cgs{} are Galerkin approximations of a solution to a certain type of Cahn-Hilliard-Cook equation\sgc{};\cgs{} see also Remark~\ref{rem:CHC_Holder} below. 

In order to apply Corollary~\ref{cor:UV2}
we define the projection $ \sgc{}\mathcal{P}\cgs{} \in L(H) $
and the operator $\invL \in L(H) $ 
by setting
\begin{equation}
  \sgc{}\mathcal{P}\cgs{} v 
= 
  ( I - P_1 ) v 
=
  v - e_1 \left< e_1, v \right>_H
\quad \textrm{ and } \quad
  \invL v = 
  -
  \sum_{ k = 2 }^{ \infty }
  (k-1)^{ -2 } \, \pi^{ -2 } 
  \left< e_k , v \right>_H 
  e_k
\end{equation}
for all $ v \in H $.
Note that
for all $ v \in D(L) $ 
it holds that
\begin{equation}
  \invL L v = 
  L \invL v = 
  \sgc{}\mathcal{P}\cgs{} v
  .
\end{equation}
Now observe that 
Young's inequality proves that
for all
\sgc{}$ 
  \delta 
  \in 
  ( 0 , \infty )
$,
$
 n \in \N
$, 
$
  x\in H_n
$\cgs{}
it holds that
\begin{equation}
\label{eq:CHC_key_estimate_Neumann}
\begin{split}
  -
  c \,
  \big\langle 
    \sgc{}\mathcal{P}\cgs{} x, x^3 
  \big\rangle_H
& =
  -
  c \,
  \left\langle 
    \sgc{}\mathcal{P}\cgs{} x, \big( \sgc{}\mathcal{P}\cgs{} x + P_1 x \big)^3 
  \right\rangle_H
\\ & 
= 
  -
  c \,
  \left\langle 
    \sgc{}\mathcal{P}\cgs{} x, \big( \sgc{}\mathcal{P}\cgs{} x \big)^3 
  \right\rangle_H
  -
  3 \, c \,
  \left\langle 
    \sgc{}\mathcal{P}\cgs{} x, \big( \sgc{}\mathcal{P}\cgs{} x \big)^2 ( P_1 x ) 
  \right\rangle_H
\\ & \quad 
  -
  3 \, c \,
  \big\langle 
    \sgc{}\mathcal{P}\cgs{} x, \big( \sgc{}\mathcal{P}\cgs{} x \big) ( P_1 x )^2 
  \big\rangle_H
  -
  c \,
  \big\langle 
    \sgc{}\mathcal{P}\cgs{} x, ( P_1 x )^3 
  \big\rangle_H
\\ &
=
  -
  c \,
  \big\| 
    ( \sgc{}\mathcal{P}\cgs{} x )^2 
  \big\|_H^2
  -
  3 \, c \,
  \left\langle 
    \sgc{}\mathcal{P}\cgs{} x, \big( \sgc{}\mathcal{P}\cgs{} x \big)^2
  \right\rangle_H
  \left< e_1, x \right>_H
\\ & \quad 
  -
  3 \, c \,
  \big\|
    \sgc{}\mathcal{P}\cgs{} x
  \big\|_H^2
  \left| \left< e_1, x \right>_H \right|^2 
  -
  c \,
  \big\langle 
    \sgc{}\mathcal{P}\cgs{} x, e_1
  \big\rangle_H
  ( \left< e_1, x \right>_H )^3 
\\ &  
\leq
  -
  c \,
  \big\| 
    ( \sgc{}\mathcal{P}\cgs{} x )^2 
  \big\|_H^2
  +
  \left[ 
    \sqrt{ 2 c \delta} 
    \big\|
      ( \sgc{}\mathcal{P}\cgs{} x )^2 
   \big\|_H
  \right]
  \left[
  3\sqrt{\tfrac{ c } { 2 \delta} } 
  \big\|
    \sgc{}\mathcal{P}\cgs{} x
  \big\|_H
  \left| \left< e_1, x \right>_H \right|
  \right]
\\ & \quad
  -
  3 \, c \,
  \big\|
    \sgc{}\mathcal{P}\cgs{} x
  \big\|_H^2
  \left| \left< e_1, x \right>_H \right|^2 
\\ & 
\leq
  -
  c \left( 1 - \delta \right)
  \big\| 
    ( \sgc{}\mathcal{P}\cgs{} x )^2 
  \big\|_H^2
  -
  3 \, c \,
  ( 1 - \tfrac{ 3 }{ 4 \delta } )
  \,
  \big\|
    \sgc{}\mathcal{P}\cgs{} x
  \big\|_H^2
  \left| \left< e_1, x \right>_H \right|^2 
\\ &  
=
  -
  c \left( 1 - \delta \right)
  \big\| 
    ( \sgc{}\mathcal{P}\cgs{} x )^2 
  \big\|_H^2
  -
  3 \, c \,
  ( 1 - \tfrac{ 3 }{ 4 \delta } )
  \,
  \big\|
    \sgc{}\mathcal{P}\cgs{} x
  \big\|_H^2
  \left[ 
    \| x \|^2_H 
    -
    \big\| \sgc{}\mathcal{P}\cgs{} x \big\|^2_H
  \right]
\\ &
\leq
  c 
  \left[
    \delta 
    +
    2 
    - \tfrac{ 9 }{ 4 \delta }
  \right]
  \big\| 
    ( \sgc{}\mathcal{P}\cgs{} x )^2 
  \big\|_H^2
  -
  3 \, c \,
  ( 1 - \tfrac{ 3 }{ 4 \delta } )
  \,
  \big\|
    \sgc{}\mathcal{P}\cgs{} x
  \big\|_H^2
  \,
  \| x \|^2_H 
  \,
  .
\end{split}
\end{equation}
In the next step observe that
for all $n\in \N$
and all $x\in H_n$
it holds that
\begin{equation}\label{eq:CHC_U_auxEst2}
\begin{split}
  \langle 
    \invL x, \mu_n( x ) 
  \rangle_H
& = 
  \langle
    \invL x, P_n( A x + F(x) )
  \rangle_H
=
  \langle
    \invL P_n x, A x + F(x)
  \rangle_H
\\ & =
  -
  \left\langle
    \invL x, L^2 x 
  \right\rangle_H
  +
  \langle
    \invL x, 
    F(x)
  \rangle_H
=
  -
  \big\langle
    \sgc{}\mathcal{P}\cgs{} x, L x 
  \big\rangle_H
  +
  c
  \big\langle
    \sgc{}\mathcal{P}\cgs{} x, 
    x^3 - x
  \big\rangle_H
\\ & =
  \big\langle
    ( - L )^{ \frac{1}{2} } \sgc{}\mathcal{P}\cgs{} x, ( - L )^{ \frac{1}{2} } 
    \sgc{}\mathcal{P}\cgs{} x 
  \big\rangle_H
  +
  c
  \big\langle
    \sgc{}\mathcal{P}\cgs{} x, 
    x^3 
  \big\rangle_H
  -
  c
  \big\langle
    \sgc{}\mathcal{P}\cgs{} x, 
    x 
  \big\rangle_H
\\ & =
  \big\|
    ( - L )^{ \frac{1}{2} } \sgc{}\mathcal{P}\cgs{} x
  \big\|^2_H
  +
  c
  \big\langle
    \sgc{}\mathcal{P}\cgs{} x, 
    x^3 
  \big\rangle_H
  -
  c
  \big\|
    \sgc{}\mathcal{P}\cgs{} x
  \big\|^2_H
  \,
  .
\end{split}
\end{equation}
\sgc{}For the remainder of the proof for every $\rho, \hat{\rho} \in (0,\infty)$ let $U_{\rho,\hat{\rho}} \colon H \rightarrow [0,\infty)$ satisfy for all $x\in H$ that\cgs{}
\begin{align}\label{eq:CHC_def_U}
  U_{\rho,\hat{\rho}}(x) 
 = 
  \tfrac{\rho}{2}
  \big\|(-\invL)^{\inv{2}} x\big\|_H^2 
  + \tfrac{\hat{\rho}}{2}
  \big\| \sgc{}\mathcal{P}\cgs{} x \big\|_H^2
  \sgc{}.\cgs{}
\end{align}
It follows from estimate
\eqref{eq:CHC_U_auxEst2}
that for all\sgc{}
$
    \rho, \hat{\rho} \in (0,\infty)
$,
$ 
  n \in \N
$,
$
    x\in H_n
$\cgs{}
it holds that
\begin{equation}
\begin{split}
&
  ( \mathcal{G}_{ \mu_n , \sigma_n } U_{\rho,\hat{\rho}} )( x )
  +
  \tfrac{ 1 }{ 2 } 
  \, \left\| \sigma_n( x )^* ( \nabla U_{\rho,\hat{\rho}} )( x ) \right\|^2_{\mathcal{H}_n} 
\\ &   =
\left[
  -
  \rho
  \,
  \langle 
    \invL x, \mu_n( x ) 
  \rangle_H
  +
  \tfrac{ \rho }{ 2 }
  \,
  \big\|
    ( - \invL )^{ \frac{1}{2} } \sigma_n( x ) 
  \big\|^2_{ \HS( \mathcal{H}_n, H) }
\right]
\\ & \quad 
  +
\left[
  \hat{ \rho }
  \,
  \big\langle 
    \sgc{}\mathcal{P}\cgs{} x , 
    \mu_n( x ) 
  \big\rangle_H
  +
  \tfrac{ \hat{ \rho } }{ 2 }
  \,
  \big\|
    \sgc{}\mathcal{P}\cgs{} \sigma_n( x ) 
  \big\|^2_{ \HS( \mathcal{H}_n, H ) ) }
\right]
  +
  \tfrac{ 1 }{ 2 } 
  \,
  \big\| 
    \sigma_n( x )^* 
    \big[ 
      \rho 
      \,
      ( - \invL ) \, x
      +
      \hat{ \rho } 
      \,
      \sgc{}\mathcal{P}\cgs{} 
      \, x
    \big] 
  \big\|^2_{\mathcal{H}_n} 
\\ & \leq
\rho
\left[
  c \,
  \big\|
    \sgc{}\mathcal{P}\cgs{} x
  \big\|^2_H
  -
  \|
    x'
  \|^2_H
  -
  c \,
  \big\langle 
    \sgc{}\mathcal{P}\cgs{} x, x^3 
  \big\rangle_H
  +
  \tfrac{ 1 }{ 2 }
  \,
  \big\|
    ( - \invL )^{ \frac{1}{2} } B( x ) 
  \big\|^2_{ \HS( \mathcal{H}, H ) }
\right]
\\ & \quad 
  +
  \hat{ \rho }
\left[
  c \,
  \| x' \|^2_H
  -
  \| x'' \|^2_H
  -
  c
  \,
  \langle 
    x' , 
    ( x^3 )'
  \rangle_H
  +
  \tfrac{ 1 }{ 2 }
  \,
  \big\|
    \sgc{}\mathcal{P}\cgs{} B( x ) 
  \big\|^2_{ \HS( \mathcal{H}, H ) }
\right]
\\ &  \quad
  +
  \tfrac{ 1 }{ 2 }
  \,
  \big\| 
    B( x )^* 
    \big[
      \hat{ \rho } \sgc{}\mathcal{P}\cgs{} - \rho \invL
    \big]
    x
  \big\|^2_{\mathcal{H}}
  .
\end{split}
\end{equation} 
Combining this with 
\eqref{eq:CHC_key_estimate_Neumann}
and the estimate \sgc{}that for all $v\in H$ it holds that\cgs{}
$
  \big\| ( - \invL )^{ \frac{1}{2} } v 
  \big\|^2_H
\leq \sgc{}
  \big\| 
    \sgc{}\mathcal{P}\cgs{} v 
  \big\|_H^2\cgs{}
$
proves that
for all \sgc{}
$
 \rho,\hat{\rho},\delta
  \in 
  (0,\infty)
$,
$ n \in \N$,
$ x \in H_n $\cgs{}
it holds that
\begin{equation}
\label{eq:CHC_U_est}
\begin{split}
&
  ( \mathcal{G}_{ \mu_n , \sigma_n } U_{\rho,\hat{\rho}} )( x )
  +
  \tfrac{ 1 }{ 2 } 
  \, \| \sigma_n( x )^* ( \nabla U_{\rho,\hat{\rho}} )( x ) \|^2_{\mathcal{H}_n} 
\\ &\leq
\rho
\left[
  c \,
  \big\|
    \sgc{}\mathcal{P}\cgs{} x
  \big\|^2_H
  -
  \left\|
    x'
  \right\|^2_H
  +
  c 
  \left[
    \delta 
    + 2
    - \tfrac{ 9 }{ 4 \delta }
  \right]
  \big\| 
    ( \sgc{}\mathcal{P}\cgs{} x )^2 
  \big\|_H^2
  -
  3 \, c \,
  ( 1 - \tfrac{ 3 }{ 4 \delta } )
  \,
  \big\|
    \sgc{}\mathcal{P}\cgs{} x
  \big\|_H^2
  \,
  \| x \|^2_H 
\right]
\\ & \quad 
  +
  \hat{ \rho }
\left[
  c \,
  \| x' \|^2_H
  -
  \| x'' \|^2_H
  -
  3 \, c
  \,
  \|
    x' x
  \|_H^2
\right]
  +
  \tfrac12
    ( \rho + \hat{ \rho } ) 
  \,
  \big\|
    \sgc{}\mathcal{P}\cgs{} B( x ) 
  \big\|^2_{ \HS( \mathcal{H}, H ) }
\\ & \quad
  +
  \tfrac12
    \| B(x) \|^2_{ \HS( \mathcal{H}, H ) }
    \big\| 
      \hat{ \rho } \sgc{}\mathcal{P}\cgs{} - \rho \invL 
    \big\|_{ L(H) }^2
    \big\| 
      \sgc{}\mathcal{P}\cgs{} x
    \big\|^2_H
  .
\end{split}
\end{equation} 
From~\eqref{eq:CHC_B_assumption} it
follows that 
for all \sgc{}
$
 \rho,\hat{\rho},\delta
  \in 
  (0,\infty)
$,
$ n \in \N$,
$ x \in H_n $\cgs{}
it holds that
\begin{equation}
\label{eq:CHC_U_est2}
\begin{split}
&
  ( \mathcal{G}_{ \mu_n , \sigma_n } U_{\rho,\hat{\rho}} )( x )
  +
  \tfrac{ 1 }{ 2 } 
  \, \| \sigma_n( x )^* ( \nabla U_{\rho,\hat{\rho}} )( x ) \|^2_{\mathcal{H}_n}
\\ &  \leq
\rho
\left[
  -
  \left\|
    x'
  \right\|^2_H
  +
  c 
  \left[
    \delta 
    + 2 
    - \tfrac{ 9 }{ 4 \delta }
  \right]
  \big\| 
    ( \sgc{}\mathcal{P}\cgs{} x )^2 
  \big\|_H^2
  -
  3 \, c \,
  ( 1 - \tfrac{ 3 }{ 4 \delta } )
  \,
  \big\|
    \sgc{}\mathcal{P}\cgs{} x
  \big\|_H^2
  \,
  \| x \|^2_H 
\right]
\\ & \quad
  +
  \hat{ \rho }
  \left[
    c \,
    \| x' \|^2_H
    -
    \| x'' \|^2_H
    -
    3 \, c
    \,
    \|
      x' x
    \|_H^2
  \right]
\\ &  \quad
  +
  \tfrac12
  ( \rho + \hat{ \rho } ) 
  \,
  \left[ 
    \eta_{ \varepsilon }
    +
    \varepsilon \,
    \big\| 
      ( \sgc{}\mathcal{P}\cgs{} x )^2 
    \big\|_H^2
    +
    \varepsilon
    \,
    \big\|
      \sgc{}\mathcal{P}\cgs{} x
    \big\|_H^2
    \left\| x \right\|^2_H
  \right]
\\ & \quad
  +
  \tfrac12
  \| B(x) - B(0) + B(0) \|^2_{ \HS( \mathcal{H}, H ) }
  \big\| 
    \hat{ \rho } \sgc{}\mathcal{P}\cgs{} - \rho \invL 
  \big\|_{ L(H) }^2
  \big\| 
    \sgc{}\mathcal{P}\cgs{} x
  \big\|^2_H
  +
  \rho \, c \,
  \big\|
    \sgc{}\mathcal{P}\cgs{} x
  \big\|^2_H
\\ & \leq
  \left[ 
    \tfrac{
      ( \rho + \hat{ \rho } ) \, \varepsilon
    }{
      2
    }
    +
    \rho
    \,
    c 
    \left[
      \delta 
      + 2
      - \tfrac{ 9 }{ 4 \delta }
    \right]
  \right]
  \big\| 
    ( \sgc{}\mathcal{P}\cgs{} x )^2 
  \big\|_H^2
\\ &  \quad
  +
  \left[
    \hat{ \rho }
    \,
    c 
    -
    \rho
  \right]
  \| x' \|^2_H
  -
  \hat{ \rho }
  \left[
    \| x'' \|^2_H
    +
    3 \, c
    \,
    \|
      x' x
    \|_H^2
  \right]
\\ &  \quad
  +
  \left[
    \tfrac{ 
      ( \rho + \hat{ \rho } ) \, \varepsilon
    }{ 2 }
    -
    \rho \, c \,
    ( 3 - \tfrac{ 9 }{ 4 \delta } )
  \right]
  \big\|
    \sgc{}\mathcal{P}\cgs{} x
  \big\|_H^2
  \,
  \| x \|^2_H 
\\ & \quad
  +
  \left[
    \varsigma^2 \, \| x \|^2_H
    +
    \| B(0) \|^2_{ \HS( \mathcal{H}, H ) }
  \right]
    \big\| 
      \hat{ \rho } \sgc{}\mathcal{P}\cgs{} - \rho \invL 
    \big\|_{ L(H) }^2
    \,
    \big\| 
      \sgc{}\mathcal{P}\cgs{} x
    \big\|^2_H 
  +
  \rho \, c \,
  \big\|
    \sgc{}\mathcal{P}\cgs{} x
  \big\|^2_H
  +
  \tfrac{ 
    \eta_{ \varepsilon }
    \,
    ( \rho + \hat{ \rho } ) 
  }{ 2 }
  .
\end{split}
\end{equation} 
Hence, we obtain that
for all\sgc{}
$
 \rho,\hat{\rho},\delta
  \in 
  (0,\infty)
$,
$ n \in \N$,
$ x \in H_n $\cgs{}
it holds that
\begin{equation}
\label{eq:CHC_U_est3}
\begin{split}
&
  ( \mathcal{G}_{ \mu_n , \sigma_n } U_{\rho,\hat{\rho}} )( x )
  +
  \tfrac{ 1 }{ 2 } 
  \, \| \sigma_n( x )^* ( \nabla U_{\rho,\hat{\rho}} )( x ) \|^2_{\mathcal{H}_n}
\\ & \leq
  \left[ 
    \tfrac{
      ( \rho + \hat{ \rho } ) \, \varepsilon
    }{
      2
    }
    +
    \rho
    \,
    c 
    \left[
      \delta 
      +
        2 - \tfrac{ 9 }{ 4 \delta }
    \right]
  \right]
  \big\| 
    ( \sgc{}\mathcal{P}\cgs{} x )^2 
  \big\|_H^2
\\ &  \quad
  +
  \left[ 
    \rho \, c 
    +
    \| B(0) \|^2_{ \HS( \mathcal{H}, H ) }
    \,
    \big\| 
      \hat{ \rho } \sgc{}\mathcal{P}\cgs{} - \rho \invL 
    \big\|_{ L(H) }^2
  \right]
  \big\|
    \sgc{}\mathcal{P}\cgs{} x
  \big\|^2_H
\\ &  \quad
  +
  \left[
    \hat{ \rho }
    \,
    c 
    -
    \rho
  \right]
  \| x' \|^2_H
  -
  \hat{ \rho }
  \left[
    \| x'' \|^2_H
    +
    3 \, c
    \,
    \|
      x' x
    \|_H^2
  \right]
\\ &\quad
  +
  \left[
    \tfrac{ 
      ( \rho + \hat{ \rho } ) \, \varepsilon
    }{ 2 }
    +
    \varsigma^2 \, 
    \big\| 
      \hat{ \rho } \sgc{}\mathcal{P}\cgs{} - \rho \invL 
    \big\|_{ L(H) }^2
    -
    \rho \, c \,
    ( 3 - \tfrac{ 9 }{ 4 \delta } )
  \right]
  \big\|
    \sgc{}\mathcal{P}\cgs{} x
  \big\|_H^2
  \,
  \| x \|^2_H 
  +
  \tfrac{ 
    \eta_{ \varepsilon }
    \,
    ( \rho + \hat{ \rho } ) 
  }{ 2 }
  .
\end{split}
\end{equation} 
\sgc{}For all $\rho, \hat{\rho} \in (0,\infty)$\cgs{} one
has
$
  \big\| 
    \hat{\rho} \sgc{}\mathcal{P}\cgs{} - \rho \invL
  \big\|_{L(H)}
  =
  \hat{\rho} + \frac{\rho}{\pi^2}.
$
It follows that for all 
\sgc{}$\rho, \hat{\rho} \in (0,\infty)$\cgs{}
it holds that
\begin{equation}\label{eq:CHC_normPx}
\begin{split}
&  \left[ 
    \rho \, c 
    +
    \| B(0) \|^2_{ \HS( \mathcal{H}, H ) }
    \,
    \big\| 
      \hat{ \rho } \sgc{}\mathcal{P}\cgs{} - \rho \invL 
    \big\|_{ L(H) }^2
  \right]
  \big\|
    \sgc{}\mathcal{P}\cgs{} x
  \big\|^2_H
\\ & \qquad 
 \leq
  \tfrac{\rho c}{ 160 }
  \big\|
    ( \sgc{}\mathcal{P}\cgs{} x )^2
  \big\|^2_H
  +
  40 \rho c
  \left(
    1
    +
    \tfrac{ 
      (\pi^2\hat{\rho}+\rho )^2
    }{ 
      \pi^{4} \rho c 
    } 
    \| B(0) \|_{\HS(\mathcal{H},H)}^2
  \right)^2.
\end{split}
\end{equation}
\sgc{}For the remainder of the proof for every $\rho, \hat{\rho} \in (0,\infty)$ let $\overline{U}_{\rho,\hat{\rho}} \colon D(L) \rightarrow \R$ satisfy for all $x\in H$ that\cgs{}
\begin{equation}\label{eq:CHC_def_Ubar}
 \overline{U}_{\rho,\hat{\rho}}(x) 
 = \hat{ \rho }
    \| x'' \|^2_H
    +
   \tfrac{\rho c}{16}
   \| x \|_H^2 \| \sgc{}\mathcal{P}\cgs{}x\|_H^2
   \sgc{}.\cgs{}
\end{equation}
\sgc{}For the remainder of the proof let $\beta\colon (0,\infty)^2 \rightarrow (-\infty,\infty]$
satisfy for all $\rho, \hat{\rho} \in (0,\infty)$ that\cgs{}
\begin{equation}
\begin{split}
\label{eq:CHC_U_est4}
\beta(\rho,\hat{\rho}) =
  \sup_{ n \in \N }
  \sup_{ x \in H_n }
  \left[
  ( \mathcal{G}_{ \mu_n , \sigma_n } 
    U_{\rho,\hat{\rho}}
  )( x )
  +
  \tfrac{ 1 }{ 2 } 
  \, \| \sigma_n( x )^* ( \nabla U_{\rho,\hat{\rho}} )( x ) \|^2_{\mathcal{H}_n}
  +
  \overline{U}_{\rho,\hat{\rho}}( x )
  \right]
  \sgc{}.\cgs{}
\end{split}
\end{equation}
From~\eqref{eq:CHC_U_est3}
and~\eqref{eq:CHC_normPx} it follows that there exist
\sgc{}$ \rho,\hat{\rho} \in (0,\infty)$\cgs{} (which we fix for the remainder of this section)
such that $\beta(\rho,\hat{\rho})<\infty$. \sgc{}For all $n\in \N$\cgs{} this and Corollary~\ref{cor:exp_mom} \sgc{} with \cgs{} $U=U_{\rho,\hat{\rho}}|_{H_n}$, $\bar{U}=\bar{U}_{\rho,\hat{\rho}}|_{H_n}-\beta(\rho,\hat{\rho})$ and $\alpha=0$ imply that \sgc{}for all $x \in H$, $t\in [0,T]$ it holds that\cgs{}
\begin{equation}\begin{aligned}
& \E \sgc{}\bigg[\cgs{}
 \text{exp}\!\left( 
    \tfrac{\rho}{2}
    \| (-\invL)^{\frac{1}{2}} X_t^{x,n} \|_H^2 
    + 
    \tfrac{\hat\rho}{2} 
    \| \sgc{}\mathcal{P}\cgs{}X_t^{x,n} \|_H^2 \right)
\\ & \quad \cdot 
    \text{exp}\!\left(
    +
    \int_0^{t}
    \big(
    \hat{\rho} 
    \| L X_s^{x,n} \|_H^2
    +
    \tfrac{\rho c}{16} 
    \| X_s^{x,n} \|_H^2 
    \| \sgc{}\mathcal{P}\cgs{}X_s^{x,n} \|_H^2 
    \big)\,ds
 \right)
\sgc{}\bigg]\cgs{}
\\ &
    \leq \text{exp}\! \left( 
        \tfrac{\rho}{2}
        \| (-\invL)^{\frac{1}{2}} P_n x\|^2_H
        +
        \tfrac{\hat{\rho}}{2}
        \| \sgc{}\mathcal{P}\cgs{} P_n x \|^2_H
    \right)
    \leq \text{exp}\! \left( 
        \tfrac{\rho}{2}
        \| (-\invL)^{\frac{1}{2}} x\|^2_H
        +
        \tfrac{\hat{\rho}}{2}
        \| \sgc{}\mathcal{P}\cgs{}  x \|^2_H
    \right).
\end{aligned}
\end{equation}
\sgc{}The\cgs{} next step is to prove that 
condition~\eqref{eq:UV2_est_mu_sigma} 
in Corollary~\ref{cor:UV2} is
satisfied for $\mu=\mu_n,\, \sigma=\sigma_n$,
$n\in \N$. To this end first observe 
that for all $n\in \N$\sgc{}, $x,y\in H_n$\cgs{}
it holds that
\begin{equation}
 \begin{split}
& \left\langle
      x-y,
      c L (x^3-y^3)
  \right\rangle_H
  =
  - c 
  \left\langle
    (x-y)',
    [(x-y)(x^2+xy+y^2)]'
  \right\rangle_H
\\ & 
  =  
  - c 
  \left\langle
    [(x-y)']^2,
    x^2+xy+y^2
  \right\rangle_H 
  - c 
  \left\langle
    (x-y)',
    (x-y)(2x'x+x'y+xy'+2y'y)
  \right\rangle_H
\\ & 
  \leq     
  - 
  \tfrac{c}{4} 
  \left\langle
    [(x-y)']^2,
    (|x|+|y|)^2
  \right\rangle_H 
  + 2 c 
  \left\langle
    |(x-y)'|,
    |x-y|(|x|+|y|)(|x'|+|y'|)
  \right\rangle_H
\\ &  
  \leq     
  - 
  \tfrac{c}{4} 
  \left\|
    (x-y)'
    (|x|+|y|)
  \right\|_H^2
  + 2 c 
  \left\|
    (x-y)'
    (|x|+|y|)
  \right\|_H
  \left\|
    (x-y)(|x'|+|y'|)
  \right\|_H  
\\ &    
  \leq     
  4c
  \left\|
    (x-y)(|x'|+|y'|)
  \right\|_H^2
  \end{split}
\end{equation}
where Young's inequality is applied
in the final estimate. It follows
from this estimate that for all $n\in \N$\sgc{}, $x,y\in H_n$\cgs{} it holds that
\begin{equation}\label{eq:CHC_est1}
 \begin{split}
&   \langle 
      x-y,
      \mu_n(x) - \mu_n(y)
    \rangle_H
    +
    \tfrac{1}{2}
    \|
      \sigma_n(x) - \sigma_n(y)
    \|_{\HS(\mathcal{H}_n,H)}^2
\\ &
  =
    -  
    \left\langle
      x-y,
      L^2(x-y)
    \right\rangle_H
    +
    \left\langle
      x-y,
      c L (x^3-y^3-x+y)
    \right\rangle_H
\\ & \quad
    +
    \tfrac{1}{2}
    \|
      P_n(B(x) - B(y))
    \|_{\HS(\mathcal{H}_n,H)}^2
\\ & 
  \leq 
    -  
    \|
      L(x-y)
    \|_H^2
    +
    4c
    \|
      (x-y)(|x'|+|y'|)
    \|_H^2
    +
    c
    \|
      (x-y)'
    \|_H^2
    +
    \tfrac{\varsigma^2}{2}
    \|
      x-y
    \|_H^2
\\ &
  \leq 
    -  
    \|
      (x-y)''
    \|_H^2
    +
    8c
    \left(
      \|
	x'
      \|_{L^{\infty}((0,1);\R)}^2
      +
      \|
	y'
      \|_{L^{\infty}((0,1);\R)}^2
    \right)
    \|
      x-y
    \|_H^2
    +
    c
    \|
      (x-y)'
    \|_H^2
\\ & \quad 
    +
    \tfrac{\varsigma^2}{2}
    \|
      x-y
    \|_H^2.
 \end{split}
\end{equation}
By integrating by parts and 
applying H\"older's and Young's
inqualities we obtain that for
all $x\in D(A)$ it holds that
\begin{equation}
  c
  \left\| x' \right\|_H^2
\leq 
  \tfrac{c^2}{4}
  \left\| x \right\|^2_H
  +
  \left\| x'' \right\|^2_H.
\end{equation}
Moreover, by similar arguments as used
to obtain~\eqref{eq:Laplace_interpolation},
it follows that 
there exists a function 
$ \kappa \colon (0,\infty) \rightarrow (0,\infty)$
such that
for all $ x \in D(A) $,
\sgc{}$ q,\rho,\hat{\rho} \in (0,\infty)$\cgs{}
it holds that
\begin{equation}\label{eq:kappa_est_CHC}
\begin{aligned}
  8c 
  \left\| 
    x' 
  \right\|^2_{ L^{ \infty }((0,1);\R)}
& =
  8c
  \big\|
    (\sgc{}\mathcal{P}\cgs{}x)'
  \big\|^2_{ L^{ \infty }((0,1);\R)}
\leq 
  c^2 
  \kappa\big(
    \tfrac{\hat{\rho}}{2q}
  \big)
  \big\|
    \sgc{}\mathcal{P}\cgs{}x
  \big\|_{H}^2
  +
  \tfrac{\hat{\rho}}{2q}
  \left\|
    x''
  \right\|_{H}^2
\\ & \leq 
  \tfrac{
    8 c^3 q \left(   
      \kappa\left(
	\frac{\hat{\rho}}{2q}
      \right)
    \right)^2
  }{
    \rho 
  }
  +
  \tfrac{1}{2q}
  \overline{U}_{\rho,\hat{\rho}}(x).
\end{aligned}
\end{equation}
Inserting this into~\eqref{eq:CHC_est1}
proves that for all \sgc{}
$
  n \in \N
$,
$
x,y\in H_n$,
$ 
    q,\rho,\hat{\rho}
  \in (0,\infty)
$\cgs{}
it holds that 
\begin{equation}\label{eq:CHC_est2}
 \begin{split}
&  \max\left\{
    0,
    \tfrac{
      \langle 
	x-y,\,
	\mu_n(x) - \mu_n(y)
      \rangle_H
      +
      \frac{1}{2}
      \|
	\sigma_n(x) - \sigma_n(y)
      \|_{\HS(\mathcal{H}_n,H)}^2
    }{
      \| x - y \|^2_H
    }
  \right\}
\\ & \qquad 
  \leq
    \tfrac{\varsigma^2}{2}
    +
    \tfrac{
      8 c^3 q \left(   
	\kappa\left(
	  \frac{\hat{\rho}}{2q}
	\right)
      \right)^2
    }{
      \rho 
    }
    +
    \tfrac{c^2}{4}
    +
    \tfrac{1}{2q}
    \big( 
      \overline{U}_{\rho,\hat{\rho}}(x)
      +
      \overline{U}_{\rho,\hat{\rho}}(y)
    \big).
 \end{split}
\end{equation}
    It follows from Corollary~\ref{cor:UV2} \sgc{}with\cgs{}
$\theta = 2$, $\rho = \infty$,
$q_{1,1} = q$,
$q_{0,0}=q_{0,1}=q_{1,0}=\infty$,
$p=p$, $r=r$, $\alpha_{i,j}=0$
for $(i,j)\in \{0,1\}^2$, 
$\beta_{1,1}=\beta(\rho,\hat{\rho})$,
$\beta_{0,0}=\beta_{0,1}=\beta_{1,0}=0$,
$c_0 \equiv \frac{p-2}{2}\varsigma^2$,
\begin{equation}
  c_1
  \equiv
    \tfrac{\varsigma^2}{2}
    +
    \tfrac{
      8 c^3 q \left(   
	\kappa\left(
	  \frac{\hat{\rho}}{2q}
	\right)
      \right)^2
    }{
      \rho 
    }
    +
    \tfrac{c^2}{4},
\end{equation}
$ U_{0,0} = U_{0,1} = U_{1,0} = \overline{U}_{0,1} \equiv 0$,
and
$ U_{1,1}= U_{\rho,\hat{\rho}}$,
$ \overline{U}_{1,1} = \overline{U}_{\rho,\hat{\rho}}$)
that for all \sgc{}
$
  n \in \N
$,
$
 x,y\in H_n
$,
$
 \rho,\hat{\rho} \in (0,\infty)
$,
$   
 r \in (2,\infty)
$,
$ p,q \in (r,\infty)$
with
$
  \frac{1}{p} + \frac{1}{q} = \frac{1}{r} 
$\cgs{} 
it holds that
\begin{equation}\label{eq:CHC_Galerkin}
 \begin{split}
&   \big\| 
      \sup\nolimits_{t\in [0,T]}
      \| X_t^{x,n} - X_t^{y,n} \|_H
    \big\|_{L^r(\Omega;\R)}
\\ & \qquad 
    \leq 
    \frac{
      \| x - y \|_H 
    }
    { \sqrt{1-\sgc{}(2/p)\cgs{}} }
    \exp\!\left(
      \left(
	\tfrac{ (p-1) \varsigma^2 }{2}
	+	
	\tfrac{
	  8 c^3 q \left(   
	    \kappa\left(
	      \frac{\hat{\rho}}{2q}
	    \right)
	  \right)^2
	}{
	  \rho 
	}
	+
	\tfrac{c^2}{4}
	+
	\tfrac{ \beta(\rho,\hat{\rho})  }{ q }
      \right)T
    \right)
\\ & \qquad \quad
    \cdot
    \exp\!\bigg(
      \tfrac{
	\rho \left\| (- \invL)^{1/2} x \right\|_H^{2}
	+
	\hat{\rho} \| \sgc{}\mathcal{P}\cgs{} x \|_H^{2}
      }
      { 4q }
      +
      \tfrac{
	\rho \left\| (- \invL)^{1/2} y \right\|_H^{2}
	+
	\hat{\rho} \| \sgc{}\mathcal{P}\cgs{} y \|_H^{2}
      }
      { 4q }
    \bigg).
 \end{split}
\end{equation}

\begin{remark}\label{rem:CHC_Holder}
Assume that $\calH = H$ and \sgc{}assume\cgs{}
that there exists a strictly positive trace class operator $ Q \in L(H) $ that commutes with $L$ and satisfies
$ ( B( v ) u )( \xi) = ( \sqrt{ Q } u )( \xi ) $
for all $ \xi \in (0,1)$, $ (u,v) \in H^2 $. Then~\cite[Theorem 2.2 and Remark 2.2]{DaPratoDebussche1996} implies for all $x\in H$ that there exists an adapted stochastic process
$ X^x \colon [0,T] \times \Omega \to H $ with 
continuous sample paths 
such that $\operatorname{im}(X^x|_{(0,T]\times \Omega})\subseteq V_{F,1}$
and such that for all 
$t\in  [0,T]$
it holds \sgc{}$\P$-a.s.\cgs{} that 
\begin{equation}
  \int_{0}^{t}
    \| e^{A(t-s)}F(X_s^x) \|_{V_{F,2}} 
    +
    \| e^{A(t-s)}B(X_s^x) \|_{\HS(\mathcal{H},H)}^2
  \,ds
  <
  \sgc{}\infty\cgs{}
\end{equation}
and
\begin{equation}
\label{eq:CHC_SPDE}
  X_t^x
  =
  e^{ A t } x
  +
  \int_0^t
  e^{ A ( t - s ) } F( X_s^x ) \, ds
  +
  \int_0^t
  e^{ A ( t - s ) } B( X_s^x ) \, dW_s
 \sgc{}.\cgs{}
\end{equation}
Note for all $x\in H$ that the process 
$(X_t^x)_{t\in [0,T]}$ provides
a mild solution to the Cahn-Hilliard-Cook type SPDE
\begin{equation}
  d X_t( \xi ) 
  =
  \left[ 
    - 
    \tfrac{ \partial^4 }{ \partial \xi^4 } X_t( \xi )
    +
    c
    \,
    \tfrac{ \partial^2 }{ \partial \xi^2 } 
    \left[
      \left( X_t( \xi ) \right)^3
      -
      X_t( \xi )
    \right]
  \right] dt
  +
  \sqrt{ Q }
  \,
  dW_t(\xi)
\end{equation}
for $ \xi \in (0,1) $, $ \sgc{}t \in [0,T]\cgs{} $
equipped with the Neumann and the non-flux boundary conditions
$
  X_t'(0) = X_t'(1) =
  X_t'''(0) = X_t'''( 1 )
  = 0
$
for $ t \in [0,T] $ and initial condition $X_t(0)=x$.
One can use~\eqref{eq:CHC_Galerkin}, Fatou's lemma and a standard localization argument
(see~\cite[Corollary 4.4]{Coxetal2020}) to show that
for all \sgc{}
$x,y \in H$,
$ \rho,\hat{\rho},T \in  (0,\infty)$,
$ r \in (2,\infty)$,
$p,q\in (r,\infty)$ 
with 
$ 
  \frac{ 1 }{ p } + \frac{ 1 }{ q }
  =
  \frac{ 1 }{ r } 
$\cgs{}
it holds that
\begin{equation}
\begin{split}
 &   \big\| 
      \sup\nolimits_{t\in [0,T]}
      \| X_t^{x} - X_t^{y} \|_H
    \big\|_{L^r(\Omega;\R)}
\\ & \qquad 
    \leq 
    \frac{
      \| x - y \|_H 
    }
    { \sqrt{1-2/p} }
    \exp\!\left(
      \left(
	\tfrac{ (p-1) \varsigma^2 }{2}
	+
	\tfrac{
	  8 c^3 q \left(   
	    \kappa\left(
	      \frac{\hat{\rho}}{2q}
	    \right)
	  \right)^2
	}{
	  \rho 
	}
	+
	\tfrac{c^2}{4} 
	+
	\tfrac{ \beta(\rho,\hat{\rho})  }{ q }
      \right)T
      \right)
\\ & \qquad \quad
    \cdot
    \exp\!\bigg(
      \tfrac{
	\rho \left\| (- \invL)^{1/2} x \right\|_H^{2}
	+
	\hat{\rho} \| \sgc{}\mathcal{P}\cgs{} x \|_H^{2}
      }
      { 4q }
      +
      \tfrac{
	\rho \left\| (- \invL)^{1/2} y \right\|_H^{2}
	+
	\hat{\rho} \| \sgc{}\mathcal{P}\cgs{} y \|_H^{2}
      }
      { 4q }
    \bigg).
\end{split}
\end{equation}
\end{remark}

\section{Non-linear wave equation}\label{ssec:stochastic.wave.equation}
\sgc{}Assume the setting of\cgs{} Section~\ref{sec:setting_SPDE}\sgc{}, let $D\subseteq \R^2$ be\cgs{} open and bounded\sgc{}, assume\cgs{} that \sgc{}$D$ has Lipschitz boundary\cgs{},  let
\begin{equation}
\begin{aligned}
&  ( H, \langle \cdot, \cdot \rangle_H , \left\| \cdot \right\|_H ) 
\\ & =
  ( 
     W^{1,2}( D; \R) \times L^2( D; \R ) , 
    \left< \cdot , \cdot \right>_{
      W^{1,2}( D; \R) \times L^2( D; \R )
    } 
    ,
    \left\| \cdot \right\|_{
      W^{1,2}( D; \R) \times L^2( D; \R ) 
    } 	
  )
,
\end{aligned}
\end{equation}
let
$ A \colon D(A) \subseteq H \rightarrow H$ be given by
$
  D(A) 
  =
  (W^{2,2}(D;\R) \cap W^{1,2}_0(D;\R)) \times W^{1,2}_0(D; \R),
$
$
  A((u_1,u_2))
  =
  (
    u_2,
    \Delta u_1 
  )
$ 
for all $(u_1,u_2)\in D(A)$, let $p\in (3,4)$
and assume that $V_{F,1}=L^p(D;\R)\times L^2(D;\R)$ and \sgc{}$V_{F,2}=W^{1,2}(D;\R)\times L^{p/3}(D;\R)$\cgs{},
assume that
$ F\colon V_{F,1} \rightarrow V_{F,2} $
is given by
$
  F((u_1,u_2)) = (0, - u_1^3)
$
for all $(u_1,u_2) \in V_{F,1}$,
let $B_2\in C(H, \HS(\calH,L^2(D;\R)))$,
\sgc{}let $\eta\in\R$
satisfy
\begin{equation}\label{eq:B_bdd}
  \eta \sgc{}=\cgs{} \sup_{v\in H} \| B_2(v) \|_{\HS(\calH,L^2(D;\R))},
\end{equation}
and assume $B \colon H\to \HS(\calH, H)$ satisfies
$
  B(u_1,u_2) = (0, B_2(u_1,u_2))
$
for all $(u_1,u_2)\in H$.\cgs{}
Let 
$e_k \in W^{2,2}(D;\R) \cap W^{1,2}_0(D;\R)$, $k\in \N$,
be such that $(e_k)_{k\in \N}$ is total in 
$W^{1,2}_0(D;\R)$ (hence
also in $L^2(D;\R)$). For
$n\in \N$ we assume that
$
  (
    H_n, 
    \langle \cdot, \cdot \rangle_{H_n},
    \| \cdot \|_{H_n}
  )
  = 
  (
    \left[\textrm{span}(\{e_1,\ldots,e_n\})\right]^2,
    \langle \cdot, \cdot \rangle_{L^2(D;\R^2)},
    \| \cdot \|_{L^2(D;\R^2)}
  ),
$
that $\tilde{P}_n \in L(L^2(D;\R))$ is 
the $L^2(D;\R)$-orthogonal projection of $L^2(D;\R)$ onto 
$\textrm{span}(\{e_1,\ldots,e_n\})$ and that
$P_n \in L(H)$ is defined by
$
  P_n((u_1,u_2))
  =
  (\tilde{P}_n u_1, \tilde{P}_n u_2)
$
for all $(u_1,u_2)\in H$. 

We define
$
  U \colon H \rightarrow \R
$
by 
\begin{equation}\label{eq:wave_U}
  U((u_1,u_2))
  =
  \tfrac14 \| u_1 \|_{L^4(D;\R)}^4
  +
  \tfrac12 \| \nabla u_1 \|_{L^2(D;\R^2)}^2
  +
  \tfrac12 \| u_2 \|_{L^2(D;\R)}^2
\end{equation}
for all $(u_1,u_2)\in H$.
Observe that 
for all $n\in \N$
and all $(u_1,u_2), (v_1,v_2), (w_1,w_2) \in H_n$
it holds that
\begin{equation}
 (U|_{H_n})'((u_1,u_2)) ((v_1,v_2)) = 
  \langle \sgc{}(u_1)^3\cgs{} - \Delta u_1, v_1 \rangle_{L^2(D;\R)}
 + 
 \langle u_2, v_2 \rangle_{L^2(D;\R)} 
\end{equation}
and  
\begin{equation}\begin{aligned}
 & (U|_{H_n})''((u_1,u_2)) ((v_1, v_2), (w_1, w_2)) 
  \\ & = 
   3\langle \sgc{}(u_1)^2\cgs{} w_1, v_1 \rangle_{L^2(D;\R)}
  -
  \langle \Delta w_1, v_1 \rangle_{L^2(D;\R)}
  +
  \langle w_2, v_2 \rangle_{L^2(D;\R)} 
  \sgc{}.\cgs{}
\end{aligned}
\end{equation}
\sgc{}Combining this with~\eqref{eq:B_bdd} shows that for all $n\in \N$, $(u_1,u_2)\in H_n$ it holds that\cgs{}
\begin{equation}\label{eq:wave_GmusigmaU}
\begin{aligned}
  \left(\calG_{\mu_n,\sigma_n}U\right)((u_1,u_2))
&  =
  \langle 
    \sgc{}(u_1)^3\cgs{} - \Delta u_1,
    u_2
  \rangle_{L^2(D;\R)}
  +
  \langle
    u_2,
    \tilde{P}_n\Delta u_1 - \tilde{P}_n \sgc{}(u_1)^3\cgs{}
  \rangle_{L^2(D;\R)}
\\ & \quad  
  +
  \tfrac12
  \|
    \tilde{P}_n B_2(u_1,u_2) 
  \|^2_{\HS(\calH_n, L^2(D;\R))}
\\ &
  =
  \tfrac12
  \|
    \tilde{P}_n B_2(u_1,u_2) 
  \|^2_{\HS(\calH_n,L^2(D;\R))}
  \leq 
  \tfrac{
    \eta^2 
  }{2}
  \sgc{}.\cgs{}
\end{aligned}
\end{equation}
It follows from~\cite[Lemma 2.2]{gk96b}
that for all $(x,n)\in H\times \N$
the $(\calF_{t})_{t\in [0,T]}$-adapted
processes $(X^{x,n}_t)_{t\in [0,T]}$ 
satisfying~\eqref{eq:Galerkin} exist.
Note that the processes $(X^{x,n}_t)_{t\in [0,T]}$ are Galerkin approximations to a certain type of non-linear stochastic wave equation\sgc{};\cgs{} see also Remark~\ref{rem:wave_Holder} below. Moreover, observe that~\eqref{eq:wave_GmusigmaU} and Corollary~\ref{cor:exp_mom} \sgc{}with\cgs{} $U=U$, $\bar{U}\equiv-\frac{\eta^2}{2}$ and $\alpha =0$ imply for all \sgc{}$x \in H$, $t\in [0,T]$\cgs{} that 
\begin{equation}
 \E \left[
    \exp( U( X^{x,n}_t) )
 \right] 
 \leq 
 e^{t \eta^2 /2 + U( P_n x)}.
\end{equation}
\sgc{}We\cgs{} now define
$ V\colon H \times H \rightarrow \R$
by
\begin{equation}\label{eq:NWE_defV}
 V(u,v) = \| u - v \|_{H}^2 
\end{equation}
for all \sgc{}$u,v \in H$. Observe that~\eqref{eq:spde_LipB} and~\eqref{eq:NWE_defV} show that 
for all $n\in \N$, $p\in (2,\infty)$, $u = (u_1,u_2), v= (v_1,v_2) \in H_n$\cgs{}
it holds that
\begin{equation}\label{eq:wave_Gsigma}
\begin{aligned}
  \left[ 
    \tfrac{1}{2/p} -1
  \right]
  \frac{ 
    \| 
      (\overline{G}_{\sigma_n}V)(u,v)
    \|_{H}^2
  }{
    | V(u,v) |^2
  } 
  & =
  \left[ 
    \tfrac{p-2}{2}
  \right]
  \sgc{}\left[
  \frac{ 
    \left\| 
      (\sigma_n(u) - \sigma_n(v))^*(u-v)
    \right\|_{L^2(D;\R)}^2
  }{
    \| u-v \|_{H}^4
  }
  \right]\cgs{}
\\ & \leq
  \left[ 
    \tfrac{p-2}{2}
  \right]
  \sgc{}\left[
    \frac{
    \| B_2 \|_{\operatorname{Lip}(H,\HS(\calH,L^2(D;\R)))}^2
    \| u - v \|_{H}^4
  }{
    | V(u,v) |^2
  }
  \right]\cgs{}
\\ & =
  \left[ 
    \tfrac{p-2}{2}
  \right]
    \varsigma^2
\end{aligned}
\end{equation}
\sgc{}and\cgs{}
\begin{equation}\label{eq:waveGmusigma1}
\begin{aligned}
&
  \frac{ 
  (\overline{\calG}_{\mu_n,\sigma_n}V)
    (u,v)
 }{
  V(u,v)
 }
\\ & 
 =
 \frac{ 
      \left\langle
	\left(
	  \begin{array}{c}
	    u_1 - v_1 - \Delta(u_1-v_1)
	  \\ 
	    u_2 - v_2
	  \end{array}
	\right)
	,
	\left(
	  \begin{array}{c}
	    u_2 - v_2
	  \\ 
	    \tilde{P}_n (\Delta u_1 - u_1^3) 
	    -
	    \tilde{P}_n (\Delta v_1 - v_1^3)
	  \end{array}
	\right)
      \right\rangle_{L^2(D;\R)}
 }{
    V(u,v)
 }
\\ & \quad
 +
 \frac{ 
   \| 
    \tilde{P}_n 
      (B_2(u) - B_2(v)) 
  \|_{\HS(\calH_n, L^2(D;\R))}^2
 }{	
  2 V(u,v)
 }
\\ & 
  \leq 
  \frac{
    \frac{1}{2}
    \left(
      \| u_1 - v_1 \|_{L^2(D;\R))}^2
      +
      \| u_2 - v_2 \|_{L^2(D;\R))}^2
    \right)
    +
    \left\langle
      u_2 - v_2,
      - u_1^3 + v_1^3 
    \right\rangle_{L^2(D;\R))}
 }{
      V(u,v)
 }
\\ & \quad
  +
  \frac{ 
   \varsigma^2
      \| u -v \|_{H}^2
   }{	
    2 | V(u,v) |
   }
.
\end{aligned}  
\end{equation}
By similar arguments as used
to obtain~\eqref{eq:Laplace_interpolation}
we have that there exists a 
$\kappa\colon (0,\infty)\rightarrow (0,\infty)$
such that for all $r\in (0,\infty)$,
$(u_1,u_2)\in H$
it holds that
\begin{equation}\label{eq:kappa_wave}
\begin{aligned}
  & \| u_1 \|_{L^6(D;\R)}^2
  \leq 
  \kappa(r) \| u_1 \|_{L^2(D;\R)}^2
  +
  r\| \nabla u_1 \|_{L^2(D;\R^2)}^2
\\ &  \leq
  \tfrac{(\kappa(r))^2|D|}{2r}
  +
  \tfrac{r}{2} \| u_1 \|_{L^4(D;\R)}^{4}
  +
  r \| \nabla u_1 \|_{L^2(D;\R^2)}^2
  \leq  
  \tfrac{(\kappa(r))^2|D|}{2r}
  +
  2rU((u_1,u_2)). 
\end{aligned}
\end{equation}
It follows that 
for all $n\in \N$,
\sgc{}$u = (u_1,u_2), v=(v_1,v_2)\in H_n$,
$r\in (0,\infty)$\cgs{}
it holds that
\begin{equation}
 \begin{aligned}    
&    \left\langle
      u_2 - v_2,
      - u_1^3 + v_1^3 
    \right\rangle_{L^2(D;\R)}
\\  & =
    \int_{D}
      (u_1^2(x) + u_1(x)v_1(x) + v_1^2(x))
      (v_1(x) - u_1(x))
      (u_2(x) - v_2(x))
    \, dx
\\ &
  \leq  
  \left(
    \int_D
      | u_1^2(x)+ u_1(x)v_1(x) + v_1^2(x)|^{3}
    \,dx
  \right)^{\frac13}
  \| v_1 - u_1\|_{L^6(D;\R)}
  \| u_2 - v_2 \|_{L^2(D;\R)}
\\ &
  \leq
  6
  \left(
    \| u_1 \|_{L^{6}(D;\R)}^2
    +
    \| v_1 \|_{L^{6}(D;\R)}^2
  \right)
  \left(
    \| v_1 - u_1\|_{L^6(D;\R)}^2
    +
    \| u_2 - v_2 \|_{L^2(D;\R)}^2
  \right)
\\ & 
  \leq
  6 \max\{ 1,\kappa(1)\}
  \left(
    \tfrac{
      (\kappa(r))^2|D|
    }{ r }
    +
    2r(U(u) + U(v))
  \right)
\\ & \quad \cdot
  \left(
    \| v_1 - u_1\|_{W^{1,2}(D;\R)}^2
    +
    \| u_2 - v_2 \|_{L^2(D;\R)}^2
  \right)
\\ &
  =
  6 \max\{ 1,\kappa(1)\}
  \left(
    \tfrac{
      (\kappa(r))^2|D|
    }{ r }
    +
    2r(U(u) + U(v))
  \right)
  V(u,v).
  \end{aligned}
\end{equation}
For all $q\in [1,\infty)$ define
$ 
  r_{q} 
  := 
  \left(
    24 \max\{ 1,\kappa(1) \}
    q T e^{\eta T}
  \right)^{-1}.
$
Taking $r = r_q$ 
in the above estimate and 
inserting this estimate into~\eqref{eq:waveGmusigma1}
and taking~\eqref{eq:wave_Gsigma} into account gives
that for all \sgc{}
$
    q\in (2,\infty)
$,
$ n\in \N $,
$u= (u_1,u_2), v = (v_1,v_2)\in H_n$ \cgs{}
it holds that
\begin{equation}
\begin{aligned}
&
  \max\left\{ 
   0,
    \frac{ 
    (\overline{\calG}_{\mu_n,\sigma_n}V)((u_1,u_2),(v_1,v_2))
    }{
      V((u_1,u_2),(v_1,v_2))
    }
  +
  \frac{
    \| 
      (\overline{G}_{\sigma_n}V)((u_1,u_2),(v_1,v_2))
    \|_{H}^2
  }{
    2 | V((u_1,u_2),(v_1,v_2)) |^2
  }
  \right\}
\\ & 
  \leq
  \tfrac{1}{2}
  +
  \tfrac{
    6 \max\{ 1,\kappa(1)\}
    (\kappa(r_q))^2|D|
  }{ r_q }
  +
  \frac{U(u) + U(v)}
  { 2 q T e^{\eta T}}
  +
  \varsigma^2
  . 
\end{aligned}  
\end{equation}
Moreover, it follows from~\eqref{eq:wave_GmusigmaU}
that for all \sgc{}$n\in \N$, $u = (u_1,u_2) \in H_n$, $t\in [0,T]$\cgs{}
it holds that
\begin{equation}
 \left(
  \calG_{\mu_n,\sigma_n} U
 \right)
 (u)
 +
 \tinv{2e^{\eta t}}\| \sigma_n(u)^*(U')(u) \|_{H}^2
 \leq 
 \tfrac{\eta^{2}}{2}(1+\| u_2\|_{L^2(D;\R)}^2)
 \leq  
 \tfrac{\eta^{2}}{2} + \eta^{2} U(u).
\end{equation} 
It thus follows from Theorem~\ref{thm:UV2} 
\sgc{}with\cgs{}
$k=1$, $\theta=2$, $\rho=\infty$, 
$ \alpha_{1,0,1} = \eta^{2}$,
$ \alpha_{0,0,1} = \alpha_{0,1,1} = \alpha_{1,1,1} = 0$,
$ \beta_{1,0,1} = \frac{\eta^{2} }{2}$,
$ \beta_{0,0,1} = \beta_{0,1,1} = \beta_{1,1,1} = 0$,
$ q_{1,0,1} = q $,
$ q_{0,0,1} = q_{0,1,1} = q_{1,1,1} = \infty$,
$ c_0 \equiv \frac{(p-2)\varsigma^2}{2}$,
$
  c_1 
  \equiv 
  \frac{1}{2}
  +
  \varsigma^2 
  + 
  \frac{	
    6 \max\{ \kappa(1), 1\}
    (\kappa(r_q))^2|D|
  }{ r_q}
$,
$ U_{1,0,1} = U$,
$ U_{0,0,1} = U_{0,1,1} = U_{1,1,1}
= \overline{U}_{0,1,1} = \overline{U}_{1,1,1} \equiv 0$
\sgc{}that\cgs{} 
for all
\sgc{}$n\in \N$,
$ x,y  \in H_n$,
$ r \in (2,\infty)$,
$p,q\in (r,\infty)$
with  
$ 
  \frac{ 1 }{ p } + \frac{ 1 }{ q }
  =
  \frac{ 1 }{ r } 
$\cgs{}
it holds that
\begin{equation}
\begin{aligned}
& \left\|
  \sup_{t\in [0,T]}
    \| X_t^{x,n} - X_t^{y,n} \|_{H}
 \right\|_{L^r(\Omega;\R)}
  \leq
  \frac{
    \| x - y \|_{H} 
  }
  { 
    \sqrt{ 1 - 2/p} 
  }
\\ & \quad \cdot
  \exp\!\left(
    \tfrac{T}{2}
    + \tfrac{p\varsigma^2T}{2}
    + \tfrac{ 6 T \max\{ \kappa(1), 1\} (\kappa(r_q))^2|D|}{ r_q }
    +  \tfrac{(e^{\eta^{2} T} + \eta^{2} T -1) }{2\eta^{2} Tq}
    + \tfrac{U(x) +U(y)}{2q}
 \right).
\end{aligned}
\end{equation}

\begin{remark}\label{rem:wave_Holder} 
Assume that $\mathcal{H} = L^2(D;\R)$, assume that there exist a Lipschitz continuous $b\colon \R\rightarrow \R$ and a positive trace class operator $Q\in L(\mathcal{H})$ such that $b(0)=0$ and $(B_2((u_1,u_2))v)(\xi) = b(u_1(\xi))\sqrt{Q}v$ for all $(u_1,u_2)\in H$, $v\in \mathcal{H}$, $\xi\in D$ and assume that there exists a non-empty and compact $K \sgc{}\subseteq\cgs{} \R^2$ such that
$
  D\supseteq
  \{
    x\in \R^2
    \colon
    \inf_{y\in K} \| x - y \| < T
  \}.
$
Under some additional regularity assumptions on $Q$~\sgc{}\cite[Proposition 4.1]{MilletMorien2001}\cgs{} implies for every 
$(u_0,v_0)\in H$ with $\operatorname{supp}(u_0)\cup \operatorname{supp}(v_0)\subseteq K$ and $u_0\in C^1(D,\R)$ that there exists a\sgc{} stochastic process 
$ u \colon [0,T] \times \Omega \rightarrow L^p(D;\R)$ with continuous sample paths\cgs{} which is a mild solution to 
the following stochastic wave equation:
\begin{equation}
\begin{aligned}
 \tfrac{\partial^2}{\partial t^2} u(t,\xi)
& =
 \Delta_s u(t,\xi)
 -
 u^3(t,\xi)
 +
 b\left(
  u(t,\xi)
 \right)
 \sqrt{Q} \dot{W}_t(\xi \big), \quad (t,\xi) \in [0,T]\times D,
\\
 u(0,\xi) 
 &
 = u_0(\xi);
\\
 \tfrac{\partial}{\partial t} u(0,\xi) 
 &
 = v_0(\xi), \quad \xi\in D.
\end{aligned}  
\end{equation}
The processes $(X^{n,(u_0,v_0)}_t)_{t\in [0,T]}$, $n\in \N$, are Galerkin approximations of $(u, \frac{\partial}{\partial t}u)$.  However, it is beyond 
the scope of this article to verify that 
$ \lim_{n\rightarrow \infty} X^{n,(u_0,v_0)} = (u, \frac{\partial}{\partial t}u) $ in some sense.
\end{remark}

\chapter*{Acknowledgments}

We thank Dirk Bl\"{o}mker for helpful 
remarks on the Cahn-Hilliard-Cook equation, 
Annika Lang for useful discussions
on the Kolmogorov-Chentsov continuity criterion,
Sara Mazzonetto for pointing out some typos in a preliminary version,
Siddhartha Mishra for references on the non-linear
wave equation,
Michael R\"{o}ckner
for useful comments on the extended generator,
Andrew Stuart for fruitful discussions on
the stochastic Lorenz equation and anonymous referees for 
useful comments.

This work has been partially funded by ETH Fellowship ``Convergence rates for approximations of stochastic (partial) differential equations with non-globally Lipschitz continuous coefficients'', 
by the Netherlands Organization for Scientific Research (NWO) under the VENI Vernieuwingsimpuls with project number 639.031.549, 
and by the Deutsche Forschungsgemeinschaft (DFG, German Research Foundation) under the research project ``Numerical approximation of stochastic differential equations with non-globally Lipschitz continuous coefficients''. 
Moreover, this work has been partially funded by the by the European Union (ERC, MONTECARLO, 101045811). 
The views and the opinions expressed in this work are however those of the authors only and do not necessarily reflect those of the European Union or the European Research Council (ERC). 
Neither the European Union nor the granting authority can be held responsible for them. In addition, the third author gratefully acknowledges the Cluster of Excellence EXC 2044-390685587, Mathematics M\"unster: Dynamics-Geometry-Structure funded by the Deutsche Forschungsgemeinschaft (DFG, German Research Foundation).
%


\bibliographystyle{amsalpha}
\bibliography{bibfile}

\end{document}